\pdfoutput=1
\RequirePackage{silence}
\documentclass[a4paper,11pt]{amsart}
\usepackage[hmarginratio={1:1},vmarginratio={1:1},lmargin=70.0pt,tmargin=70.0pt]{geometry}
\usepackage{lmodern}
\usepackage{multicol}

\usepackage{ifdraft}
\usepackage{microtype}
\usepackage{calligra}
\usepackage[T1]{fontenc}
\usepackage[ugly]{nicefrac}
\usepackage{mathtools}
\mathtoolsset{mathic}
\usepackage{multicol}
\usepackage{tabularx,tabu}
\usepackage{latexsym,exscale,amsfonts,amssymb,mathtools}
\usepackage{amsmath,amsthm,amsfonts,amssymb,amscd,textcomp,bbm}
\usepackage{stmaryrd}
\SetSymbolFont{stmry}{bold}{U}{stmry}{m}{n}
\usepackage[normalem]{ulem}
\usepackage{thmtools}

\makeatletter
\@ifundefined{newcounteralias}{}{%
  \renewcommand\thmt@autorefsetup{%
    \@xa\def\csname\thmt@envname autorefname\@xa\endcsname
    \@xa{\thmt@thmname}}%
}
\makeatother

\usepackage{enumitem}
\setlist[enumerate]{itemsep=0.15cm,label=\emph{\upshape(\alph*)}}
\usepackage[table]{xcolor}
\newcommand{\fix}[1]{#1}
\usepackage{mathrsfs}
\DeclareMathAlphabet{\mathscrbf}{OMS}{mdugm}{b}{n}
\DeclareFontEncoding{LS1}{}{}
\DeclareFontSubstitution{LS1}{stix}{m}{n}
\DeclareMathAlphabet{\mathpzc}{LS1}{stixscr}{m}{n}

\let\emph\relax
\DeclareTextFontCommand{\emph}{\bfseries\em}

\usepackage{caption} 
\definecolor{mygray}{gray}{0.6}
\definecolor{mygraydark}{gray}{0.4}
\definecolor{mygraylight}{gray}{0.8}

\definecolor{cherry}{RGB}{222,49,99}
\definecolor{cream}{RGB}{255,253,208}
\definecolor{corn}{RGB}{251,236,93}
\definecolor{citron}{RGB}{190,180,90}

\definecolor{spinach}{RGB}{46,139,87}
\definecolor{tomato}{RGB}{255,99,71}
\definecolor{pumpkin}{RGB}{224,180,80}

\definecolor{orchid}{RGB}{143,40,194}
\definecolor{lava}{RGB}{207,16,32}
\definecolor{mydarkblue}{RGB}{10,10,150}

\definecolor{myorange}{RGB}{225,127,0}
\definecolor{mygreen}{RGB}{0,225,0}
\definecolor{mypurple}{RGB}{128,0,128}
\definecolor{myred}{RGB}{255,0,0}
\definecolor{myblue}{RGB}{0,0,195}
\definecolor{myyellow}{RGB}{210,210,0}

\usepackage[all]{xy}
\SelectTips{cm}{}

\usepackage{tikz}
\tikzstyle{densely dotted}=[dash pattern=on \pgflinewidth off .5pt]
\tikzset{anchorbase/.style={baseline={([yshift=-0.5ex]current bounding box.center)}},
tinynodes/.style={font=\tiny, text height=0.25ex, text depth=0.05ex},
smallnodes/.style={font=\scriptsize, text height=0.75ex, text depth=0.15ex},
crossline/.style={preaction={draw=white,line width=5.0pt,-},preaction={draw=black,line width=0.9pt,-}},
usual/.style={line width=1.0,color=black},
mor/.style={line width=1.0,black,fill=mygray,fill opacity=0.35},
morl/.style={draw,rectangle,minimum height=0.5cm,minimum width=0.75cm,
line width=1.0,fill=mygray,fill opacity=0.35,path picture={\draw[solid,line width=5.0,black]
([xshift=0pt] current path bounding box.south west)--([xshift=0pt] current path bounding box.north west)
;}},
morr/.style={draw,rectangle,minimum height=0.5cm,minimum width=0.75cm,
line width=1.0,black,fill=mygray,fill opacity=0.35,path picture={
\draw[solid,line width=5.0,black]
(current path bounding box.south east)--(current path bounding box.north east);
}},
JW/.style={line width=1.0,color=black},
pQJWl/.style={draw,rectangle,minimum height=0.5cm,minimum width=0.75cm,
line width=1.0,color=black,fill=corn!60,path picture={
\draw[solid,line width=5.0,black]
(current path bounding box.south east)--(current path bounding box.north east);
}},
pJWl/.style={draw,rectangle,minimum height=0.5cm,minimum width=0.75cm,
line width=1.0,color=black,fill=orchid!70,path picture={\draw[solid,line width=5.0,black]
([xshift=0pt] current path bounding box.south east)--([xshift=0pt] current path bounding box.north east)
;}},
pQJW/.style={draw,rectangle,minimum height=0.5cm,minimum width=0.75cm,
line width=1.0,color=black,fill=corn!60,path picture={\draw[solid,line width=5.0,black]
([xshift=0pt] current path bounding box.south west)--([xshift=0pt] current path bounding box.north west)
;}},
pJW/.style={draw,rectangle,minimum height=0.5cm,minimum width=0.75cm,
line width=1.0,color=black,fill=orchid!70,path picture={\draw[solid,line width=5.0,black]
([xshift=0pt] current path bounding box.south west)--([xshift=0pt] current path bounding box.north west)
;}},
}
\usetikzlibrary{cd}
\usetikzlibrary{decorations}
\usetikzlibrary{decorations.markings}
\usetikzlibrary{decorations.pathreplacing}
\usetikzlibrary{decorations.pathmorphing}
\usetikzlibrary{arrows.meta,shapes,positioning,matrix,calc}
\usetikzlibrary{shapes.callouts}
\tikzstyle directed=[postaction={decorate,decoration={markings,
mark=at position #1 with {\arrow[line width=0.15mm, black]{>}}}}]
\tikzstyle odirected=[postaction={decorate,decoration={markings,
mark=at position #1 with {\arrow[line width=0.15mm, myorange]{>}}}}]
\tikzstyle smarked=[postaction={decorate,decoration={markings,
mark=at position #1 with {\fill[spinach] (0,0) circle (.065cm);}}}]
\tikzstyle tmarked=[postaction={decorate,decoration={markings,
mark=at position #1 with {\fill[tomato] (0,0) circle (.065cm);}}}]
\tikzstyle wmarked=[postaction={decorate,decoration={markings,
mark=at position #1 with {\fill[white] (0,0) circle (.1cm);}}}]

\allowdisplaybreaks

\newcommand{\tm}{\text{-}}

\newcommand{\ie}{\textsl{i.e.}}
\newcommand{\eg}{\textsl{e.g.}}
\newcommand{\cf}{\textsl{cf.}}

\newcommand{\aka}{\textsl{a.k.a.}}
\newcommand{\ver}{\textsl{verbatim}}
\newcommand{\vive}{\textsl{vice versa}}
\newcommand{\muta}{\textsl{mutatis mutandis}}

\renewcommand{\dots}{\text{...}}

\renewcommand{\geq}{\geqslant}
\renewcommand{\leq}{\leqslant}

\newcommand{\C}{\mathbb{C}}
\newcommand{\Z}{\mathbb{Z}}
\newcommand{\R}{\mathbb{R}}
\newcommand{\Q}{\mathbb{Q}}
\newcommand{\K}{\mathbb{K}}
\newcommand{\kk}{\mathbbm{k}}
\newcommand{\N}[1][]{\mathbb{N}_{#1}}

\newcommand{\xpar}{\mathbbm{x}}
\newcommand{\vpar}{\mathbbm{v}}
\newcommand{\qpar}{\mathbbm{q}}
\newcommand{\lpar}{\ell}
\newcommand{\ppar}{\mathsf{p}}
\newcommand{\plpar}{\ppar{\vee}\lpar}
\newcommand{\Zv}{\Z[\vpar^{\pm 1}]}
\newcommand{\Zvs}{\Z[\vpar^{\pm 1/2}]}
\newcommand{\Nv}[1][0]{\mathbb{N}_{#1}[\vpar^{\pm 1}]}
\newcommand{\Qv}{\Q(\vpar)}

\newcommand{\kkv}{\kk(\vpar)}
\newcommand{\kkvs}{\kk(\vpar^{1/2})}

\newcommand{\kklocal}{\mathbb{L}}

\newcommand{\mchar}[1][\kk,\qpar]{m\mathrm{char}(#1)}

\newcommand{\qnum}[2]{[#1]_{#2}}
\newcommand{\qfac}[2]{[#1]_{#2}!}

\newcommand{\qbin}[3]{{\textstyle\genfrac{[}{]}{0pt}{}{#1}{#2}_{#3}}}
\newcommand{\qbinn}[3]{\genfrac{[}{]}{0pt}{}{#1}{#2}_{#3}}

\newcommand{\placeholder}{{}_{-}}

\newcommand{\vcirc}{\circ}
\newcommand{\hcirc}{\otimes}
\newcommand{\munit}{\mathbbm{1}}

\newcommand{\setstuff}[1]{\mathrm{#1}}
\newcommand{\catstuff}[1]{\mathbf{#1}}
\newcommand{\functorstuff}[1]{\mathpzc{#1}}

\newcommand{\obstuff}[1]{\mathtt{#1}}
\newcommand{\morstuff}[1]{\mathrm{#1}}

\newcommand{\idmor}{\morstuff{id}}

\newcommand{\End}{\setstuff{End}}
\newcommand{\Hom}{\setstuff{Hom}}

\newcommand{\mainfunctor}{\functorstuff{F}}
\newcommand{\losp}{losp}

\newcommand{\lucas}{\epsilon}
\newcommand{\F}[1][\ppar]{\mathbb{F}_{#1}}
\newcommand{\ord}[1][{\ppar,\lpar}]{\nu_{#1}}

\newcommand{\eve}{\setstuff{Eve}}
\newcommand{\pbase}[2]{[#1]_{#2}}
\newcommand{\plbase}[1]{\pbase{#1}{\ppar,\lpar}}
\newcommand{\ancest}[1][v]{\setstuff{A}(#1)}
\newcommand{\supp}[1][v]{\nabla\setstuff{supp}(#1)}

\newcommand{\mother}[1][v]{\obstuff{m}_{#1}}
\newcommand{\motherr}[2]{\obstuff{m}^{#2}_{#1}}
\newcommand{\generation}[1][v]{\obstuff{gen}_{#1}}
\newcommand{\fancest}[2]{a_{#1,#2}}

\newcommand{\zigzag}[1][v]{\obstuff{Z}^{\kk,\qpar}}
\newcommand{\zigzagmod}{\catstuff{Proj}\tm\zigzag}

\newcommand{\SLtwo}{\mathrm{SL}_{2}}
\newcommand{\SLtwolie}{\mathfrak{sl}_{2}}
\newcommand{\tilt}[1][{\kk,\qpar}]{\catstuff{Tilt}^{#1}}
\newcommand{\allmod}[1][{\K,\qpar}]{\catstuff{fdMod}^{#1}}
\newcommand{\verlinde}[1][{\K,\qpar}]{\catstuff{Ver}^{#1}}

\newcommand{\tmod}{\obstuff{T}}
\newcommand{\wmod}{\Delta}
\newcommand{\dwmod}{\nabla}
\newcommand{\lmod}{\obstuff{L}}

\newcommand{\ch}[1]{\chi_{#1}}

\newcommand{\wmult}[2]{\big(\tmod(#1-1):\wmod(#2-1)\big)}
\newcommand{\wdmult}[2]{\big(\tmod(#1-1):\dwmod(#2-1)\big)}
\newcommand{\lmult}[2]{\big[\wmod(#1-1):\lmod(#2-1)\big]}
\newcommand{\ldmult}[2]{\big[\dwmod(#1-1):\lmod(#2-1)\big]}

\newcommand{\lsupp}[1][v]{\lmod\setstuff{supp}(#1)}

\newcommand{\MTr}[1]{\setstuff{Tr}^{#1}}
\newcommand{\MtTr}[1]{\tilde{\setstuff{Tr}}^{#1}}
\newcommand{\mTr}[2]{\setstuff{Tr}^{#1}_{#2}}
\newcommand{\kdim}[2]{\setstuff{dim}^{(#1)}_{#2}}

\newcommand{\digscalar}[2]{d_{#1}^{\,#2}} 
\newcommand{\twebs}[1][v]{\digscalar{#1,\emptyset}{a\ppar^{(k)},b\ppar^{(t)}}}
\newcommand{\twebss}[1][v]{\digscalar{#1,S}{a\ppar^{(k)},b\ppar^{(t)}}}

\newcommand{\digscalartild}[2]{\tilde{d}_{#1}^{\,#2}} 
\newcommand{\twebstild}[1][v]{\digscalartild{#1,\emptyset}{a\ppar^{(k)},b\ppar^{(t)}}}
\newcommand{\twebsstild}[1][v]{\digscalartild{#1,S}{a\ppar^{(k)},b\ppar^{(t)}}}

\newcommand{\fusscalar}[2]{x^{#1}_{#2}} 

\newcommand{\TL}[1][{\K,\qpar}]{\catstuff{TL}^{#1}}
\newcommand{\tlfunctor}[1][{\K,\qpar}]{\functorstuff{D}^{#1}}

\newcommand{\sbas}[1][v-1,w-1]{\setstuff{B}^{int}_{#1}}
\newcommand{\awbas}[1][v-1,w-1]{\setstuff{B}^{\vpar}_{#1}}
\newcommand{\pbas}[1][v-1,w-1]{\setstuff{B}^{\qpar}_{#1}}

\newcommand{\dist}{\mathsf{d}}

\newcommand{\idtl}[1][v{-}1]{\idmor_{#1}}

\newcommand{\fliph}[1][\phantom{a}]{{#1}^{\updownarrow}}
\newcommand{\flipv}[1][\phantom{a}]{{#1}^{\leftrightarrow}}

\newcommand{\qjw}[1][v{-}1]{\tilde{\obstuff{e}}_{#1}}
\newcommand{\pqjw}[1][v{-}1]{\overline{\obstuff{e}}_{#1}}
\newcommand{\pjw}[1][v{-}1]{\obstuff{E}_{#1}}

\newcommand{\qjwm}[1][v{-}1]{\mathtt{#1}}
\newcommand{\pqjwm}[1][v{-}1]{\mathtt{#1}}
\newcommand{\pjwm}[1][v{-}1]{\mathtt{#1}}

\newcommand{\Up}[1]{\mathrm{U}_{#1}}

\newcommand{\upo}[1]{\tilde{\mathrm{u}}_{#1}}
\newcommand{\Down}[1]{\mathrm{D}_{#1}}

\newcommand{\downo}[1]{\tilde{\mathrm{d}}_{#1}}

\newcommand{\loopdowngen}[2]{\tilde{\mathrm{L}}^{#1}_{#2}}

\newcommand{\loopdown}[2]{\mathrm{L}^{#1}_{#2}}

\newcommand{\funcf}{\obstuff{f}}
\newcommand{\funcg}{\obstuff{g}}

\newcommand{\funcF}[1][S]{\obstuff{f}_{#1}}
\newcommand{\funcG}[1][S]{\obstuff{g}_{#1}}
\newcommand{\funcH}[1][S]{\obstuff{h}_{#1}}
\newcommand{\hull}[1][S]{\overline{#1}}

\newcommand{\xmor}[2]{\morstuff{X}^{#2}_{#1}}
\newcommand{\dmor}[3]{{}_{#3}^{\phantom{v}}\morstuff{D}{}^{#1}_{#2}}

\newcommand{\fusidem}[2]{\morstuff{B}^{#2}_{#1}}
\newcommand{\fusidema}[2]{\morstuff{A}^{#2}_{#1}} 
\newcommand{\fusidemb}[2]{(\fusidema{#1}{#2})^{\fix{\prime}}} 
\newcommand{\fuscaseidem}[2]{\morstuff{P}^{#2}_{#1}}
\newcommand{\taill}[1]{\mathrm{tl}(#1)}
\newcommand{\fus}{\mathbf{F}}
\newcommand{\trp}{\mathbf{PT}}

\newcommand{\tru}[5]{
\draw [JW] (-#4-#2/2-#2/4,#5) to (-#4-#1,#5) to (-#4-#1,#5+#2) to (-#4,#5+#2) to (-#4-#2/2-#2/4,#5);
\node at (-#4-#1/2-#1/16,#5+#2/2-0.05) {#3};
}

\newcommand{\trd}[5]{
\draw [JW] (-#4-#2/2-#2/4,#5) to (-#4-#1,#5) to (-#4-#1,#5-#2) to (-#4,#5-#2) to (-#4-#2/2-#2/4,#5);
\node at (-#4-#1/2-#1/16,#5-#2/2-0.05) {#3};
}

\newcommand{\tr}[5]{
\draw [JW] (-#4-#1,#5) to (-#4-#1,#5-#2) 
to (-#4,#5-#2) to (-#4-#2/2-#2/4,#5) to (-#4,#5+#2) to (-#4-#1,#5+#2) to (-#4-#1,#5);
\node at (-#4-#1/2-#1/16,#5-0.05) {#3};
}

\newcommand{\TRU}[1]{
\begin{tikzpicture}[anchorbase,tinynodes]
\tru{1.1}{0.6}{#1}{0}{0}
\end{tikzpicture}
}

\newcommand{\TRD}[1]{
\begin{tikzpicture}[anchorbase,tinynodes]
\trd{1.1}{0.6}{#1}{0}{0}
\end{tikzpicture}
}
\newcommand{\TRUD}[2]{
\begin{tikzpicture}[anchorbase,tinynodes]
\tru{1.1}{0.3}{#1}{0}{0}
\trd{1.1}{0.3}{#2}{0}{0}
\end{tikzpicture}
}

\newcommand{\TR}[1]{
\begin{tikzpicture}[anchorbase,tinynodes]
\tr{1.1}{0.3}{#1}{0}{0}
\end{tikzpicture}
}

\newcommand{\ptru}[5]{
\draw [pJW] (-#4-#2/2-#2/4,#5) to (-#4-#1,#5) to (-#4-#1,#5+#2) to (-#4,#5+#2) to (-#4-#2/2-#2/4,#5);
\node at (-#4-#1/2-#1/16,#5+#2/2-0.05) {#3};
}

\newcommand{\ptrd}[5]{
\draw [pJW] (-#4-#2/2-#2/4,#5) to (-#4-#1,#5) to (-#4-#1,#5-#2) to (-#4,#5-#2) to (-#4-#2/2-#2/4,#5);
\node at (-#4-#1/2-#1/16,#5-#2/2-0.05) {#3};
}

\newcommand{\pTRU}[1]{
\begin{tikzpicture}[anchorbase,tinynodes]
\ptru{1.1}{0.6}{#1}{0}{0}
\end{tikzpicture}
}

\newcommand{\pTRD}[1]{
\begin{tikzpicture}[anchorbase,tinynodes]
\ptrd{1.1}{0.6}{#1}{0}{0}
\end{tikzpicture}
}

\newtheorem{theoremm}{Theorem}[section]

\declaretheoremstyle[
headfont=\bfseries, 
notebraces={[}{]},
bodyfont=\normalfont\itshape,
headpunct={},
postheadspace=1em,
spacebelow=10pt,
spaceabove=10pt, 
]{ourtheo}

\declaretheoremstyle[
headfont=\normalfont\bfseries,
notefont=\mdseries,
notebraces={(}{)},
bodyfont=\normalfont\slshape,
headpunct={},
postheadspace=1em,
spacebelow=10pt,
spaceabove=10pt, 
]{ourdef}

\declaretheorem[style=ourtheo,name=Theorem,numberlike=theoremm]{theorem}
\declaretheorem[style=ourtheo,name=Lemma,numberlike=theoremm]{lemma}
\declaretheorem[style=ourtheo,name=Proposition,numberlike=theoremm]{proposition}

\declaretheorem[style=ourtheo,name=Corollary,qed=$\square$,numberlike=theoremm]{corollary}

\declaretheorem[style=ourdef,name=Definition,numberlike=theorem]{definition}
\declaretheorem[style=ourdef,name=Example,numberlike=theorem]{example}
\declaretheorem[style=ourdef,name=Remark,numberlike=theorem]{remark}
\declaretheorem[style=ourdef,name=Convention,numberlike=theorem]{convention}

\allowdisplaybreaks

\def\changed#1{{#1}}
\def\ochanged#1{{#1}}

\numberwithin{equation}{section}
\usepackage[hypertexnames=false]{hyperref}
\usepackage{youngtab}
\usepackage{ytableau}
\usepackage{cleveref,bookmark}
\hypersetup{
pdftoolbar=true,        
pdfmenubar=true,        
pdffitwindow=false,     
pdfstartview={FitH},    
pdftitle={\texorpdfstring{$\mathrm{SL}_{2}$}{SL2} tilting modules in the mixed case},    
pdfauthor={Louise Sutton, Daniel Tubbenhauer, Paul Wedrich, and Jieru Zhu},     
pdfsubject={},   
pdfcreator={Louise Sutton, Daniel Tubbenhauer, Paul Wedrich, and Jieru Zhu},   
pdfproducer={Louise Sutton, Daniel Tubbenhauer, Paul Wedrich, and Jieru Zhu}, 
pdfkeywords={}, 
pdfnewwindow=true,      
colorlinks=true,       
linkcolor=mydarkblue,          
citecolor=teal,        
filecolor=magenta,      
urlcolor=orchid,          
linkbordercolor=lava,
citebordercolor=teal,
urlbordercolor=orchid,  
linktocpage=true
}
\renewcommand\thesubsection{\thesection\Alph{subsection}}

\renewcommand{\theequation}{\thesection-\arabic{equation}}
\let\fullref\autoref
\def\makeautorefname#1#2{\expandafter\def\csname#1autorefname\endcsname{#2}}
\makeautorefname{equation}{equation}%
\makeautorefname{footnote}{footnote}%
\makeautorefname{item}{item}%
\makeautorefname{figure}{Figure}%
\makeautorefname{table}{Table}%
\makeautorefname{part}{Part}%
\makeautorefname{appendix}{Appendix}%
\makeautorefname{chapter}{Chapter}%
\makeautorefname{section}{Section}%
\makeautorefname{subsection}{Section}%
\makeautorefname{subsubsection}{Section}%
\makeautorefname{paragraph}{Paragraph}%
\makeautorefname{subparagraph}{Paragraph}%
\makeautorefname{theorem}{Theorem}%
\makeautorefname{theo}{Theorem}%
\makeautorefname{thm}{Theorem}%
\makeautorefname{addendum}{Addendum}%
\makeautorefname{addend}{Addendum}%
\makeautorefname{add}{Addendum}%
\makeautorefname{maintheorem}{Main theorem}%
\makeautorefname{mainthm}{Main theorem}%
\makeautorefname{corollary}{Corollary}%
\makeautorefname{corol}{Corollary}%
\makeautorefname{coro}{Corollary}%
\makeautorefname{cor}{Corollary}%
\makeautorefname{lemma}{Lemma}%
\makeautorefname{lemm}{Lemma}%
\makeautorefname{lem}{Lemma}%
\makeautorefname{sublemma}{Sublemma}%
\makeautorefname{sublem}{Sublemma}%
\makeautorefname{subl}{Sublemma}%
\makeautorefname{proposition}{Proposition}%
\makeautorefname{proposit}{Proposition}%
\makeautorefname{propos}{Proposition}%
\makeautorefname{propo}{Proposition}%
\makeautorefname{prop}{Proposition}%
\makeautorefname{property}{Property}
\makeautorefname{proper}{Property}
\makeautorefname{scholium}{Scholium}%
\makeautorefname{step}{Step}%
\makeautorefname{conjecture}{Conjecture}%
\makeautorefname{conject}{Conjecture}%
\makeautorefname{conj}{Conjecture}%
\makeautorefname{question}{Question}
\makeautorefname{questn}{Question}
\makeautorefname{quest}{Question}
\makeautorefname{ques}{Question}
\makeautorefname{qn}{Question}
\makeautorefname{definition}{Definition}%
\makeautorefname{defin}{Definition}%
\makeautorefname{defi}{Definition}%
\makeautorefname{def}{Definition}%
\makeautorefname{dfn}{Definition}%
\makeautorefname{notation}{Notation}
\makeautorefname{nota}{Notation}
\makeautorefname{notn}{Notation}
\makeautorefname{remark}{Remark}%
\makeautorefname{rema}{Remark}%
\makeautorefname{rem}{Remark}%
\makeautorefname{rmk}{Remark}%
\makeautorefname{rk}{Remark}%
\makeautorefname{remarks}{Remarks}%
\makeautorefname{rems}{Remarks}%
\makeautorefname{rmks}{Remarks}%
\makeautorefname{rks}{Remarks}%
\makeautorefname{example}{Example}%
\makeautorefname{examp}{Example}%
\makeautorefname{exmp}{Example}%
\makeautorefname{exam}{Example}%
\makeautorefname{exa}{Example}%
\makeautorefname{algorithm}{Algorithm}%
\makeautorefname{algo}{Algorithm}%
\makeautorefname{alg}{Algorithm}%
\makeautorefname{axiom}{Axiom}%
\makeautorefname{axi}{Axiom}%
\makeautorefname{ax}{Axiom}%
\makeautorefname{case}{Case}%
\makeautorefname{claim}{Claim}%
\makeautorefname{clm}{Claim}%
\makeautorefname{assumption}{Assumption}%
\makeautorefname{assumpt}{Assumption}%
\makeautorefname{conclusion}{Conclusion}%
\makeautorefname{concl}{Conclusion}%
\makeautorefname{conc}{Conclusion}%
\makeautorefname{condition}{Condition}%
\makeautorefname{condit}{Condition}%
\makeautorefname{cond}{Condition}%
\makeautorefname{construction}{Construction}%
\makeautorefname{construct}{Construction}%
\makeautorefname{const}{Construction}%
\makeautorefname{cons}{Construction}%
\makeautorefname{criterion}{Criterion}%
\makeautorefname{criter}{Criterion}%
\makeautorefname{crit}{Criterion}%
\makeautorefname{exercise}{Exercise}%
\makeautorefname{exer}{Exercise}%
\makeautorefname{exe}{Exercise}%
\makeautorefname{problem}{Problem}%
\makeautorefname{problm}{Problem}%
\makeautorefname{probm}{Problem}%
\makeautorefname{prob}{Problem}%
\makeautorefname{solution}{Solution}%
\makeautorefname{soln}{Solution}%
\makeautorefname{sol}{Solution}%
\makeautorefname{summary}{Summary}%
\makeautorefname{summ}{Summary}%
\makeautorefname{sum}{Summary}%
\makeautorefname{operation}{Operation}%
\makeautorefname{oper}{Operation}%
\makeautorefname{observation}{Observation}%
\makeautorefname{observn}{Observation}%
\makeautorefname{obser}{Observation}%
\makeautorefname{obs}{Observation}%
\makeautorefname{ob}{Observation}%
\makeautorefname{convention}{Convention}%
\makeautorefname{convent}{Convention}%
\makeautorefname{conv}{Convention}%
\makeautorefname{cvn}{Convention}%
\makeautorefname{warning}{Warning}%
\makeautorefname{warn}{Warning}%
\makeautorefname{note}{Note}%
\makeautorefname{fact}{Fact}%
\makeautorefname{hope}{Expectation}%
\makeautorefname{hope2}{Conjecture}%

\newcommand{\nnfootnote}[1]{%
\begin{NoHyper}
\renewcommand\thefootnote{}\footnote{#1}%
\addtocounter{footnote}{-1}%
\end{NoHyper}
}

\begin{document}
\title[\texorpdfstring{$\mathrm{SL}_{2}$}{SL2} tilting modules in the mixed case]
{\texorpdfstring{$\mathrm{SL}_{2}$}{SL2} tilting modules in the mixed case}
\author[Louise Sutton, Daniel Tubbenhauer, Paul Wedrich, and Jieru Zhu]{Louise Sutton, Daniel Tubbenhauer, Paul Wedrich, and Jieru Zhu}

\address{L.S.: Okinawa Institute of Science and Technology, 1919-1 Tancha, Onna-son, Kunigami-gun Okinawa, Japan 904-0495,
\href{https://louisesuttonblog.wordpress.com/}{louisesuttonblog.wordpress.com}}
\email{louise.sutton@oist.jp}

\address{D.T.: The University of Sydney,
School of Mathematics and Statistics,
F07 - Carslaw Building,
Office Carslaw 827,
NSW 2006, Australia, \href{www.dtubbenhauer.com}{www.dtubbenhauer.com}}
\email{daniel.tubbenhauer@sydney.edu.au}

\address{P.W.: Universit{\"a}t Hamburg, Department of Mathematics, Bundesstraße 55, 20146 Hamburg, Germany
\href{http://paul.wedrich.at}{paul.wedrich.at}}
\email{p.wedrich@gmail.com}

\address{J.Z.: Okinawa Institute of Science and Technology, 1919-1 Tancha, Onna-son, Kunigami-gun
Okinawa, Japan 904-0495
\href{http://cyclotomic.squarespace.com}{cyclotomic.squarespace.com}}
\email{jieruzhu699@gmail.com}

\nnfootnote{\textit{Mathematics Subject Classification 2020.} Primary: 18M15, 20G42; Secondary: 17B10, 20C20, 20G05.}
\nnfootnote{\textit{Keywords.} Tilting modules in the mixed case, 
diagrammatic algebra, Temperley--Lieb algebras and categories, fusion rules, braided structures.}

\begin{abstract}
Using the non-semisimple Temperley--Lieb calculus, we study the additive and
monoidal structure of the category of tilting modules for $\SLtwo$ in the mixed
case. This simultaneously generalizes the semisimple situation, the case of the complex quantum
group at a root of unity, and the algebraic group case in positive
characteristic.
We describe character formulas and give a presentation of the category of
tilting modules as an additive category via a quiver with relations. Turning to
the monoidal structure, we describe fusion rules and obtain an explicit
recursive description of the appropriate analog of Jones--Wenzl projectors. 
\end{abstract}

\maketitle

\tableofcontents

\renewcommand{\theequation}{\thesection-\arabic{equation}}

\addtocontents{toc}{\protect\setcounter{tocdepth}{1}}

\section{Introduction}\label{section:intro}

Let $\kk$ be a field of characteristic $\ppar$, fix a non-zero element
$\qpar\in\kk^{*}$, and let $\K$ be an algebraically closed field containing
$\kk$. Tilting modules for $\SLtwo$, the reductive group $\SLtwo(\K)$ if
$\qpar=\pm 1$ or Lusztig's divided power quantum group for $\SLtwolie$ if
$\qpar\neq\pm 1$, are among the most well-studied objects in representation
theory. In this paper, we use diagrammatic methods to study monoidal categories
of tilting modules in the mixed case, {\ie} for arbitrary $(\kk,\qpar)$. As a
modern day perspective, the mixed case can be thought of as a simultaneous
generalization of the case of the complex quantum group (at a root of unity),
the case of the algebraic group in positive characteristic, as well as the
classical semisimple situation.

Tilting modules form a monoidal category, so one can ask questions concerning
objects, morphisms, and how these behave under the tensor product. Concentrating
on objects and their characters is the classical approach in representation
theory. Recently, the focus has shifted to understanding morphisms between
tilting modules, especially from a monoidal perspective, which has been driven
by work from quantum topology and categorification. A more thorough
understanding of the associated diagrammatic and combinatorial model that
underpins the behavior of these tilting modules, known as the Temperley--Lieb
category, was a key ingredient in recent progress.

\makeautorefname{section}{Sections}

In this paper, we let $\tilt$ for arbitrary $(\kk,\qpar)$ denote the monoidal
category obtained by idempotent completion from the Temperley--Lieb category
$\TL[{\kk,\qpar}]$, see \fullref{rem:abuse}. We study $\tilt$ with a focus on
the behavior of objects and morphisms with respect to its monoidal structure, a
natural progression of previous work \cite{TuWe-quiver} and \cite{TuWe-center}.
The main results of this paper are contained within
\changed{\fullref{section:TL-additive} and \ref{section:fusion}} and can be summarized as
follows.

\makeautorefname{section}{Section}

In \fullref{subsection:projectors} we define \emph{mixed JW projectors}
$\pjw[v{-}1]$ in $\TL[{\kk,\qpar}]$ for $v\in\N$ and show that they correspond
to indecomposable tilting modules $\tmod(v-1)$ of highest weight $v-1$. These
idempotents have been constructed independently in \cite{MaSp-lp-jw} and they
are a simultaneous generalization of the classical Jones--Wenzl (short: JW) projectors
\cite{Jo-index-subfactors}, \cite{We-projectors}, the projectors of Goodman--Wenzl \cite{GoWe-tl-at-roots-of-1},
and the $\ppar$Jones--Wenzl projectors of Burrull--Libedinsky--Sentinelli
\cite{BuLiSe-tl-char-p}.

In \fullref{subsection:quiver} we study morphisms between mixed JW projectors in
$\TL[{\kk,\qpar}]$ and obtain a presentation of $\tilt$ as an additive category
by generators and relations. Specifically, we exhibit $\tilt$ as the category of
projective modules for the path algebra of a quiver with relations explicitly
described in \fullref{theorem:main-tl-section}, which can be interpreted as the
(semi-infinite) Ringel dual of $\SLtwo$. 

In \fullref{section:fusion} we turn to the monoidal structure and study fusion
rules for $\tilt$ and their categorified analogs in $\TL[{\kk,\qpar}]$.
Classically, fusion rules express the structure constants for the representation
ring, {\ie} the decomposition multiplicities of tensor products of modules, such
as $\tmod(v)\hcirc\tmod(w)$, into indecomposable modules. On the categorified
level one is interested in explicitly describing the projection and inclusion
maps realizing such decompositions. In the Temperley--Lieb context this means
decomposing the tensor products $\pjw[v]\otimes \pjw[w]$ of projectors into idempotents that
project onto the indecomposable summands predicted by the fusion rule. 

A famous example is the recursion for the classical Jones--Wenzl projectors 
\begin{gather}\label{eq:jw-recursion}
\begin{tikzpicture}[anchorbase,scale=0.25,tinynodes]
\draw[JW] (-2,-1) rectangle (2,1);
\node at (0,-0.2) {$\pjwm[v{-}1]$};
\draw[usual] (2.5,-1) to (2.5,1);
\end{tikzpicture}
=
\begin{tikzpicture}[anchorbase,scale=0.25,tinynodes]
\draw[JW] (-2,-1) rectangle (2,1);
\node at (0,-0.2) {$\pjwm[v]$};
\end{tikzpicture}
+
\frac{1}{\funcG[*]}\cdot
\begin{tikzpicture}[anchorbase,scale=0.4,tinynodes]
\draw[JW] (-0.5,-1.5) rectangle (2.5,-2.5);
\node at (1,-2.2) {$\pjwm[v{-}1]$};
\draw[JW] (-0.5,-4.5) rectangle (2.5,-5.5);
\node at (1,-5.2) {$\pjwm[v{-}1]$};
\draw[JW] (-0.5,-3.25) rectangle (1.75,-3.75);
\node at (0.625,-3.6) {\scalebox{0.8}{$\pjwm[v{-}2]$}};
\draw[usual] (2,-2.5) to[out=270,in=180] (2.5,-3) to[out=0,in=270] 
(3,-2.5) to (3,-1.5);
\draw[usual] (2,-4.5) to[out=90,in=180] (2.5,-4) to[out=0,in=90] 
(3,-4.5) to (3,-5.5);
\draw[usual] (0,-2.5) to (0,-3.25);
\draw[usual] (0,-3.75) to (0,-4.5);
\end{tikzpicture}
,\quad
\frac{1}{\funcG[*]}=-\frac{\qnum{v-1}{\qpar}}{\qnum{v}{\qpar}}
,
\end{gather}
which describes the decomposition $\tmod(v-1)\hcirc\tmod(1)
\cong\tmod(v)\oplus\tmod(v-2)$ whenever all involved tilting modules
are simple. In fact, the Jones--Wenzl recursion \eqref{eq:jw-recursion} is often
taken as (part of) the definition of the Jones--Wenzl projectors.    

In \fullref{theorem:timest1} we establish decompositions analogous to
\eqref{eq:jw-recursion} in the mixed setting of $\TL[{\kk,\qpar}]$. These
provide a recursive description of the mixed JW projectors, which appear to be
new in this generality, even new when specialized to the positive characteristic
or complex quantum group cases, \cf~ \cite{BlDeReMu-dia-small-qgroup}. As an
example, we show an instance going beyond \eqref{eq:jw-recursion}, which
describes a decomposition
$\tmod(v-1)\hcirc\tmod(1)\cong\tmod(v)\oplus\tmod(v-2)$ with summands that need
not be simple:
\begin{gather}\label{eq:intro-fusion}
\begin{tikzpicture}[anchorbase,scale=0.25,tinynodes]
\draw[pJW] (-2,-1) rectangle (2,1);
\node at (0,-0.2) {$\pjwm[v{-}1]$};
\draw[usual] (2.5,-1) to (2.5,1);
\end{tikzpicture}
=
\begin{tikzpicture}[anchorbase,scale=0.25,tinynodes]
\draw[pJW] (-2,-1) rectangle (2,1);
\node at (0,-0.2) {$\pjwm[v]$};
\end{tikzpicture}
+
\left(
\frac{1}{\funcG[*]}\cdot
\begin{tikzpicture}[anchorbase,scale=0.4,tinynodes]
\draw[pJW] (-0.5,-1.5) rectangle (2.5,-2.5);
\node at (1,-2.2) {$\pjwm[v{-}1]$};
\draw[pJW] (-0.5,-4.5) rectangle (2.5,-5.5);
\node at (1,-5.2) {$\pjwm[v{-}1]$};
\draw[pJW] (-0.5,-3.25) rectangle (1.75,-3.75);
\node at (0.625,-3.6) {\scalebox{0.8}{$\pjwm[v{-}2]$}};
\draw[usual] (2,-2.5) to[out=270,in=180] (2.5,-3) to[out=0,in=270] 
(3,-2.5) to (3,-1.5);
\draw[usual] (2,-4.5) to[out=90,in=180] (2.5,-4) to[out=0,in=90] 
(3,-4.5) to (3,-5.5);
\draw[usual] (0,-2.5) to (0,-3.25);
\draw[usual] (0,-3.75) to (0,-4.5);
\end{tikzpicture}  
-
\frac{\funcF[*]}{\funcG[*]}\cdot
\begin{tikzpicture}[anchorbase,scale=0.4,tinynodes]
\draw[pJW] (-0.5,-1.5) rectangle (2.5,-2.5);
\node at (1,-2.2) {$\pjwm[v{-}1]$};
\draw[pJW] (-0.5,-4.5) rectangle (2.5,-5.5);
\node at (1,-5.2) {$\pjwm[v{-}1]$};
\draw[pJW] (-0.5,-2.5) rectangle (1.75,-3);
\node at (0.625,-2.85) {\scalebox{0.8}{$\pjwm[v{-}2]$}};
\draw[pJW] (-0.5,-4.5) rectangle (1.75,-4);
\node at (0.625,-4.35) {\scalebox{0.8}{$\pjwm[v{-}2]$}};
\draw[usual] (2,-2.5) to[out=270,in=180] (2.5,-3) to[out=0,in=270] 
(3,-2.5) to (3,-1.5);
\draw[usual] (2,-4.5) to[out=90,in=180] (2.5,-4) to[out=0,in=90] 
(3,-4.5) to (3,-5.5);
\draw[usual] (0,-3) to (0,-4);
\draw[usual] (0.5,-4) to[out=90,in=180] (1,-3.6) to[out=0,in=90] 
(1.5,-4);
\draw[usual] (0.5,-3) to[out=270,in=180] (1,-3.4) to[out=0,in=270] 
(1.5,-3);
\end{tikzpicture}\,
\right).
\end{gather}
(Here and throughout the paper we use colored boxes to encode mixed JW
projectors corresponding to tilting modules that need not be simple.) The middle
part of the rightmost diagram in \eqref{eq:intro-fusion} corresponds to a
nilpotent endomorphism of $\tmod(v-2)$. In particular, if $\tmod(v-2)$ is
simple, then the rightmost diagram is zero and we recover
\eqref{eq:jw-recursion}. In general, however, the decompositions provided by
\fullref{theorem:timest1} are more complex than suggested by the example
\eqref{eq:intro-fusion}. In particular, arbitrary many summands can appear, with
multiplicities up to two.


\subsection*{A bit of historical background and other works}

Tilting modules for $\SLtwo$ have 
played a crucial role 
in representation theory and 
low-dimensional topology, even before their 
introduction by Donkin 
\cite{Do-tilting-alg-groups} 
and Ringel \cite{Ri-good-filtrations}.
Let us recall parts of this story.

In the semisimple case, $\tilt$ is well-understood on the level of objects and
morphisms: The characters are given by Weyl's character formula and the fusion
rules by the Clebsch--Gordan rule. On the morphism level, $\tilt$ was given a
diagrammatic presentation early on by Rumer--Teller--Weyl
\cite{RuTeWe-cup-diagrams} using what is nowadays called the Temperley--Lieb
algebra or category $\TL[{\kk,\qpar}]$. This diagrammatic presentation, in its
quantum version, lies at the heart of constructions and calculations for the
Jones-type invariants of links and $3$-manifolds via Jones--Wenzl projectors and recoupling theory, see {\eg}
\cite{KaLi-TL-recoupling}.

In the complex quantum group case, many of our results have previously appeared
in the literature. The fusion rules on the object level 
in this case are certainly well-known since the end of the
1980s, but a bit hard to track down, see however \cite{Do-q-schur} for a
slightly later reference. The category $\TL[{\kk,\qpar}]$ plays a major role as
it provides the diagrammatic and integral model of $\tilt$. (While we don't know
an explicit exposition, this can be deduced from \cite{DuPaSc-schur-weyl}.) The
appropriate analog of JW projectors in this case were defined by Goodman--Wenzl
\cite{GoWe-tl-at-roots-of-1}, the Ringel dual quiver was computed in
\cite{AnTu-tilting}, and (parts of) recoupling theory was developed under the
umbrella of non-semisimple $3$-manifold invariants, see {\eg}
\cite{BlDeReMu-dia-small-qgroup} or \cite{DGGPR19}.

Historically speaking, the 
characteristic $\ppar$ case came long before the 
complex quantum group case, for example, 
due to its relationship to 
projective modules of the finite 
group $\SLtwo(\F[\ppar^{k}])$.
On the level 
of objects, the characteristic $\ppar$ 
case is intensively studied throughout the literature, 
see {\eg} \cite{CaCl-submodule-weyl-a1}, 
\cite{AnJoLa-sl2-projectives}, 
\cite{Do-tilting-alg-groups},
\cite{ErHe-ringel-schur}, 
\cite{ErHe-ringel-schur-symmetric-group} 
and \cite{DoHe-char-p-sl2}. Of particular importance, are Steinberg's and
Donkin's tensor product formulas, which give the characters of simple and
tilting modules. However, not much appears to be known about fusion of objects
beyond special cases, {\eg} coming from studying Verlinde quotients, see for
example \cite{An-simple-tl} or \cite{BeEtOs-pverlinde}, or the situation of the
finite group $\SLtwo(\F[\ppar^{k}])$, see for example
\cite{Cr-tensor-simple-modules}. On the morphism level, the use of
$\TL[{\kk,\qpar}]$ is crucial, specifically in the work of
Burrull--Libedinsky--Sentinelli \cite{BuLiSe-tl-char-p}, which introduced the
$\ppar$JW projectors, a main ingredient to find the quiver with relations and
the center of $\tilt$, see \cite{TuWe-quiver} and \cite{TuWe-center}. When it
comes to \changed{fusion for morphisms}, our results are new.

In the (strictly) mixed case, most of the results in the present paper are new.
See, however, {\eg} \cite{Do-q-schur} and \cite{An-simple-tl} for character and
fusion formulas, and \cite{MaSp-lp-jw} for their (independent) construction of
the mixed JW projectors.

\subsection*{Further directions}

A potential application concerns quotients of $\tilt$, especially in the
characteristic $\ppar$ and the strictly mixed cases, when the category $\tilt$
has infinitely many $\hcirc$-ideals. This is in stark contrast to the semisimple
and the complex quantum group case, where one has no or only one non-trivial
$\hcirc$-ideal. The strictly mixed cases turn out to be very appealing in two
directions. First, in generalizing {\eg} the results of \cite{BeEtOs-pverlinde}
to the mixed situation, where $\tilt$ may have universal properties that are
similar to those studied in the characteristic $\ppar$ case. \changed{Second, it is easy to see 
that $\tilt[\kk,\qpar^{1/2}]$ has a non-degenerate braiding in the mixed case.}
This is particularly interesting from the viewpoint of non-semisimple $3d$ TQFT,
where one could try to apply the strategy of modified traces from {\eg}
\cite{GeKuPaMi-gen-traces-modified-dimensions} and
\cite{GeKuPaMi-ambidextrous-objects}, to obtain new non-semisimple $3d$ TQFTs.

The fusion rules for $\tilt$ are also of importance in physics (from which its
name arose), although the focus has been on the semisimple and complex quantum
group cases. In fact, this was one motivation to develop the Temperley--Lieb
calculus \cite{KaLi-TL-recoupling} and its variations, which appear under
different names in the physics- and mathematics-oriented literature. For
example, idempotent truncations by tensor products of classical JW projectors
are studied under the names valenced Temperley--Lieb algebras in
\cite{FlPe-jwalgebras1} and symmetric webs in \cite{RoTu-symmetric-howe}. (See
also \cite{Sp-modular-jwtl-algebra} for a discussion using the $\ppar\lpar$JW
projectors.) Other recent work concerns the non-semisimple complex quantum group
case and its relation to mathematical physics, see {\eg} \cite{KoSt-fusion}, but
a non-semisimple recoupling theory along the lines of \cite{KaLi-TL-recoupling}
seems largely undeveloped.

Finally, the algorithm given in \cite{JeWi-p-canonical} to compute
$\ppar$-Kazhdan--Lusztig basis elements of affine type $A_{1}$ played a key role
in \cite{BuLiSe-tl-char-p}, and one could hope that this is a two-way street.
For example, via quantum Satake \cite{El-q-satake} and the approach in
\cite{RiWi-tilting-p-canonical} it might be possible to study analogs of mixed
Kazhdan--Lusztig bases.

\subsection*{Further remarks}

\changed{Before we start with the main bulk of the paper, we remark:}

\begin{enumerate}

\item \changed{We tried to make the exposition of this paper self-contained by
introducing all relevant concepts and definitions, many of which are identical
or very similar to those in \cite{TuWe-quiver}. Some results that we need, for
example \fullref{theorem:main-tl-section}, can then be proved analogously as in
\cite{TuWe-quiver}. Instead of repeating these proofs here, we decided to only
include detailed commentary about the necessary changes. In summary, while 
some of our results here depend on those of \cite{TuWe-quiver}, a
reader who wants to skip the proofs does not need to be familiar with that paper.}

\item \changed{The comparison of the various (non-)semisimple JW projectors in the
literature and their relation to this paper is as follows. Our mixed projectors
agree (modulo illustrations conventions) with the projectors constructed
independently in \cite{MaSp-lp-jw}. In the semisimple case they agree with the 
classical JW projectors, and over $\C$ and at a root of unity they agree with 
the projectors from \cite{GoWe-tl-at-roots-of-1}, again modulo conventions. 
For $\ppar=\lpar<\infty$ and $\qpar=1$, our projectors
agree with the projectors from \cite{BuLiSe-tl-char-p}. In each case,
constructing the projectors and proving that they are well-defined, {\eg} our
\fullref{theorem:well-defined}, require non-trivial numerical data. In our
approach, these are the tilting characters, see
\fullref{subsection:character-formulas}. There are two other ways to get
equivalent numerical data: First, one could use the Soergel category for affine
type $A_{1}$ and the $p$-Kazhdan--Lusztig basis as in \cite{BuLiSe-tl-char-p}.
This works for the quantum parameter being $1$, but for general quantum
parameter the situation is more complicated. A second method is to calculate the
simple multiplicities within the projective cover of the trivial Temperley--Lieb
module, as done in \cite{MaSp-lp-jw} (which is a follow-up of
\cite{Sp-modular-tl-algebra} where the decomposition numbers of the
Temperley--Lieb algebra are computed). This works in the mixed case. However, a
crucial upshot of the tilting characters approach taken in this paper is that it
might generalize beyond $\SLtwo$, {\eg} see \cite{So-tilting-a} and
\cite{So-tilting-b} for the complex quantum group case, very explicitly worked
out in \cite{St-diplom}.}

\end{enumerate}

\noindent\textbf{Acknowledgments.} 
We thank Elijah Bodish, David Craven, 
Nicolle Gonz{\'a}lez, Amit Hazi, Robert Spencer 
and Catharina Stroppel for stimulating 
discussions and helpful exchanges of emails. 
Special thanks to \ochanged{a referee and} Robert Spencer for
comments on a draft of this paper\fix{, and to Jonas Antor for drawing our attention to an error in an earlier version}.

Significant parts of this work were done when all the authors met at
the Hausdorff Research Institute for Mathematics (HIM) in the fall 2020 
during the Junior Trimester Program New Trends in Representation Theory. 
Hospitality and support of the HIM are gratefully acknowledged.

Parts of this paper needed extensive computer-based experiments, and all these
calculations were done using Mathematica 12. The first author thanks the
Heilbronn Institute for Mathematical Research for financial support; the second
author's computer literacy was sponsored by Covid19.

\section{Preliminaries}\label{section:prelim}

In this section we introduce necessary $\ppar\lpar$-adic notation, 
and recall how tilting modules of $\SLtwo$ and the Temperley--Lieb 
calculus are related.

\subsection{Basics of $\ppar\lpar$-adic expansions}\label{subsection:ellp}

Let $\kk$ denote a field, and fix an invertible element 
$\qpar\in\kk$ throughout. 
We also let $\vpar$ denote a formal variable.   
For $\xpar\in\kk$ and $a\in\N$ 
we consider the quantum numbers:
\begin{gather*}
\qnum{0}{\xpar}=0
,\quad
\qnum{a}{\xpar}=
\xpar^{-(a-1)}+\xpar^{-(a-3)}+\dots+\xpar^{a-3}+\xpar^{a-1}
,\quad
\qnum{-a}{\xpar}=-\qnum{a}{\xpar}.
\end{gather*}

\begin{definition} The \emph{mixed characteristic} of the pair $(\kk, \qpar)$ is
$\mchar=(\ppar,\lpar)$, where $\ppar\in\N\cup\{\infty\}$ denotes the additive
order of $1$ in $\kk$, and $\lpar\in\N\cup\{\infty\}$ is minimal such that
$\qnum{\lpar}{\qpar}=0$ for $\qpar\neq\pm 1$, and $\lpar=\ppar$ for $\qpar=\pm
1$. 
\end{definition}

Note that $\ppar$ is a prime number, if finite, but $\lpar$ can be any element in
$\N[\geq 2]\cup\{\infty\}$. Moreover, for finite $\lpar$ the equation
$\qnum{\lpar}{\qpar}=
\tfrac{(1-\qpar^{-\lpar})(\qpar^{\lpar}+1)}{\qpar-\qpar^{-1}}=0$ implies that
$\qpar^{\lpar}=\pm 1$.
Conversely, the order $n=\mathrm{ord}(\qpar)$
of the root of unity $\qpar$, if finite, determines $\lpar$ 
and the signs $\qpar^{\lpar}$ and $(-\qpar)^{\lpar}$ as
follows:
\begin{gather*}
\lpar=
\begin{cases}
n&\text{if }n\equiv 1\bmod 2,
\\
n/2&\text{if }n\equiv 0\bmod 2,
\end{cases}
\;
\qpar^{\lpar}
=
\begin{cases}
1&\text{if }n\equiv 1\bmod 2,
\\
-1&\text{if }n\equiv 0\bmod 2,
\end{cases}
\;
(-\qpar)^{\lpar}
=
\begin{cases}
-1&\text{if }n\equiv 0,1,3\bmod 4,
\\
1&\text{if }n\equiv 2\bmod 4.
\end{cases}  
\end{gather*}
The signs $\qpar^{\lpar}$ and $(-\qpar)^{\lpar}$ will appear in \eqref{eq:f-g}.

\begin{example}\label{example:pl}
The examples for $(\kk,\qpar)$ that the reader should keep in mind are: 
\begin{enumerate}

\item The \emph{integral case}, where the pair is $\big(\Zv,\vpar\big)$. Beware
that here $\kk$ is not a field, and we will always treat this case separately.

\item The \emph{semisimple case}, where $\ppar$ is arbitrary and $\lpar=\infty$.
Explicit examples include $\big(\Q(\vpar),\vpar\big)$,
and in fact $\big(\kkv,\vpar\big)$ for any field $\kk$.

\item The \emph{complex quantum group case} (at a root of unity), where
$\ppar=\infty$ and $\lpar<\infty$. For example, one could take $\kk=\C$ with
$\qpar=\exp(\pi i/\lpar)$ or $\qpar=\exp(2\pi i/\lpar)$, the former for 
all possible $\lpar$ and the latter for odd $\lpar$.

\item The \emph{characteristic $\ppar$ case}, where $\ppar=\lpar<\infty$ and
$\qpar=1$, {\eg} $(\F,1)$ or $(\F[\ppar^{k}],1)$. 

\item The \emph{(strictly) mixed cases} are all other cases, {\ie}
$\ppar<\infty$, $\lpar<\infty$ with $\ppar\neq\lpar$. An explicit example is the
pair $(\F[7],2)$ for which the mixed characteristic is $(7,3)$.

\end{enumerate}
\end{example}

For the rest of this paper, with the exception of concrete examples, we fix a
pair $(\kk,\qpar)$ of mixed characteristic $(\ppar,\lpar)$.
The numbers $\ppar$ and $\lpar$ will play a crucial role in this
paper, {\eg} via \emph{$\ppar\lpar$-adic expansions}:

\begin{definition}
Set $\ppar^{(0)}=1$, and for $i\in\N$ let $\ppar^{(i)}=\ppar^{i-1}\lpar$. For
any $v\in\N$ we write $\plbase{a_{j},\dots,a_{0}}=\sum_{i=0}^{j}a_{i}
\ppar^{(i)}=v$ with $a_{j}\neq 0$. The digits are from the sets
$a_{i}\in\{0,\dots,\ppar-1\}$ for $i>0$, and $a_{0}\in\{0,\dots,\lpar-1\}$. The
higher digits are declared to be zero: $a_{>j}=0$.

Conversely, any tuple $(b_{j},\dots,b_{0})\in\Z^{j+1}$ defines an integer
$\plbase{b_{j},\dots,b_{0}}=\sum_{i=0}^{j}b_{i}\ppar^{(i)}\in\Z$. Here we
explicitly allow negative digits.
\end{definition}

The $\ppar\lpar$-adic expansion of a natural number $v$ as defined above is
clearly unique: $a_{0}$ is uniquely determined as the remainder of $v$ upon
division by $\lpar$, and the remaining digits $\pbase{a_{j},\dots,a_{1}}{\ppar}$
are determined by the usual $\ppar$-adic expansion of the quotient
$\frac{v-a_{0}}{\lpar}$. The \emph{leading digit} $a_{j}$ and the \emph{zeroth
digit} $a_{0}$ will play slightly different roles than the other digits. If an
index $i\geq 0$ is implicit from the context, \ochanged{then} the symbol $\plpar$ refers to
$\ppar$ if $i>0$ and to $\lpar$ if $i=0$.

\begin{remark}\label{lospremark}
We will repeatedly encounter the 
\emph{law of small primes}, {\losp} for short: we see
special behavior in cases when relevant digits are (close to) $0$ modulo
$\plpar$. For large characteristics 
such cases are exceptions, while for small
ones they are the rule. 
\end{remark}

The following is taken from 
\cite{TuWe-quiver}, but for $\ppar\lpar$-adic
expansions.

\begin{definition}\label{definition:ancestry}
If $v=\plbase{a_{j},\dots,a_{0}}\in\N$ has only 
a single non-zero digit, then $v$ is 
called an \emph{eve}. The set of eves is 
denoted by $\eve$.
If $v\notin\eve$, then the \emph{mother} 
$\mother$ of $v$ is obtained by setting
the rightmost non-zero digit of $v$ to zero.

Assume that $v\notin\eve$ has $k$ non-zero, non-leading digits.
We will also consider the set
$\ancest=\{\mother,\motherr{v}{2}=\mother[\mother],\dots,\changed{\motherr{v}{k-1}},\plbase{a_{j},0,\dots,0}=\motherr{v}{\infty}\}$ of (matrilineal)
\emph{ancestors} of $v$, whose size $\generation$ is called the
\emph{generation} of $v$. By convention, $\ancest=\emptyset$ 
and $\generation=0$ for $v\in\eve$.
The \emph{support} $\supp\subset\N$ is 
the set of the $2^{\generation}$ integers of the form
$w=\plbase{a_{j},\pm a_{j-1},\dots,\pm a_{0}}$.
\end{definition}

Note that every $v\notin\eve$ has an 
associated eve \changed{$\motherr{v}{k}=\motherr{v}{\infty}$ 
with $k$ as in \fullref{definition:ancestry}, and $\generation=k$.} We think of the generation
and the ancestry chart as a measure of 
the complexity of the associated
$\SLtwo$ modules. For example, in
\fullref{proposition:multiplicities} 
we will see that a tilting module is simple
if and only if its $\rho$-shifted highest weight is an eve.

\begin{example}\label{example:support} 
In the semisimple case $\lpar=\infty$ every $v\in\N[0]$ is an eve and has no ancestors. In the complex quantum
group case $\ppar=\infty$, $\lpar<\infty$ every $v\in\N[0]$ is either
an eve or of generation $1$. In the other cases the generation can
be any number in $\N[0]$. For example,
$68=\pbase{68}{\infty,\infty}=\pbase{66,2}{\infty,3}=\pbase{1,2,5}{7,7}=\pbase{3,1,2}{7,3}$
has generation $0$, $1$, $2$ and $2$ in the listed mixed characteristics. In mixed
characteristic $(7,3)$ we have
$\changed{\ancest[68]}=\{66=\pbase{3,1,0}{7,3},63=\pbase{3,0,0}{7,3}\}$ and $\supp[{68=\pbase{3,1,2}{7,3}}]=\{68=\pbase{3,1,2}{7,3},64=\pbase{3,1,-2}{7,3},62=\pbase{3,-1,2}{7,3},58=\pbase{3,-1,-2}{7,3}\}$.
\end{example}

The elements $w$ in the support $\supp$ of 
$v\in\N$ can be described by the
sets of indices of digits of 
$v$, which are negated (or ``reflected'') to obtain
an expression for $w$. To obtain a bijection 
between elements in $\supp$ and
sets of indices, we enforce certain 
admissibility conditions on the latter:

\begin{definition}\label{definition:adm}
For $S\subset\N[0]$ a finite set, we consider partitions $S=\bigsqcup_{i}S_{i}$
of $S$ into subsets $S_{i}$ of consecutive integers that we call
\emph{stretches}. For the rest of the definition, we let $S=\bigsqcup_{i}S_{i}$
be the coarsest such partition into stretches.

The set $S$ is called \emph{down-admissible} for $v=\plbase{a_{j},\dots,a_{0}}$ if:
\begin{enumerate}[label=(\roman*)]

\item $a_{\min(S_{i})}\neq 0$ for every $i$, and

\item if $s\in S$ and $a_{s+1}=0$, then $s+1\in S$.
\end{enumerate}
If $S\subset\N[0]$ is down-admissible for $v=\plbase{a_{j},\dots,a_{0}}$, then we define its 
\emph{downward reflection} along $S$ as
\begin{gather*}
v[S]=\plbase{a_{j},
\epsilon_{j-1}a_{j-1},\dots,\epsilon_{0}\,a_{0}},\quad
\epsilon_{k}
=
\begin{cases}
1 &\text{if }k\notin S,
\\
-1 &\text{if }k\in S.
\end{cases}
\end{gather*}

Conversely, $S$ is 
\emph{up-admissible} for $v=\plbase{a_{j},\dots,a_{0}}$ if the following conditions are satisfied:
\begin{enumerate}[label=(\roman*)]

\item $a_{\min(S_{i})}\neq 0$ for every $i$, and

\item if $s\in S$ and $a_{s{+}1}=\ppar-1$, then we also have $s+1\in S$.

\end{enumerate}
If $S\subset\N[0]$ is up-admissible 
for $v=\pbase{a_{j},\dots,a_{0}}{\ppar}$, then we define its 
\emph{upward reflection} along $S$ as
\begin{gather*}
v(S)=
\plbase{a_{r(S)}^{\prime},\dots,a_{0}^{\prime}},
\quad
a_{k}^{\prime}= 
\begin{cases}
a_{k} &\text{if }k\notin S,k-1\notin S,
\\
a_{k}+2 &\text{if }k\notin S,k-1\in S,
\\
-a_{k} &\text{if }k\in S,
\end{cases}
\end{gather*}
where we extend the digits of $v$ by $a_{h}=0$ for $h>j$ if necessary, and
$r(S)$ is the biggest integer such that $a_{k}^{\prime}\neq 0$.

Any down- or up-admissible set $S$ has a 
unique finest partition into down- or
up-admissible sets, each of which consist 
of consecutive integers and which we
call \emph{minimal down-} respectively 
\emph{up-admissible stretches}. 
\end{definition}

A stretch $\{k,k-1,\dots,l+1,l\}$ is minimal 
down-admissible if and only if
\begin{gather*}
(a_{k+1},a_{k},\dots,a_{l+1},a_{l})=(a_{k+1},0,\dots,0,a_{l})
\quad\text{with }a_{k+1}\neq 0,a_{l}\neq 0.
\end{gather*}
It is minimal up-admissible if and only if
\begin{gather*}
(a_{k+1},a_{k},\dots,a_{l+1},a_{l})=(a_{k+1},\ppar-1,\dots,
\ppar-1,a_{l})\quad\text{with }a_{k+1}\neq\ppar-1,a_{l}\neq 0.
\end{gather*}
Very often (unless {\losp} applies), 
the minimal stretches will just be singleton
sets $\{i\}$ specifying a single digit 
in which we reflect. We also tend to omit
the set brackets of down- or 
up-admissible sets if no confusion can arise, {\eg}
we write $v[i]$ instead of $v[\{i\}]$.

For $v\in\N$ a finite set $S\subset\N[0]$ 
is down-admissible if and only if it
is up-admissible for $v[S]$, and in this case $v[S](S)=v$. For a
representation-theoretic interpretation of the admissibility conditions see
\fullref{remark:admissibility}: 
Note that $\supp=\{v[S]\mid S\text{ down-admissible}\}$.

\begin{definition}
If $S$ is up-admissible for $v\in N$, 
then we denote by $\hull\subset\N[0]$ the
down-admissible hull of $S$, the 
smallest down-admissible set containing $S$, if
it exists.
\end{definition}

Note that $\hull$ is only defined 
for up-admissible $S$, which excludes
stretches with rightmost digit 
zero. The singleton containing the leading digit
is always up-admissible and its 
down-admissible hull does not exist.

\begin{remark}
The above admissibility condition it is taken from \cite[Definition
2.8]{TuWe-quiver}. The whole discussion after \cite[Definition 2.8]{TuWe-quiver}
works {\ver}. Explicit examples appear there and in \cite[Example 2.9]{TuWe-center}.
\end{remark}

\subsection{Tilting modules and their diagrams}\label{subsection:TL-tilting}

Let $\kk\subset\K$ denote an 
algebraically closed field containing $\kk$. We use the symbol
$\SLtwo$ to denote the reductive group 
$\SLtwo$ over $\K$ if $\qpar=\pm 1\in\K$ and
Lusztig's divided power quantum group 
(using the conventions from \cite{AnPoWe-representation-qalgebras})
associated to $\SLtwolie$ for
other values of $\qpar$. We will identify 
dominant integral weights of $\SLtwo$ 
with $\N[0]$ and weights with $\Z$ in the 
usual way.

We consider finite-dimensional (left) 
$\SLtwo$-modules of type $1$ over $\K$. 
\changed{(Restricting to $\SLtwo$-modules of type $1$ is, as usual, for convenience, 
see {\eg} \cite[Section 5.2]{Ja-lectures-qgroups} for details.)}
These form an abelian, $\K$-linear 
category $\allmod=\SLtwo\text{-}\allmod$, for
which we additionally choose a monoidal and a pivotal structure using the
comultiplication $\SLtwo$, and the antipode of $\SLtwo$ and the analog of the
involution $\omega$ from \cite[Lemma 4.6]{Ja-lectures-qgroups}. The category
$\allmod$ contains four families of highest weight modules of particular
interest for our purpose, all parameterized by $\N[0]$. 

\begin{remark}
Here and in the following sections, we write the 
highest weights of these modules 
often as $v-1$ for
$v\in\N$. This puts an emphasis on the 
quantity $v$, the $\rho$-shifted highest
weight, which will play a greater role 
than the highest weight itself.
\end{remark}

The first two families are formed by 
the \emph{Weyl} modules $\wmod(v-1)$ and the \emph{dual Weyl}
modules $\dwmod(v-1)$. These do not depend on the mixed
characteristic in the sense that they can be defined integrally, {\ie} for
$\big(\Zv,\vpar\big)$. \ochanged{Their characters are given by Weyl character formula.}

The other two families of modules are 
formed by the \emph{simple} modules $\lmod(v-1)$
and the \emph{indecomposable tilting} modules $\tmod(v-1)$. These modules
do not admit a construction independent of the mixed characteristic. 
Their characters are given by 
\fullref{proposition:multiplicities} below.

Let
$\tilt[{\K,\qpar}]=\SLtwo\text{-}\tilt[{\K,\qpar}]$ 
be the full subcategory of $\allmod$ whose objects
are direct sums of $\tmod(v-1)$ for $v\in\N$. We also write 
$\tmod(z)$ for $z<0$ which is zero, by 
convention. The category $\tilt[{\K,\qpar}]$ is
additive, idempotent closed, Krull--Schmidt 
(meaning there is a essential unique 
decompositions into indecomposables, and an object 
is indecomposable if and only 
if its endomorphism ring is local), $\K$-linear, and pivotal
(restricting the structures from 
$\allmod$ to $\tilt[{\K,\qpar}]$). It is the main object
under study in this paper and 
called the \emph{category of tilting modules} of
$\SLtwo$.

\begin{remark}\label{remark:filtrations} 
Classically $\tilt[{\K,\qpar}]$ would
be defined as the full subcategory of 
$\allmod$ whose objects have Weyl and dual
Weyl filtrations, and its closure under 
tensor product would be a theorem. The
above definition is equivalent to the 
classical one for $\SLtwo$, because the
sole fundamental representation (it is $\tmod(1)$) 
is tilting and thus all
indecomposable tiltings appear as 
direct summands of tensor powers thereof. This
may fail for other types in small characteristic.
\end{remark}

Generally these four types of 
modules (Weyl, dual Weyl, simple, and
indecomposable tilting) for a 
fixed highest weight are distinct from one
another. If, however, two are 
isomorphic {\eg} $\tmod(v-1)\cong\dwmod(v-1)$,
then it follows that all four 
types of modules of the same highest weight are
isomorphic. An example is
$\tmod(0)\cong\wmod(0)\cong\dwmod(0)\cong\lmod(0)\cong\K$, which is the
monoidal unit of $\tilt[{\K,\qpar}]$ and which we denote by $\munit$.

\begin{remark}
Let us comment on the references for the above and some of the material below,
using the terminology from \fullref{example:pl}. In the semisimple case,
$\tilt[{\K,\qpar}]$ is equivalent to $\allmod$, is semisimple and has the
classical combinatorics of $\mathrm{SL}_{2}(\C)$, which is covered in many
textbooks. Otherwise $\tilt[{\K,\qpar}]$ is non-semisimple and we refer to
\cite{AnPoWe-representation-qalgebras} and \cite{An-tensor-q-tilting-modules}
in the complex quantum group case, to \cite{Ri-good-filtrations} and \cite{Do-tilting-alg-groups} 
in the characteristic $\ppar$ case, and to
\cite{AnWe-mixed-qgroup} as well as \cite{Do-q-schur} and \cite{An-simple-tl}
in the mixed case. A summary for tilting modules 
can also be found in \cite{AnStTu-cellular-tilting-2}.
\end{remark}

The diagrammatic incarnation of $\tilt[{\K,\qpar}]$ is sometimes called the
\emph{Temperley--Lieb category} 
(abbreviated to TL category) and can be defined as follows.
Let $\TL[{\Zv,\vpar}]$ 
denote the $\Zv$-linear category with objects indexed by
$m\in\N[0]$, and with 
morphisms from $m$ to $n$ being $\Zv$-linear 
combinations
of unoriented string diagrams drawn in a horizontal strip $\R\times[0,1]$
between $m$ marked points on the lower 
boundary $\R\times\{0\}$ and $n$ marked
points on the upper boundary $\R\times\{1\}$, 
considered up to planar isotopy
relative to the boundary and the 
relation that a circle evaluates to
$-\qnum{2}{\vpar}$. The category 
$\TL[{\Zv,\vpar}]$ is (strict) monoidal with $\hcirc$ given by
horizontal concatenation and admits a 
(strict) pivotal structure given by the duality
maps being cups and caps, and all objects being self-dual.

We write $\morstuff{F}\morstuff{G}=\morstuff{F}\vcirc\morstuff{G}$ 
for the composition of morphisms in $\TL[{\Zv,\vpar}]$, and 
we read diagrams from bottom to top and left to right, {\eg}
\begin{gather*}
(\idmor\hcirc\morstuff{G})(\morstuff{F}\hcirc\idmor) 
=
\begin{tikzpicture}[anchorbase,scale=0.22,tinynodes]
\draw[mor] (-5.5,2) rectangle (-2,0.5);
\draw[mor] (-0.5,3) rectangle (3,4.5);
\draw[thick,densely dotted] (-5.5,2.5) node[left,yshift=-2pt]{$\vcirc$} 
to (3.5,2.5) node[right,yshift=-2pt]{$\vcirc$};
\draw[thick,densely dotted] (-1.25,5) 
node[above,yshift=-2pt]{$\hcirc$} to (-1.25,0) node[below]{$\hcirc$};
\draw[usual] (-5,2) to (-5,3.75) node[right,xshift=1pt]{$\dots$} to (-5,5);
\draw[usual] (-2.5,2) to (-2.5,5);
\draw[usual] (0,0) to (0,1.25) node[right,xshift=1pt]{$\dots$} to (0,3);
\draw[usual] (2.5,0) to (2.5,3);
\draw[usual] (-5,0) to (-5,0.25) node[right,xshift=1pt]{$\dots$} to (-5,0.5);
\draw[usual] (-2.5,0) to (-2.5,0.5);
\draw[usual] (2.5,4.5) to (2.5,5);
\draw[usual] (0,4.5) to (0,4.75) node[right,xshift=1pt]{$\dots$} to (0,5);
\node at (-3.75,1.0) {$\morstuff{F}$};
\node at (1.25,3.5) {$\morstuff{G}$};
\end{tikzpicture}
=
\begin{tikzpicture}[anchorbase,scale=0.22,tinynodes]
\draw[mor] (-5.5,1.75) rectangle (-2,3.25);
\draw[mor] (-0.5,1.75) rectangle (3,3.25);
\draw[usual] (-5,3.25) to (-5,4.125) node[right,xshift=1pt]{$\dots$} to (-5,5);
\draw[usual] (-2.5,3.25) to (-2.5,5);
\draw[usual] (0,0) to (0,0.625) node[right,xshift=1pt]{$\dots$} to (0,1.75);
\draw[usual] (2.5,0) to (2.5,1.75);
\draw[usual] (-5,0) to (-5,0.625) node[right,xshift=1pt]{$\dots$} to (-5,1.75);
\draw[usual] (-2.5,0) to (-2.5,1.75);
\draw[usual] (2.5,3.25) to (2.5,5);
\draw[usual] (0,3.25) to (0,4.125) node[right,xshift=1pt]{$\dots$} to (0,5);
\node at (-3.75,2.25) {$\morstuff{F}$};
\node at (1.25,2.25) {$\morstuff{G}$};
\end{tikzpicture}
=
\begin{tikzpicture}[anchorbase,scale=0.22,tinynodes]
\draw[mor] (5.5,2) rectangle (2,0.5);
\draw[mor] (0.5,3) rectangle (-3,4.5);
\draw[thick, densely dotted] (5.5,2.5) 
node[right,yshift=-2pt]{$\vcirc$} to (-3.5,2.5) node[left,yshift=-2pt]{$\vcirc$};
\draw[thick, densely dotted] (1.25,5) 
node[above,yshift=-2pt]{$\hcirc$} to (1.25,0) node[below]{$\hcirc$};
\draw[usual] (2.5,2) to (2.5,3.75) node[right,xshift=1pt]{$\dots$} to (2.5,5);
\draw[usual] (5,2) to (5,5);
\draw[usual] (0,0) to (0,3);
\draw[usual] (-2.5,0) to (-2.5,1.25) node[right,xshift=1pt]{$\dots$} to (-2.5,3);
\draw[usual] (2.5,0) to (2.5,0.25) node[right,xshift=1pt]{$\dots$} to (2.5,0.5);
\draw[usual] (5,0) to (5,0.5);
\draw[usual] (-2.5,4.5) to (-2.5,4.75) node[right,xshift=1pt]{$\dots$} to (-2.5,5);
\draw[usual] (0,4.5) to (0,5);
\node at (3.75,1.0) {$\morstuff{G}$};
\node at (-1.25,3.5) {$\morstuff{F}$};
\end{tikzpicture}
=(\morstuff{F}\hcirc\idmor)(\idmor\hcirc\morstuff{G}).
\end{gather*} 
There is an antiinvolution 
$\fliph[{(\placeholder)}]$ on $\TL[{\Zv,\vpar}]$ which fixes 
objects and reflects diagrams in a horizontal line, as well as an involution
$\flipv[{(\placeholder)}]$ which mirrors long the vertical axis. 
The following summarizes the important relations and conventions:
\begin{gather*}
\begin{tikzpicture}[anchorbase,scale=0.25,tinynodes]
\draw[usual] (0,-2) to (0,0) to[out=90,in=180] (1,1) to[out=0,in=90] (2,0);
\draw[usual] (2,0) to[out=270,in=180] (3,-1) to[out=0,in=270] (4,0) to (4,2);
\end{tikzpicture}
=
\begin{tikzpicture}[anchorbase,scale=0.25,tinynodes]
\draw[usual] (0,-2) to (0,2);
\end{tikzpicture}
\;,\quad
\begin{tikzpicture}[anchorbase,scale=0.25,tinynodes]
\draw[usual] (0,-2) to (0,0) to[out=90,in=0] (-1,1) to[out=180,in=90] (-2,0);
\draw[usual] (-2,0) to[out=270,in=0] (-3,-1) to[out=180,in=270] (-4,0) to (-4,2);
\end{tikzpicture}
=
\begin{tikzpicture}[anchorbase,scale=0.25,tinynodes]
\draw[usual] (0,-2) to (0,2);
\end{tikzpicture}
,\quad
\begin{tikzpicture}[anchorbase,scale=0.25,tinynodes]
\draw[usual] (0,0) to[out=90,in=180] (1,1) to[out=0,in=90] (2,0);
\draw[usual] (0,0) to[out=270,in=180] (1,-1) to[out=0,in=270] (2,0);
\end{tikzpicture}
=-\qnum{2}{\vpar}
,\quad
\fliph[n]=\flipv[n]=n
,\quad
\bigg(\begin{tikzpicture}[anchorbase,scale=0.22,tinynodes]
\draw[thin,black,fill=mygray,fill opacity=0.35] (-5.5,2) rectangle (-2,0.5);
\draw[usual] (-5,0) to (-5,0.25) node[right,xshift=1pt]{$\dots$} to (-5,0.5);
\draw[usual] (-2.5,0) to (-2.5,0.5);
\draw[usual] (-2.5,2) to (-2.5,2.5);
\draw[usual] (-5,2) to (-5,2.25) node[right,xshift=1pt]{$\dots$} to (-5,2.5);
\node at (-3.75,1.0) {$\morstuff{F}$};
\end{tikzpicture}
\bigg)^{\updownarrow}
=
\begin{tikzpicture}[anchorbase,scale=0.22,tinynodes]
\draw[thin,black,fill=mygray,fill opacity=0.35] (-5.5,2) rectangle (-2,0.5);
\draw[usual] (-5,0) to (-5,0.25) node[right,xshift=1pt]{$\dots$} to (-5,0.5);
\draw[usual] (-2.5,0) to (-2.5,0.5);
\draw[usual] (-2.5,2) to (-2.5,2.5);
\draw[usual] (-5,2) to (-5,2.25) node[right,xshift=1pt]{$\dots$} to (-5,2.5);
\node at (-3.75,1.75) {\reflectbox{\rotatebox{180}{$\morstuff{F}$}}};
\end{tikzpicture}
,\quad
\bigg(\begin{tikzpicture}[anchorbase,scale=0.22,tinynodes]
\draw[thin,black,fill=mygray,fill opacity=0.35] (-5.5,2) rectangle (-2,0.5);
\draw[usual] (-5,0) to (-5,0.25) node[right,xshift=1pt]{$\dots$} to (-5,0.5);
\draw[usual] (-2.5,0) to (-2.5,0.5);
\draw[usual] (-2.5,2) to (-2.5,2.5);
\draw[usual] (-5,2) to (-5,2.25) node[right,xshift=1pt]{$\dots$} to (-5,2.5);
\node at (-3.75,1.0) {$\morstuff{F}$};
\end{tikzpicture}
\bigg)^{\leftrightarrow}
=
\begin{tikzpicture}[anchorbase,scale=0.22,tinynodes]
\draw[thin,black,fill=mygray,fill opacity=0.35] (-5.5,2) rectangle (-2,0.5);
\draw[usual] (-5,0) to (-5,0.25) node[right,xshift=1pt]{$\dots$} to (-5,0.5);
\draw[usual] (-2.5,0) to (-2.5,0.5);
\draw[usual] (-2.5,2) to (-2.5,2.5);
\draw[usual] (-5,2) to (-5,2.25) node[right,xshift=1pt]{$\dots$} to (-5,2.5);
\node at (-3.75,1) {\reflectbox{$\morstuff{F}$}};
\end{tikzpicture}
.
\end{gather*}

Let $\TL[{\kk,\qpar}]=\TL[{\Zv,\vpar}]\hcirc_{\Zv}\kk$ be the scalar extension
and specialization 
$\Zv\ni\vpar\mapsto\qpar\in\kk$. Recall that $\K$ denotes an
algebraically closed field containing $\kk$. Recall also that $\tmod(1)$
generates $\tilt[{\K,\qpar}]$ as a monoidal category.

\begin{proposition}\label{proposition:TLtilt} 
We have a $\K$-linear, pivotal functor 
\begin{gather*}
\tlfunctor\colon
\TL[{\K,\qpar}]\to\tilt[{\K,\qpar}],
\quad
\tlfunctor(d)=\tmod(1)^{\hcirc(d)},
\end{gather*}
which induces an equivalence of 
$\K$-linear, pivotal 
categories upon additive idempotent completion.
\end{proposition}

\begin{proof}
This is folklore, the semisimple case dates back to 
\cite{RuTeWe-cup-diagrams}, and a proof 
in general can be found in {\eg} \cite[Theorem 2.58]{El-ladders-clasps} 
or \cite[Proposition 2.3]{AnStTu-semisimple-tilting}.
\end{proof}

Recall that $\TL[{\Zvs,\vpar}]$ 
(we need to add a formal square root of $\vpar$)
admits the structure of a braided category. 
The braiding is determined on the
generating object by Kauffman's \emph{skein relation}
\begin{gather}\label{eq:kauffman}
\begin{tikzpicture}[anchorbase,scale=0.25,tinynodes]
\draw[usual] (2,0) to (0,2);
\draw[usual,crossline] (0,0) to (2,2);
\end{tikzpicture}
=
\vpar^{1/2}\cdot
\begin{tikzpicture}[anchorbase,scale=0.25,tinynodes]
\draw[usual] (0,0) to (0,2);
\draw[usual] (2,0) to (2,2);
\end{tikzpicture}
+
\vpar^{-1/2}\cdot
\begin{tikzpicture}[anchorbase,scale=0.25,tinynodes]
\draw[usual] (0,0) to[out=90,in=180] (1,0.75) to[out=0,in=90] (2,0);
\draw[usual] (0,2) to[out=270,in=180] (1,1.25) to[out=0,in=270] (2,2);
\end{tikzpicture}
,\quad
\begin{tikzpicture}[anchorbase,scale=0.25,tinynodes]
\draw[usual] (0,0) to (2,2);
\draw[usual,crossline] (2,0) to (0,2);
\end{tikzpicture}
=
\vpar^{-1/2}\cdot
\begin{tikzpicture}[anchorbase,scale=0.25,tinynodes]
\draw[usual] (0,0) to (0,2);
\draw[usual] (2,0) to (2,2);
\end{tikzpicture}
+
\vpar^{1/2}\cdot
\begin{tikzpicture}[anchorbase,scale=0.25,tinynodes]
\draw[usual] (0,0) to[out=90,in=180] (1,0.75) to[out=0,in=90] (2,0);
\draw[usual] (0,2) to[out=270,in=180] (1,1.25) to[out=0,in=270] (2,2);
\end{tikzpicture}
.
\end{gather}
There is also a braiding on $\tilt[{\K,\qpar}]$, assuming that $\qpar$ has a
square root in $\K$, given by the so-called R-matrix, see {\eg} \cite[Section
IX.7]{Ka-quantum-groups}. (We clear 
the denominators in these formulas by using
divided powers, and observe that the expression is well-defined on all
finite-dimensional modules without further adjustments.) These two braidings,
the only ones we will consider in this paper, are compatible, as can be seen, {\eg} by comparing on generating objects:

\begin{proposition}
If $\qpar$ has a square root in $\K$, then the 
functor $\tlfunctor$ from \fullref{proposition:TLtilt} 
is an equivalence of braided categories.\qed
\end{proposition}

\begin{remark}
\label{rem:abuse}
Note that \fullref{proposition:TLtilt} allows us to identify the additive
Karoubi closure of $\TL[{\K,\qpar}]$ with $\tilt[{\K,\qpar}]$. Motivated by
this, we will also denote the additive Karoubi closure of $\TL[{\kk,\qpar}]$ by
$\tilt[{\kk,\qpar}]$ in case of the not necessary algebraically closed field
$\kk$, and the objects $\tmod(v-1)\in\tilt[{\kk,\qpar}]$ are defined as the
images of primitive idempotents under $\tlfunctor[{\kk,\qpar}]$. In this
notation the functor $\tlfunctor[{\kk,\qpar}]$ will be the universal embedding
of $\TL[{\kk,\qpar}]$ into its Karoubi closure, and we omit it from the
notation.
\end{remark}

\begin{remark}
At this point, the reader is warned that $\tilt[{\kk,\qpar}]$ may not be
equivalent to the category 
of tilting modules over $\kk$ (as defined via (dual)
Weyl filtrations), even in semisimple cases, if $\kk$ is finite, see \cite[Section 5]{BeDo-schur-weyl-finite-fields}.
\end{remark}

In the semisimple situation 
of $\TL[\kkv,\vpar]$, the primitive idempotents that
are mapped to the indecomposable tilting modules $\tmod(v-1)$ in
$\tmod(1)^{\hcirc(v{-}1)}$ are the well-known Jones--Wenzl projectors. 
Since $\tmod(v-1)$ is a simple module in this case, we will
call these idempotents 
\emph{simple Jones--Wenzl projectors} (simple JW projectors for
short), also to 
distinguish them from their non-simple analogs. All we need to
know about these projectors is summarized in the following proposition, 
see {\eg} \cite{KaLi-TL-recoupling} for a proof.

\begin{proposition}\label{proposition:generic-jw-properties}
For all $v\in\N$ there exists a unique idempotent
$\qjw[v{-}1]\in\End_{\TL[\kkv,\vpar]}(v-1)$, which is invariant under duality
$\fliph[{(\qjw)}]=\flipv[{(\qjw)}]=\qjw$ (this implies that the 
following relations hold mirrored as well) and which satisfies:\\
\noindent
\begin{minipage}{0.29\textwidth}
\begin{gather}\label{eq:0absorb}
\begin{tikzpicture}[anchorbase,scale=0.25,tinynodes]
\draw[JW] (-0.6,0) rectangle (1.6,2);
\draw[usual] (-1,0) to (-1,2);
\draw[usual] (2,0) to (2,2);
\node at (0.5,0.9) {$\qjwm[w{-}1]$};
\draw[JW] (-1.5,-2) rectangle (2.5,0);
\node at (0.5,-1.1) {$\qjwm$};
\node at (0,2.5) {$\phantom{a}$};
\node at (0,-2.5) {$\phantom{a}$};
\end{tikzpicture}
=
\begin{tikzpicture}[anchorbase,scale=0.25,tinynodes]
\draw[JW] (-1.5,-1) rectangle (2.5,1);
\node at (0.5,-0.1) {$\qjwm$};
\node at (0,2.5) {$\phantom{a}$};
\node at (0,-2.5) {$\phantom{a}$};
\end{tikzpicture},
\end{gather}
\end{minipage}
\begin{minipage}{0.24\textwidth}
\begin{gather}\label{eq:0kill}
\quad
\begin{tikzpicture}[anchorbase,scale=0.25,tinynodes]
\draw[JW] (-1.2,-1) rectangle (1.2,1);
\node at (0,-0.1) {$\qjwm$};
\draw[usual] (-0.5,1) to[out=90,in=180] (0,1.5) 
node[above,yshift=-2pt]{$k$} to[out=0,in=90] (0.5,1);
\node at (0,2.5) {$\phantom{a}$};
\node at (0,-2.5) {$\phantom{a}$};
\end{tikzpicture}
=
0,
\end{gather}
\end{minipage}
\begin{minipage}{0.45\textwidth}
\begin{gather}\label{eq:0trace}
\begin{tikzpicture}[anchorbase,scale=0.25,tinynodes]
\draw[JW] (1.5,-1) rectangle (-1,1);
\node at (0.25,-0.1) {$\qjwm$};
\draw[usual] (1,1) to[out=90,in=180] (1.5,1.5) to[out=0,in=90] 
(2,1) to (2,0) node[right,xshift=-2pt]{$k$} to (2,-1) 
to[out=270,in=0] (1.5,-1.5) to[out=180,in=270] (1,-1);
\node at (0,2.5) {$\phantom{a}$};
\node at (0,-2.5) {$\phantom{a}$};
\end{tikzpicture}
=(-1)^{k}
\tfrac{\qnum{v}{\vpar}}{\qnum{v{-}k}{\vpar}}\cdot
\begin{tikzpicture}[anchorbase,scale=0.25,tinynodes]
\draw[JW] (-1.8,-1) rectangle (1.8,1);
\node at (0,-0.1) {$\qjwm[{v{-}1{-}k}]$};
\node at (0,2.5) {$\phantom{a}$};
\node at (0,-2.5) {$\phantom{a}$};
\end{tikzpicture}
.
\end{gather}
\end{minipage}

\noindent \changed{Here we use the usual box notation for these projectors, a number $k$ next to a strand means $k$ parallel strands}, and the 
projector $\qjw[w{-}1]$ in \eqref{eq:0absorb} respectively the cup or cap in 
\eqref{eq:0kill} can be at arbitrary positions. The idempotent 
$\qjw[v{-}1]$ satisfies the recursion in \eqref{eq:jw-recursion}
\qed
\end{proposition}

In \fullref{definition:lightladders} and \fullref{convention:du} we will define
various different bases of morphism spaces in Temperley--Lieb categories. The
first example are the \emph{integral bases} given by sets of crossingless
matchings ({\aka} \emph{Temperley--Lieb diagrams}) $\sbas$ of $v+w-2$ points.
These are integral in the sense that they provide isomorphisms
$\Hom_{\TL[{\Zv,\vpar}]}(v-1,w-1)\cong\Zv\sbas$ of $\Zv$-modules. Second,
$\TL[{\kk,\qpar}]$ has \emph{projector bases} $\pbas$ given by decomposing
$\tmod(1)^{\hcirc(v{-}1)}$ into indecomposable summands. (For
$\TL[{\kkv,\vpar}]$ a basis of the form $\awbas$ is an \emph{Artin--Wedderburn
basis} since these summands are simple.) We stress that these bases are not
unique unless one specifies further properties that these should satisfy. The
existence of these bases follows from abstract theory, see
\cite{AnStTu-cellular-tilting}, and all of these are cellular and related by
unitriangular basis change matrices. To construct these bases explicitly we can
use the \emph{light ladder strategy}, see \cite{El-ladders-clasps} and
\cite{AnStTu-cellular-tilting}.

\begin{definition}\label{definition:lightladders} 
Fix a family of morphisms
$\morstuff{G}_{v{-}1}\in\End_{\TL[{\kk,\qpar}]}(v-1)$ for $v\in\N$. Then for
each $\morstuff{F}\in\Hom_{\TL[{\kk,\qpar}]}(w-1,v-1)$ we define morphisms
$\tilde{\varepsilon}_{1}(\morstuff{F})\in\Hom_{\TL[{\kk,\qpar}]}(w,v)$ and
(provided $v>1$) $\tilde{\varepsilon}_{-1}(\morstuff{F})\in
\Hom_{\TL[{\kk,\qpar}]}(w,v-2)$ by sending:
\begin{gather*}
F=
\begin{tikzpicture}[anchorbase,scale=1,tinynodes]
\draw[mor] (0,0) rectangle (1,0.5) node[pos=0.5,text opacity=1,yshift=-2pt]{$\morstuff{F}$};
\draw[usual] (0.5,0.5) to (0.5,0.75) node[above,yshift=-2pt]{$v-1$};
\end{tikzpicture}
\mapsto
\begin{tikzpicture}[anchorbase,scale=1,tinynodes]
\draw[mor] (0,0) rectangle (1,0.5) node[pos=0.5,text opacity=1,yshift=-2pt]{$\morstuff{F}$};
\draw[mor] (0,0.5) rectangle (1.5,1) node[pos=0.5,text opacity=1,yshift=-2pt]{$\morstuff{G}_{v}$};
\draw[usual] (1.25,0) to (1.25,0.5);
\end{tikzpicture}
=\tilde{\varepsilon}_{1}(\morstuff{F})
,\quad
F=
\begin{tikzpicture}[anchorbase,scale=1,tinynodes]
\draw[mor] (0,0) rectangle (1,0.5) node[pos=0.5,text opacity=1,yshift=-2pt]{$\morstuff{F}$};
\draw[usual] (0.5,0.5) to (0.5,0.75) node[above,yshift=-2pt]{$v-1$};
\end{tikzpicture}
\mapsto
\begin{tikzpicture}[anchorbase,scale=1,tinynodes]
\draw[mor] (0,0) rectangle (1,0.5) node[pos=0.5,text opacity=1,yshift=-2pt]{$\morstuff{F}$};
\draw[mor] (0,0.5) rectangle (0.75,1) node[pos=0.5,text opacity=1,yshift=-2pt]{$\morstuff{G}_{v{-}2}$};
\draw[usual] (1.25,0) to (1.25,0.5) to [out=90,in=0] (1.05,0.7) to [out=180,in=90] (0.85,0.5);
\end{tikzpicture}
=\tilde{\varepsilon}_{-1}(\morstuff{F}).
\end{gather*}
For any path $\pi$ in the positive Weyl chamber, considered as a finite sequence
of $\pm 1$ whose partial sums are non-negative, we associated a \emph{down
morphism} $\delta(\pi)$ by using the operators $\tilde{\varepsilon}_{\pm 1}$
in order specified by $\pi$, starting with $\morstuff{F}$ being the empty diagram. Similarly, we
define an \emph{up morphism} $\upsilon(\pi)$ as
$\fliph[\delta(\pi)]$, and for a pair $(\pi,\pi^{\prime})$ an element
$\morstuff{c}_{\pi,\pi^{\prime}}^{\lambda}=\upsilon(\pi^{\prime})\delta(\pi)$
whenever that makes sense, {\ie} for
$\delta(\pi)\in\Hom_{\TL[{\kk,\qpar}]}(v-1,\lambda)$ and
$\upsilon(\pi^{\prime})\in\Hom_{\TL[{\kk,\qpar}]}(\lambda,w-1)$.
\end{definition}

\begin{convention}\label{convention:du}
We will use the light ladder strategy from
\fullref{definition:lightladders} in several different contexts. The associated
down and up morphisms are consistently distinguished throughout this paper by
the following notation convention.
\begin{enumerate}

\item For $\morstuff{G}_{v-1}=\idmor_{v-1}$, which works for any ground ring
(in particular for $\big(\Zv,\vpar\big)$), we obtain the integral bases
$\sbas$ for morphism spaces. We reserve the following notation for these
morphisms: $ \morstuff{G}_{v-1}=\idmor_{v-1}\implies\morstuff{d}(\pi),
\morstuff{u}(\pi)$.

\item For $\morstuff{G}_{v-1}=\qjw[v{-}1]$ and working over
$\big(\kkv,\vpar\big)$ we get the Artin--Wedderburn basis. The associated
morphisms will be denoted with tilde symbols: $
\morstuff{G}_{v-1}=\qjw[v{-}1]\implies\tilde{\morstuff{d}}(\pi),
\tilde{\morstuff{u}}(\pi)$.

\item For $\morstuff{G}_{v-1}=\pjw[v{-}1]$, {\ie} for the projectors constructed
in \fullref{subsection:projectors} for non-semisimple situations, we will
use capital letters: $\morstuff{G}_{v-1}=\pjw[v{-}1]\implies\Down{}(\pi),
\Up{}(\pi)$.
These are specializations of morphisms that one gets 
for $\morstuff{G}_{v-1}=\pqjw[v{-}1]$ 
\changed{(with $\pqjw[v{-}1]$ as in \fullref{definition:pl-jw-q} below)}, and we will use an overline 
in this situation: $\morstuff{G}_{v-1}=\pqjw[v{-}1]\implies\overline{\morstuff{d}}(\pi),
\overline{\morstuff{u}}(\pi)$.

\end{enumerate}
\end{convention}

\begin{definition}\label{definition:left-right}
A family of morphisms $\morstuff{G}_{v{-}1}\in\End_{\TL[{\kk,\qpar}]}(v-1)$
for $v\in\N$ is \emph{left-aligned} if
\begin{gather*}
\morstuff{G}_{v{-}1}(\morstuff{G}_{w{-}1}\hcirc\idtl[v{-}w]) =(\morstuff{G}_{w{-}1}\hcirc\idtl[v{-}w])\morstuff{G}_{v{-}1} =\morstuff{G}_{v{-}1},\quad\text{for all }1\leq w\leq v,
\end{gather*}
and \emph{right-aligned} if
\begin{gather*}
\morstuff{G}_{v{-}1}(\idtl[v{-}w]\hcirc\morstuff{G}_{w{-}1}) =(\idtl[v{-}w]\hcirc\morstuff{G}_{w{-}1})\morstuff{G}_{v{-}1} =\morstuff{G}_{v{-}1}
,\quad\text{for all }1\leq w\leq v. 
\end{gather*}
\end{definition}

We draw morphisms from a left-aligned family as boxes with a bar at the
left-hand side, and {\vive} for right-aligned. Using this notation the two
conditions in \fullref{definition:left-right} read:
\begin{gather*}
\begin{tikzpicture}[anchorbase,scale=0.25,tinynodes]
\draw[morl] (0.5,0) rectangle (-2.5,2);
\draw[usual] (1,0) to (1,2);
\node at (-1,0.9) {\scalebox{0.75}{$\morstuff{G}_{w{-}1}$}};
\draw[morl] (1.5,-2) rectangle (-2.5,0);
\node at (-0.5,-1.1) {$\morstuff{G}_{v{-}1}$};
\node at (-1,2.5) {$\phantom{a}$};
\node at (-1,-2.5) {$\phantom{a}$};
\end{tikzpicture}
=
\begin{tikzpicture}[anchorbase,scale=0.25,tinynodes]
\draw[morl] (1.5,-1) rectangle (-2.5,1);
\node at (-0.5,-.1) {$\morstuff{G}_{v{-}1}$};
\node at (-1,2.5) {$\phantom{a}$};
\node at (-1,-2.5) {$\phantom{a}$};
\end{tikzpicture}
=
\begin{tikzpicture}[anchorbase,scale=0.25,tinynodes]
\draw[morl] (0.5,0) rectangle (-2.5,2);
\draw[usual] (1,0) to (1,2);
\node at (-1,0.9) {\scalebox{0.75}{$\morstuff{G}_{w{-}1}$}};
\draw[morl] (1.5,2) rectangle (-2.5,4);
\node at (-0.5,2.9) {$\morstuff{G}_{v{-}1}$};
\node at (-1,4.5) {$\phantom{a}$};
\node at (-1,-0.5) {$\phantom{a}$};
\end{tikzpicture}
,\quad
\begin{tikzpicture}[anchorbase,scale=0.25,tinynodes]
\draw[morr] (-0.5,0) rectangle (2.5,2);
\draw[usual] (-1,0) to (-1,2);
\node at (1,0.9) {\scalebox{0.75}{$\morstuff{G}_{w{-}1}$}};
\draw[morr] (-1.5,-2) rectangle (2.5,0);
\node at (0.5,-1.1) {$\morstuff{G}_{v{-}1}$};
\node at (1,2.5) {$\phantom{a}$};
\node at (1,-2.5) {$\phantom{a}$};
\end{tikzpicture}
=
\begin{tikzpicture}[anchorbase,scale=0.25,tinynodes]
\draw[morr] (1.5,-1) rectangle (-2.5,1);
\node at (-0.5,-.1) {$\morstuff{G}_{v{-}1}$};
\node at (1,2.5) {$\phantom{a}$};
\node at (1,-2.5) {$\phantom{a}$};
\end{tikzpicture}
=
\begin{tikzpicture}[anchorbase,scale=0.25,tinynodes]
\draw[morr] (-0.5,0) rectangle (2.5,2);
\draw[usual] (-1,0) to (-1,2);
\node at (1,0.9) {\scalebox{0.75}{$\morstuff{G}_{w{-}1}$}};
\draw[morr] (-1.5,2) rectangle (2.5,4);
\node at (0.5,2.9) {$\morstuff{G}_{v{-}1}$};
\node at (1,4.5) {$\phantom{a}$};
\node at (1,-0.5) {$\phantom{a}$};
\end{tikzpicture}
.
\end{gather*}
Note that left- and right-aligned families 
of morphisms are always idempotents,
by the $v=w$ case of the defining relation.

\begin{remark}
The families of identity morphisms 
$\idtl$ are both left- and right-aligned 
and so are 
simple JW projectors by
\eqref{eq:0absorb}. However, in the 
mixed case the corresponding projectors $\pjw[v{-}1]$
form a family that is only left-aligned, 
see \fullref{example:left-right}. (Of
course, there are also right-aligned 
versions $\flipv[{(\pjw[v{-}1])}]$.) This
asymmetry will play an important role 
within our setup. For example, in
\fullref{definition:lightladders} we 
presented a version of the light ladders
strategy that favors left-aligned families 
of projectors and when discussing
fusion rules for morphisms this will play 
an important role.
\end{remark}

\section{Additive structure}\label{section:TL-additive}

In this section, we explain the 
additive structure of the category of tilting
modules. Some of results in this 
section are well-known, others generalize the
results from \cite{BuLiSe-tl-char-p} 
and \cite{TuWe-quiver}. We also added a few
new observations.

\subsection{Character formulas}\label{subsection:character-formulas}

The Weyl and dual Weyl modules have classical Weyl characters, {\ie}
$\ch{\wmod(v-1)}=\ch{\dwmod(v-1)}=\qnum{v}{\vpar}$, which we view as elements of
$\Nv$ where the coefficient of $\vpar^{k}$ is the dimensions of the weight space
of weight $k$.
Each $\tmod(v-1)$ has a (dual) Weyl filtration and we denote the \emph{(dual)
Weyl multiplicities} by $\wmult{v}{w}=\wdmult{v}{w}$.

\begin{remark}\label{remark:admissibility} 
The purpose of the admissibility
conditions on finite sets $S\subset \N[0]$ from \fullref{definition:adm} is
so that for $v\in\N$ we have bijections
\begin{align*}
\{S\subset\N[0]\mid S\text{ is down-admissible for } v\}
\to\{w\in\N\mid\wmult{v}{w}=1\},\;&S\mapsto v[S], 
\\
\{S\subset\N[0]\mid S \text{ is up-admissible for } v\}
\to\{w\in\N\mid\wmult{w}{v}=1\},\;&S\mapsto v(S).
\end{align*}
\end{remark}

Moreover, each (dual) Weyl module has a filtration by simple modules, and we denote the
corresponding \emph{simple multiplicities} by $\lmult{v}{w}=\ldmult{v}{w}$.
These have a similar description as the Weyl multiplicities:

\begin{definition}\label{definition:bgg} Let $v=\plbase{a_{j},\dots,a_{0}}$. The
\emph{$\lmod$-support} $\lsupp[v]\subset\N$ is defined as follows.

\begin{enumerate}

\item If $a_{i}\neq 0,\ppar-1$ for all $j>i>0$, then for all $v[S]\in\supp[v]$
we set 
\[
v[S]^{\lmod}=v[S]-2\sum_{S_{i},i>0}\ppar^{\min(S_{i})}
\qquad\text{and}\qquad
\lsupp[v]=\{v[S]^{\lmod}\mid S\}.
\]
\changed{Here $i$ denotes the index for possible non-leading digits.}

\item Otherwise, there is a recursive description of $\lsupp[v]$ as in
\cite[Section 5B]{TuWe-quiver}, working with $\plpar$ instead of $\ppar$. 

\end{enumerate}
\end{definition}

One can check that $\supp[v]$ is always of order $2^{\generation}$, while the size of 
$\lsupp[v]$ \changed{can be different (for example when {\losp} applies and digits are zero).}

\begin{proposition}\label{proposition:multiplicities}
Let $v=\plbase{a_{j},\dots,a_{0}}$.

\begin{enumerate}

\item We have 
\begin{align*}
\wmult{v}{w}&=
\begin{cases}
1,&\text{if }w\in\supp[v],
\\
0,&\text{else},
\end{cases},\\
\lmult{v}{w}&=
\begin{cases}
1,&\text{if }w\in\lsupp[v],
\\
0,&\text{else}.
\end{cases}
\end{align*}
Thus, the tilting characters are
\[
\ch{\tmod(v-1)}={
\sum_{w\in\supp[v]}}\,\ch{\wmod(w-1)}=
{
\sum_{w\in\supp[v]}}\,\qnum{w}{\vpar},
\] 
while the simple characters can be obtained by inverting $\ch{\wmod(v-1)}={\textstyle\sum_{w\in\lsupp[v]}}\,\ch{\lmod(w-1)}$.

\item We have a version of (Brauer--Humphreys or BGG) \emph{reciprocity}, {\ie} if $a_{i}\neq 0,\ppar-1$ for all $j>i>0$, then
\begin{gather*}
\wmult{v}{w}=\lmult{v}{w^{\lmod}}
=
\begin{cases}
1,&\text{if }w=v[S]\text{ and }w^{\lmod}=v[S]^{\lmod},
\\
0,&\text{if }w\neq v[S]\text{ and }w^{\lmod}\neq v[S]^{\lmod}.
\end{cases}
\end{gather*}
In particular, $\wmult{v}{w}=\lmult{v}{w}$ for $v=\plbase{a,b}$.

\end{enumerate}
\end{proposition}

\begin{proof}
The Weyl multiplicities are known, see \cite[Section 3.4]{Do-q-schur} for the
potentially first written account in the mixed case. The simple
multiplicities can be obtained by direct calculation using the simple characters
in \eqref{eq:steinberg} below. The reciprocity follows immediately from these.
\end{proof}

Two remarkable results describing the structure of objects in $\allmod$ are
Donkin's \eqref{eq:donkin} and Steinberg's \eqref{eq:steinberg} \emph{tensor
product formulas}, which we recall in the following proposition. Both formulas
describe modules of highest weight $v-1$ in terms of tensor products of
\emph{Frobenius--Lusztig twist} of modules of lower weight, following the
$\ppar\lpar$-adic expansion of $v$. The $i$th Frobenius--Lusztig twist will
be denoted by $(\placeholder)^{\ppar^{(i)}}$. It acts as the Frobenius twist on
digits $a_{i}$ for $i>0$ and as its quantum analog on the zeroth digit.
Furthermore, we will accompany the two famous tensor product formulas with a
third one. To this end, we note that we can naively apply
$(\placeholder)^{a_{i}\ppar^{(i)}}$ to weight spaces, although we loose the
module structure for $a_{i}\neq 1$.

\begin{proposition}\label{proposition:donkin-etc}
Let $v=\plbase{a_{j},\dots,a_{0}}$ and $v-1=\plbase{b_{j},\dots,b_{0}}$.

\begin{enumerate}

\item We have
\begin{gather}\label{eq:donkin}
\tmod(v-1)\cong
\tmod(a_{j}-1)^{\ppar^{(j)}}\hcirc{\textstyle\bigotimes_{a_{i}}}\,\tmod(a_{i}+\plpar-1)^{\ppar^{(i)}},
\end{gather}
where the monoidal product runs over all non-leading digits of $v$. 
Thus, $\ch{\tmod(v-1)}=\qnum{a_{j}}{\vpar^{\ppar^{(j)}}}{\textstyle\prod_{a_{i}}}\,
\big(\qnum{a_{i}+\plpar}{\vpar^{\ppar^{(i)}}}+\qnum{-a_{i}+\plpar}{\vpar^{\ppar^{(i)}}}\big)$.

\item We have
\begin{gather}\label{eq:steinberg}
\lmod(v-1)\cong
{\textstyle\bigotimes_{b_{i}\neq 0}}\,\lmod(b_{i})^{\ppar^{(i)}},
\end{gather}
where the monoidal product runs over all (non-zero) digits of $v-1$. Thus,
$\ch{\lmod(v-1)}={\textstyle\prod_{b_{i}\neq
0}}\,\qnum{b_{i}+1}{\vpar^{\ppar^{(i)}}}$.

\item We have an isomorphism of $\Z$-graded vector spaces
\begin{gather}\label{eq:partial-trace-result}
\tmod(v-1)\cong
\tmod(\motherr{v}{\infty}-1)\hcirc{\textstyle\bigotimes_{a_{i}\neq 0}}\,\tmod(1)^{(a_{i}\ppar^{(i)})},
\end{gather}
where the monoidal product runs over all non-zero and non-leading digits of $v$. 
Thus, $\ch{\tmod(v-1)}=\qnum{\motherr{v}{\infty}}{\vpar}{\textstyle\prod_{a_{i}\neq 0}}\,\qnum{2}{\vpar^{a_{i}\ppar^{(i)}}}$.

\item \eqref{eq:partial-trace-result} can be realized as an isomorphism of
$\SLtwo$-modules if all non-zero digits $a_{i}$ are equal to $1$. In this case
$\tmod(v-1)$ is a tensor product of simple modules.

\end{enumerate}
\end{proposition}

\begin{proof}
For the tensor product formulas \eqref{eq:donkin} and \eqref{eq:steinberg} see
\cite[Proposition 5.2]{An-simple-tl} (to be precise, the above is
\cite[Proposition 4.7]{TuWe-center} adjusted to mixed characteristic) and
\cite[Theorem 1.10]{AnWe-mixed-qgroup} for the mixed versions. We will give a
diagrammatic proof of the (apparently new) character formula in (c) in
\fullref{proposition:qdim} below. For the final statement, by (c), it suffices
to observe that ${\textstyle\bigotimes_{a_{i}\neq
0}}\,\tmod(1)^{(a_{i}\ppar^{(i)})}$ is simple by \eqref{eq:steinberg}, which
implies that the right-hand side of \eqref{eq:partial-trace-result} is tilting
by the mixed characteristic analog of \cite[Lemma 3.3]{BeEtOs-pverlinde}.
\end{proof}

\begin{example}
Recall from \fullref{example:support} that $\supp[{68=\pbase{3,1,2}{7,3}}]=\{68,64,62,58\}$.

\begin{enumerate}	

\item For $68$ we get $\ch{\tmod(68-1)}=\qnum{68}{\vpar}
+\qnum{64}{\vpar}+\qnum{62}{\vpar}+\qnum{58}{\vpar}$, as well as 
$\ch{\tmod(68-1)}=\qnum{3}{\vpar^{21}}(\qnum{8}{\vpar^{3}}+\qnum{6}{\vpar^{3}})(\qnum{5}{\vpar}+\qnum{1}{\vpar})$ 
and $\ch{\tmod(68-1)}=\qnum{63}{\vpar}\qnum{2}{\vpar^{3}}\qnum{2}{\vpar^{2}}$.

\item From $\supp[68]$ we obtain $\lsupp[68]=\{68,64,58,48\}$, since we need to adjust
$62=\pbase{3,-1,2}{7,3}$ to $62-2\cdot 7=48$. We thus 
get $\ch{\lmod(68-1)}=\qnum{4}{\vpar^{21}}\qnum{2}{\vpar^{3}}\qnum{2}{\vpar}$, using $67=\pbase{3,1,1}{7,3}$, 
and $\ch{\lmod(68-1)}=\qnum{68}{\vpar}-\ch{\lmod(64-1)}-\ch{\lmod(58-1)}-\ch{\lmod(48-1)}$.

\end{enumerate}

\end{example}

Note that (d) of \fullref{proposition:donkin-etc} implies 
a remarkable appearance of {\losp}:

\begin{corollary}
All indecomposable tilting modules 
are tensor products of simple modules in characteristic $\ppar=2$.
\end{corollary}

The tilting modules $\tmod(v-1)$ for $v\in\eve$ will also be called \emph{eves}.
By \fullref{proposition:multiplicities} these are the only simple tilting
modules, {\ie} $\tmod(v-1)\cong\lmod(w-1)$ if and only of $v=w\in\eve$. The
\emph{prime eves} are those where $v=\ppar^{(i)}$, and they play special role \changed{in the theory}.

\subsection{Non-semisimple projectors}\label{subsection:projectors}

In order to define the projectors $\pjw[v{-}1]$, we need a few notions. A crucial role will be played by certain down and up morphisms that are defined in the same spirit as
those in \fullref{definition:lightladders}, but with an emphasis on good
compositional properties.

\begin{definition}\label{definition:jw-cupscaps} 
Fix a left-aligned family of
morphisms $\morstuff{G}_{v{-}1}\in\End_{\TL[{\kk,\qpar}]}(v-1)$ 
for $v\in\N$
as in \fullref{definition:lightladders}. 
Let $v=\plbase{a_{j},\dots,a_{0}}$
and $0\leq i<j$ with $a_{i}\neq 0$. Consider the ancestors
$v^{\prime}=\plbase{a_{j},\dots,a_{i},0\dots,0}$ and
$v^{\prime\prime}=\plbase{a_{j},\dots,a_{i+1},0\dots,0}$ 
as well as the difference
$x=v-v^{\prime}=\plbase{a_{i-1},\dots,a_{0}}$. 
\changed{Then we define morphisms in
$\delta_{i}\idtl,\idtl\upsilon_{i}\in\TL[{\kk,\qpar}]$ as follows.}
\begin{gather*}
\delta_{i}\idtl
=
\begin{tikzpicture}[anchorbase,scale=0.25,tinynodes]
\draw[morl] (0,0) rectangle (-4,1);
\node at (0,1.85) {\scalebox{0.9}{$a_{i}\ppar^{(i)}$}};	
\draw[usual] (2,0) to (2,3) node[above,yshift=-2pt]{$x$};
\draw[usual] (-1.5,1) to(-1.5,2) to[out=90,in=180] (0,3) to[out=0,in=90]  (1.5,2) to (1.5,0);
\draw[usual] (-2,1) to (-2,3)node[above,yshift=-2pt]{$~$};
\node at (0,-.25) {$\phantom{.}$};
\node at (0,3.25) {$\phantom{.}$};
\end{tikzpicture}
,\quad
\idtl\upsilon_{i}
=
\fliph[(\delta_{i}\idtl)]
.
\end{gather*}
The box represents the morphisms 
$\morstuff{G}_{v^{\prime\prime}-1}$, and we will
consider the three variations with their corresponding notation (meaning $\morstuff{d}_{i}$, $\tilde{\morstuff{d}}_{i}$, $\Down{i}$) that were
introduced in \fullref{convention:du}. 

Similarly, if $S=\{s_{k}>\cdots>s_{1}>s_{0}\}$ is a down-admissible stretch
for $v$, then we define
\begin{gather}\label{Eq:Delta}
\delta_{S}\idtl
=
\idtl[{v[S]{-}1}]\delta_{s_{0}}\cdots\delta_{s_{k}}\idtl
=
\begin{tikzpicture}[anchorbase,scale=0.25,tinynodes]
\draw[morl] (0,0) rectangle (-4,1);
\node at (0,1.85) {\scalebox{0.9}{$S$}};	
\draw[usual] (2,0) to (2,3);
\draw[usual] (-1.5,1) to(-1.5,2) to[out=90,in=180] (0,3) to[out=0,in=90]  (1.5,2) to (1.5,0);
\draw[usual] (-2,1) to (-2,3);
\node at (0,-.25) {$\phantom{.}$};
\node at (0,3.25) {$\phantom{.}$};
\end{tikzpicture}
.
\end{gather}
For the final equation we have used that the morphisms
$\morstuff{G}_{v{-}1}$ form a left-aligned family. 

\ochanged{Although we do not draw the corresponding up morphisms, 
we define them symmetrically using the down morphisms above.  The corresponding upwards
version of these morphisms are defined  by $\upsilon_{S}=\fliph[\delta_{S}]$.}

We will also use the case $S=\emptyset$ for which all involved operations are
identities.\changed{ In \fullref{definition:jw-cupscaps2} we will extend these
definitions to down- resp.\ up-admissible sets $S$.}.
\end{definition}

\begin{example}
\changed{Note that $\delta_{i}$ (and likewise $\upsilon_{i}$) itself does not
specify a well-defined morphism in $\TL[{\kk,\qpar}]$; we need to include
information about the (co-)domain, {\eg} by including an idempotent in the
notation as in $\delta_{i}\idtl$. Note that the morphisms $\delta_{s_{i}}$ in
\eqref{Eq:Delta} are composed in an order that leads to a nested configuration
of caps. A concrete example of this can be seen in \fullref{example:trapezes},
namely the bottom right part of the morphism $\qjwm[\{1,0\}]$. }
\end{example}

For simplicity of notation we often only indicate the number of strands at the
beginning or end of a composite of such morphisms, since the other numbers are
then determined.

\begin{definition}\label{definition:jw-cupscaps2}
Suppose that $S=\{s_{k}>\cdots>s_{1}>s_{0}\}$ is down-admissible for $v$ 
and $S^{\prime}=\{s_{l}^{\prime}>\cdots>s_{1}^{\prime}>s_{0}^{\prime}\}$ is 
up-admissible for $v$. Then we define \emph{simple trapezes and loops}
\begin{gather*}
\begin{gathered}
\TRD{$\qjwm[S]$}\changed{:}=
\qjw[{v[S]{-}1}]\downo{S}
\changed{:}=\qjw[{v[S]{-}1}]\downo{s_{0}}\cdots\downo{s_{k}}\idtl 
,\quad
\TRU{$\qjwm[S^{\prime}]$}\changed{:}=
\upo{S^{\prime}}\qjw[{v{-}1}]\changed{:}=\idtl[{v(S^{\prime})-1}]\upo{s_{l}^{\prime}}\cdots\upo{s_{0}^{\prime}}\qjw[{v{-}1}],
\\
\TR{$\qjwm[S]$}\changed{:}=\loopdowngen{S}{v{-}1}\changed{:}=\upo{S}\qjw[{v[S]{-}1}]\downo{S},
\end{gathered}
\end{gather*}
which we also define for the other \changed{two} variations from \fullref{convention:du},
with the appropriate adjustment of notation.
\end{definition}

Note that in all cases loops carry an
idempotent $\morstuff{G}_{v[S]-1}$ in the center and down and up morphisms
carry this idempotent on their thin end.

\begin{remark}
Using \fullref{convention:du}, we can give an alternative description of the
simple trapezes. If $S$ is down-admissible for
$v=\plbase{a_{j},\dots,a_{0}}$, then we define a sign sequence
\begin{gather*}
\pi_{S}(v)= 
\underbrace{\dots+\dots+}_{a_{j}\ppar^{(j)}-1}
\underbrace{\epsilon_{j-1}\dots\epsilon_{j-1}}_{a_{j-1}\ppar^{(j-1)}}\dots
\underbrace{\epsilon_{0}\dots\epsilon_{0}}_{a_{0}\ppar^{(0)}}
\in\{+,-\}^{v-1}
,\quad 
\epsilon_{i}=
\begin{cases}
- & \text{ if } i\in S,
\\
+ & \text{ if } i\notin S.
\end{cases}
\end{gather*}
\changed{That is, $\pi_{S}(v)$ is a concatenation of signs 
with multiplicity given by the digits of $v$.}
We get
\begin{gather*}
\TRD{$\qjwm[S]$}=\tilde{\morstuff{d}}(\pi_{S}(v))
,\quad
\TRU{$\qjwm[S]$}=\tilde{\morstuff{u}}(\pi_{S}(v)). 
\end{gather*}
Note the subtle, but
important difference that $\tilde{\morstuff{d}}(\pi_{S}(v))$ includes an
idempotent $\qjw[{v[S]{-}1}]$ on the left, while $\downo{S}$ does not.
As a consequence, composites of morphisms of type
$\tilde{\morstuff{d}}(\pi)$ are automatically zero, while the morphisms
of type $\downo{S}$ can be composed in interesting ways. Remarkably,
this distinction disappears when considering analogs of such morphisms
built from mixed projectors, see \fullref{proposition:jw-properties}.
\end{remark}

\begin{example}\label{example:trapezes}
For $v=\plbase{a,b,c}$ we have: 
\begin{gather*}
\TRD{$\qjwm[\emptyset\,]$}
=
\begin{tikzpicture}[anchorbase,scale=0.25,tinynodes]	
\draw[JW] (-1.5,1) rectangle (1.5,2);
\draw[JW] (-1.5,-1) rectangle (1.5,-2);
\draw[usual] (0,-1) to (0,1);
\node at (0,-2.3) {$\phantom{a}$};
\end{tikzpicture}
,\quad
\TRD{$\qjwm[\{0\}]\,$}
=
\begin{tikzpicture}[anchorbase,scale=0.25,tinynodes]
\draw[JW] (-1.5,1) rectangle (1.5,2);
\draw[JW] (-1.5,-1) rectangle (1.5,-2);
\draw[usual] (0,-1) to (0,1);
\draw[usual] (1,-1) to[out=90,in=180] (1.5,-0.5) 
to[out=0,in=90] (2,-1) to (2,-2) node[below,yshift=0pt]{$c$};
\node at (0,-2.1) {$\phantom{a}$};
\end{tikzpicture}
,\quad
\TRD{$\qjwm[\{1\}]\,$}
=
\begin{tikzpicture}[anchorbase,scale=0.25,tinynodes]
\draw[JW] (-1.5,1) rectangle (1.5,2);
\draw[JW] (-1.5,-1) rectangle (1.5,-2);
\draw[usual] (0,-1) to (0,1);
\draw[usual] (1,-1) to[out=90,in=180] (1.5,-0.5) 
to[out=0,in=90] (2,-1) to (2,-2) node[below,yshift=0pt]{\scalebox{0.75}{$b\ppar^{(1)}$}};
\draw[usual] (1,1) to[out=270,in=90] (3,-0.5) to (3,-2) node[below,yshift=0pt]{$c$};
\node at (0,-2.1) {$\phantom{a}$};
\end{tikzpicture}
,\quad
\TRD{$\qjwm[\{1,0\}]\,$}
=
\begin{tikzpicture}[anchorbase,scale=0.25,tinynodes]
\draw[JW] (-1.5,1) rectangle (1.5,2);
\draw[JW] (-1.5,-1) rectangle (1.5,-2);
\draw[usual] (-1,-1) to (-1,1);
\draw[usual] (1,-1) to[out=90,in=180] (1.5,-0.5) 
to[out=0,in=90] (2,-1) to (2,-2) node[below,yshift=0pt]{\scalebox{0.75}{$b\ppar^{(1)}$}};
\draw[usual] (0,-1) to[out=90,in=180] (1.5,0) 
to[out=0,in=90] (3,-1) to (3,-2) node[below,yshift=0pt]{$c$};
\node at (0,-2.1) {$\phantom{a}$};
\end{tikzpicture}
.
\end{gather*}
\end{example}

Recall that we use $v=\plbase{a_{j},\dots,a_{0}}$ and write
$\ppar^{(i)}=\ppar^{i-1}\lpar$ for $i>0$. For $v\in\N$ and $s\in\N[0]$ let
$\fancest{v}{s}$ denote the youngest ancestor of $v$ whose $s$th digit is zero.
\changed{(When $s=-1$, we define $\fancest{v}{{-}1}=v$.)} For each down-admissible $S$ for $v$ we
let 
\begin{gather*}
\lambda_{v,S}
={\textstyle\prod_{s\in S}}\,
(-1)^{a_{s}\ppar^{(s-1)}}\tfrac{\qnum{\fancest{v}{s{-}1}[S]}{\vpar}}{\qnum{\fancest{v}{s}[S]}{\vpar}}
\in\kkv.
\end{gather*}

\fullref{definition:pl-jw-q} generalizes \cite[Definition 2.22]{TuWe-quiver} and \fullref{lemma:pl-jw-q} below 
generalizes \cite[Proposition 3.3]{BuLiSe-tl-char-p} to the mixed case.

\begin{definition}\label{definition:pl-jw-q} 
For $v\in\N$ the
\emph{semisimple $\ppar\lpar$JW projector}
$\pqjw[v{-}1]\in\End_{\TL[{\kkv,\vpar}]}(v-1)$ is defined to be
\begin{gather}\label{eq:pljwdef}
\begin{tikzpicture}[anchorbase,scale=0.25,tinynodes]
\node[pQJW,label={[yshift=-0.5cm]$\pqjwm[v{-}1]$}] at (0,0){};
\end{tikzpicture}
=
\pqjw[v{-}1]
=
{\textstyle\sum_{v[S]\in\supp[v]}}\,
\lambda_{v,S}\,\loopdowngen{S}{v{-}1}
=
{\textstyle\sum_{v[S]\in\supp[v]}}\,
\lambda_{v,S}\cdot\TR{$\qjwm[S]$}.
\end{gather}
\end{definition}

The choice of name for $\pqjw[v{-}1]$ is because the associated tilting 
module is a direct sum of simple tilting modules, {\ie}
\begin{gather}\label{eq:tilting-semisimple}
\overline{\tmod}(v-1)\cong
{\textstyle\bigoplus_{v[S]\in\supp[v]}}\,\tmod(v[S]-1)
\cong
{\textstyle\bigoplus_{v[S]\in\supp[v]}}\,\wmod(v[S]-1).
\end{gather}

Note that $\mchar[{\kkv,\vpar}]=(\ppar,\infty)$, so \eqref{eq:pljwdef} is well-defined.
In \fullref{theorem:well-defined} we will see that the 
semisimple $\ppar\lpar$JW
projectors can be base changed to $(\kk,\qpar)$ with
$\mchar=(\ppar,\lpar)$.

\begin{example}\label{example:left-right}
By construction, $\fliph[{(\pqjw[v{-}1])}]=\pqjw[v{-}1]$. However, 
$\flipv[{(\pqjw[v{-}1])}]\neq\pqjw[v{-}1]$ in general (we will address 
this in \fullref{lemma:left-right}), as can be seen by
\begin{gather*}
\begin{tikzpicture}[anchorbase,scale=0.25,tinynodes]
\draw[pQJW] (0,-1) rectangle (4,1);
\node at (2,-0.1) {$\pqjwm[3]$};
\end{tikzpicture}
\;\;=\;\;
\begin{tikzpicture}[anchorbase,scale=0.25,tinynodes]
\draw[usual] (0,0) to (0,2);
\draw[usual] (2,0) to (2,2);
\draw[usual] (4,0) to (4,2);
\end{tikzpicture}
+\tfrac{1}{2}\cdot
\begin{tikzpicture}[anchorbase,scale=0.25,tinynodes]
\draw[usual] (0,0) to[out=90,in=180] (1,0.75) to[out=0,in=90] (2,0);
\draw[usual] (0,2) to[out=270,in=180] (1,1.25) to[out=0,in=270] (2,2);
\draw[usual] (4,0) to (4,2);
\end{tikzpicture}
\;\;\neq\;\;
\begin{tikzpicture}[anchorbase,scale=0.25,tinynodes]
\draw[usual] (0,0) to (0,2);
\draw[usual] (2,0) to (2,2);
\draw[usual] (4,0) to (4,2);
\end{tikzpicture}
+\tfrac{1}{2}\cdot
\begin{tikzpicture}[anchorbase,scale=0.25,tinynodes]
\draw[usual] (0,0) to (0,2);
\draw[usual] (2,0) to[out=90,in=180] (3,0.75) to[out=0,in=90] (4,0);
\draw[usual] (2,2) to[out=270,in=180] (3,1.25) to[out=0,in=270] (4,2);
\end{tikzpicture}
\;\;=\;\;
\begin{tikzpicture}[anchorbase,scale=0.25,tinynodes]
\draw[pQJWl] (0,-1) rectangle (4,1);
\node at (2,-0.1) {$\pqjwm[3]$};
\end{tikzpicture}
\end{gather*}
which is an example in characteristic $\ppar=3$. 
\end{example}

\begin{remark}
Further concrete examples for projector expansions \eqref{eq:pljwdef} can be found in
\cite[Examples 2.20 and 2.23]{TuWe-quiver}. Relative to the treatment there, we
allow the following two generalizations. First, the $\ppar$-adic expansions should be replaced by the
$\ppar\lpar$-adic expansions, and $\ppar^{k}$ therein by $\ppar^{(k)}$.
Second, all coefficients use quantum numbers instead of integers.
\end{remark}

The following is reproduced from \cite[Section 3B]{TuWe-quiver}, adjusting 
the scalars.

\begin{lemma}\label{lemma:trapezes}
\leavevmode

\begin{enumerate}
\item Suppose that $S$ and $S^{\prime}$ are down-admissible for $v$. Then we have 
\begin{gather*}
\qjw[{v[S]{-}1}]\downo{S}\upo{S^{\prime}}\qjw[{v[S^{\prime}]{-}1}] 
=
\begin{tikzpicture}[anchorbase,tinynodes]
\draw[JW] (-0.15,0.3) rectangle (-1.25,0.6);
\node at (-0.7,0.4) {$\qjwm[{v[S]{-}1}]$};
\draw[JW] (-0.15,-0.3) rectangle (-1.25,-0.6);
\node at (-0.7,-0.5) {\scalebox{0.875}{$\qjwm[{v[S^{\prime}]{-}1}]$}};
\trd{1.325}{0.3}{$\qjwm[S]$}{-0.075}{0.3}
\tru{1.325}{0.3}{\scalebox{0.9}{$\qjwm[S^{\prime}]$}}{-0.075}{-0.3}
\end{tikzpicture}
=
\delta_{S,S^{\prime}}\lambda^{-1}_{v,S^{\prime}}\cdot
\begin{tikzpicture}[anchorbase,tinynodes]
\draw[JW] (0.15,-0.3) rectangle (1.25,0.3);
\node at (0.7,-0.05) {$\qjwm[{v[S^{\prime}]{-}1}]$};
\end{tikzpicture}
=
\delta_{S,S^{\prime}}\lambda^{-1}_{v,S^{\prime}}\qjw[{v[S^{\prime}]{-}1}].
\end{gather*}
\changed{Here $\delta_{S,S^{\prime}}$ is the Kronecker delta, so $\delta_{S,S^{\prime}}=1$ if $S=S^{\prime}$ and zero otherwise.}

\item Suppose $S$ is down-admissible for $v$, 
and $S^{\prime}=\{s,\dots,s^{\prime}-1\}$ is a minimal 
down-admissible stretch for $v$. Then we have 
\begin{gather*}
\begin{tikzpicture}[anchorbase,tinynodes]
\tru{1.2}{0.6}{$\qjwm[S]$}{0}{0}
\draw[usual] (-0.125,0.6) 
to[out=90,in=0] (-0.25,0.725) node[above,yshift=-2pt]{$S^{\prime}$} to[out=180,in=90] (-0.375,0.6);
\end{tikzpicture}
=	
\begin{cases}
(-1)^{a_{s}\ppar^{(s)}}
\tfrac{\qnum{\fancest{v}{s}[S]}{\vpar}}{\qnum{\fancest{v}{s{-}1}[S]}{\vpar}}\cdot
\TRU{\,$\qjwm[S{\setminus}S^{\prime}]$}
& \text{ if }s\in S,s^{\prime}\notin S,
\\[5pt]
\TRU{\,$\qjwm[S\cup S^{\prime}]$}
& \text{ if }s\notin S,s^{\prime}\in S,
\\[5pt]
0 & \text{ otherwise}.
\end{cases}
\end{gather*}

\item Suppose that $S^{\prime}=\{s,\dots,s^{\prime}-1\}$ is the smallest 
minimal down-admissible stretch for $v$ and let $S$ be down-admissible for $\fancest{v}{s}=\mother$. Then we have
\begin{gather*}
\begin{tikzpicture}[anchorbase,scale=0.3,tinynodes]
\tr{2.4}{0.6}{\raisebox{-0.05cm}{$\qjwm[S]$}}{0}{0}
\draw[usual] (0.6,-0.6) to (0.6,0.6) 
to[out=90,in=0] (0,1.2) node[above,yshift=-2pt]{$S^{\prime}$} to[out=180,in=90] (-0.6,0.6);
\node at (-1,1.2) {$\phantom{a}$};
\node at (-1,-0.6) {$\phantom{a}$};
\end{tikzpicture}	
=
\begin{cases}
\TRUD{$\qjwm[S]$}{\scalebox{0.875}{$\qjwm[S\cup S^{\prime}]$}}
=
\upo{S}\qjw[{v[S\cup S^{\prime}]{-}1}]\downo{S\cup S^{\prime}} 
& \text{ if }s^{\prime}\notin S,
\\[5pt]
\TRUD{\scalebox{0.875}{$\qjwm[S\cup S^{\prime}]$}}{$\qjwm[S]$}
=
\upo{S\cup S^{\prime}}\qjw[{v[S]{-}1}]\downo{S}
& \text{ if }s^{\prime}\in S.
\end{cases}
\end{gather*}
\end{enumerate}
\end{lemma}

\begin{proof}
Word-by-word as in \cite[Lemmas 3.7, 3.8 and 3.9]{TuWe-quiver}.
\end{proof}

\begin{lemma}\label{lemma:pl-jw-q}  
The semisimple $\ppar\lpar$JW projectors can be expanded as
\begin{gather}\label{eq:recursion-formula}
\begin{tikzpicture}[anchorbase,scale=0.25,tinynodes]
\draw[pQJW] (-1.5,0) rectangle (1.5,2);
\node at (0,0.9) {$\qjwm[v{-}1]$};
\end{tikzpicture}
=
{\textstyle\sum_{\mother[v][S]\in\supp[{\mother[v]}]}}\,
\lambda_{\mother[v],S}\,
\left(
\begin{tikzpicture}[anchorbase,tinynodes]
\tru{1.65}{0.3}{$\qjwm[S]$}{0}{0.2}
\draw[JW] (0.3,0.2) rectangle (-1.65,-0.2);
\trd{1.65}{0.3}{$\qjwm[S]$}{0}{-0.2}
\node at (-0.675,-0.05) {$\qjwm[{v[S]{-}1}]$};
\draw[usual] (0.15,0.2) to (0.15,0.5) node[above,yshift=-2pt]{$a_{s}\ppar^{(s)}$};
\draw[usual] (0.15,-0.2) to (0.15,-0.5) node[below,yshift=-3pt]{$a_{s}\ppar^{(s)}$};
\end{tikzpicture}
+(-1)^{a_{s}\ppar^{(s)}}\tfrac{\qnum{v[S][s]}{\vpar}}{\qnum{\mother{[S]}}{\vpar}}
\cdot
\begin{tikzpicture}[anchorbase,tinynodes]
\tru{1.65}{0.3}{$\qjwm[S]$}{0}{0.2}
\draw[JW] (-0.45,-0.2) rectangle (-1.65,0.2);
\trd{1.65}{0.3}{$\qjwm[S]$}{0}{-0.2}
\node at (-1.05,-0.05) {$\qjwm[{v[S][s]{-}1}]$};
\draw[usual] (-0.3,0.2) to[out=270,in=180] (-0.075,0.075) 
to[out=0,in=270] (0.15,0.2) to (0.15,0.5) node[above,yshift=-2pt]{$a_{s}\ppar^{(s)}$};
\draw[usual] (-0.3,-0.2) to[out=90,in=180] (-0.075,-0.075) 
to[out=0,in=90] (0.15,-0.2) to (0.15,-0.5) node[below,yshift=-3pt]{$a_{s}\ppar^{(s)}$};
\end{tikzpicture}
\right),
\end{gather}
where $a_{s}$ is the first non-zero digit of $v$. As a consequence, for any
ancestor $\motherr{v}{j}$ of $v$, the projector $\pqjw[v{-}1]$ absorbs
$\pqjw[\motherr{v}{j}{-}1]$ when left-aligned as in
\fullref{definition:left-right}.
\end{lemma}

\begin{proof}
With the properties listed in \fullref{lemma:trapezes}, 
the proof follows {\ver} as in \cite[Lemma 2.24]{TuWe-quiver} and \cite[Proposition 3.3]{BuLiSe-tl-char-p}. 
\end{proof}

The next statement of this section enables us to relate the left and right
version of the $\ppar\lpar$JW projectors.
Let $v=\plbase{a_{j},\dots,a_{0}}$, as
usual, and let $\setstuff{S}_{v-1}$ 
denote the symmetric group on $v-1$ letters.
Assuming the existence of square roots, 
we can use \eqref{eq:kauffman} to define
$\morstuff{g}=\morstuff{g}(v-1)$ to be the positive braid lift of the longest
element of $\setstuff{S}_{v-1}$ (the positive half twist, a Garside element),
and $\morstuff{r}=\morstuff{r}(a_{j},\dots,a_{0})$ to be the positive braid lift of a shortest coset representative 
for $\setstuff{S}_{v-1}/(\setstuff{S}_{a_{j}\ppar^{(j)}-1}\times
\setstuff{S}_{a_{j-1}\ppar^{(j-1)}}\times\dots\times
\setstuff{S}_{a_{0}})$. 

\begin{lemma}\label{lemma:left-right}
Assume that $\vpar$ has a square root\changed{, {\ie} we are working in $\kkvs$,} we have
\begin{gather*}
\begin{tikzpicture}[anchorbase,tinynodes]
\draw[mor] (-0.15,0.6) rectangle (-1.25,0.9);
\node at (-0.7,0.7) {\scalebox{0.8}{$\morstuff{g}^{-1}$}};
\draw[pQJW] (-0.15,0.3) rectangle (-1.25,0.6);
\node at (-0.7,0.4) {$\pqjwm[{v{-}1}]$};
\draw[mor] (-0.15,0) rectangle (-1.25,0.3);
\node at (-0.7,0.1) {$\morstuff{g}$};
\end{tikzpicture}
=
\begin{tikzpicture}[anchorbase,tinynodes]
\draw[mor] (-0.15,0.6) rectangle (-1.25,0.9);
\node at (-0.7,0.65) {$\morstuff{r}^{-1}$};
\draw[pQJW] (-0.15,0.3) rectangle (-1.25,0.6);
\node at (-0.7,0.4) {$\pqjwm[{v{-}1}]$};
\draw[mor] (-0.15,0) rectangle (-1.25,0.3);
\node at (-0.7,0.1) {$\morstuff{r}$};
\end{tikzpicture}
=
\begin{tikzpicture}[anchorbase,tinynodes]
\draw[pQJWl] (-0.15,0.3) rectangle (-1.25,0.6);
\node at (-0.7,0.4) {$\pqjwm[{v{-}1}]$};
\end{tikzpicture}
.
\end{gather*}
\end{lemma}

\begin{proof}
A \ochanged{straightforward} consequence of two facts: conjugation with the half twist
$\morstuff{g}$ acts as the involution $\flipv[{(\placeholder)}]$ on the integral
basis (and thus on everything) and JW projectors absorb crossings up to scalars
$\vpar^{\pm 1/2}$.
\end{proof}

The semisimple $\ppar\lpar$JW projectors $\pqjw[v{-}1]$ are defined over
$\F[\ppar](\vpar)$, but the algorithm to construct them generates coefficient
(with respect to the integral basis) that we can view as elements of $\Qv$,
\changed{and we will do this below.} \ochanged{We write 
$\mathrm{sp}_{\ppar,\lpar}(\placeholder)$ for the specialization 
of morphisms to $\kk$, if it exists.}

\begin{theorem}\label{theorem:well-defined}
We have that 
\begin{gather*}
\pjw[v{-}1]:=
\begin{tikzpicture}[anchorbase,scale=0.25,tinynodes]
\draw[pJW] (-1.5,0) rectangle (1.5,2);
\node at (0,0.9) {$\pqjwm[v{-}1]$};
\end{tikzpicture}
:=\mathrm{sp}_{\ppar,\lpar}(\pqjw[v{-}1])
\in\End_{\TL[{\kk,\qpar}]}(v-1)
\end{gather*}
is a well-defined idempotent whose coefficients are elements of $\F(\qpar)$.
Moreover, $\tlfunctor[{\kk,\qpar}](\pjw[v{-}1])=\idmor_{\tmod(v{-}1)}$, and
$\tilt[{\kk,\qpar}]$ is Krull--Schmidt.
\end{theorem}

In particular, under the equivalence induced by $\tlfunctor$, see
\fullref{proposition:TLtilt}, the image of the idempotent $\pjw[v{-}1]$ is
mapped to $\tmod(v-1)$. We call the $\pjw[v{-}1]$ \emph{mixed JW projectors}.

\begin{proof}
\ochanged{We start by explaining the lifting strategy. Assume that we are in the
mixed cases (the other cases are easier and omitted). Let
$\bar{\delta}=-\qnum{2}{\qpar}=-\qpar-\qpar^{-1}$, the circle value. Let
$p(\delta)$ be the minimal polynomial in $\mathbb{F}_{\ppar}[\delta]$ satisfied
by $\bar{\delta}$. We lift $p(\delta)$ to $\Z[\delta]$, and denote this lift by
the same symbol. Let $L$ denote the localization of $\Z[\delta]$ at the maximal
ideal $m=(p(\delta),\ppar)$. (That $m$ is maximal in $\Z[\delta]$ follows since
$(\ppar)$ is maximal in $\Z$ and $p(\delta)$ is irreducible in
$\mathbb{F}_{\ppar}[\delta]$.) Then we have that $mL$ is a maximal ideal in $L$,
and the completion of $L$ is a complete Noetherian local domain $\kklocal$, and
its field of fraction $F$ is a characteristic zero field. The residue field is contained in
$\kk$. This triple $(F,\kklocal,\kk)$ satisfies
$\kk\hookleftarrow\kklocal/m\kklocal\twoheadleftarrow\kklocal\hookrightarrow F$ and then
allows comparison of idempotents similarly as in the $\ppar$-adic case and the
classical theory of idempotent lifts. In fact, by construction of the TL
category and our idempotents, we will work in the subfield $\Q(\delta)$ of $F$,
which we immediately extend to $\Q(\vpar)$ along $\delta =-\vpar-\vpar^{-1}$.}	

\ochanged{As above we assume that we are in the mixed cases 
(the other cases are easier and omitted). Note that 
we could have defined $\pqjw[v{-}1]$ directly in 
$\Q(\delta)$, and the ones for $\Qv$ are scalar extensions of the 
idempotents defined over $\Q(\delta)$.
We can further interprete the idempotents $\pqjw[v{-}1]$ as being in $F$ 
by scalar extension from $\Qv$. We can also 
scalar extend from the residue field to $\kk$ below. and we will do both in the proof to be consistent 
with the above lifting strategy. The situation over $F$ is semisimple as 
$F$ contains $\Q(\delta)$, so the combinatorics is the same as in $\Qv$.} 

Note that $\pqjw[v{-}1]$ has the correct character,
namely $\ch{\mathrm{Im}\pqjw[v{-}1]}=
\ch{\tmod(v-1)}$. By \fullref{lemma:pl-jw-q},
$\pqjw[v{-}1]\in\End_{\tilt[{F}]}
\big(\tmod(1)^{\hcirc(v-1)}\big)$ is an idempotent and it absorbs the tensor
product $\pqjw[\mother{-}1]\hcirc\idtl[{v{-}\mother}]$
of the idempotent for the mother with extra strands. Now we claim there is
exactly one idempotent in
$\End_{\tilt[{F}]}\big(\tmod(1)^{\hcirc(v-1)}\big)$ with this property
and the correct character. To see that this is true let us denote by
$\overline{\tmod}(\mother-1)\in\tilt[{F}]$ the direct sum of Weyl
modules with the correct character. Now $\overline{\tmod}(\mother-1)\hcirc
\tmod(1)^{\hcirc(v{-}\mother)}\in\tilt[{F}]$ contains each Weyl factor
of $\tmod(v-1)$ exactly once, see \fullref{lemma:translation}, so there is
exactly one idempotent in $\End_{\tilt[{F}]}
\big(\tmod(1)^{\hcirc(v-1)}\big)$ with the correct character and absorption
property, and the claim follows.

Now let $\kklocal$ be the completion as described above. 
There are specialization maps and functors
\begin{gather*}
\begin{tikzcd}[ampersand replacement=\&,column sep=3em]
\& \kklocal\ar[dl,swap]\ar[dr] \&
\\
\kk \& \& F
\end{tikzcd}
,\quad
\begin{tikzcd}[ampersand replacement=\&,column sep=-3em]
\& \End_{\tilt[{\kklocal}]}\big(\tmod(1)^{\hcirc(v-1)}\big)
\ar[dl,"\mathcal{F}_{\kk}",swap]
\ar[dr,"\mathcal{F}_{F}"] \&
\\
\End_{\tilt[{\kk}]}\big(\tmod(1)^{\hcirc(v-1)}\big) 
\& \& \End_{\tilt[{F}]}\big(\tmod(1)^{\hcirc(v-1)}\big)
\end{tikzcd}
.
\end{gather*}
Next, we show that
$\pqjw[v{-}1]$ can be lifted to
$\End_{\tilt[{\kklocal}]}\big(\tmod(1)^{\hcirc(v-1)}\big)$ and 
its specialization to $\kk$ projects to 
$\tmod(v-1)$. To this end, we use induction over the ancestry of  
$v$, with the case of $v\in\eve$ being clear since 
$\tmod(v-1)\cong\wmod(v-1)$ in these cases.
So let $\pqjw[\mother]$ be liftable and let 
$l_{F}(\pqjw[\mother])$ 
be its lift. Induction implies
\begin{gather*}
\mathcal{F}_{\kk}
\big(l_{F}(\pqjw[\mother])\big)
=\idmor_{\tmod(\mother-1)}.
\end{gather*}
We further know that $\tmod(v-1)$ is a direct summand of $\tmod(\mother-1)\hcirc
\tmod(1)^{\hcirc(v{-}\mother)}\in\tilt[{\kk}]$, so, there is some
projector
$E^{\prime}\in\End_{\tilt[{\kk}]}\big(\tmod(1)^{\hcirc(v-1)}\big)$
projecting to this summand, which absorbs the idempotent corresponding to the
mother tensor product with strands. By idempotent lifting, {\cf} \cite[Theorem
21.31]{La-first-non-commutative-rings}, we can pull $E^{\prime}$ back to
$\tilt[{\kklocal}]$ giving us another projector
$l_{\kk}(E^{\prime})$. Pushing this forward gives a projector
$\mathcal{F}_{F}\big(l_{\kk}(E^{\prime})\big)$ in
the semisimple case with the correct character and absorption property. However, as
we have seen, such a projector is unique and thus, has to be $\pqjw[v{-}1]$.
Hence, we get that $l_{\kk}(E^{\prime})$
is a lift of $\pqjw[v{-}1]$.

Thus, we can specialize $\pqjw[v{-}1]$ to $\pjw[v{-}1]=E^{\prime}$, and the
claims about the coefficients and
$\pjw[v{-}1]=\idmor_{\tmod(v{-}1)}$ follow by
construction of $\pqjw[v{-}1]$. The Krull--Schmidt property then follows
inductively as the above constructs all highest weight projectors\ochanged{, and the claim about 
the coefficients being in $\F(\qpar)$ is evident by the construction of the projectors.}
\end{proof}

Diagrammatically, the three types of projectors are distinguished 
as follows:
\begin{gather}\label{eq:danger}
\qjw[v{-}1]=
\begin{tikzpicture}[anchorbase,scale=0.25,tinynodes]
\draw[JW] (-1.5,0) rectangle (1.5,2);
\node at (0,0.9) {$\qjwm[v{-}1]$};
\end{tikzpicture}
,\quad
\pqjw[v{-}1]=
\begin{cases*}
\begin{tikzpicture}[anchorbase,scale=0.25,tinynodes]
\draw[pQJW] (-1.5,0) rectangle (1.5,2);
\node at (0,0.9) {$\pqjwm[v{-}1]$};
\end{tikzpicture}
\\
\begin{tikzpicture}[anchorbase,scale=0.25,tinynodes]
\draw[JW] (-1.5,0) rectangle (1.5,2);
\node at (0,0.9) {$\pqjwm[v{-}1]$};
\end{tikzpicture}\text{ for }v\in\eve,
\end{cases*}
\quad
\pjw[v{-}1]=
\begin{cases*}
\begin{tikzpicture}[anchorbase,scale=0.25,tinynodes]
\draw[pJW] (-1.5,0) rectangle (1.5,2);
\node at (0,0.9) {$\pqjwm[v{-}1]$};
\end{tikzpicture}
\\
\begin{tikzpicture}[anchorbase,scale=0.25,tinynodes]
\draw[JW] (-1.5,0) rectangle (1.5,2);
\node at (0,0.9) {$\pqjwm[v{-}1]$};
\end{tikzpicture}\text{ for }v\in\eve.
\end{cases*}
\end{gather}
The middle and the rightmost projector have the same character, but
$\pqjw[v{-}1]$ corresponds to a direct sum of simple tilting modules in the semisimple
setting, {\cf} \eqref{eq:tilting-semisimple}, and $\pjw[v{-}1]$ corresponds to the indecomposable $\tmod(v-1)$. We
will use the middle projectors to deduce properties of the right projectors.
Moreover, as illustrated in \eqref{eq:danger}, we also use white boxes for eves
to indicate that these satisfy the same diagrammatic properties as the simple
JW projectors.

We stress again that the non-semisimple projectors do not have a
left-right-symmetry, and their properties do not have such a symmetry either.
For the remainder of the paper, each cup and cap in the illustrations is a parallel bundle 
of cups and caps, depending on $S$ respectively $S^{\prime}$, or a plain 
number. (We also omit to illustrate these if no confusion can arise.)

We have the following generalizations of 
\fullref{proposition:generic-jw-properties}, called 
\emph{classical absorption}, \emph{non-classical absorption}, 
\emph{shortening} and \emph{partial trace}.

\begin{proposition}\label{proposition:jw-properties}
\leavevmode
\begin{enumerate}

\item The projectors $\pqjw[v{-}1]$ form a left-aligned family in the sense of \fullref{definition:left-right}.

\item Let $S$ be a
down-admissible stretch for $v$. Then we have
\begin{gather*}
\begin{tikzpicture}[anchorbase,scale=0.25,tinynodes]
\draw[pQJW] (2,0.5) rectangle (-2,-0.5);
\draw[pQJW] (2,2.5) rectangle (-2,3.5);
\node at (0,-0.2) {$\pjwm[v{-}1]$};
\draw[usual] (-0.5,0.5) to[out=90,in=180] (0,1)node[above,yshift=-0.1cm]{$S$} to[out=0,in=90] (0.5,0.5);
\draw[usual] (-1,0.5) to (-1,2.5);
\draw[usual] (1,0.5) to (1,2.5);
\node at (1,3.5) {$\phantom{a}$};
\node at (1,-0.5) {$\phantom{a}$};
\end{tikzpicture}
=
\begin{tikzpicture}[anchorbase,scale=0.25,tinynodes]
\draw[pQJW] (2,0.5) rectangle (-2,-0.5);
\node at (0,-0.2) {$\pjwm[v{-}1]$};
\draw[usual] (-0.5,0.5) to[out=90,in=180] (0,1)node[above,yshift=-0.1cm]{$S$} to[out=0,in=90] (0.5,0.5);
\draw[usual] (-1,0.5) to (-1,3.5);
\draw[usual] (1,0.5) to (1,3.5);
\node at (1,3.5) {$\phantom{a}$};
\node at (1,-0.5) {$\phantom{a}$};
\end{tikzpicture}
,\quad
\begin{tikzpicture}[anchorbase,scale=0.25,tinynodes]
\draw[pQJW] (-2,0.5) rectangle (0,-0.5);
\draw[pQJW] (-2,2.5) rectangle (2,3.5);
\draw[usual] (0.5,-0.5)to (0.5,0.5) to[out=90,in=0] (0,1)node[above,yshift=-0.1cm]{$S$} to[out=180,in=90] (-0.5,0.5);
\draw[usual] (1,-0.5) to (1,2.5);
\draw[usual] (-1,0.5) to (-1,2.5);
\node at (-1,3.5) {$\phantom{a}$};
\node at (-1,-0.5) {$\phantom{a}$};
\end{tikzpicture}
=
\begin{tikzpicture}[anchorbase,scale=0.25,tinynodes]
\draw[pQJW] (2,0.5) rectangle (-2,-0.5);
\node at (0,-0.2) {$\pjwm[v{-}1]$};
\draw[usual] (-0.5,0.5) to[out=90,in=180] (0,1)node[above,yshift=-0.1cm]{$S$} to[out=0,in=90] (0.5,0.5);
\draw[usual] (-1,0.5) to (-1,3.5);
\draw[usual] (1,0.5) to (1,3.5);
\node at (1,3.5) {$\phantom{a}$};
\node at (1,-0.5) {$\phantom{a}$};
\end{tikzpicture}
,
\end{gather*}
where \changed{the top boxes are labeled $v[S]-1$, and} the small box is labeled by $\fancest{v}{S}{-}1$ for $\fancest{v}{S}$
being the youngest ancestor of $v$ for which all digits indexed by elements of $S$ are zero. 

\item For $v=\plbase{a_{j},\dots,a_{0}}\notin\eve$ let
$w=\plbase{a_{k},\dots,a_{0}}$ for some $k<j$.
Then we have
\begin{gather*}
\begin{tikzpicture}[anchorbase,scale=0.25,tinynodes]
\draw[pQJW] (1.5,-1) rectangle (-2.5,1);
\node at (-0.5,-0.1) {$\pjwm[v{-}1]$};
\draw[usual] (1,1) to[out=90,in=180] (1.5,1.5) to[out=0,in=90] 
(2,1) to (2,0) node[right,xshift=-2pt]{$w$} to (2,-1) 
to[out=270,in=0] (1.5,-1.5) to[out=180,in=270] (1,-1);
\node at (0,2.5) {$\phantom{a}$};
\node at (0,-2.5) {$\phantom{a}$};
\end{tikzpicture}
=
(-1)^{w}
\prod_{a_{i}\neq 0}
\qnum{2}{\vpar^{a_{i}\ppar^{(i)}}}\cdot
\begin{tikzpicture}[anchorbase,scale=0.25,tinynodes]
\draw[pQJW] (-1.5,-1) rectangle (2.5,1);
\node at (0.5,-0.1) {$\pjwm[{\motherr{v}{x}{-}1}]$};
\node at (0,2.5) {$\phantom{a}$};
\node at (0,-2.5) {$\phantom{a}$};
\end{tikzpicture}
,
\end{gather*}
where the product runs over all non-zero digits of $w$, and $\motherr{v}{x}=v-w$
is the corresponding ancestor. (Note that for $\qpar^{\lpar}=\pm 1$ and $i>0$ we
have $(-1)^{a_{i}\ppar^{(i)}}\qnum{2}{\qpar^{a_{i}\ppar^{(i)}}}=(-q)^{a_{i}\ppar^{(i)}} 2$.) For
$v=\plbase{a_{j},0,\dots,0}\in\eve$, $v\geq\lpar$ and $k\leq v$ such that
$v-k=\plbase{b_{i},0,\dots,0}\in\eve$ with $i<j$ we additionally have
\begin{gather*}
\begin{tikzpicture}[anchorbase,scale=0.25,tinynodes]
\draw[JW] (1.5,-1) rectangle (-2.5,1);
\node at (-0.5,-0.1) {$\pjwm[v{-}1]$};
\draw[usual] (1,1) to[out=90,in=180] (1.5,1.5) to[out=0,in=90] 
(2,1) to (2,0) node[right,xshift=-2pt]{$k$} to (2,-1) 
to[out=270,in=0] (1.5,-1.5) to[out=180,in=270] (1,-1);
\node at (0,2.5) {$\phantom{a}$};
\node at (0,-2.5) {$\phantom{a}$};
\end{tikzpicture}
=
0
=
\begin{tikzpicture}[anchorbase,scale=0.25,tinynodes]
\draw[JW] (-1.5,-1) rectangle (2.5,1);
\node at (0.5,-0.1) {$\pjwm[v{-}1]$};
\draw[usual] (-1,1) to[out=90,in=0] (-1.5,1.5) to[out=180,in=90] 
(-2,1) to (-2,0) node[left,xshift=2pt]{$k$} to (-2,-1) 
to[out=270,in=180] (-1.5,-1.5) to[out=0,in=270] (-1,-1);
\node at (0,2.5) {$\phantom{a}$};
\node at (0,-2.5) {$\phantom{a}$};
\end{tikzpicture}
\in\End_{\TL[{\kk,\qpar}]}(v-k).
\end{gather*}
(A special case of this is the trace down to the empty diagram.)

\end{enumerate}
\end{proposition}

\begin{proof}
All except the final statement can be shown as in \cite[Propositions 3.11, 3.13 and 3.14]{TuWe-quiver}. The final statement follows 
by using \eqref{eq:0trace} and observing that the projector
after taking partial trace satisfies $\qjw[v{-}k]=\pqjw[v{-}k]$ 
and the zero obtained by \eqref{eq:0trace} annihilates it.
\end{proof}

The projectors $\pqjw[v{-}1]$ typically do not form a right-aligned family. \fullref{example:left-right} gives
a counterexample to right-aligned absorption of $\idtl[1]\hcirc\pqjw[2]$ into $\pqjw[3]$.

Recall the definition of the \emph{categorical dimension} 
$\dim_{\catstuff{C}}$ of objects in a pivotal 
category $\catstuff{C}$\changed{, see {\eg} \cite[Definition 4.7.1]{EtGeNiOs-tensor-categories} (the categorical dimension is the trace of the identity therein).}

\begin{proposition}\label{proposition:qdim}
For $v=\plbase{a_{j},\dots,a_{0}}$ we have  
\begin{gather*}
\dim_{\tilt[{\kk,\qpar}]}\big(\tmod(v-1)\big)
=
(-1)^{v-1} {\textstyle \sum_{S\in\supp[v]}} \qnum{v[S]}{\qpar}
=(-1)^{v-1}\qnum{\motherr{v}{\infty}}{\qpar}{\textstyle\prod_{a_{i}\neq 0}}\,\qnum{2}{\qpar^{a_{i}\ppar^{(i)}}},
\end{gather*}
where the product runs over all non-zero and non-leading digits of $v$.
\end{proposition}

\begin{proof}
The categorical dimension in $\TL[{\kk,\qpar}]$ is given by closing 
pictures in the usual way, and for the first 
equality we calculate
\begin{gather*}
\begin{tikzpicture}[anchorbase,scale=0.25,tinynodes]
\draw[pQJW] (0,0) rectangle (4,1);
\node at (2,0.3) {$\qjwm[v{-}1]$};
\draw[usual] (2,1) to[out=90,in=180] (4,2) to[out=0,in=90] (6,1) to (6,0) to[out=270,in=0] (4,-1) to[out=180,in=270] (2,0);
\end{tikzpicture}
=
\sum_{S\in\supp[v]}
\lambda_{v,S}
\cdot
\begin{tikzpicture}[anchorbase,scale=0.25,tinynodes]
\draw[JW] (0,0) rectangle (4,1.5);
\node at (2,0.55) {$\qjwm[{v[S]{-}1}]$};
\trd{5}{1.5}{$\qjwm[S]$}{-5}{0}
\tru{5}{1.5}{$\qjwm[S]$}{-5}{1.5}
\draw[usual] (2,3) to[out=90,in=180] (4,4) to[out=0,in=90] (6,3) to (6,-1.5) to[out=270,in=0] (4,-2.5) to[out=180,in=270] (2,-1.5);
\end{tikzpicture}
=
\sum_{S\in\supp[v]}
\lambda_{v,S}
\cdot
\begin{tikzpicture}[anchorbase,scale=0.25,tinynodes]
\draw[JW] (0,-1.5) rectangle (4,0);
\node at (2,-0.95) {$\qjwm[{v[S]{-}1}]$};
\trd{5}{1.5}{$\qjwm[S]$}{-5}{3}
\tru{5}{1.5}{$\qjwm[S]$}{-5}{0}
\draw[usual] (2,3) to[out=90,in=180] (4,4) to[out=0,in=90] (6,3) to (6,-1.5) to[out=270,in=0] (4,-2.5) to[out=180,in=270] (2,-1.5);
\end{tikzpicture}
=
\sum_{S\in\supp[v]}
\begin{tikzpicture}[anchorbase,scale=0.25,tinynodes]
\draw[JW] (0,-1.5) rectangle (4,0);
\node at (2,-0.95) {$\qjwm[{v[S]{-}1}]$};
\draw[usual] (2,0) to[out=90,in=180] (4,1) to[out=0,in=90] (6,0) to (6,-1.5) to[out=270,in=0] (4,-2.5) to[out=180,in=270] (2,-1.5);
\end{tikzpicture}.
\end{gather*}
Observing that $(-1)^{v[S]-1}=(-1)^{v-1}$ 
for all $S\in\supp[v]$, the first equality follows by classical theory, see {\eg} \cite[Section 9.5]{KaLi-TL-recoupling}. 
The second equation follows then from \fullref{proposition:jw-properties}.
\end{proof}

Note that the categorical dimension of $\tmod(v-1)$
is an element of the underlying field, but interpreted 
in $\Nv$ we obtain the character $\ch{\tmod(v-1)}$.

\subsection{Tilting modules as an additive category}
\label{subsection:quiver}

Let us define a (locally unital) $\kk$-algebra via
\begin{gather*}
\zigzag=
{\textstyle\bigoplus_{v,w\in\N}}\Hom_{\tilt[{\kk,\qpar}]}\big(\tmod(v-1),\tmod(w-1)\big).
\end{gather*}
Let $\zigzagmod$ denote the category of finitely generated, projective (right) 
$\zigzag$-modules. By construction we obtain, as instance 
of \emph{Ringel duality} (semi-infinite in the sense of \cite{BrSt-semi-infinite}), that
\begin{gather*}
\mainfunctor\colon
\tilt[{\kk,\qpar}]\to\zigzagmod,
\tmod\mapsto
{\textstyle\bigoplus_{v\in\N}}\Hom_{\tilt[{\kk,\qpar}]}\big(\tmod(v-1),\tmod\big)
\end{gather*}
is an equivalence of additive, $\kk$-linear categories, sending 
indecomposable tiltings to indecomposable projectives. Let us describe 
$\zigzag$ explicitly.

By construction, morphisms in 
$\Hom_{\tilt[{\kk,\qpar}]}\big(\tmod(v-1),\tmod(w-1)\big)$ are given by flanking 
TL morphisms with $\pjw[v{-}1]$ from the bottom and with $\pjw[w{-}1]$ 
from the top, and the primitive idempotents (which are local units) are
the $\pjw[v{-}1]$ for $v\in\N$. Other morphisms, called 
\emph{mixed trapezes and loops}, are diagrammatically 
given by the analog of \fullref{definition:jw-cupscaps2}: if
$S$ and $S^{\prime}$ be down- and up-admissible for $v$, respectively, 
and assume that $S$ and $S^{\prime}$ are minimal admissible stretches of 
consecutive integers then we define
\begin{gather}\label{eq:trapz}
\begin{gathered}
\pTRD{$\pjw[S]$}\ochanged{:}=
\pjw[{v[S]{-}1}]\Down{S}\pjw[v{-}1]=
\begin{tikzpicture}[anchorbase,scale=0.25,tinynodes]
\draw[pJW] (2,0.5) rectangle (-2,-0.5);
\draw[pJW] (2,2.5) rectangle (-2,3.5);
\node at (0,-0.2) {$\pjwm[v{-}1]$};
\draw[usual] (-0.5,0.5) to[out=90,in=180] (0,1)node[above,yshift=-0.1cm]{$S$} to[out=0,in=90] (0.5,0.5);
\draw[usual] (-1,0.5) to (-1,2.5);
\draw[usual] (1,0.5) to (1,2.5);
\node at (1,3.5) {$\phantom{a}$};
\node at (1,-0.5) {$\phantom{a}$};
\end{tikzpicture}
,\quad
\pTRU{$\pjw[S^{\prime}]$}\ochanged{:}=
\pjw[v(S){-}1]\Up{S^{\prime}}\pjw[v{-}1]=
\begin{tikzpicture}[anchorbase,scale=0.25,tinynodes]
\draw[pJW] (2,-0.5) rectangle (-2,-1.5);
\draw[pJW] (2,1.5) rectangle (-2,2.5);
\node at (0,-1.2) {$\pjwm[v{-}1]$};
\draw[usual] (-0.5,1.5) to[out=270,in=180] (0,1)node[below,yshift=-0.05cm]{$S^{\prime}$} to[out=0,in=270] (0.5,1.5);
\draw[usual] (-1,-0.5) to (-1,1.5);
\draw[usual] (1,-0.5) to (1,1.5);
\node at (1,2.5) {$\phantom{a}$};
\node at (1,-1.5) {$\phantom{a}$};
\end{tikzpicture}
.
\end{gathered}
\end{gather}
These are the generators of $\zigzag$, and (up to {\losp}) the respective
minimal stretches are singleton sets $S=\{i\}$, reflecting along the $i$th
digit. The corresponding $S$-labeled cups and caps in \eqref{eq:trapz} consist
of $a_{i}\ppar^{(i)}$ parallel strands. We also define the loops
\changed{$\loopdown{S}{v{-}1}:=\pjw[v{-}1]\Up{S}\pjw[{v[S]{-}1}]\Down{S}\pjw[v{-}1]$}.
Finally, note that these morphisms can be defined more generally for any down-
and up-admissible stretches, but then their diagrammatic incarnation can involve
multiple stretch-labeled cups and caps. 

To describe the relations between expressions in the generating morphisms, we
will use the same scalars (depending on the digits) as in \cite[Section
3A]{TuWe-quiver}, namely
\begin{gather}\label{eq:f-g}
\begin{gathered}
\funcf(a)=
\begin{cases} 
(-\qpar)^{(a+1)\lpar}{\cdot}\tfrac{-2}{a}
&\text{if } 1\leq a\leq\ppar-2,
\\
0 
&\text{if }a=0\text{ or }a=\ppar-1,
\end{cases}
\funcg(a) =
\begin{cases}
(-\qpar)^{\lpar}\big(\tfrac{a+1}{a}\big)
&\text{if } 1\leq a\leq\ppar-1,
\\
(-\qpar)^{\lpar}2 &\text{if }a=0,
\end{cases}
\\
\funcF\pjw[v{-}1] 
=
\funcf(a_{\max(S)+1})\pjw[v{-}1], 
\quad
\funcG\pjw[v{-}1] 
= 
\funcg(a_{\max(S)+1})\pjw[v{-}1],
\\
\funcH\pjw[v{-}1] 
=
\funcg(a_{\max(S)+1}-1)\pjw[v{-}1].
\end{gathered}
\end{gather}
In fact, as we will see later, these scalars can be seen 
as (inverses of higher order) 
\emph{local intersection forms} in the language of 
\cite{El-ladders-clasps}.

We obtain the mixed characteristic version of \cite[Theorem 3.2]{TuWe-quiver}:

\begin{theorem}\label{theorem:main-tl-section}
The algebra $\zigzag$ is generated by $\pjw[v{-}1]$ for $v\in\N$, 
and elements $\Down{S}\pjw[v{-}1]$ and $\Up{S^{\prime}}\pjw[v{-}1]$, 
where $S$ and $S^{\prime}$ denote minimal down- and up-admissible stretches for $v$, respectively.
These generators are subject to the following complete set of relations. (As before, we omit idempotents from the notation if they can be recovered from
the given data.)
\begin{enumerate}[label=(\arabic*)]

\setlength\itemsep{0.15cm}

\item \label{theorem:main-tl-section-1} \emph{Idempotents.}
\begin{gather*}
\begin{gathered}
\pjw[v{-}1]\pjw[w{-}1]=\delta_{v,w}\pjw[v{-}1],
\\ 
\pjw[{v[S]{-}1}]\Down{S}\pjw[v{-}1]
=
\pjw[{v[S]{-}1}]\Down{S}
=
\Down{S}\pjw[v{-}1],
\quad
\pjw[v(S^{\prime}){-}1]\Up{S^{\prime}}\pjw[v{-}1]
=
\pjw[v(S^{\prime}){-}1]\Up{S^{\prime}}
=
\Up{S^{\prime}}\pjw[v{-}1].
\end{gathered}
\end{gather*}

\item \label{theorem:main-tl-section-2} \emph{Containment.}
If $S^{\prime}\subset S$, then we have
\begin{gather*}
\Down{S^{\prime}}\Down{S}\pjw[v{-}1]=0,
\quad
\Up{S}\Up{S^{\prime}}\pjw[v{-}1]=0.
\end{gather*}

\item \label{theorem:main-tl-section-3} \emph{Far-commutativity.}
If $\dist(S,S^{\prime})>1$, then 
\begin{gather*}
\Down{S}\Down{S^{\prime}}\pjw[v{-}1]=\Down{S^{\prime}}\Down{S}\pjw[v{-}1],
\quad
\Down{S}\Up{S^{\prime}}\pjw[v{-}1]=\Up{S^{\prime}}\Down{S}\pjw[v{-}1],
\quad
\Up{S}\Up{S^{\prime}}\pjw[v{-}1]=\Up{S^{\prime}}\Up{S}\pjw[v{-}1].
\end{gather*}

\item \label{theorem:main-tl-section-4} \emph{Adjacency relations.}
If $\dist(S,S^{\prime})=1$ and $S^{\prime}>S$, then
\begin{gather*}
\Down{S^{\prime}}\Up{S}\pjw[v{-}1]=\Down{S{\cup}S^{\prime}}\pjw[v{-}1],
\quad
\Down{S}\Up{S^{\prime}}\pjw[v{-}1]=\Up{S^{\prime}{\cup}S}\pjw[v{-}1],
\\
\Down{S^{\prime}}\Down{S}\pjw[v{-}1]=\Up{S}\Down{S^{\prime}}\funcH[S]\pjw[v{-}1],
\quad
\Up{S}\Up{S^{\prime}}\pjw[v{-}1]=\funcH[S]\Up{S^{\prime}}\Down{S}\pjw[v{-}1].
\end{gather*}

\item \label{theorem:main-tl-section-5} \emph{Overlap relations.}
If $S^{\prime}\geq S$ with $S^{\prime}\cap S=\{s\}$ and $S^{\prime}\not\subset S$, then we have
\begin{gather*}
\Down{S^{\prime}}\Down{S}\pjw[v{-}1]=\Up{\{s\}}\Down{S}\Down{S^{\prime}{\setminus}\{s\}}\pjw[v{-}1],
\quad
\Up{S}\Up{S^{\prime}}\pjw[v{-}1]=\Up{S^{\prime}{\setminus}\{s\}}\Up{S}\Down{\{s\}}\pjw[v{-}1].
\end{gather*}

\item \label{theorem:main-tl-section-6} \emph{Zigzag.}
\begin{gather*}
\Down{S}\Up{S}\pjw[v{-}1]=\Up{\hull[S]}\Down{\hull[S]}\funcG[S]\pjw[v{-}1] 
+\Up{T}\Up{\hull[S]}\Down{\hull[S]}\Down{T}\funcF[S]\pjw[v{-}1].
\end{gather*}
Here, if the down-admissible hull $\hull[S]$, or 
the smallest minimal down-admissible stretch $T$ with $T>\hull[S]$ 
does not exist, then the involved symbols are 
zero by definition.

\noindent The elements of the form
\end{enumerate}
\begin{gather*}
\tag{Basis}\qquad\pjw[w{-}1]\Up{S_{i_{l}}^{\prime}}
\cdots\Up{S_{i_{0}}^{\prime}}\Down{S_{i_{0}}}\cdots\Down{S_{i_{k}}}\pjw[v{-}1],
\end{gather*}
\hspace{1cm} with $S_{i_{l}}^{\prime}>\cdots>S_{i_{0}}^{\prime}$, and $S_{i_{0}}<\cdots<S_{i_{k}}$, 
form a basis for $\pjw[w{-}1]\zigzag\pjw[v{-}1]$. 
\begin{gather*}
\begin{gathered}
\tag{Complete}
\text{Any word $\pjw[w{-}1]\morstuff{X}\pjw[v{-}1]$ in the generators 
of $\zigzag$ can be rewritten as a linear}\\
\text{combination of basis elements from (Basis) using only the above relations.}
\end{gathered}
\end{gather*}
\end{theorem}

\begin{proof}
\changed{This is analogous to the ten page proof of the characteristic $\ppar$
case in \cite{TuWe-quiver}. The proof given therein splits into the following steps which one can copy up to some adjustments detailed below.}
\begin{enumerate}

\item \changed{The proof starts with \cite[Lemma 3.6]{TuWe-quiver} which proves (Complete). The arguments given therein work {\ver}.}

\item \changed{The proof continues with \cite[Section 3B]{TuWe-quiver} listing 
properties of the $\ppar$JW projectors. The $\ppar\lpar$-versions of these properties are given above.}

\item \changed{Then (Basis) is proven, see \cite[Section 3C]{TuWe-quiver}. The arguments 
given therein work again {\ver}.}

\item \changed{Then some of the relations are proven in the following order. Relation (1) is clear, relation (2) is proven in \cite[Lemma 3.22]{TuWe-quiver}, 
relation (3) is proven in \cite[Lemma 3.23]{TuWe-quiver}, and the first pair of the relations (4) is proven in \cite[Lemma 3.24]{TuWe-quiver}. All of these proofs go through {\ver}.}

\item \changed{The gist is the simultaneous proof of the first pair of the relations (4), relation (5), relation (6), and the identification of $\End_{\TL[{\kk,\qpar}]}(\pjw[v{-}1])$, which are 
\cite[Lemmas 3.25-3.28]{TuWe-quiver}. The arguments given in \cite{TuWe-quiver} 
carry over to the mixed case, but some subtle changes need to made as detailed below.}

\end{enumerate}

\changed{As mentioned above, the proofs given in \cite{TuWe-quiver}} need some adjustment due to {\eg} 
the appearance of signs in $\funcf$ and $\funcg$ 
from \eqref{eq:f-g}. We record the necessary modification to the
numerical arguments used in \cite{TuWe-quiver}. 

First of all, the scalars $\lambda_{v,S}$ and
partial trace formulas for the various 
JW projectors now involve fractions of quantum
numbers. Moreover, a few signs that have started their lives as
$-1=(-1)^{\ppar^{i}}$ now 
have to replaced by $(-1)^{\ppar^{(i)}}=(-1)^{\lpar}$ when
$i>0$. This concerns the sign of the fraction of quantum numbers in
\cite[(4-2)]{TuWe-quiver} (this replacement leads to the desired interpretation
in terms of $\funcg$), the sign in $q$ from \cite[(4-8)]{TuWe-quiver} should be
$(-1)^{\ppar^{(s)}}$, which balances against the sign of $\lambda_{w,R}$ in
the following display. Further, in the \textit{Proof, which caveat} for
\cite[Lemma 4.9]{TuWe-quiver}, the signs $(-1)^{w-u}$ and $(-1)^{w+1-u}$ are to
be replaced with $(-q)^{w-u}$ and $(-q)^{w+\ppar^{(i)}-u}$, which is again
compatible with $\funcf$ and $\funcg$ as desired. The vanishing of $q^{\prime}$ from
\cite[(4-4)]{TuWe-quiver} follows using a similar argument using quantum
numbers. Finally, the zigzag relation are established by an inductive argument
based on the case of generation $2$, which is proved exactly as outlined in
\cite[Lemma 4.8]{TuWe-quiver}.
\end{proof}

Note that the non-idempotent generators of $\zigzag$ are 
given by down and up morphisms for minimal stretches, {\aka} 
singleton sets if we ignore {\losp} (we will write {\eg} 
$\Down{i}$ instead of $\Down{\{i\}}$ for these to simplify notation). 
By using the relations, 
{\eg} \fullref{theorem:main-tl-section}.\ref{theorem:main-tl-section-4}, one 
obtains down and up morphisms for more general stretches.

\begin{example}
For the complex quantum group case the only possible 
stretch is $S=\{0\}$, which is 
down-admissible unless $a_{0}=0$, where we note that 
$\hull[\{0\}]$ does not exist in this case 
(and so $T$ does not exist either).
The only relevant relations are 
the ones in 
\fullref{theorem:main-tl-section}.\ref{theorem:main-tl-section-1} and
\begin{gather*}
\Down{0}\Down{0}\pjw[v{-}1]=0,
\quad
\Up{0}\Up{0}\pjw[v{-}1]=0,
\quad
\Down{0}\Up{0}\pjw[v{-}1]=
\begin{cases}
\funcg(a_{1})\Up{0}\Down{0}\pjw[v{-}1]&\text{if }a_{0}\neq 0,
\\
0&\text{if }a_{0}=0,
\end{cases}
\end{gather*}
and $\zigzag$ has connected components corresponding to
(scaled) zigzag algebras for each $v<\lpar$, and single 
vertices for $v=\pbase{a_{1},0}{\infty,\lpar}$. Thus, we 
recover \cite[Theorem 3.12]{AnTu-tilting}.
\end{example}

\begin{example}
One can show the useful relation that
\begin{gather}\label{eq:dud}
\Down{S}\Up{S}\Down{S}\pjw[v{-}1]=0,
\quad
\pjw[v{-}1]\Up{S}\Down{S}\Up{S}=0,
\end{gather}
for any down-admissible stretch $S$, {\cf} \cite[(3-13)]{TuWe-quiver}. 
In particular, loops square to zero.
\end{example}

The following proposition gives explicit versions of the cellular bases
constructed in \cite{An-tilting-cellular} and \cite{AnStTu-cellular-tilting}. To state it recall from
\cite{GrLe-cellular} that the crossingless matching basis of $\TL[{\Zv,\vpar}]$
together with $(\N[0],<)$ and $\fliph[(\placeholder)]$ endows $\TL[{\Zv,\vpar}]$
with the structure of a (strictly object adapted) cellular category. (We refer
to \cite[Definition 2.1]{We-tensors-cellular-categories} and \cite[Definition
2.4]{ElLa-trace-hecke} for the terminology.)

\begin{proposition}\label{proposition:cellular} Let $\setstuff{B}$ denote the
set given by the elements in \fullref{theorem:main-tl-section}.(Basis).
Moreover, let $\tilde{\setstuff{b}}$ respectively $\overline{\setstuff{b}}$
denote the sets obtained from the analogous expressions based on the projectors
$\qjw[v{-}1]$ and $\pqjw[q{-}1]$ respectively.

\begin{enumerate}

\item The sets $\tilde{\setstuff{b}}$ and 
$\overline{\setstuff{b}}$ give bases for the hom-spaces in 
$\tilt[{\kkv,\vpar}]$ while the set $\setstuff{B}$ gives 
a bases for the hom-spaces in $\tilt[{\kk,\qpar}]$.

\item All of these bases are unitriangularly equivalent to the crossingless
matching bases (with respect to $(\N[0],<)$).

\item All of these bases
together with $(\N[0],<)$ and $\fliph[(\placeholder)]$ endow 
$\tilt[{\kkv,\vpar}]$ respectively $\tilt[{\kk,\qpar}]$ with the 
structure of a (strictly object adapted) cellular category.

\end{enumerate} 
\end{proposition}

\begin{proof}
\fullref{theorem:main-tl-section} shows that 
$\overline{\setstuff{b}}$ and $\setstuff{B}$ give 
bases of the respective hom-spaces, and the former is 
unitriangularly equivalent to $\tilde{\setstuff{b}}$, 
by construction. Moreover, $\tilde{\setstuff{b}}$ is 
unitriangularly equivalent to the crossingless matching 
basis, and a basis change that is unitriangular 
with respect to the cell order preserves 
all structures defining a (strictly object adapted) cellular category.
\end{proof}

Finally, we record a useful consequence of \fullref{theorem:main-tl-section}.Basis.

\begin{lemma}\label{lemma:no-delta-overlap}
Suppose that $v,w\in\N$ are such that $\supp[v]\cap\supp[w]=\emptyset$.
We have
\begin{gather*}
\Hom_{\tilt[{\kk,\qpar}]}\big(\tmod(v-1),\tmod(w-1)\big)
\cong
\pjw[w{-}1]\zigzag\pjw[v{-}1]
=\{0\}.
\end{gather*}
In particular, this holds true if the zeroth digit $b_{0}$ of $w$ 
satisfies $b_{0}\neq a_{0}$ and $b_{0}\neq\lpar-a_{0}$, or 
$b_{0}=a_{0}\neq\tfrac{\lpar}{2}$ 
but $\ppar>2$ and the parity of the sum of the remaining digits 
of $v$ and $w$ is different.
\end{lemma}

\begin{proof}
The first part of the statement is clear by
\fullref{theorem:main-tl-section}.Basis. The first condition for when
$\supp[v]\cap\supp[w]=\emptyset$ is immediate from the definitions as the
corresponding $\ppar\lpar$-adic expansions of the elements of $\supp[v]$ and
$\supp[w]$ have to agree on the zeroth digit. For the second condition note that
$b_{0}=a_{0}\neq\tfrac{\lpar}{2}$ ensures that every element of $\supp[v]$ is
distinguished from the elements of $\supp[w]$ by its zeroth digit or by the
parity of the sum of higher digits (here we use that $\ppar$ is odd).
\end{proof}

Note that the conditions given at the end of \fullref{lemma:no-delta-overlap}
are in general only sufficient to ensure $\supp[v]\cap\supp[w]=\emptyset$.

\begin{remark}
The statement in \fullref{lemma:no-delta-overlap} is 
known as a \emph{Weyl factor overlap 
criterion} in the theory of tilting modules 
and follows from \emph{Ext-vanishing}, 
see see {\eg}
\cite[Section 2B]{AnStTu-cellular-tilting}, using the integrality of these
statements which follows from \cite{RyHa-integral-kempf}.
One can see \fullref{lemma:no-delta-overlap} as an explicit 
incarnation of these (general) facts about tilting modules.
\end{remark}

\subsection{More partial trace formulas}\label{subsection:ptrace}

The next lemma deals with partial traces that do not reach an ancestor. 
Hence, they are complementary to part (c) of \fullref{proposition:jw-properties}.

\begin{lemma}\label{lemma:othertrace}
Let $v=\plbase{a_{j},\dots,a_{k},0\dots,0}$ with 
$k>0$, and $a_{k}\neq 0$. Suppose
that $w=(\ppar-a_{i})\ppar^{(i)}$ 
for some $1\leq i<k$. 
Then we have
\begin{gather*}
\begin{tikzpicture}[anchorbase,scale=0.25,tinynodes]
\draw[pQJW] (1.5,-1) rectangle (-2.5,1);
\node at (-0.5,-0.1) {$\pjwm[v{-}1]$};
\draw[usual] (1,1) to[out=90,in=180] (1.5,1.5) to[out=0,in=90] 
(2,1) to (2,0) node[right,xshift=-2pt]{$w$} to (2,-1) 
to[out=270,in=0] (1.5,-1.5) to[out=180,in=270] (1,-1);
\node at (0,2.5) {$\phantom{a}$};
\node at (0,-2.5) {$\phantom{a}$};
\end{tikzpicture}
=
\begin{tikzpicture}[anchorbase,xscale=-0.25,yscale=0.25,tinynodes]
\draw[pQJW] (-2.5,-0.5) rectangle (2.5,0.5);
\node at (0,-0.15) {$\pjwm[v{+}w{-}1]$};
\draw[usual] (-1,0.5) to[out=90,in=0] (-1.5,1) to[out=180,in=90] 
(-2,0.5)  (-2,-0.5) 
to[out=270,in=180] (-1.5,-1) to[out=0,in=270] (-1,-0.5);
\node at (0,2.5) {$\phantom{a}$};
\node at (0,-2.5) {$\phantom{a}$};
\draw[pQJW] (-0.5,-1.5) rectangle (2.5,-2.5);
\node at (0.875,-2.15) {$\pjwm[x{-}1]$};
\draw[pQJW] (-0.5,1.5) rectangle (2.5,2.5);
\node at (0.875,1.85) {$\pjwm[x{-}1]$};
\draw[usual] 
(0,-1.5) to (0,-0.5) 
(2,-1.5) to (2,-0.5)
(0,0.5) to (0,1.5)
(2,0.5) to (2,1.5);
\end{tikzpicture}
=
\Down{S}\Up{S}\pjw[x-1],
\end{gather*}
where $x=(v+w)[S] 
=\plbase{a_{j},\dots,a_{k}-1,\ppar-1,\dots,\ppar-1,a_{i},0,\dots,0}$ 
with $S=\{i,\dots,k-1\}$. 
The same holds for $0<w^{\prime}=\lpar-a_{0}<\lpar$, 
where $x^{\prime}=(v+w^{\prime})[S^{\prime}] 
=\plbase{a_{j},\dots,a_{k}-1,\ppar-1\dots,\ppar-1,a_{0}}$
and $S=\{0,\dots,k-1\}$.
\end{lemma}

\begin{proof}
This follows from 
projector absorption and
shortening where we have used that $S$ is the minimal
down-admissible stretch of $v+w$. 
The formula for $w^{\prime}$ and $x^{\prime}$ can be proven {\ver}.
\end{proof}

We can now use the zigzag relation from \fullref{theorem:main-tl-section}.(6) to
simplify $\Down{S}\Up{S}\pjw[x-1]$ further. To use this relation, recall that if
the down-admissible hull $\hull[S]$, or the smallest minimal down-admissible
stretch $T$ with $T>\hull[S]$ does not exist, then the involved symbols are zero
by definition. 

\begin{proposition}\label{proposition:more-ptrace}
Retain notation as in \fullref{lemma:othertrace}. We have
\begin{gather*}
\begin{aligned}
\begin{tikzpicture}[anchorbase,scale=0.25,tinynodes]
\draw[pJW] (1.5,-1) rectangle (-2.5,1);
\node at (-0.5,-0.1) {$\pjwm[v{-}1]$};
\draw[usual] (1,1) to[out=90,in=180] (1.5,1.5) to[out=0,in=90] 
(2,1) to (2,0) node[right,xshift=-2pt]{$w$} to (2,-1) 
to[out=270,in=0] (1.5,-1.5) to[out=180,in=270] (1,-1);
\node at (0,2.5) {$\phantom{a}$};
\node at (0,-2.5) {$\phantom{a}$};
\end{tikzpicture}   
&=
\funcg(a_{k}-1)\cdot\Up{\hull[S]}\Down{\hull[S]}\pjw[x{-}1] 
+\funcf(a_{k}-1)\cdot\Up{T}\Up{\hull[S]}\Down{\hull[S]}\Down{T}\pjw[x{-}1].
\end{aligned}
\end{gather*}
Moreover, the formula holds for $w^{\prime}$ and 
$x^{\prime}$.
\end{proposition}

Note that the $\funcf$-terms in \fullref{proposition:more-ptrace} vanish  if
$a_{k}=1$. On the other hand, the $\funcg$-terms can be zero only if no
admissible hull $\hull[S]$ exists. Hence, the whole partial trace vanishes if
and only if $a_{k}=a_{j}=1$, {\ie} if and only if $v=\ppar^{(k)}$ is a prime
eve.

We will state more partial trace formulas 
in \fullref{theorem:timest1} later on.

\section{Monoidal structure}\label{section:fusion}

In this section, we study $\tilt[{\kk,\qpar}]$ as a monoidal category. In the
semisimple and complex quantum group cases, the results in this section appear
throughout the literature\changed{, see for example the remarks in this section
for a small (and somewhat biased) collection of references.}

\subsection{Fusion rules}

We start by recalling the well-known fusion rules for tilting modules {lying} in the
fundamental alcove,  {in which tilting modules are simple}.  {Note that in the semisimple case, {\ie} when $\lpar=\infty$, the fundamental alcove is the whole of $\N[0]$.}

\begin{lemma}\label{lemma:clebsch-gordan}
For $1\leq v,w\leq\lpar$ and $v+w-2<\lpar$ we have
\begin{gather}\label{eq:clebsch-gordan}
\tmod(v-1)\hcirc\tmod(w-1)\cong
{\textstyle\bigoplus_{i=1}^{\min(v,w)}}\,
\tmod(v+w-2i).
\end{gather} 
\changed{Let us write $x^{\prime}$ 
to denote $x$, if $x$ is odd, and $x-1$, if $x$ is even.}
For $1\leq v,w\leq\lpar$ and $v+w-2\geq\lpar$ we
have
\begin{gather}\label{eq:qCG}
\begin{aligned}
\tmod(v-1)\hcirc\tmod(w-1)\cong&
{\textstyle\bigoplus_{i=1}^{\lpar-\max(v,w)}}\,
\tmod\big(\plbase{|v-w|-1+2i}-1\big)
\\
&\oplus
{\textstyle\bigoplus_{i=0}^{((v+w-\lpar)^{\prime}-1)/2}}\,
\tmod\big(\plbase{1,v+w-1-\lpar-2i}-1\big),
\end{aligned}
\end{gather}
where, for later use, we indicate the $\ppar\lpar$-adic expansions of the
occurring terms. (The first direct sum \changed{in \eqref{eq:qCG}} is empty if $\max(v,w)=\lpar$.)
\end{lemma}

\begin{proof}
The first part is classical; the second part is easy using \fullref{proposition:multiplicities}.
\end{proof} 

The summands with highest weights in the
fundamental alcove appear in the first direct sum in \eqref{eq:qCG}. 
The second direct sum collects all remaining summands. If
$v+w$ and $\lpar$ have the same parity, then each of the summands are of generation one,
otherwise there exists a simple summand $\tmod(\lpar-1)$.
Equation \eqref{eq:clebsch-gordan} is the \emph{Clebsch--Gordan rule}.

We note the following consequence, used in the proof of 
\fullref{theorem:well-defined}:

\begin{lemma}\label{lemma:translation} 
For all $v,w\in\N$ such that
$v\notin\eve$, suppose that $\wmult{v}{w}=1$, 
then 
\begin{gather*}
\big(\tmod(\mother-1)\hcirc
\tmod(1)^{\hcirc(v{-}\mother)}:\wmod(w-1)\big)=1.
\end{gather*}
\end{lemma}

\begin{proof}
This follows a character argument: we compute the Weyl factors of
$\tmod(\mother-1)\hcirc \tmod(1)^{\hcirc(v{-}\mother)}$ by repeatedly
raising or lowering the corresponding highest weights of these Weyl factors
by $\pm 1$ since
\begin{gather}\label{eq:wmod-rule}
\wmod(0)\hcirc\wmod(1)\cong\wmod(1),\quad
\wmod(n)\hcirc\wmod(1) 
\cong\wmod(n+1)\oplus\wmod(n-1)\text{ if }n\neq 0. 
\end{gather}
In other words, 
we multiply the character of $\tmod(\mother-1)$ as it appears in part (a) of \fullref{proposition:multiplicities}
by $\qnum{2}{\vpar}^{v{-}\mother}$.
Now observe that any Weyl factor of $\tmod(v-1)$ is uniquely obtained from the
summand $\qnum{w}{\vpar}$ that appears in the character of $\tmod(\mother-1)$,
and so the statement follows.
\end{proof}

\begin{definition}	
We define the \emph{tail-length} $\taill{v}$ of $v=\plbase{a_{j},\dots,a_{0}}$
to be the \changed{maximal} $k\in\N[0]$ such that \changed{$a_{0}=\lpar-1$ and
$a_{i}=\ppar-1$ for all $k>i>0$. If $a_{0}\neq\lpar-1$, then $\taill{v}=0$.}
\end{definition}

The \changed{next fusion rule, that goes beyond the fundamental alcove,} and it involves tensoring with the monoidal generator $\tmod(1)$.

\begin{proposition}\label{proposition:times-t1}
Let $v=\plbase{a_{j},\dots,a_{0}}$. We have
\begin{gather*}
\tmod(v-1)\hcirc\tmod(1)
\cong
\tmod(v)\oplus
{
\bigoplus_{i=0}^{\taill{v}}}\,\tmod(v-2\ppar^{(i)})^{\oplus x_{i}}
,\quad
x_{i}=
\begin{cases}
0&\text{if }a_{i}=0 \text{ or } i=j \text{ and } a_{j}=1,
\\
2&\text{if }a_{i}=1,
\\
1&\text{if }a_{i}>1.
\end{cases}
\end{gather*}
\end{proposition}

\begin{proof}
A character computation based on \eqref{eq:partial-trace-result}.
\end{proof}

Let us comment on
the qualitative differences between the cases {found} in
\fullref{proposition:times-t1}. For the sake of exposition, we consider $\lpar\geq 4$. Focusing on the
zeroth digit, we see:
\begin{gather*}
\tmod(v-1)\hcirc\tmod(1)
\cong
\begin{cases}
\tmod(v)&\text{if }a_{0}=0,
\\
\tmod(v)\oplus\tmod(v-2)^{\oplus 2}&\text{if }a_{0}=1,
\\
\tmod(v)\oplus\tmod(v-2)&\text{if }a_{0}\in\{2,\dots,\lpar-2\},
\\
\text{generation drop case}&\text{if }a_{0}=\lpar-1.
\end{cases}
\end{gather*}

{The \emph{generation drop case} occurs when the zeroth digit is maximal, in which case additional direct summands may appear.}
Recall that $\tmod(v-1)$ has
$2^{\generation[v]}$ Weyl factors. Under tensoring with $\tmod(1)\cong
\wmod(1)$, most of them produce two new Weyl factors by \eqref{eq:wmod-rule}. In
total, $\tmod(v-1)\hcirc\tmod(1)$ will have $2^{\generation[v]+1}$ or
$2^{\generation[v]+1}-1$ Weyl factors. Observe that we are guaranteed to find a
direct summand $\tmod(v)$ in $\tmod(v-1)\hcirc\tmod(1)$. Now we have three cases
depending on whether the generation increases, stays constant or drops, which
are precisely the {respective} cases $a_{0}=0$, $a_{0}\in\{1,\dots,\lpar-2\}$ {and}
$a_{0}=\lpar-1$,  as above. In the first case, $\tmod(v)$ exhausts all newly generated
Weyl factors, so it is the only summand that appears.  In the second case, it
exhausts roughly {half of all} Weyl factors, so one expects {a further summand to appear}.  In
the generation drop case, we have $\generation[v{+}1]=\generation[v]-\taill{v}$.  Hence $\tmod(v)$ only
accounts for a small proportion of the Weyl modules, and we expect several other
tilting summands {to appear.}

\fullref{proposition:times-t1} immediately implies the following 
appearance of {\losp}.

\begin{proposition}
Let $d\in\N$.
If $\mchar=(2,2)$, then $\munit$ is never a direct summand of 
$\tmod(1)^{\hcirc 2d}$, whereas if $\mchar=(3,3)$, then $\munit$ 
appears exactly once.
\end{proposition}

Note the contrast to the semisimple situation, where the multiplicities of
$\munit$ in tensor products $\tmod(1)^{\hcirc 2d}$ are given by the Catalan
numbers, which grow exponentially.

\begin{proof}
Let us prove the (harder) case $\mchar=(3,3)$
by induction on $d$. The case $d=1$ is 
just $\tmod(1)^{\hcirc 2}\cong\munit\oplus\tmod(2)$.
For $d>1$, we observe that $\tmod(v)\hcirc\tmod(1)$ 
for $v>1$ will never contain a summand below $2$ by \fullref{proposition:times-t1}. 
Hence, we are done 
since the summand with the second lowest highest weight in 
$\tmod(1)^{\hcirc(d-1)}$ is $\tmod(2)$, by induction.
\end{proof}

\begin{remark}
Questions about the structure constants of the representation ring have 
been studied for the
finite group $\SLtwo(\mathbb{F}_{\ppar^{k}})$ for a long time. 
(The connection to our setup is
to embed $\SLtwo(\mathbb{F}_{\ppar^{k}})$ into
$\SLtwo(\overline{\mathbb{F}}_{\ppar^{k}})$ via 
fixed points under the Frobenius
twist.) For example, \fullref{lemma:clebsch-gordan} and \cite[Lemma
5]{ErHe-ringel-schur} is used in \cite[Section 3]{Cr-tensor-simple-modules} to
find the finite group analog of fusion rules.
\end{remark}

\subsection{Categorified fusion rules for tensoring with the vector
representation}\label{subsection:fusion-mor} 

The fusion rule from \fullref{proposition:times-t1} describes the multiplicities
of indecomposable tilting modules in the tensor product
$\tmod(v-1)\hcirc\tmod(1)$. In this section, we consider the refined problem of
describing the morphisms that project onto such summands using the
Temperley--Lieb calculus. Specifically, in \fullref{theorem:timest1} we will
decompose the idempotent $\pjw[v{-}1]\hcirc\idtl[1]$ into a sum of orthogonal,
primitive idempotents factoring through $\pjw[v]$ as well as the other $\pjw[v{-}2\ppar^{(i)}]$ predicted by \fullref{proposition:times-t1}.
Conversely, such a decomposition can also be read as a recursive description of
the mixed JW projector $\pjw[v]$ in terms of mixed JW projectors of lower order.

For the following definition, we use scalars determined by evaluating the functions
$\funcg_{\qpar}$ and $\funcf_{\qpar}$ on digits. On all digits, except for the zeroth
one, we use \eqref{eq:f-g}. For the zeroth digit, we instead use:
\begin{gather}\label{eq:quantum-fg}
\funcf_{\xpar}(a)=
\begin{cases} 
(-1)^{a+1}\cdot \tfrac{-2}{\qnum{a}{\xpar}}
&\text{if } 1\leq a\leq\lpar-2,
\\
0 
&\text{if }a=0\text{ or }a=\lpar-1,
\end{cases}
\quad
\funcg_{\xpar}(a)=
\begin{cases}
-\tfrac{\qnum{a+1}{\xpar}}{\qnum{a}{\xpar}}
&\text{if }1\leq a\leq\lpar-1,
\\
-\qnum{2}{\xpar} &\text{if }a=0.
\end{cases}
\end{gather}

Armed with this notation, we now define the morphisms that will feature in the
decomposition of $\pjw[v{-}1]\hcirc\idtl[1]$ into orthogonal, primitive
idempotents. 

\begin{definition}\label{definition:fusionidempotents} Let
$v=\plbase{a_{j},\dots,a_{0}}$ and $0\leq i \leq \taill{v}$. If $a_{i}=1$ and $i\neq j$, then \fix{choose $\alpha,\beta\in\kk$ such that $\funcg_{\qpar}(0)+\alpha+\beta=0$ and} we define

\begin{gather*}
\fusidema{v}{i}:=
\begin{tikzpicture}[anchorbase,scale=0.4,tinynodes]
\draw[pJW] (-0.5,-1.5) rectangle (3.2,-2.5);
\node at (1.35,-2.2) {$\pjwm[v{-}w]$};
\draw[pJW] (-0.5,-3.5) rectangle (2.5,-4.5);
\node at (1,-4.2) {$\pjwm[v{-}2w]$};
\draw[pJW] (-0.5,-5.5) rectangle (3.2,-6.5);
\node at (1.35,-6.2) {$\pjwm[v{-}w]$};
\draw[usual] (2,-2.5) to[out=270,in=180] (2.5,-2.85)node[left,xshift=-1pt,yshift=-2.5pt]{$w$} 
to[out=0,in=270] (3,-2.5);
\draw[usual] (2,-3.5) to[out=90,in=180] (2.5,-3.2) 
to (3,-3.2) to [out=0,in=270] (3.5,-1.5)
node[right,xshift=-2pt,yshift=-2pt]{$w$};
\draw[usual] (2,-4.5) to (2,-5.5);
\draw[usual] (2.5,-5.5) to[out=90,in=180] (3,-5.15) 
to[out=0,in=90] (3.5,-5.5) to (3.5,-6.5)node[right,xshift=-2pt,yshift=1pt]{$w$};
\draw[usual] (0.5,-3.5) to (0.5,-2.5);
\draw[usual] (0.5,-4.5) to (0.5,-5.5);
\end{tikzpicture}
+\,\fix{\alpha\cdot}
\begin{tikzpicture}[anchorbase,scale=0.4,tinynodes]
\draw[pJW] (-0.5,-1.5) rectangle (3.2,-2.5);
\node at (1.35,-2.2) {$\pjwm[v{-}w]$};
\draw[pJW] (-0.5,-3.5) rectangle (2.5,-4.5);
\node at (1,-4.2) {$\pjwm[v{-}2w]$};
\draw[pJW] (-0.5,-5.5) rectangle (3.2,-6.5);
\node at (1.35,-6.2) {$\pjwm[v{-}w]$};
\draw[usual] (2,-2.5) to[out=270,in=180] (2.5,-2.85)node[left,xshift=-1pt,yshift=-2.5pt]{$w$} 
to[out=0,in=270] (3,-2.5);
\draw[usual] (2,-3.5) to[out=90,in=180] (2.5,-3.2) 
to (3,-3.2) to [out=0,in=270] (3.5,-1.5)
node[right,xshift=-2pt,yshift=-2pt]{$w$};
\draw[usual] (2,-4.5) to[out=270,in=180] (2.5,-4.8) 
to (3,-4.8) to [out=0,in=90] (3.5,-6.5)
node[right,xshift=-2pt,yshift=1pt]{$w$};
\draw[usual] (2,-5.5) to[out=90,in=180] (2.5,-5.15)node[left,xshift=-1pt]{$w$}
to[out=0,in=90] (3,-5.5);
\draw[usual] (0.5,-3.5) to (0.5,-2.5);
\draw[usual] (0.5,-4.5) to (0.5,-5.5);
\end{tikzpicture}
,\quad 
w=\ppar^{(i)}
.
\end{gather*}
Here the caps and cups have thickness $w=\ppar^{(i)}$, and are thus admissible.
If $a_{i}=1$ and $i=j$, then we declare $\fusidema{v}{i}=\fix{\fusidemb{v}{i}=}0$ (the diagram makes no sense in this case since $v<2w$). We will also consider the
morphism $\fusidemb{v}{i}$ \fix{obtained by reflecting the first diagram} along a horizontal line \fix{and using $\beta$ as the coefficient of the second diagram}.

For $a_{i}>1$ and $i \neq j$ we consider
\begin{gather*}
\fusidem{v}{i}:=
\tfrac{1}{\funcg_{\qpar}(a_{i}-1)}\cdot
\begin{tikzpicture}[anchorbase,scale=0.4,tinynodes]
\draw[pJW] (-0.5,-1.5) rectangle (2.5,-2.5);
\node at (1,-2.2) {$\pjwm[v{-}w]$};
\draw[pJW] (-0.5,-4.5) rectangle (2.5,-5.5);
\node at (1,-5.2) {$\pjwm[v{-}w]$};
\draw[pJW] (-0.5,-3.25) rectangle (0.5,-3.75);
\draw[usual] (2,-2.5) to[out=270,in=180] (2.5,-3) to[out=0,in=270] 
(3,-2.5) to (3,-1.5) node[right,xshift=-2pt,yshift=-2pt]{$w$};
\draw[usual] (2,-4.5) to[out=90,in=180] (2.5,-4) to[out=0,in=90] 
(3,-4.5) to (3,-5.5) node[right,xshift=-2pt,yshift=1pt]{$w$};
\draw[usual] (0,-2.5) to (0,-3.25);
\draw[usual] (0,-3.75) to (0,-4.5);
\end{tikzpicture}
-
\tfrac{\funcf_{\qpar}(a_{i})}{\funcg_{\qpar}(a_{i}-1)}\cdot
\begin{tikzpicture}[anchorbase,scale=0.4,tinynodes]
\draw[pJW] (-0.5,-1.5) rectangle (2.5,-2.5);
\node at (1,-2.2) {$\pjwm[v{-}w]$};
\draw[pJW] (-0.5,-4.5) rectangle (2.5,-5.5);
\node at (1,-5.2) {$\pjwm[v{-}w]$};
\draw[pJW] (-0.5,-3.25) rectangle (0.5,-3.75);
\draw[usual] (2,-2.5) to[out=270,in=180] (2.5,-3) to[out=0,in=270] 
(3,-2.5) to (3,-1.5) node[right,xshift=-2pt,yshift=-2pt]{$w$};
\draw[usual] (2,-4.5) to[out=90,in=180] (2.5,-4) to[out=0,in=90] 
(3,-4.5) to (3,-5.5) node[right,xshift=-2pt,yshift=1pt]{$w$};
\draw[usual] (0,-2.5) to (0,-3.25);
\draw[usual] (0,-3.75) to (0,-4.5);
\draw[usual] (0.5,-4.5) to[out=90,in=180] (1,-4) to[out=0,in=90] 
(1.5,-4.5);
\draw[usual] (0.5,-2.5) to[out=270,in=180] (1,-3) to[out=0,in=270] 
(1.5,-2.5);
\node at (1.3,-3.65) {$S_{(i)}$};
\end{tikzpicture}
,\quad 
w=\ppar^{(i)}.
\end{gather*}
Here the caps and cups are of thickness $a_{i}\ppar^{(i)}$ and are thus
admissible. If $i=j$, then we use the same formula to define $\fusidem{v}{i}$,
except we omit the second summand. 
\end{definition}	

The \emph{categorified fusion rule} for $\tmod(v-1)\hcirc\tmod(1)$ is now given by the following theorem.

\begin{theorem}\label{theorem:timest1} Let $v=\plbase{a_{j},\dots,a_{0}}$.
\leavevmode

\begin{enumerate}

\item We have the following 
decomposition of $\pjw[v{-}1]\hcirc\idtl[1]$ into a sum of
orthogonal, primitive idempotents.
\begin{gather}\label{eq:timest1-first}
\begin{tikzpicture}[anchorbase,scale=0.25,tinynodes]
\draw[pJW] (-2,-1) rectangle (2,1);
\node at (0,-0.2) {$\pjwm[v{-}1]$};
\draw[usual] (2.5,-1) to (2.5,1);
\end{tikzpicture}
=
\begin{tikzpicture}[anchorbase,scale=0.25,tinynodes]
\draw[pJW] (-2,-1) rectangle (2,1);
\node at (0,-0.2) {$\pjwm[v]$};
\end{tikzpicture}
+
{\textstyle\sum_{i=0}^{\taill{v}}}\,
\fuscaseidem{v}{i}
\quad\text{where } 
\fuscaseidem{v}{i}=
\begin{cases}
0 & \text{if }a_{i}=0,
\\
\fusidema{v}{i}+\fusidemb{v}{i} & \text{if }a_{i}=1,
\\
\fusidem{v}{i} & \text{if }a_{i}>1.
\end{cases}
\end{gather}
For $a_{i}=1$ \fix{and $i\neq j$}, both the summands $\fusidema{v}{\fix{i}}$ and $\fusidemb{v}{\fix{i}}$ are orthogonal, primitive idempotents.

\item Further, we have the following \emph{partial trace rules}. Let $0\leqslant
i\leqslant\taill{v}$, $a_{i}\neq 0$, and $w=\ppar^{(i)}$. For the case $i=j$, we
additionally assume $a_j>1$. Then we have
\begin{gather}\label{eq:ptrstatement}
\begin{tikzpicture}[anchorbase,scale=0.25,tinynodes]
\draw[pJW] (1.5,-1) rectangle (-2.5,1);
\node at (-0.5,-0.1) {$\pjwm[v{-}w]$};
\draw[usual] (1,1) to[out=90,in=180] (1.5,1.5) to[out=0,in=90] 
(2,1) to (2,0) node[right,xshift=-2pt]{$w$} to (2,-1) 
to[out=270,in=0] (1.5,-1.5) to[out=180,in=270] (1,-1);
\node at (0,2.5) {$\phantom{a}$};
\node at (0,-2.5) {$\phantom{a}$};
\end{tikzpicture} 
=
\funcg_{\qpar}(a_{i}-1)\cdot
\begin{tikzpicture}[anchorbase,scale=0.25,tinynodes]
\draw[pJW] (1.5,-1) rectangle (-2.5,1);
\node at (-0.5,-0.1) {$\pjwm[v{-}2w]$};
\node at (0,2.5) {$\phantom{a}$};
\node at (0,-2.5) {$\phantom{a}$};
\end{tikzpicture}
+\funcf_{\qpar}(a_{i}-1)\cdot
\loopdown{i}{v{-}2w}.
\end{gather}
(If $a_{i}=1$ or $i=j$, then the second summand is zero. Even though
$\loopdown{0}{v{-}2w}$ is not defined on its own, this is meaningful in both
cases because $\funcf_{\qpar}(0)=0$ or $v-2w+1$ is an eve.)
\end{enumerate}
\end{theorem}

\begin{example}
For $v=\pbase{4,1,6,6,6,10}{7,11}$, we have $\taill{v}=4$ and
\begin{gather*}
\pjw[v{-}1]\hcirc\idtl[1]=
\pjw[v]+\fusidem{v}{0}+\fusidem{v}{1}+
\fusidem{v}{2}+\fusidem{v}{3}+\big(\fusidema{v}{4}+\fusidemb{v}{4}\big).
\end{gather*}
For $v=\pbase{1,1,1,1}{2,2}$, we have $\taill{v}=4$ and
\begin{gather*}
\pjw[v{-}1]\hcirc\idtl[1]=\pjw[v]+\big(\fusidema{v}{0}
+\fusidemb{v}{0}\big)+
\big(\fusidema{v}{1}+\fusidemb{v}{1}\big)+\big(\fusidema{v}{2}
+\fusidemb{v}{2}\big),
\end{gather*}
where we note $\fusidema{v}{3}=\fusidemb{v}{3}=0$ since the leading
digit is $a_{3}=1$, and $\fuscaseidem{v}{4}=0$ since $a_{4}=0$. The occurrence of
multiple pairs $\fusidema{v}{i}+\fusidemb{v}{i}$ is an instance of {\losp}. For
$\lpar\neq 2$ and $\ppar\neq 2$ we encounter at most one pair of the form
$\fusidema{v}{i}+ \fusidemb{v}{i}$ since $a_{i}=1$ implies that $\taill{v}\leq i$. 
\end{example}

\begin{remark}
The fusion rule \eqref{eq:jw-recursion} can be used to express classical JW
projectors in terms of JW projectors of lower order. Analogously,
\fullref{theorem:timest1} gives a recursion of $\ppar\lpar$JW projectors in
terms of $\ppar\lpar$JW projectors of lower order. This is in contrast to
the defining description in \eqref{eq:recursion-formula}, which uses
classical JW projectors.
\end{remark}

\begin{remark}
In the complex quantum group case and for $v\leq 2\lpar-2$, the fusion rule
\eqref{eq:timest1-first} can be deduced from \cite[Lemma
3.2]{BlDeReMu-dia-small-qgroup}. The three cases of their rule reflect the
trichotomy of $a_{0}=0$, $a_{0}=1$, and $a_{0}>1$.
\end{remark}

\begin{remark}
We do not know a good partial trace formula of type
\eqref{eq:ptrstatement} in the case $a_{i}=1$, not even for $i=0$ and $w=1$.
(One can write down a formula 
using \eqref{eq:timest1-first}, of course.)
We expect this formula to be more complicated, because it deals with a
generation increase (on comparison with the increased complexity of the fusion
rule when the generation drops).
\end{remark}

\begin{proof}[Proof of \fullref{theorem:timest1}]
The proof proceeds by induction on $v$. To do so, we
split the statement of the theorem into the following two types of assertions.
\begin{enumerate}

\item[$\fus(v)$] denotes: {\it The categorified fusion rules are given by
\eqref{eq:timest1-first} for $v$.}

\item[$\trp(v)$] denotes: {\it The partial trace rules
\eqref{eq:ptrstatement} hold for $v$.}
\end{enumerate}
The former makes sense for all $v\geq 1$ and the latter for $v\geq 2$. 
We will also write $\fus(<v)$ to express the assertion that $\fus(w)$ holds for all $1\leq w <v$, and similarly for $\trp(<v)$.

The base cases for the inductive argument is
$\fus(1)$, which is immediate.

The induction step will be accomplished by two arguments that we separate into two distinct statements below. 
\fullref{lemma:fusiontotrace} shows the implication
$\fus(<v)\implies\trp(v)$ for all $v\geq 2$.
\fullref{lemma:tracetofusion} shows the 
implication $\trp(v)\implies\fus(v)$ for all $v\geq 2$. 
Induction then shows that both assertions 
hold for all relevant values of $v$. 
\end{proof}

We now turn to the two lemmas
that form the heart of the proof of \fullref{theorem:timest1}.

\begin{lemma}\label{lemma:fusiontotrace}
We have $\fus(<v)\implies\trp(v)$ for all $v\geq 2$.
\end{lemma}

\begin{proof}
We first consider $i=0$, where we have $w=1$ and 
assume $a_{0}\neq 0$, {\ie} we aim to prove
\begin{gather}\label{eq:ptrclaim}
\begin{tikzpicture}[anchorbase,scale=0.25,tinynodes]
\draw[pJW] (1.5,-1) rectangle (-2.5,1);
\node at (-0.5,-0.1) {$\pjwm[v{-}1]$};
\draw[usual] (1,1) to[out=90,in=180] (1.5,1.5) to[out=0,in=90] 
(2,1) to (2,0) to (2,-1) 
to[out=270,in=0] (1.5,-1.5) to[out=180,in=270] (1,-1);
\node at (0,2.5) {$\phantom{a}$};
\node at (0,-2.5) {$\phantom{a}$};
\end{tikzpicture} 
=
\funcg_{\qpar}(a_{0}-1)\cdot
\begin{tikzpicture}[anchorbase,scale=0.25,tinynodes]
\draw[pJW] (1.5,-1) rectangle (-2.5,1);
\node at (-0.5,-0.1) {$\pjwm[v{-}2]$};
\node at (0,2.5) {$\phantom{a}$};
\node at (0,-2.5) {$\phantom{a}$};
\end{tikzpicture}
+\funcf_{\qpar}(a_{0}-1)\cdot
\loopdown{0}{v{-}2}.
\end{gather}
To verify this, we will use the fusion rule for $v-1$ in reverse to expand the
projector $\pjw[v{-}1]$. If $a_{0}=1$, then we have
$\pjw[v{-}1]=\pjw[v{-}2]\hcirc\idtl[1]$. The claimed statement follows since
the circle value is $-\qnum{2}{\qpar}=\funcg_{\qpar}(0)$ and the second term is zero by definition. If $\lpar=2$, then we are done. Thus, from now on suppose
that $\lpar>2$. 

We consider the case
$a_{0}=2$, where the fusion 
rule involves $\fusidema{v{-}1}{0}+\fusidemb{v{-}1}{0}$:
\begin{gather*}
\begin{tikzpicture}[anchorbase,scale=0.25,tinynodes]
\draw[pJW] (1.5,-1) rectangle (-2.5,1);
\node at (-0.5,-0.1) {$\pjwm[v{-}1]$};
\draw[usual] (1,1) to[out=90,in=180] (1.5,1.5) to[out=0,in=90] 
(2,1) to (2,0) to (2,-1) 
to[out=270,in=0] (1.5,-1.5) to[out=180,in=270] (1,-1);
\node at (0,2.5) {$\phantom{a}$};
\node at (0,-2.5) {$\phantom{a}$};
\end{tikzpicture}
=
\begin{tikzpicture}[anchorbase,scale=0.25,tinynodes]
\draw[pJW] (0.5,-1) rectangle (-2.5,1);
\node at (-1,-0.1) {$\pjwm[v{-}2]$};
\draw[usual] (1,1) to[out=90,in=180] (1.5,1.5) to[out=0,in=90] 
(2,1) to (2,0) to (2,-1) 
to[out=270,in=0] (1.5,-1.5) to[out=180,in=270] (1,-1) to (1,1);
\node at (0,2.5) {$\phantom{a}$};
\node at (0,-2.5) {$\phantom{a}$};
\end{tikzpicture}
-
\left(
\begin{tikzpicture}[anchorbase,scale=0.4,tinynodes]
\draw[pJW] (-0.5,-1.5) rectangle (3.2,-2.5);
\node at (1.35,-2.2) {$\pjwm[v{-}2]$};
\draw[pJW] (-0.5,-3.5) rectangle (2.5,-4.5);
\node at (1,-4.2) {$\pjwm[v{-}3]$};
\draw[pJW] (-0.5,-5.5) rectangle (3.2,-6.5);
\node at (1.35,-6.2) {$\pjwm[v{-}2]$};
\draw[usual] (2,-2.5) to[out=270,in=180] (2.5,-2.85) 
to[out=0,in=270] (3,-2.5);
\draw[usual] (2,-3.5) to[out=90,in=180] (2.5,-3.2) 
to (3,-3.2) to [out=0,in=270] (3.5,-1.5);
\draw[usual] (2,-4.5) to (2,-5.5);
\draw[usual] (2.5,-5.5) to[out=90,in=180] (3,-5.15) 
to[out=0,in=90] (3.5,-5.5) to (3.5,-6.5);
\draw[usual] (0.5,-3.5) to (0.5,-2.5);
\draw[usual] (0.5,-4.5) to (0.5,-5.5);
\draw[usual] (3.5,-6.5) to[out=270,in=180] (4,-7) to[out=0,in=270]
(4.5,-6.5) to (4.5,-1.5) to[out=90,in=0] (4,-1) to[out=180,in=90]
(3.5,-1.5);
\end{tikzpicture}
+\,
\begin{tikzpicture}[anchorbase,scale=0.4,tinynodes]
\draw[pJW] (-0.5,-1.5) rectangle (3.2,-2.5);
\node at (1.35,-2.2) {$\pjwm[v{-}2]$};
\draw[pJW] (-0.5,-3.5) rectangle (2.5,-4.5);
\node at (1,-4.2) {$\pjwm[v{-}3]$};
\draw[pJW] (-0.5,-5.5) rectangle (3.2,-6.5);
\node at (1.35,-6.2) {$\pjwm[v{-}2]$};
\draw[usual] (2,-2.5) to[out=270,in=180] (2.5,-2.85) 
to[out=0,in=270] (3,-2.5);
\draw[usual] (2,-3.5) to[out=90,in=180] (2.5,-3.2) 
to (3,-3.2) to [out=0,in=270] (3.5,-1.5);
\draw[usual] (2,-4.5) to[out=270,in=180] (2.5,-4.8) 
to (3,-4.8) to [out=0,in=90] (3.5,-6.5);
\draw[usual] (2,-5.5) to[out=90,in=180] (2.5,-5.15) 
to[out=0,in=90] (3,-5.5);
\draw[usual] (0.5,-3.5) to (0.5,-2.5);
\draw[usual] (0.5,-4.5) to (0.5,-5.5);
\draw[usual] (3.5,-6.5) to[out=270,in=180] (4,-7) to[out=0,in=270]
(4.5,-6.5) to (4.5,-1.5) to[out=90,in=0] (4,-1) to[out=180,in=90]
(3.5,-1.5);
\end{tikzpicture}
\right)
-\left(
\begin{tikzpicture}[anchorbase,scale=0.4,tinynodes]
\draw[pJW] (-0.5,-1.5) rectangle (3.2,-2.5);
\node at (1.35,-2.2) {$\pjwm[v{-}2]$};
\draw[pJW] (-0.5,-3.5) rectangle (2.5,-4.5);
\node at (1,-4.2) {$\pjwm[v{-}3]$};
\draw[pJW] (-0.5,-5.5) rectangle (3.2,-6.5);
\node at (1.35,-6.2) {$\pjwm[v{-}2]$};
\draw[usual] (2.5,-2.5) to[out=270,in=180] (3,-2.85) 
to[out=0,in=270] (3.5,-2.5) to (3.5,-1.5);
\draw[usual] (2,-3.5) to (2,-2.5);
\draw[usual] (2,-4.5) to[out=270,in=180] (2.5,-4.8) 
to (3,-4.8) to [out=0,in=90] (3.5,-6.5);
\draw[usual] (2,-5.5) to[out=90,in=180] (2.5,-5.15) 
to[out=0,in=90] (3,-5.5);
\draw[usual] (0.5,-3.5) to (0.5,-2.5);
\draw[usual] (0.5,-4.5) to (0.5,-5.5);
\draw[usual] (3.5,-6.5) to[out=270,in=180] (4,-7) to[out=0,in=270]
(4.5,-6.5) to (4.5,-1.5) to[out=90,in=0] (4,-1) to[out=180,in=90]
(3.5,-1.5);
\end{tikzpicture}
+\,
\begin{tikzpicture}[anchorbase,scale=0.4,tinynodes]
\draw[pJW] (-0.5,-1.5) rectangle (3.2,-2.5);
\node at (1.35,-2.2) {$\pjwm[v{-}2]$};
\draw[pJW] (-0.5,-3.5) rectangle (2.5,-4.5);
\node at (1,-4.2) {$\pjwm[v{-}3]$};
\draw[pJW] (-0.5,-5.5) rectangle (3.2,-6.5);
\node at (1.35,-6.2) {$\pjwm[v{-}2]$};
\draw[usual] (2,-2.5) to[out=270,in=180] (2.5,-2.85) 
to[out=0,in=270] (3,-2.5);
\draw[usual] (2,-3.5) to[out=90,in=180] (2.5,-3.2) 
to (3,-3.2) to [out=0,in=270] (3.5,-1.5);
\draw[usual] (2,-4.5) to[out=270,in=180] (2.5,-4.8) 
to (3,-4.8) to [out=0,in=90] (3.5,-6.5);
\draw[usual] (2,-5.5) to[out=90,in=180] (2.5,-5.15) 
to[out=0,in=90] (3,-5.5);
\draw[usual] (0.5,-3.5) to (0.5,-2.5);
\draw[usual] (0.5,-4.5) to (0.5,-5.5);
\draw[usual] (3.5,-6.5) to[out=270,in=180] (4,-7) to[out=0,in=270]
(4.5,-6.5) to (4.5,-1.5) to[out=90,in=0] (4,-1) to[out=180,in=90]
(3.5,-1.5);
\end{tikzpicture}
\right)
.
\end{gather*}
The first term in each of the brackets 
is $\loopdown{0}{v{-}2}$; the second is
$(\loopdown{0}{v{-}2})^{2}$, which is 
zero by \eqref{eq:dud}. Since 
$\funcg_{\qpar}(1)=-\qnum{2}{\qpar}$ and
$\funcf_{\qpar}(1)=-2$ the claim 
follows for $a_{0}=2$.

Next we consider $a_{0}\in\{3,\dots,\lpar-1\}$.  
Here we use $-\qnum{2}{\qpar}+\tfrac{\qnum{a_{0}-2}{\qpar}}
{\qnum{a_{0}-1}{\qpar}}=-
\tfrac{\qnum{a_{0}}{\qpar}}{\qnum{a_{0}-1}{\qpar}}$ 
to get
\begin{gather}\label{eq:ptrfusion}
\begin{aligned}
\begin{tikzpicture}[anchorbase,scale=0.25,tinynodes]
\draw[pJW] (1.5,-1) rectangle (-2.5,1);
\node at (-0.5,-0.1) {$\pjwm[v{-}1]$};
\draw[usual] (1,1) to[out=90,in=180] (1.5,1.5) to[out=0,in=90] 
(2,1) to (2,0) to (2,-1) 
to[out=270,in=0] (1.5,-1.5) to[out=180,in=270] (1,-1);
\node at (0,2.5) {$\phantom{a}$};
\node at (0,-2.5) {$\phantom{a}$};
\end{tikzpicture}
&=
\begin{tikzpicture}[anchorbase,scale=0.25,tinynodes]
\draw[pJW] (0.5,-1) rectangle (-2.5,1);
\node at (-1,-0.1) {$\pjwm[v{-}2]$};
\draw[usual] (1,1) to[out=90,in=180] (1.5,1.5) to[out=0,in=90] 
(2,1) to (2,0) to (2,-1) 
to[out=270,in=0] (1.5,-1.5) to[out=180,in=270] (1,-1) to (1,1);
\node at (0,2.5) {$\phantom{a}$};
\node at (0,-2.5) {$\phantom{a}$};
\end{tikzpicture}
-
\left(
\tfrac{1}{\funcg_{\qpar}(a_{0}-2)}\cdot
\begin{tikzpicture}[anchorbase,scale=0.4,tinynodes]
\draw[pJW] (-0.5,-1.5) rectangle (2.5,-2.5);
\node at (1,-2.2) {$\pjwm[v{-}2]$};
\draw[pJW] (-0.5,-4.5) rectangle (2.5,-5.5);
\node at (1,-5.2) {$\pjwm[v{-}2]$};
\draw[pJW] (-0.5,-3.25) rectangle (0.5,-3.75);
\draw[usual] (2,-2.5) to[out=270,in=180] (2.5,-3) to[out=0,in=270] 
(3,-2.5) to (3,-1.5);
\draw[usual] (2,-4.5) to[out=90,in=180] (2.5,-4) to[out=0,in=90] 
(3,-4.5) to (3,-5.5);
\draw[usual] (0,-2.5) to (0,-3.25);
\draw[usual] (0,-3.75) to (0,-4.5);
\draw[usual] (3,-5.5) to[out=270,in=180] (3.5,-6) to[out=0,in=270]
(4,-5.5) to (4,-1.5) to[out=90,in=0] (3.5,-1) to[out=180,in=90]
(3,-1.5);
\end{tikzpicture}  
-\tfrac{\funcf_{\qpar}(a_{0}-1)}{\funcg_{\qpar}(a_{0}-2)}\cdot
\begin{tikzpicture}[anchorbase,scale=0.4,tinynodes]
\draw[pJW] (-0.5,-1.5) rectangle (2.5,-2.5);
\node at (1,-2.2) {$\pjwm[v{-}2]$};
\draw[pJW] (-0.5,-4.5) rectangle (2.5,-5.5);
\node at (1,-5.2) {$\pjwm[v{-}2]$};
\draw[pJW] (-0.5,-3.25) rectangle (0.5,-3.75);
\draw[usual] (2,-2.5) to[out=270,in=180] (2.5,-3) to[out=0,in=270] 
(3,-2.5) to (3,-1.5);
\draw[usual] (2,-4.5) to[out=90,in=180] (2.5,-4) to[out=0,in=90] 
(3,-4.5) to (3,-5.5);
\draw[usual] (0,-2.5) to (0,-3.25);
\draw[usual] (0,-3.75) to (0,-4.5);
\draw[usual] (0.5,-4.5) to[out=90,in=180] (1,-4) to[out=0,in=90] 
(1.5,-4.5);
\draw[usual] (0.5,-2.5) to[out=270,in=180] (1,-3) to[out=0,in=270] 
(1.5,-2.5);
\node at (1,-3.65) {$S_{0}$};
\draw[usual] (3,-5.5) to[out=270,in=180] (3.5,-6) to[out=0,in=270]
(4,-5.5) to (4,-1.5) to[out=90,in=0] (3.5,-1) to[out=180,in=90]
(3,-1.5);
\end{tikzpicture}
\right)
\\
&=
\funcg_{\qpar}(a_{0}-1)\cdot
\begin{tikzpicture}[anchorbase,scale=0.25,tinynodes]
\draw[pJW] (1.5,-1) rectangle (-2.5,1);
\node at (-0.5,-0.1) {$\pjwm[v{-}2]$};
\node at (0,2.5) {$\phantom{a}$};
\node at (0,-2.5) {$\phantom{a}$};
\end{tikzpicture}
+
\tfrac{\funcf_{\qpar}(a_{0}-1)}{\funcg_{\qpar}(a_{0}-2)}\cdot
\begin{tikzpicture}[anchorbase,scale=0.4,tinynodes]
\draw[pJW] (-0.5,-1.5) rectangle (2.5,-2.5);
\node at (1,-2.2) {$\pjwm[v{-}2]$};
\draw[pJW] (-0.5,-4.5) rectangle (2.5,-5.5);
\node at (1,-5.2) {$\pjwm[v{-}2]$};
\draw[pJW] (-0.5,-3.25) rectangle (0.5,-3.75);
\draw[usual] (2,-4.5) to (2,-2.5);
\draw[usual] (0,-2.5) to (0,-3.25);
\draw[usual] (0,-3.75) to (0,-4.5);
\draw[usual] (0.5,-4.5) to[out=90,in=180] (1,-4) to[out=0,in=90] 
(1.5,-4.5);
\draw[usual] (0.5,-2.5) to[out=270,in=180] (1,-3) to[out=0,in=270] 
(1.5,-2.5);
\node at (1,-3.65) {$S_{0}$};
\end{tikzpicture}
,
\end{aligned}
\end{gather}
and it remains to compute the final term. To this end, 
we will use the fusion rule for $v-2a_{0}+2$ 
on the mini box (using induction), which corresponds to
$\pjw[v{-}2a_{0}{+}1]$. The last digit of its
relevant $\ppar\lpar$-adic expansion 
is an element of $\{3,\dots,\lpar-1\}$, namely 
$\lpar-a_{0}+2$. We claim that fusion results in:
\begin{gather}\label{eq:middleminifusion}
\begin{tikzpicture}[anchorbase,scale=0.4,tinynodes]
\draw[pJW] (-0.5,-1.5) rectangle (2.5,-2.5);
\node at (1,-2.2) {$\pjwm[v{-}2]$};
\draw[pJW] (-0.5,-4.5) rectangle (2.5,-5.5);
\node at (1,-5.2) {$\pjwm[v{-}2]$};
\draw[pJW] (-0.5,-3.25) rectangle (0.5,-3.75);
\draw[usual] (2,-4.5) to (2,-2.5);
\draw[usual] (0,-2.5) to (0,-3.25);
\draw[usual] (0,-3.75) to (0,-4.5);
\draw[usual] (0.5,-4.5) to[out=90,in=180] (1,-4) to[out=0,in=90] 
(1.5,-4.5);
\draw[usual] (0.5,-2.5) to[out=270,in=180] (1,-3) to[out=0,in=270] 
(1.5,-2.5);
\node at (1,-3.65) {$S_{0}$};
\end{tikzpicture}
=
\begin{tikzpicture}[anchorbase,scale=0.4,tinynodes]
\draw[pJW] (-0.5,-1.5) rectangle (2.5,-2.5);
\node at (1,-2.2) {$\pjwm[v{-}2]$};
\draw[pJW] (-0.5,-4.5) rectangle (2.5,-5.5);
\node at (1,-5.2) {$\pjwm[v{-}2]$};
\draw[pJW] (-0.5,-3.25) rectangle (2.5,-3.75);
\draw[usual] (2,-3.25) to (2,-2.5);
\draw[usual] (2,-4.5) to (2,-3.75);
\draw[usual] (0,-2.5) to (0,-3.25);
\draw[usual] (0,-3.75) to (0,-4.5);
\draw[usual] (0.5,-4.5) to[out=90,in=180] (1,-4) to[out=0,in=90] 
(1.5,-4.5);
\draw[usual] (0.5,-2.5) to[out=270,in=180] (1,-3) to[out=0,in=270] 
(1.5,-2.5);
\end{tikzpicture}
+
\tfrac{1}{\funcg_{\qpar}(\lpar-a_{0}+1)}
\loopdown{0}{v{-}2}
+\text{lower order terms}
=
\tfrac{1}{\funcg_{\qpar}(\lpar-a_{0}+1)}
\loopdown{0}{v{-}2}
.
\end{gather}
Here, we have three things to check. To start, the first term in the middle is
zero because $\tmod(v-2)$ and $\tmod(v-2a_{0}+2)$ do not share any common Weyl
factors. Second, the fusion rule includes a term of the form
$\fusidem{v{-}2a_{0}{+}2}{0}$, which is typically a sum of two diagrams
(although the second may not appear in some cases). The first diagram combines
with the present caps and cups to form
$\tfrac{1}{\funcg_{\qpar}(\lpar-a_{0}+1)}\loopdown{0}{v{-}2}$. The second
diagram (if it is present at all) vanishes when it is sandwiched because of the
containment relation $\Up{0}\Up{0}=0$. Third, one checks that all possible terms
of even lower order arising from the fusion rule become zero when sandwiched.
Such terms only arise if $\taill{v-2a_{0}+2}>0$, {\ie} if $a_{0}=3$. This
finishes the verification of \eqref{eq:middleminifusion}, which we now use to
rewrite \eqref{eq:ptrfusion}. The coefficient of $\loopdown{0}{v{-}2}$ is
\begin{gather*}
\tfrac{\funcf_{\qpar}(a_{0}-1)}{\funcg_{\qpar}(a_{0}-2)}\cdot
\tfrac{1}{\funcg_{\qpar}(\lpar-a_{0}+1)}
=
\tfrac{\funcf_{\qpar}(a_{0}-1)}{\funcg_{\qpar}(a_{0}-2)}\cdot
\funcg_{\qpar}(a_{0}-2)
=
\funcf_{\qpar}(a_{0}-1),
\end{gather*}
where we have simplified $\funcg_{\qpar}(\lpar-a_{0}+1)^{-1}=
\funcg_{\qpar}(a_{0}-2)$ using $\qnum{\lpar}{\qpar}=0$ (this does not hold in
the semisimple case). Thus we have verified the partial trace claim
\eqref{eq:ptrclaim}. 

Finally we consider the case $0<i\leq\taill{v}$, which, surprisingly, is much
easier to prove. Recall that we assume that $a_{i}\neq 0$ and aim to prove: 
\begin{gather*}
\begin{tikzpicture}[anchorbase,scale=0.25,tinynodes]
\draw[pJW] (1.5,-1) rectangle (-2.5,1);
\node at (-0.5,-0.1) {$\pjwm[v{-}w]$};
\draw[usual] (1,1) to[out=90,in=180] (1.5,1.5) to[out=0,in=90] 
(2,1) to (2,0) node[right,xshift=-2pt]{$w$} to (2,-1) 
to[out=270,in=0] (1.5,-1.5) to[out=180,in=270] (1,-1);
\node at (0,2.5) {$\phantom{a}$};
\node at (0,-2.5) {$\phantom{a}$};
\end{tikzpicture} 
=
\funcg_{\qpar}(a_{i}-1)\cdot
\begin{tikzpicture}[anchorbase,scale=0.25,tinynodes]
\draw[pJW] (1.5,-1) rectangle (-2.5,1);
\node at (-0.5,-0.1) {$\pjwm[v{-}2w]$};
\node at (0,2.5) {$\phantom{a}$};
\node at (0,-2.5) {$\phantom{a}$};
\end{tikzpicture}
+\funcf_{\qpar}(a_{i}-1)\cdot
\loopdown{i}{v{-}2w}.
\end{gather*}
To verify this claim, we calculate that 
$v-w+1=\plbase{a_{j},\dots,a_{i},0,\dots,0}$.
In particular, we can use (a slight generalization of) \fullref{proposition:more-ptrace} 
to trace off $w-1=\ppar^{(i)}-1$ strands
and we get
\begin{gather*}
\begin{tikzpicture}[anchorbase,scale=0.25,tinynodes]
\draw[pJW] (1.5,-1) rectangle (-2.5,1);
\node at (-0.5,-0.1) {$\pjwm[v{-}w]$};
\draw[usual] (1,1) to[out=90,in=180] (1.5,1.5) to[out=0,in=90] 
(2,1) to (2,0) node[right,xshift=-2pt]{$w-1$} to (2,-1) 
to[out=270,in=0] (1.5,-1.5) to[out=180,in=270] (1,-1);
\end{tikzpicture}   
=
\funcg_{\qpar}(a_{i}-1)\cdot
\begin{tikzpicture}[anchorbase,scale=0.25,tinynodes]
\draw[pJW] (2,0.5) rectangle (-1,-0.5);
\draw[pJW] (2,3.5) rectangle (-1,4.5);
\node at (0.5,-0.2) {$\pjwm[x{-}1]$};
\node at (0.5,3.8) {$\pjwm[x{-}1]$};
\draw[usual] (0,0.5) to[out=90,in=180] (0.5,1) to[out=0,in=90] (1,0.5);
\draw[usual] (0,3.5) to[out=270,in=180] (0.5,3) to[out=0,in=270] (1,3.5);
\draw[usual] (-0.5,0.5) to (-0.5,3.5);
\end{tikzpicture}
+
\funcf_{\qpar}(a_{i}-1)\cdot
\begin{tikzpicture}[anchorbase,scale=0.25,tinynodes]
\draw[pJW] (2,0.5) rectangle (-1,-0.5);
\draw[pJW] (2,3.5) rectangle (-1,4.5);
\node at (0.5,-0.2) {$\pjwm[x{-}1]$};
\node at (0.5,3.8) {$\pjwm[x{-}1]$};
\draw[usual] (0,0.5) to[out=90,in=180] (0.5,1) to[out=0,in=90] (1,0.5);
\draw[usual] (-0.25,0.5) to[out=90,in=180] (0.5,1.25) to[out=0,in=90] (1.25,0.5);
\draw[usual] (0,3.5) to[out=270,in=180] (0.5,3) to[out=0,in=270] (1,3.5);
\draw[usual] (-0.25,3.5) to[out=270,in=180] (0.5,2.75) to[out=0,in=270] (1.25,3.5);
\draw[usual] (-0.5,0.5) to (-0.5,3.5);
\node at (1,3.5) {$\phantom{a}$};
\node at (1,-0.5) {$\phantom{a}$};
\end{tikzpicture},
\end{gather*}
where $x=\plbase{a_{j},\dots,a_{i}-1,0,\dots,1}$ (if
$i=j$, the $\funcf$-term vanishes but the $\funcg$-term does not because we
assume $a_j>1$). Now we use shortening to compute the full partial trace as:
\begin{gather*}
\begin{tikzpicture}[anchorbase,scale=0.25,tinynodes]
\draw[pJW] (1.5,-1) rectangle (-2.5,1);
\node at (-0.5,-0.1) {$\pjwm[v{-}w]$};
\draw[usual] (1,1) to[out=90,in=180] (1.5,1.5) to[out=0,in=90] 
(2,1) to (2,0)node[right]{$w$} to (2,-1) 
to[out=270,in=0] (1.5,-1.5) to[out=180,in=270] (1,-1);
\end{tikzpicture} 
=
\funcg_{\qpar}(a_{i}-1)\cdot
\begin{tikzpicture}[anchorbase,scale=0.25,tinynodes]
\draw[pJW] (2,0.5) rectangle (-1,-0.5);
\draw[pJW] (2,3.5) rectangle (-1,4.5);
\node at (0.5,-0.2) {$\pjwm[x{-}1]$};
\node at (0.5,3.8) {$\pjwm[x{-}1]$};
\draw[usual] (0,0.5) to[out=90,in=180] (0.5,1) to[out=0,in=90] (1,0.5);
\draw[usual] (0,3.5) to[out=270,in=180] (0.5,3) to[out=0,in=270] (1,3.5);
\draw[usual] (-0.5,0.5) to (-0.5,3.5);
\draw[usual] (1.5,4.5) to[out=90,in=180] (2,5) to[out=0,in=90] 
(2.5,4.5) to (2.5,-0.5) 
to[out=270,in=0] (2,-1) to[out=180,in=270] (1.5,-0.5);
\end{tikzpicture}
+
\funcf_{\qpar}(a_{i}-1)\cdot
\begin{tikzpicture}[anchorbase,scale=0.25,tinynodes]
\draw[pJW] (2,0.5) rectangle (-1,-0.5);
\draw[pJW] (2,3.5) rectangle (-1,4.5);
\node at (0.5,-0.2) {$\pjwm[x{-}1]$};
\node at (0.5,3.8) {$\pjwm[x{-}1]$};
\draw[usual] (0,0.5) to[out=90,in=180] (0.5,1) to[out=0,in=90] (1,0.5);
\draw[usual] (-0.25,0.5) to[out=90,in=180] (0.5,1.25) to[out=0,in=90] (1.25,0.5);
\draw[usual] (0,3.5) to[out=270,in=180] (0.5,3) to[out=0,in=270] (1,3.5);
\draw[usual] (-0.25,3.5) to[out=270,in=180] (0.5,2.75) to[out=0,in=270] (1.25,3.5);
\draw[usual] (-0.5,0.5) to (-0.5,3.5);
\node at (1,3.5) {$\phantom{a}$};
\node at (1,-0.5) {$\phantom{a}$};
\draw[usual] (1.5,4.5) to[out=90,in=180] (2,5) to[out=0,in=90] 
(2.5,4.5) to (2.5,-0.5) 
to[out=270,in=0] (2,-1) to[out=180,in=270] (1.5,-0.5);
\end{tikzpicture}
=
\funcg_{\qpar}(a_{i}-1)\cdot
\begin{tikzpicture}[anchorbase,scale=0.25,tinynodes]
\draw[pJW] (0.75,0.5) rectangle (-1,-0.5);
\draw[pJW] (0.75,3.5) rectangle (-1,4.5);
\draw[usual] (0,0.5) to[out=90,in=180] (0.5,1) to[out=0,in=90] (1,0.5);
\draw[usual] (0,3.5) to[out=270,in=180] (0.5,3) to[out=0,in=270] (1,3.5);
\draw[usual] (-0.5,0.5) to (-0.5,3.5);
\draw[usual] (1,3.5) to (1,4.5) to[out=90,in=180] (1.5,5) to[out=0,in=90] 
(2,4.5) to (2,-0.5) 
to[out=270,in=0] (1.5,-1) to[out=180,in=270] (1,-0.5) to (1,0.5);
\end{tikzpicture}
+
\funcf_{\qpar}(a_{i}-1)\cdot
\begin{tikzpicture}[anchorbase,scale=0.25,tinynodes]
\draw[pJW] (1.025,0.5) rectangle (-1,-0.5);
\draw[pJW] (1.025,3.5) rectangle (-1,4.5);
\draw[usual] (0,0.5) to[out=90,in=180] (0.5,1) to[out=0,in=90] (1,0.5);
\draw[usual] (-0.25,0.5) to[out=90,in=180] (0.5,1.25) to[out=0,in=90] (1.25,0.5);
\draw[usual] (0,3.5) to[out=270,in=180] (0.5,3) to[out=0,in=270] (1,3.5);
\draw[usual] (-0.25,3.5) to[out=270,in=180] (0.5,2.75) to[out=0,in=270] (1.25,3.5);
\draw[usual] (-0.5,0.5) to (-0.5,3.5);
\node at (1,3.5) {$\phantom{a}$};
\node at (1,-0.5) {$\phantom{a}$};
\draw[usual] (1.25,3.5) to (1.25,4.5) to[out=90,in=180] (1.75,5) to[out=0,in=90] 
(2.25,4.5) to (2.25,-0.5) 
to[out=270,in=0] (1.75,-1) to[out=180,in=270] (1.25,-0.5) to (1.25,0.5);
\end{tikzpicture}
\; ,
\end{gather*}
which we pull straight to get the claimed partial trace formula.
\end{proof}

\begin{lemma}\label{lemma:tracetofusion}
We have $\trp(v)\implies\fus(v)$ for all $v\geq 2$. 
\end{lemma}

\begin{proof}
We start with a few observations. First,
\fullref{proposition:times-t1} ensures that we know how many orthogonal,
primitive idempotents to expect in the categorified fusion rule. Second, by the
same arguments as in the proof of 
\fullref{theorem:well-defined}
(however, it is easier in this case since we only need to tensor with $\tmod(1)$) the idempotents
projecting onto isotypic components are uniquely determined by the
property of absorbing $\pjw[v{-}1]\hcirc\idtl[1]$. These isotypic idempotents
are automatically orthogonal because a straightforward computation, using
\fullref{proposition:multiplicities} and \fullref{lemma:clebsch-gordan}, shows
that the isotypic components share no Weyl factors, which implies that there are
no non-zero morphisms between them by \fullref{lemma:no-delta-overlap}.

Combining these observations, it remains to show that the morphisms
$\fusidem{v}{i}$, $\fusidema{v}{i}$, $\fusidemb{v}{i}$ from
\fullref{definition:fusionidempotents} satisfy the absorption property and are indeed
idempotents whenever they appear in \eqref{eq:timest1-first}. Finally, we also
check that $\fusidema{v}{i}$ and $\fusidemb{v}{i}$ are orthogonal.

\textit{Absorption.} Let us first check that all of the candidate
idempotents appearing in \fullref{theorem:timest1} absorb
$\pjw[v{-}1]\hcirc\idtl[1]$. For $\pjw[v]$, $\fusidem{v}{0}$, $\fusidema{v}{0}$,
and $\fusidemb{v}{0}$, this follows immediately from the absorption properties of
the $\ppar\lpar$JW projectors -- see \fullref{proposition:jw-properties}.(a). For
the cases with $i>0$ we use the shortening property from
\fullref{proposition:jw-properties}.(b).

We will now verify the idempotency of the candidate expressions on a case-by-case basis.

\textit{Idempotency (and orthogonality) of $\fusidema{v}{i}$ and $\fusidemb{v}{i}$.}
We start with the \fix{terms $\fusidema{v}{i}$ and $\fusidemb{v}{i}$}, which \fix{are} defined \fix{using} the two diagrams
\begin{gather*}
\fusidema{v}{i}=\morstuff{X}+\fix{\alpha}\morstuff{Y},\quad
\fix{\fusidemb{v}{i}=\fliph[\morstuff{X}]+\beta\morstuff{Y}},\quad\text{where}\quad
\morstuff{X}=
\begin{tikzpicture}[anchorbase,scale=0.4,tinynodes]
\draw[pJW] (-0.5,-1.5) rectangle (3.2,-2.5);
\node at (1.35,-2.2) {$\pjwm[v{-}w]$};
\draw[pJW] (-0.5,-3.5) rectangle (2.5,-4.5);
\node at (1,-4.2) {$\pjwm[v{-}2w]$};
\draw[pJW] (-0.5,-5.5) rectangle (3.2,-6.5);
\node at (1.35,-6.2) {$\pjwm[v{-}w]$};
\draw[usual] (2,-2.5) to[out=270,in=180] (2.5,-2.85)node[left,xshift=-1pt,yshift=-2.5pt]{$w$} 
to[out=0,in=270] (3,-2.5);
\draw[usual] (2,-3.5) to[out=90,in=180] (2.5,-3.2) 
to (3,-3.2) to [out=0,in=270] (3.5,-1.5)
node[right,xshift=-2pt,yshift=-2pt]{$w$};
\draw[usual] (2,-4.5) to (2,-5.5);
\draw[usual] (2.5,-5.5) to[out=90,in=180] (3,-5.15) 
to[out=0,in=90] (3.5,-5.5) to (3.5,-6.5)node[right,xshift=-2pt,yshift=1pt]{$w$};
\draw[usual] (0.5,-3.5) to (0.5,-2.5);
\draw[usual] (0.5,-4.5) to (0.5,-5.5);
\end{tikzpicture}
\quad\text{and}\quad
\morstuff{Y}=
\begin{tikzpicture}[anchorbase,scale=0.4,tinynodes]
\draw[pJW] (-0.5,-1.5) rectangle (3.2,-2.5);
\node at (1.35,-2.2) {$\pjwm[v{-}w]$};
\draw[pJW] (-0.5,-3.5) rectangle (2.5,-4.5);
\node at (1,-4.2) {$\pjwm[v{-}2w]$};
\draw[pJW] (-0.5,-5.5) rectangle (3.2,-6.5);
\node at (1.35,-6.2) {$\pjwm[v{-}w]$};
\draw[usual] (2,-2.5) to[out=270,in=180] (2.5,-2.85)node[left,xshift=-1pt,yshift=-2.5pt]{$w$} 
to[out=0,in=270] (3,-2.5);
\draw[usual] (2,-3.5) to[out=90,in=180] (2.5,-3.2) 
to (3,-3.2) to [out=0,in=270] (3.5,-1.5)
node[right,xshift=-2pt,yshift=-2pt]{$w$};
\draw[usual] (2,-4.5) to[out=270,in=180] (2.5,-4.8) 
to (3,-4.8) to [out=0,in=90] (3.5,-6.5)
node[right,xshift=-2pt,yshift=1pt]{$w$};
\draw[usual] (2,-5.5) to[out=90,in=180] (2.5,-5.15)node[left,xshift=-1pt]{$w$}
to[out=0,in=90] (3,-5.5);
\draw[usual] (0.5,-3.5) to (0.5,-2.5);
\draw[usual] (0.5,-4.5) to (0.5,-5.5);
\end{tikzpicture}
,
\end{gather*}
in which $w=\ppar^{(i)}$. Next we compute the pairwise products of
$\morstuff{X}$, $\fliph[\morstuff{X}]$, and $\morstuff{Y}$. First we use
shortening and absorption of mixed JW projectors to compute
$\morstuff{X}^{2}=\morstuff{X}$ and $\morstuff{X}\morstuff{Y}=\morstuff{Y}$.

Symmetrically, we also have 
$(\fliph[\morstuff{X}])^{2}=\fliph[\morstuff{X}]$ and
$\morstuff{Y}(\fliph[\morstuff{X}])=\morstuff{Y}$. 
Now we claim that \fix{the following}
products are zero, namely
$(\fliph[\morstuff{X}])\morstuff{X}=\morstuff{Y}\morstuff{X}=
(\fliph[\morstuff{X}])\morstuff{Y}=\morstuff{Y}^{2}=0$.
This can be seen as follows. Up to symmetry, 
these statements all follow from
\begin{gather*}
\begin{tikzpicture}[anchorbase,scale=0.4,tinynodes]
\draw[pJW] (-0.5,-3.5) rectangle (2.5,-4.5);
\node at (1,-4.2) {$\pjwm[v{-}2w]$};
\draw[pJW] (-0.5,-5.5) rectangle (3.5,-6.5);
\node at (1.5,-6.2) {$\pjwm[v{-}w]$};
\draw[usual] (2,-4.5) to[out=270,in=180] (2.5,-4.8) 
to (3,-4.8) to [out=0,in=90] (4,-6.5) to (4,-7.5)
node[right,xshift=-1pt,yshift=1pt]{$w$};
\draw[usual] (2,-5.5) to[out=90,in=180] (2.5,-5.2)
to[out=0,in=90] (3,-5.5);
\draw[usual] (3,-6.5) to[out=270,in=0] (2.5,-6.8)node[below]{$w$} 
to[out=180,in=270] (2,-6.5);
\draw[usual] (0.5,-4.5) to (0.5,-5.5);
\draw[usual] (0.5,-6.5) to (0.5,-7.5);
\end{tikzpicture}
=0
\quad \Leftrightarrow\quad 
\begin{tikzpicture}[anchorbase,scale=0.4,tinynodes]
\draw[pJW] (-0.5,-3.5) rectangle (2.5,-4.5);
\node at (1,-4.2) {$\pjwm[v{-}2w]$};
\draw[pJW] (-0.5,-5.5) rectangle (3.5,-6.5);
\node at (1.5,-6.2) {$\pjwm[v{-}w]$};
\draw[usual] (2,-4.5) to[out=270,in=180] (2.5,-4.8) 
to (3,-4.8) to [out=0,in=90] (4,-6.5) to[out=270,in=180] (4.5,-6.8) 
to[out=0,in=270] (5,-6.5) to (5,-3.5)node[right,xshift=-1pt,yshift=-5pt]{$w$};
\draw[usual] (2,-5.5) to[out=90,in=180] (2.5,-5.2)
to[out=0,in=90] (3,-5.5);
\draw[usual] (3,-6.5) to[out=270,in=0] (2.5,-6.8)node[below]{$w$} 
to[out=180,in=270] (2,-6.5);
\draw[usual] (0.5,-4.5) to (0.5,-5.5);
\draw[usual] (0.5,-6.5) to (0.5,-7.5);
\end{tikzpicture}
=
\begin{tikzpicture}[anchorbase,scale=0.4,tinynodes]
\draw[pJW] (-0.5,-3.5) rectangle (3.5,-4.5);
\node at (1.5,-4.2) {$\pjwm[v{-}w]$};
\draw[pJW] (-0.5,-5.5) rectangle (3.5,-6.5);
\node at (1.5,-6.2) {$\pjwm[v{-}w]$};
\draw[usual] (2,-4.5) to[out=270,in=180] (2.5,-4.8) 
to[out=0,in=270] (3,-4.5);
\draw[usual] (2,-5.5) to[out=90,in=180] (2.5,-5.2)node[right,yshift=0.5pt]{$w$}  
to[out=0,in=90] (3,-5.5);
\draw[usual] (3,-6.5) to[out=270,in=0] (2.5,-6.8)node[below]{$w$}  
to[out=180,in=270] (2,-6.5);
\draw[usual] (0.5,-4.5) to (0.5,-5.5);
\draw[usual] (0.5,-6.5) to (0.5,-7.5);
\end{tikzpicture}
=\pjw[v{-}w]\Up{i}\Down{i}\Up{i} = 0.
\end{gather*}
The equivalence is given by bending, as illustrated. 
On the right-hand side we undid
shortening and translated the caps and cups into morphisms $\Up{i}$ and
$\Down{i}$ respectively (this is possible since $a_{i}=1$), and then applied \eqref{eq:dud}. 
\fix{The remaining product is
$\morstuff{X}(\fliph[\morstuff{X}])=\funcg_{\qpar}(0)\morstuff{Y}$ by the zigzag relation.}
Taking all of these together shows that
\begin{gather*}
(\morstuff{X}+\fix{\alpha}\morstuff{Y})^{2}
=
\morstuff{X}+\fix{\alpha}\morstuff{Y}
,\quad
(\fliph[\morstuff{X}]+\fix{\beta}\morstuff{Y})^{2}
=
\fliph[\morstuff{X}]+\fix{\beta}\morstuff{Y}
,\\
(\morstuff{X}+\fix{\alpha}\morstuff{Y})
(\fliph[\morstuff{X}]+\fix{\beta}\morstuff{Y})
=
\fix{(\funcg_{\qpar}(0)+\alpha+\beta)\morstuff{Y}=0},
\\
(\fliph[\morstuff{X}]+\fix{\beta}\morstuff{Y})
(\morstuff{X}+\fix{\alpha}\morstuff{Y})=0,
\end{gather*}
which expresses $\fusidema{v}{i}$ and $\fusidemb{v}{i}$ as orthogonal idempotents.

\textit{Idempotency of $\fusidem{v}{i}$.} Next we check that the terms
$\fusidem{v}{i}$ are idempotents. Recall that $\fusidem{v}{i}$ for $i\neq j$ is
defined as a linear combination of the following two morphisms
\begin{align*}
\morstuff{X}=
\begin{tikzpicture}[anchorbase,scale=0.4,tinynodes]
\draw[pJW] (-0.5,-1.5) rectangle (2.5,-2.5);
\node at (1,-2.2) {$\pjwm[v{-}w]$};
\draw[pJW] (-0.5,-4.5) rectangle (2.5,-5.5);
\node at (1,-5.2) {$\pjwm[v{-}w]$};
\draw[pJW] (-0.5,-3.25) rectangle (0.5,-3.75);
\draw[usual] (2,-2.5) to[out=270,in=180] (2.5,-3) to[out=0,in=270] 
(3,-2.5) to (3,-1.5) node[above]{$w$};
\draw[usual] (2,-4.5) to[out=90,in=180] (2.5,-4) to[out=0,in=90] 
(3,-4.5) to (3,-5.5) node[below]{$w$};
\draw[usual] (0,-2.5) to (0,-3.25);
\draw[usual] (0,-3.75) to (0,-4.5);
\end{tikzpicture}
=
\begin{tikzpicture}[anchorbase,scale=0.4,tinynodes]
\draw[pJW] (-0.5,-1.5) rectangle (2.5,-2.5);
\node at (1,-2.2) {$\pjwm[v{-}w]$};
\draw[pJW] (-0.5,-4.5) rectangle (2.5,-5.5);
\node at (1,-5.2) {$\pjwm[v{-}w]$};
\draw[usual] (2,-2.5) to[out=270,in=180] (2.5,-3) to[out=0,in=270] 
(3,-2.5) to (3,-1.5) node[above]{$w$};
\draw[usual] (2,-4.5) to[out=90,in=180] (2.5,-4) to[out=0,in=90] 
(3,-4.5) to (3,-5.5) node[below]{$w$};
\draw[usual] (0,-2.5) to (0,-4.5);
\end{tikzpicture}
\quad\text{and}\quad
\morstuff{Y}=
\begin{tikzpicture}[anchorbase,scale=0.4,tinynodes]
\draw[pJW] (-0.5,-1.5) rectangle (2.5,-2.5);
\node at (1,-2.2) {$\pjwm[v{-}w]$};
\draw[pJW] (-0.5,-4.5) rectangle (2.5,-5.5);
\node at (1,-5.2) {$\pjwm[v{-}w]$};
\draw[pJW] (-0.5,-3.25) rectangle (0.5,-3.75);
\draw[usual] (2,-2.5) to[out=270,in=180] (2.5,-3) to[out=0,in=270] 
(3,-2.5) to (3,-1.5) node[above]{$w$};
\draw[usual] (2,-4.5) to[out=90,in=180] (2.5,-4) to[out=0,in=90] 
(3,-4.5) to (3,-5.5) node[below]{$w$};
\draw[usual] (0,-2.5) to (0,-3.25);
\draw[usual] (0,-3.75) to (0,-4.5);
\draw[usual] (0.5,-4.5) to[out=90,in=180] (1,-4) to[out=0,in=90] 
(1.5,-4.5);
\draw[usual] (0.5,-2.5) to[out=270,in=180] (1,-3) to[out=0,in=270] 
(1.5,-2.5);
\node at (1,-3.65) {$S_{i}$};
\end{tikzpicture}
=
\begin{tikzpicture}[anchorbase,scale=0.4,tinynodes]
\draw[pJW] (-0.5,-1.5) rectangle (2.5,-2.5);
\node at (1,-2.2) {$\pjwm[v{-}w]$};
\draw[pJW] (-0.5,-4.5) rectangle (2.5,-5.5);
\node at (1,-5.2) {$\pjwm[v{-}w]$};
\draw[usual] (2,-2.5) to[out=270,in=180] (2.5,-3) to[out=0,in=270] 
(3,-2.5) to (3,-1.5) node[above]{$w$};
\draw[usual] (2,-4.5) to[out=90,in=180] (2.5,-4) to[out=0,in=90] 
(3,-4.5) to (3,-5.5) node[below]{$w$};
\draw[usual] (0,-2.5) to (0,-4.5);
\draw[usual] (0.5,-4.5) to[out=90,in=180] (1,-4) to[out=0,in=90] 
(1.5,-4.5);
\draw[usual] (0.5,-2.5) to[out=270,in=180] (1,-3) to[out=0,in=270] 
(1.5,-2.5);
\node at (1,-3.65) {$S_{i}$};
\end{tikzpicture}
, 
\end{align*}
where we also write $w=\ppar^{(i)}$. For $i=j$, we use the same definition for
$\morstuff{X}$, but set $\morstuff{Y}=0$. To compute the various products of
these elements, we will use the partial trace rule \eqref{eq:ptrstatement} for
$\pjw[v{-}w]$, which holds by assumption $\trp(v)$. Since we know from
\eqref{eq:dud} that the loop $\loopdown{i}{v{-}2w}$ is annihilated by
postcomposing with another down morphism $\Down{i}\pjw[v{-}2w]$, we also obtain
from \eqref{eq:ptrclaim} that
\begin{gather*}
\begin{tikzpicture}[anchorbase,scale=0.25,tinynodes]
\draw[pJW] (1.5,-1) rectangle (-2.5,1);
\node at (-0.5,-0.1) {$\pjwm[v{-}w]$};
\draw[usual] (1,1) to[out=90,in=180] (1.5,1.5) to[out=0,in=90] 
(2,1) to (2,0)node[right,xshift=-2pt]{$w$} to (2,-1) 
to[out=270,in=0] (1.5,-1.5) to[out=180,in=270] (1,-1);
\draw[usual] (0.5,1) to[out=90,in=0] (0,1.5) to[out=180,in=90] (-0.5,1);
\node at (0,2.5) {$\phantom{a}$};
\node at (0,-2.5) {$\phantom{a}$};
\end{tikzpicture}=
\funcg_{\qpar}(a_{i}-1)\cdot
\begin{tikzpicture}[anchorbase,scale=0.25,tinynodes]
\draw[pJW] (1.5,-1) rectangle (-2.5,1);
\node at (-0.5,-0.1) {$\pjwm[v{-}2w]$};
\draw[usual] (0.5,1) to[out=90,in=0] (0,1.5) to[out=180,in=90] (-0.5,1);
\node at (0,2.5) {$\phantom{a}$};
\node at (0,-2.5) {$\phantom{a}$};
\end{tikzpicture}
\;\quad\text{and}\quad
\begin{tikzpicture}[anchorbase,scale=0.25,tinynodes]
\draw[pJW] (1.5,-1) rectangle (-2.5,1);
\node at (-0.5,-0.1) {$\pjwm[v{-}w]$};
\draw[usual] (1,1) to[out=90,in=180] (1.5,1.5) to[out=0,in=90] 
(2,1) to (2,0)node[right,xshift=-2pt]{$w$} to (2,-1) 
to[out=270,in=0] (1.5,-1.5) to[out=180,in=270] (1,-1);
\draw[usual] (0.5,1) to[out=90,in=0] (0,1.5) to[out=180,in=90] (-0.5,1);
\draw[usual] (0.5,-1) to[out=270,in=0] (0,-1.5) to[out=180,in=270] (-0.5,-1);
\node at (0,2.5) {$\phantom{a}$};
\node at (0,-2.5) {$\phantom{a}$};
\end{tikzpicture}
=
\funcg_{\qpar}(a_{i}-1)\cdot
\begin{tikzpicture}[anchorbase,scale=0.25,tinynodes]
\draw[pJW] (1.5,-1) rectangle (-2.5,1);
\node at (-0.5,-0.1) {$\pjwm[v{-}2w]$};
\draw[usual] (0.5,1) to[out=90,in=0] (0,1.5) to[out=180,in=90] (-0.5,1);
\draw[usual] (0.5,-1) to[out=270,in=0] (0,-1.5) to[out=180,in=270] (-0.5,-1);
\node at (0,2.5) {$\phantom{a}$};
\node at (0,-2.5) {$\phantom{a}$};
\end{tikzpicture}.
\end{gather*}
Thus a diagrammatic calculation shows
\begin{align*}
\morstuff{X}^{2}&=
\funcg_{\qpar}(a_{i}-1)\morstuff{X}
+\funcf_{\qpar}(a_{i}-1)\morstuff{Y}
,\quad
\morstuff{X}\morstuff{Y}=
\funcg_{\qpar}(a_{i}-1)\morstuff{Y}
\quad\text{and}\quad
\morstuff{Y}^{2}
=0.
\end{align*}
The latter holds since $(\loopdown{i}{v{-}2w})^{2}=0$, which follows from
\eqref{eq:dud}. (Observe that the above relations also hold in the special case
$i=j$ where $\morstuff{Y}=0$.) These equations verify that $\fusidem{v}{i}$ is an idempotent.
\end{proof}

\addtocontents{toc}{\protect\setcounter{tocdepth}{1}}
%

\renewcommand{\thesection}{Appendix \Alph{section}}
\renewcommand{\thesubsection}{\Alph{section}.\arabic{subsection}}
\renewcommand{\theequation}{\Alph{section}-\arabic{equation}}
\renewcommand{\thefigure}{\Alph{section}-\arabic{figure}}
\setcounter{theoremm}{0}
\setcounter{equation}{0}
\setcounter{figure}{0}
\setcounter{section}{0}
\setcounter{subsection}{0}

\section{Fusion rules for eves}\label{subsection:theta}

For this section, we fix $a,b\in\N$ with $1\leq a,b<\plpar$ and $k,t\in\N[0]$.

\subsection{On the level of objects}

\changed{For any $\xpar\in\kk$ we will also use quantum factorials
and binomials:
\begin{gather*}
\qfac{0}{\xpar}=1
,\quad
\qfac{a}{\xpar}=\qnum{a}{\xpar}\qnum{a-1}{\xpar}\dots\qnum{1}{\xpar},
\quad
\qbin{b}{0}{\xpar}=1
,\quad
\qbin{b}{a}{\xpar}=\tfrac{\qnum{b}{\xpar}\qnum{b-1}{\xpar}\dots\qnum{b-a+1}{\xpar}}{\qnum{a}{\xpar}\qnum{a-1}{\xpar}\dots\qnum{1}{\xpar}}.
\end{gather*}
Note that (quantum) binomials are zero for $0\leq b<a$.
Here $a\in\N$ and $b\in\Z$, and all of these are elements of
$\kk$. Note that $\qbin{b}{a}{1}=\binom{b}{a}$.
We recall the \emph{quantum Lucas' theorem}:}

\begin{proposition}\label{proposition:qlucas}
\changed{For $v=\plbase{a_{j},\dots,a_{0}}$, $w=\plbase{b_{j},\dots,b_{0}}$ 
and $v-w=\plbase{c_{j},\dots,c_{0}}$ we have
\begin{gather*}
\qbinn{v}{w}{\qpar}
=
\lucas
\binom{a_{j}}{b_{j}}
\dots
\binom{a_{1}}{b_{1}}
\qbinn{a_{0}}{b_{0}}{\qpar}
=
\lucas
\binom{a_{j}}{c_{j}}
\dots
\binom{a_{1}}{c_{1}}
\qbinn{a_{0}}{c_{0}}{\qpar}
=
\qbinn{v}{v-w}{\qpar}.
\end{gather*}
The scalar is $\lucas=(\qpar^{\lpar})^{A_{1}b_{0}+a_{0}B_{1}
+\lpar(A_{1}B_{1}-B_{1}^{2})}=\pm 1$, where 
$A_{1}=\frac{v-a_{0}}{\lpar}$ and $B_{1}=\frac{w-b_{0}}{\lpar}$.}
\end{proposition}

\begin{proof}
\changed{This is folklore, but the first written account might be 
\cite[(1.2.4)]{Ol-gpower}. 
(Translating from that paper, which uses a non-symmetric definition of 
quantum numbers often used in combinatorics, 
to our conventions gives the sign $\lucas$.)}
\end{proof}

\begin{example}
\changed{For $(\F[7],2)$ and $a=71=\pbase{3,2,2}{7,3}$ and $b=1=\pbase{0,0,1}{7,3}$
the quantum Lucas' theorem gives $\qnum{71}{2}=\qbin{71}{1}{2}=
\binom{3}{0}
\binom{2}{0}
\qbinn{2}{1}{2}=\qnum{2}{2}=6$. However, note that 
$\qnum{71}{3}=-\qnum{2}{2}$
because $\lucas=-1$ in this case since $3^{3}=-1\bmod 7$.
In general, $\qnum{v}{\qpar}=(\qpar^{\lpar})^{A_{1}}\qnum{a_{0}}{\qpar}
=\pm\qnum{a_{0}}{\qpar}$. 
One can also see the 
order of vanishing, say for $a=63=\pbase{3,0,0}{7,3}$. 
Using the second expression one gets $\qnum{63}{2}=\qbin{63}{62}{2}=
\binom{3}{2}\binom{0}{6}\qbinn{0}{2}{2}$, which is thus divisible by $7$ and $\qnum{3}{2}$.}
\end{example}

Another useful tool for quantum calculations is 
using \emph{a change of variables method}:

\begin{example}
\ochanged{Take $v=\plbase{a_{j},\dots,a_{k},0,\dots,0}$ and
$w=\plbase{b_{j},\dots,b_{k},0,\dots,0}$ for $k>0$ and $b_{k}\neq 0$.  Clearly,
$\qnum{v}{\qpar}=\qnum{w}{\qpar}=0$, but the quotient
$\tfrac{\qnum{v}{\qpar}}{\qnum{w}{\qpar}}$ is well-defined (and non-zero if
$a_{k}\neq 0$). To see this we use $\qpar^{\ppar^{(k)}}=
(\qpar^{\lpar})^{\ppar^{k-1}}=\qpar^{\lpar}=\pm 1$ (the second equation holds
trivially if $\ppar=2$) to calculate
\begin{gather*}
\tfrac{\qnum{v}{\vpar}}{\qnum{w}{\vpar}}
=
\tfrac{\tfrac{\vpar^{v}-\vpar^{-v}}{\vpar-\vpar^{-1}}}{\tfrac{\vpar^{w}-\vpar^{-w}}{\vpar-\vpar^{-1}}}
=
\tfrac{\tfrac{\vpar^{v}-\vpar^{-v}}{\vpar^{\ppar^{(k)}}-\vpar^{-\ppar^{(k)}}}}{\tfrac{\vpar^{w}-\vpar^{-w}}{\vpar^{\ppar^{(k)}}-\vpar^{-\ppar^{(k)}}}}
=
\tfrac{\qnum{v/\ppar^{(k)}}{\vpar^{\ppar^{(k)}}}}
{\qnum{w/\ppar^{(k)}}{\vpar^{\ppar^{(k)}}}}
\xmapsto{\vpar \mapsto \qpar}
\pm
\tfrac{\pbase{a_{j},\dots,a_{k}}{\ppar,\ppar}}
{\pbase{b_{j},\dots,b_{k}}{\ppar,\ppar}}
=
\pm
\tfrac{a_{k}}{b_{k}},
\end{gather*}
where the sign is $(\qpar^{\lpar})^{\sum_{i=k}^{j}a_{i}-b_{i}}$.}
\end{example}

The fusion rules for eves are as follows. These, together with \fullref{lemma:clebsch-gordan}, give a complete collection of fusion rules for simple tilting modules.

\begin{proposition}\label{proposition:eve-times-eve} 
Let $a,b\in\N$,
$k,t\in\N[0]$, $1\leq a,b<\plpar$, and $a\ppar^{(k)}\geq b\ppar^{(t)}$. 
\ochanged{As before, l}et us write $x^{\prime}$ 
to denote $x$, if $x$ is odd, and $x-1$, if $x$ is even.
\begin{enumerate}

\item For $k>t$ we have
\begin{gather}\label{eq:fusion-eve-eve}
\begin{gathered}
\tmod(a\ppar^{(k)}-1)\hcirc\tmod(b\ppar^{(t)}-1)\cong
\\
{
\bigoplus_{i_{t}=0}^{\frac{b^{\prime}-1}{2}}
\bigoplus_{i_{t{-}1}=0}^{\frac{\ppar^{\prime}-1}{2}}
\dots
\bigoplus_{i_{0}=0}^{\frac{\lpar^{\prime}-1}{2}}}\,
\tmod\big(\plbase{a,0,\dots,0,b-1-2i_{t}
,\ppar-1-2i_{t-1},\dots,\lpar-1-2i_{0}} -1\big),
\end{gathered}
\end{gather}
where the potentially non-zero digits 
appear at positions $k$ and between $t$
and $0$.

\item For $k=t$ we have
\begin{gather*}
\tmod(a\ppar^{(k)}-1)\hcirc\tmod(b\ppar^{(k)}-1)\cong
\begin{cases}
\text{\eqref{eq:eve-fusion-second}}&\text{if }a+b-2<\plpar,
\\
\text{\eqref{eq:eve-fusion-third}}&\text{if }a+b-2\geq\plpar.
\end{cases}
\end{gather*}
\begin{gather}\label{eq:eve-fusion-second}
{
\bigoplus_{x}
\bigoplus_{i_{t{-}1}=0}^{\frac{\ppar^{\prime}-1}{2}}
\dots
\bigoplus_{i_{0}=0}^{\frac{\lpar^{\prime}-1}{2}}}\,
\tmod\big(\plbase{x-1,\ppar-1-2i_{k-1},\dots,\lpar-1-2i_{0}}-1\big),
\end{gather}
where the leftmost sum is running over all $x=a+b-2i\in\big\{|a-b|+2,\dots,a+b\big\}$.
\begin{gather}\label{eq:eve-fusion-third}
\begin{aligned}
&{
\bigoplus_{x}
\bigoplus_{i_{t{-}1}=0}^{\frac{\ppar^{\prime}-1}{2}}
\dots
\bigoplus_{i_{0}=0}^{\frac{\lpar^{\prime}-1}{2}}}\,
\tmod\big(\plbase{x-1,\ppar-1-2i_{k-1},\dots,\lpar-1-2i_{0}}-1\big)
\\
\oplus
&{
\bigoplus_{y}
\bigoplus_{i_{t{-}1}=0}^{\frac{\ppar^{\prime}-1}{2}}
\dots
\bigoplus_{i_{0}=0}^{\frac{\lpar^{\prime}-1}{2}}}\,
\tmod\big(\plbase{1,y-1,\ppar-1-2i_{k-1},\dots,\lpar-1-2i_{0}}-1\big),
\end{aligned}
\end{gather}
where the two leftmost sums are running over
$x=|a-b|+2i\in\big\{|a-b|+2,\dots,2\plpar-a-b\big\}$, respectively 
$y=a+b-\plpar-2i\in\big\{(a+b-\plpar)-(a+b-\plpar)^{\prime}+1
,\dots,a+b-\plpar\big\}$.
\end{enumerate}
\end{proposition}

Note that the case $k=t$ 
is a variation of \eqref{eq:fusion-eve-eve} 
where the
$k$th and $k+1$th digits for the tilting summands that appear are given by the
first two digits in \eqref{eq:clebsch-gordan} or \eqref{eq:qCG} 
for $\tmod(a-1)\hcirc\tmod(b-1)$ and $\plpar$,
while the other digits are given as in \eqref{eq:fusion-eve-eve}.

\begin{example}\label{example:eve-times-eve}
Here are a few examples of \eqref{eq:fusion-eve-eve}.
\begin{enumerate}

\item The formula is particularly nice if $b=1$, {\ie} one of the tensor factors
is a prime eve, since then only terms with $i_l=0$ appear. Explicitly, for
$\mchar=(7,3)$, the tilting summands are precisely
$\tmod\big(\plbase{a,0,\dots,0,0,y_{t{-}1},\dots,y_{1},y_{0}}-1\big)$ with
$y_{i}\in\{0,2,4,6\}$ for $i>0$ and $y_{0}\in\{0,2\}$. So, {\eg} for $t=3$ and
$k\geq 3$, there are $32$ such summands.

\item For $\ppar=\lpar=2$, we get
$\tmod(2^{k}-1)\hcirc\tmod(2^{t}-1)\cong\tmod(2^{k}+2^{t}-2)$.

\item The fusion rule for
$\tmod\big(\pbase{5,0,0,0}{7,3}-1\big)
\hcirc\tmod\big(\pbase{2,0,0}{7,3}-1\big)$ 
is illustrated in \fullref{figure:fusion-graph2}.

\end{enumerate}

\end{example}

\begin{proof}[Proof of \fullref{proposition:eve-times-eve}.]
As before, we use
character computations. We will first 
focus on the case $k>t$, and
then on the case $k=t$.

We will use 
\begin{gather}\label{eq:character-use}
\ch{\tmod(w-1)}=\qnum{\motherr{v}{\infty}}{\vpar}{
\prod_{b_{i}\neq 0}}\,\qnum{2}{\vpar^{b_{i}\ppar^{(i)}}}
,\quad\ch{\tmod(v-1)\hcirc\tmod(1)}
=\qnum{\motherr{v}{\infty}}{\vpar}{
\prod_{a_{i}\neq 0}}\,\qnum{2}{\vpar^{a_{i}\ppar^{(i)}}}\cdot\qnum{2}{\vpar},
\end{gather}
where $w=\plbase{b_{j+1},\dots,b_{0}}$ 
are the relevant $\ppar\lpar$-adic expansions for the summands that appear. 
Hereby, $b_i=a_i$ whenever $\taill{v}<i\geqslant j$ and $b_{j+1}\neq 0$ if and only if $\taill{v}=j+1$.

We decompose
$\qnum{a\ppar^{(k)}}{\vpar}\qnum{b\ppar^{(t)}}{\vpar}$ into a sum of
indecomposable tilting characters, presented as in \eqref{eq:character-use}. 
Note that $\qnum{a\ppar^{(k)}}{\vpar}$ is already one of the factors expected
in the character of every summand, hence we focus on rewriting
$\qnum{b\ppar^{(t)}}{\vpar}$. We claim:
\begin{align*}
\qnum{b\ppar^{(t)}}{\vpar}
&=
\vpar^{-b\ppar^{(t)}+1}+\vpar^{-b\ppar^{(t)}+3}
+\dots+\vpar^{b\ppar^{(t)}-3}+\vpar^{b\ppar^{(t)}-1}
\\
&=
{
\sum_{i_{t}=0}^{\frac{b^{\prime}-1}{2}}
\sum_{i_{t{-}1}=0}^{\frac{\ppar^{\prime}-1}{2}}
\dots
\sum_{i_{0}=0}^{\frac{\lpar^{\prime}-1}{2}}}\,
\big(
\qnum{2}{\vpar^{(b-1-2i_{t})\ppar^{(t)}}}
\qnum{2}{\vpar^{(\ppar-1-2i_{t-1})\ppar^{(t-1)}}}
\dots
\qnum{2}{\vpar^{(\lpar-1-2i_{0})\ppar^{(0)}}}
\big),
\end{align*}
where we interpret $\qnum{2}{\vpar^{0}}$ as $1$ and not $2$.
(This is due to the same reason why we need to take the product over the 
non-zero digits in \eqref{eq:character-use}.)
In order to verify this, we rewrite the second line as
\begin{align*}
&{
\sum_{i_{t}=0}^{\frac{b^{\prime}-1}{2}}
\sum_{i_{t{-}1}=0}^{\frac{\ppar^{\prime}-1}{2}}
\dots
\sum_{i_{0}=0}^{\frac{\lpar^{\prime}-1}{2}}}\,
\big(
\qnum{2}{\vpar^{(b-1-2i_{t})\ppar^{(t)}}}
\qnum{2}{\vpar^{(\ppar-1-2i_{t-1})\ppar^{(t-1)}}}
\dots
\qnum{2}{\vpar^{(\lpar-1-2i_{0})\ppar^{(0)}}}
\big)
\\
=&
{
\sum_{i_{t}=0}^{\frac{b^{\prime}-1}{2}}
\qnum{2}{\vpar^{(b-1-2i_{t})\ppar^{(t)}}}
\sum_{i_{t{-}1}=0}^{\frac{\ppar^{\prime}-1}{2}}
\qnum{2}{\vpar^{(\ppar-1-2i_{t-1})\ppar^{(t-1)}}}
\dots
\sum_{i_{0}=0}^{\frac{\lpar^{\prime}-1}{2}}
\qnum{2}{\vpar^{(\lpar-1-2i_{0})\ppar^{(0)}}}}\,
\\
=&
\qnum{b}{\vpar^{\ppar^{(t)}}}
\qnum{\ppar}{\vpar^{\ppar^{(t-1)}}}
\dots
\qnum{\lpar}{\vpar^{\ppar^{(0)}}}\\
=&
\qnum{b\ppar^{(t)}}{\vpar}
.
\end{align*}
The last equation is a telescoping product when writing
$\qnum{n}{\xpar}=\tfrac{\xpar^{n}-\xpar^{-n}}{\xpar-\xpar^{-1}}$. 

For $k=t$, the same type of argument leaves us with rewriting
$\qnum{a\ppar^{(k)}}{\vpar}\qnum{b}{\vpar^{\ppar^{(k)}}}$. 
To this end, we will use the quantum number calculations underlying 
\fullref{lemma:clebsch-gordan}.
Let us assume that we are in the case of \eqref{eq:eve-fusion-second}, where we calculate
\begin{gather*}
\qnum{a\ppar^{(k)}}{\vpar}\qnum{b}{\vpar^{\ppar^{(k)}}}
=
\qnum{a}{\vpar^{\ppar^{(k)}}}\qnum{b}{\vpar^{\ppar^{(k)}}}
\qnum{\ppar}{\vpar^{\ppar^{(k-1)}}}
\dots
\qnum{\lpar}{\vpar^{\ppar^{(0)}}}
={
\sum_{x}}\,\qnum{x}{\vpar^{\ppar^{(k)}}}
\qnum{\ppar}{\vpar^{\ppar^{(k-1)}}}
\dots
\qnum{\lpar}{\vpar^{\ppar^{(0)}}}
={
\sum_{x}}\,\qnum{x\ppar^{(k)}}{\vpar},
\end{gather*}
where $x$ runs over the same set as in \eqref{eq:eve-fusion-second}, 
and the second equality is \eqref{eq:clebsch-gordan} for 
$\vpar$ substituted by $\vpar^{\ppar^{(k)}}$. Note that 
$\qnum{x\ppar^{(k)}}{\vpar}$ are 
terms of the form
$\qnum{\motherr{v}{\infty}}{\vpar}$ in \eqref{eq:character-use}, 
so we are done in this case. The case 
of \eqref{eq:eve-fusion-second} uses the same argument with 
\eqref{eq:qCG}
instead of \eqref{eq:clebsch-gordan}.
\end{proof}

\begin{remark}
Questions about the structure constants of the representation ring have 
been studied for the
finite group $\SLtwo(\mathbb{F}_{\ppar^{k}})$ for a long time. 
(The connection to our setup is
to embed $\SLtwo(\mathbb{F}_{\ppar^{k}})$ into
$\SLtwo(\overline{\mathbb{F}}_{\ppar^{k}})$ via 
fixed points under the Frobenius
twist.) For example, \fullref{lemma:clebsch-gordan} and \cite[Lemma
5]{ErHe-ringel-schur} is used in \cite[Section 3]{Cr-tensor-simple-modules} to
find the finite group analog of fusion rules.

Actually, a bit more is true and worthwhile to point out: Recall that for a
finite group a module $\tmod$ is \emph{algebraic} if there exists $f\in\N[0][X]$
such that $f\big([\tmod]\big)=0$ holds in the representation ring; otherwise
modules are called \emph{transcendental}, {\cf}
\cite{Al-modules-fractional-groups}. As shown in \cite[Section
3]{Cr-tensor-simple-modules} the tilting module $\tmod(1)$ considered as a
$\SLtwo(\mathbb{F}_{\ppar^{k}})$-module is algebraic, and so are eves. Being
algebraic is a measure of how complex fusion rules get. 
This was a motivation to stick to the
fusion rules presented above. In fact, this is a rare
property for modules of groups of Lie type as \cite[Theorem
A]{Cr-tensor-simple-modules} shows and most of the time even the
vector representation is transcendental, {\eg} 
for all $\mathrm{SL}_{a}$ for $a>2$.
\end{remark}

\subsection{Categorified eve fusion rules}

We have seen in \fullref{proposition:eve-times-eve} that the decomposition of a
tensor product of simple tilting modules
$\tmod(a\ppar^{(k)}-1)\hcirc\tmod(b\ppar^{(t)}-1)$ into indecomposable tilting
modules is \emph{multiplicity-free}. In fact, its decomposition is \emph{Weyl-multiplicity-free},
meaning that no Weyl factor appears twice. This implies that the problem of
computing categorified fusion rules for eves is well-posed.

\begin{lemma}\label{lemma:catfusion-problem} 
There exists a unique
decomposition of $\pjw[a\ppar^{(k)}-1]\hcirc\pjw[b\ppar^{(t)}-1]$ into
orthogonal, primitive idempotents $\pjw[{a\ppar^{(k)}-1,b\ppar^{(t)}-1}]^{v-1}$ 
factoring through $\pjw[v-1]$ 
\begin{gather*}
\begin{tikzpicture}[anchorbase,scale=0.25,tinynodes]
\draw[JW] (-2.5,3) rectangle (2.5,5);
\node at (0,3.8) {$\pjwm[a\ppar^{(k)}{-}1]$};
\draw[JW] (3.5,3) rectangle (8.5,5);
\node at (6,3.8) {$\pjwm[b\ppar^{(t)}{-}1]$};
\end{tikzpicture}
={\sum_{v}}\,
\pjw[{a\ppar^{(k)}-1,b\ppar^{(t)}-1}]^{v-1}
= 
{\sum_{v}}\,
\begin{tikzpicture}[anchorbase,scale=0.25,tinynodes]
\draw[JW] (-2.5,3) rectangle (2.5,5);
\node at (0,3.8) {$\pjwm[a\ppar^{(k)}{-}1]$};
\draw[JW] (3.5,3) rectangle (8.5,5);
\node at (6,3.8) {$\pjwm[b\ppar^{(t)}{-}1]$};
\draw[pJW] (0.5,-1) rectangle (5.5,1);
\node at (3,-0.2) {$\pjwm[v{-}1]$};
\draw[JW] (-2.5,-3) rectangle (2.5,-5);
\node at (0,-4.2) {$\pjwm[a\ppar^{(k)}{-}1]$};
\draw[JW] (3.5,-3) rectangle (8.5,-5);
\node at (6,-4.2) {$\pjwm[b\ppar^{(t)}{-}1]$};
\draw[mor] (-1.5,-3) to (7.5,-3) to (4.5,-1) to (1.5,-1) to (-1.5,-3);
\node at (3,-2) {$\morstuff{f}$};
\draw[mor] (-1.5,3) to (7.5,3) to (4.5,1) to (1.5,1) to (-1.5,3);
\node at (3,1.9) {$\morstuff{f}^{\prime}$};
\end{tikzpicture}       
,
\end{gather*}
where $v$ ranges over the set specified by \eqref{eq:fusion-eve-eve}. 
\end{lemma}

\begin{proof} 
The existence of an expansion into idempotents as illustrated follows from the
fact that the tilting modules $\tmod(v-1)$ form a complete collection of indecomposable objects
(up to isomorphism) in $\tilt[{\kk,\qpar}]$ and the associated primitive idempotents 
are the mixed JW projectors $\pjw[v{-}1]$. The summands that occur are
determined by \eqref{eq:fusion-eve-eve} since $\pjw[v{-}1]$ represents
$\tmod(v-1)$. The orthogonality and uniqueness follow from the fact that
$\tmod(a\ppar^{(k)}-1)\hcirc\tmod(b\ppar^{(t)}-1)$ is 
Weyl-multiplicity-free since 
$\tmod(a\ppar^{(k)}-1)\cong\wmod(a\ppar^{(k)}-1)$, and 
$\tmod(b\ppar^{(t)}-1)\cong\wmod(b\ppar^{(t)}-1)$. 
\end{proof}

\begin{remark}
In general, $\pjw[v_{1}{-}1]\hcirc\pjw[v_{2}{-}1]$ still has a decomposition into
idempotents factoring through various $\pjw[v{-}1]$. However, if
$\tmod(v_{1}-1)\hcirc\tmod(v_{2}-1)$ is not Weyl-multiplicity-free or not even
multiplicity-free, then such 
decompositions are no longer unique. In
\fullref{subsection:fusion-mor}, we will encounter this issue when giving a
complete description of the case $\tmod(v-1)\hcirc\tmod(1)$.
\end{remark}

The problem of explicitly computing the idempotents in the decompositions from
\fullref{lemma:catfusion-problem} in full generality is beyond the scope of this
paper. Instead, we will compute the idempotents projecting onto summands
$\tmod(v-1)$ of generation at most one. In particular, we completely determine the decomposition in the complex quantum group case.

We will consider morphisms of the form 
\begin{gather}\label{eq:XandD}
\xmor{a\ppar^{(k)}-1,b\ppar^{(t)}-1}{w-1}
=
\begin{tikzpicture}[anchorbase,scale=0.25,tinynodes]
\draw[JW] (-2.5,3) rectangle (2.5,5);
\node at (0,3.8) {$\pjwm[a\ppar^{(k)}{-}1]$};
\draw[JW] (3.5,3) rectangle (8.5,5);
\node at (6,3.8) {$\pjwm[b\ppar^{(t)}{-}1]$};
\draw[pJW] (0.5,-1) rectangle (5.5,1);
\node at (3,-0.2) {$\pjwm[w{-}1]$};
\draw[JW] (-2.5,-3) rectangle (2.5,-5);
\node at (0,-4.2) {$\pjwm[a\ppar^{(k)}{-}1]$};
\draw[JW] (3.5,-3) rectangle (8.5,-5);
\node at (6,-4.2) {$\pjwm[b\ppar^{(t)}{-}1]$};
\draw[usual] (2,-3) to[out=90,in=180] (3,-2)node[above,xshift=0.15cm,yshift=-0.1cm]{$x$} to[out=0,in=90] (4,-3);
\draw[usual] (-1.5,-3)node[left,yshift=0.3cm,xshift=0.3cm]{$m$} to[out=90,in=270] (1.5,-1);
\draw[usual] (7.5,-3)node[left,yshift=0.3cm,xshift=0.1cm]{$n$}to[out=90,in=270] (4.5,-1);
\draw[usual] (2,3) to[out=270,in=180] (3,2)node[below,xshift=0.15cm,yshift=0.05cm]{$x$} to[out=0,in=270] (4,3);
\draw[usual] (-1.5,3) node[left,yshift=-0.4cm,xshift=0.3cm]{$m$}to[out=270,in=90] (1.5,1);
\draw[usual] (7.5,3)node[left,yshift=-0.4cm,xshift=0.1cm]{$n$} to[out=270,in=90] (4.5,1);
\end{tikzpicture}       
\quad\text{and}\quad
\dmor{a\ppar^{(k)}-1,b\ppar^{(t)}-1}{w-1}{w^{\prime}-1}
=
\begin{tikzpicture}[anchorbase,scale=0.25,tinynodes]
\draw[JW] (-2.5,3) rectangle (2.5,5);
\node at (0,3.8) {$\pjwm[a\ppar^{(k)}{-}1]$};
\draw[JW] (3.5,3) rectangle (8.5,5);
\node at (6,3.8) {$\pjwm[b\ppar^{(t)}{-}1]$};
\draw[pJW] (0.5,-1) rectangle (5.5,1);
\node at (3,-0.2) {$\pjwm[w{-}1]$};
\draw[pJW] (0.5,7) rectangle (5.5,9);
\node at (3,7.8) {$\pjwm[w^{\prime}{-}1]$};
\draw[usual] (2,5) to[out=90,in=180] (3,6)node[above,xshift=0.15cm,yshift=-0.1cm]{$x^{\prime}$} to[out=0,in=90] (4,5);
\draw[usual] (-1.5,5) to[out=90,in=270] (1.5,7);
\draw[usual] (7.5,5) to[out=90,in=270] (4.5,7);
\draw[usual] (2,3) to[out=270,in=180] (3,2)node[below,xshift=0.15cm,yshift=0.05cm]{$x$} to[out=0,in=270] (4,3);
\draw[usual] (-1.5,3) to[out=270,in=90] (1.5,1);
\draw[usual] (7.5,3) to[out=270,in=90] (4.5,1);
\end{tikzpicture}.
\end{gather} 
To describe the idempotent of $\pjw[a\ppar^{(k)}-1]\hcirc\pjw[b\ppar^{(t)}-1]$ 
that factors through $\pjw[v{-}1]$ we will consider the ansatz
\begin{gather}\label{eq:ansatz-idempotent}
\pjw[{a\ppar^{(k)}-1,b\ppar^{(t)}-1}]^{v-1} 
\overset{?}{=}{
\sum_{S}}\,
\fusscalar{(v,S)}{a\ppar^{(k)},b\ppar^{(t)}}
\cdot\xmor{a\ppar^{(k)}-1,b\ppar^{(t)}-1}{v[S]-1}
\end{gather} 
with scalars $\fusscalar{(v,S)}{a\ppar^{(k)},b\ppar^{(t)}}\in\kk$ , and where the sum runs over all $S$ that are down-admissible for $v$. 
To check
whether such expressions give the desired idempotent, we have to compute
composites of the morphisms $\xmor{\placeholder,\placeholder}{\placeholder}$. This, in
turn, requires us to expand the \emph{digon morphisms}  
$\dmor{\placeholder,\placeholder}{\placeholder}{\placeholder}$ in a fixed basis of the
relevant morphism space.

\begin{lemma}
Let $\tmod(v-1)$ and $\tmod(v^{\prime}-1)$ be summands of
$\tmod(a\ppar^{(k)}-1)\hcirc\tmod(b\ppar^{(t)}-1)$ as specified by
\eqref{eq:fusion-eve-eve}. For $v=v^{\prime}$ we have
\begin{gather}\label{eq:ansatz-digon}
\dmor{a\ppar^{(k)}-1,b\ppar^{(t)}-1}{v-1}{v^{\prime}-1}
=
{
\sum_{S}}\,
\digscalar{(v,S)}{a\ppar^{(k)},b\ppar^{(t)}}
\cdot\loopdown{S}{v{-}1}
\end{gather} 
with uniquely determined scalars
$\digscalar{(v,S)}{a\ppar^{(k)},b\ppar^{(t)}}\in\kk$, and the 
sum running over all down-admissible $S$. If
$v\neq v^{\prime}$, then
$\dmor{a\ppar^{(k)}-1,b\ppar^{(t)}-1}{v-1}{v^{\prime}-1}=0$.
\end{lemma}

\begin{proof}
This is a direct consequence of the fact that loops form a basis for the
endomorphism algebra of $\pjw[v{-}1]$ and that
$\tmod(a\ppar^{(k)}-1)\hcirc\tmod(b\ppar^{(t)}-1)$ is
Weyl-multiplicity-free.
\end{proof}

To warm up, we will compute the digon expansion \eqref{eq:ansatz-digon} in the
case when $v$ is of generation zero, {\ie} when the direct summand $\tmod(v-1)$
of $\tmod(a\ppar^{(k)}-1)\hcirc\tmod(b\ppar^{(t)}-1)$ is simple. In this case,
the endomorphisms of $\pjw[v{-}1]$ are just its scalar multiples, and the single
relevant scalar $\twebs$ can be determined using the values of \emph{theta webs}
from \cite[Chapter 6]{KaLi-TL-recoupling}. 

For every triple $\alpha,\beta,\gamma\in\N[0]$ such that $\alpha+\beta+\gamma$
is even and the triangle inequalities $\alpha\leq\beta+\gamma$, 
$\beta\leq\alpha+\gamma$ and $\gamma\leq\alpha+\beta$ are satisfied, one can consider the
associated theta value in $\TL[\kkv,\vpar]$:
\begin{gather*}
\Theta(\alpha,\beta,\gamma)
:=
\begin{tikzpicture}[anchorbase,scale=0.25,tinynodes]
\draw[JW] (0,3) rectangle (2,5);
\node at (1,3.8) {$\pjwm[\alpha]$};
\draw[JW] (3,3) rectangle (5,5);
\node at (4,3.8) {$\pjwm[\beta]$};
\draw[JW] (6,3) rectangle (8,5);
\node at (7,3.8) {$\pjwm[\gamma]$};
\draw[usual] (1.5,5) to[out=90,in=180] (2.5,6)node[left]{$\mu$} to[out=0,in=90] (3.5,5);
\draw[usual] (1.5,3) to[out=270,in=180] (2.5,2) to[out=0,in=270] (3.5,3);
\draw[usual] (4.5,5) to[out=90,in=180] (5.5,6)node[right]{$\nu$} to[out=0,in=90] (6.5,5);
\draw[usual] (4.5,3) to[out=270,in=180] (5.5,2) to[out=0,in=270] (6.5,3);
\draw[usual] (0.5,5) to[out=90,in=180] (4,7.5)node[below]{$\chi$} to[out=0,in=90] (7.5,5);
\draw[usual] (0.5,3) to[out=270,in=180] (4,.5) to[out=0,in=270] (7.5,3);
\end{tikzpicture}
=(-1)^{\mu+\nu+\chi}
\tfrac{\qfac{\mu+\mu+\chi+1}{\vpar}\qfac{\mu}{\vpar}\qfac{\nu}{\vpar}\qfac{\chi}{\vpar}}
{\qfac{\mu+\nu}{\vpar}\qfac{\mu+\chi}{\vpar}\qfac{\nu+\chi}{\vpar}}
\in\kkv.
\end{gather*}
Let $x\in\N[0]$ be such that $(a\ppar^{(k)}-1)+(b\ppar^{(t)}-1)=2x+(v-1)$ 
and set $m=a\ppar^{(k)}-1-x$ and $n=b\ppar^{(t)}-1-x$. (We refer to
\eqref{eq:XandD} with $w=v$ for a diagrammatic interpretation of these
quantities.) We now define a scalar in $\kkv$ by
\begin{gather}\label{eq:theta-scalars1}
\twebstild
:=
(-1)^{(v-1)}\tfrac{\Theta(a\ppar^{(k)}-1,b\ppar^{(t)}-1,v-1)}{\qnum{v}{\vpar}}
= 
(-1)^{x}
\tfrac{\qbin{m+n+x+1}{x}{\vpar}}
{\qbin{m+x}{x}{\vpar}\qbin{n+x}{x}{\vpar}}.
\end{gather}
The last equality follows from a straightforward computation with quantum
binomials.

\begin{lemma}\label{lemma:ssdigon} Let $\tmod(v-1)$ be a summand of
$\tmod(a\ppar^{(k)}-1)\hcirc\tmod(b\ppar^{(t)}-1)$ as specified by
\eqref{eq:fusion-eve-eve}, which is simple, {\ie} of generation zero.  Then
we have 
\begin{gather*}
\dmor{a\ppar^{(k)}-1,b\ppar^{(t)}-1}{v-1}{v-1}
=\twebs\cdot\pjw[v{-}1],
\end{gather*}
where $\twebs$ is obtained from $\twebstild$ by specializing $\vpar$ to $\qpar$.
\end{lemma}

\begin{proof} 
We consider the analogous digon morphism in $\kkv$. (This is possible since all of the
involved projectors are simple JW projectors). Abusing notation, we find that
\begin{gather*}
\dmor{a\ppar^{(k)}-1,b\ppar^{(t)}-1}{v-1}{v-1}
=c\cdot\qjw[v{-}1]
\end{gather*} 
for some scalar $c\in\kkv$. Now taking the trace on both sides gives
\begin{gather*}
\Theta(a\ppar^{(k)}-1,b\ppar^{(t)}-1,v-1)=c(-1)^{v-1}\qnum{v}{\vpar},
\end{gather*} 
which implies $c=\twebstild$. It follows that this scalar specializes to a
well-defined scalar $\twebs$ in $\kk$ and the claim follows.
\end{proof}

\begin{proposition}\label{proposition:nonzeroscalar} 
Suppose
$0\leq a\ppar^{(k)}+b\ppar^{(t)}-v-1$ is even. Then
$\ord\big(\twebstild\big)\geq 0$, so $\twebstild$ descends to a well-defined
scalar $\twebs\in\kk$. Moreover, $\ord\big(\twebstild\big)=0$, and thus
$\twebs\neq 0$ if and only if $\tmod(v-1)$ appears as a summand in
$\tmod(a\ppar^{(k)}-1)\hcirc\tmod(b\ppar^{(t)}-1)$.  
\end{proposition}

Note that we have no assumptions on $\generation[v]$ in \fullref{proposition:nonzeroscalar}.

\begin{proof}
Throughout the proof we will use the quantities $x$, $m$, and $n$, defined as
above. 

First, we shall argue that $\twebstild$ descends to $\kk$ since the denominator
$\qbin{m+x}{x}{\vpar}\qbin{n+x}{x}{\vpar}$ in \eqref{eq:theta-scalars1} does not
vanish upon specializing from $\vpar$ to $\qpar$. Let us expand
$x=\plbase{x_{k},\dots,x_{0}}=\plbase{x_{t},\dots,x_{0}}$ (allowing zeros on the
left). As both $a\ppar^{(k)}=(m+x)+1$ and $b\ppar^{(t)}=(n+x)+1$ are eves, the quantum Lucas' theorem 
\fullref{proposition:qlucas} implies
\begin{gather*}
\qbin{m+x}{x}{\qpar}
=\lucas\qbin{a-1}{x_{k}}{\qpar}\qbin{\ppar-1}{x_{k-1}}{\qpar}
\dots\qbin{\lpar-1}{x_{0}}{\qpar}\neq 0\quad\text{and}\quad 
\qbin{n+x}{x}{\qpar}
=\lucas\qbin{b-1}{x_{t}}{\qpar}\qbin{\ppar-1}{x_{t-1}}{\qpar}
\dots\qbin{\lpar-1}{x_{0}}{\qpar}\neq 0,
\end{gather*}
since all digits of $(m+x)$ and $(n+x)$ are greater than or equal to the corresponding
digit of $x$.

Next we assume that $\tmod(v-1)$ is a summand of $\tmod(a\ppar^{(k)}-1)\hcirc\tmod(b\ppar^{(t)}-1)$.  
To show $\twebs\neq 0$, first assume $k>t$ and $\tmod(v-1)$ takes the form as a
typical summand in \eqref{eq:fusion-eve-eve}. The relevant $\ppar\lpar$-adic
expansions in $\twebs$ are given by
\begin{align*}
v+x
&=\plbase{a,0,\dots,0,b-1-x_{t},\ppar-1-x_{t-1},\dots,\lpar-1-x_{0}},
\\
v&=\plbase{a,0,\dots,0,b-1-2x_{t},\ppar-1-2x_{t-1},\dots,\lpar-1-2x_{0}}.
\end{align*}
The digits of $x$ are constrained by $x_{t}\leq\tfrac{b-1}{2}$ and
$x_{j}\leq\tfrac{\plpar-1}{2}$ for $0\leq j\leq t-1$. 
Consequently we also have $x_{t}\leq b-1-x_{t}$, 
$x_{j}\leq\plpar-1-x_{j}$ for $0\leq j\leq t-1$. 
The quantum Lucas' theorem with $m+n+x+1=v+x$ thus implies
\begin{gather*}
\qbin{v+x}{x}{\qpar}=
\qbin{m+n+x+1}{x}{\qpar}=\lucas\qbin{a}{0}{\qpar}\qbin{0}{0}{\qpar}
\dots\qbin{0}{0}{\qpar}\qbin{b-1-x_{t}}{x_{t}}{\qpar}\qbin{\ppar-1-x_{t-1}}{x_{t-1}}{\qpar}\dots
\qbin{\lpar-1-x_{0}}{x_{0}}{\qpar}\neq 0.
\end{gather*}
Now suppose $k=t$, where we again want to use the quantum Lucas' theorem. To
this end, observe that all digits below the $k$th one behave as in the  
$k>t$ case.
So let us focus on the digit of $m+n+x+1$ which involves $x_{k}$. 
For \eqref{eq:eve-fusion-second}, this digit is 
$a+b-x_{k}-1$, which is bigger or equal to $x_{k}$ because $2x_{k}\leq a+b-1$.
For \eqref{eq:eve-fusion-third}, this digit for 
the $x$-sum therein 
is $|a-b|+3x_{k}-1$ for $x_{k}$ at least $1$, 
which is clearly bigger or equal to $x_{k}$. 
For the $y$-sum, this digit is $a+b-\plpar-x_{k}-1$, which is 
bigger or equal to $x_{k}$ by the
allowed range for $x_{k}$ which gives $2x_{k}\leq a+b-\plpar-1$.

For the final part of this proof, we assume that $\tmod(v-1)$ is not a summand of
$\tmod(a\ppar^{(k)}-1)\hcirc\tmod(b\ppar^{(t)}-1)$. To see that the scalar
vanishes in this case, we again use the quantum Lucas' theorem. In particular, the
relevant calculations of $v+x$ and $v$ stay the same, but now at least one of the
digits of $x$ leaves the specified range and the corresponding (quantum)
binomial in the factorization for $\qbin{v+x}{x}{\qpar}$ vanishes.
\end{proof}

We now return to the task of computing digon expansions
\eqref{eq:ansatz-digon}. Suppose that $\generation[v]=1$ with corresponding
minimal down-admissible stretch $S$, which reflects down along $c=d\ppar^{(i)}$.
Then we define another scalar in $\kkv$ by 
\begin{gather}\label{eq:theta-scalars2}
\begin{aligned}
\twebsstild
=&
\scalebox{0.85}{$(-1)^{c+x}\left(-\frac{\qnum{m+n-2c+1}{\vpar}}{\qnum{m+n-c+1}{\vpar}}
\frac{\qbin{m+n+x+1}{x}{\vpar}}
{\qbin{m+x}{x}{\vpar}\qbin{n+x}{x}{\vpar}}
+
\big(\frac{\qnum{m+n-2c+1}{\vpar}!\qnum{m}{\vpar}!}
{\qnum{m+n-c+1}{\vpar}!\qnum{m-c}{\vpar}!}\big)^{2}
\frac{\qbin{m+n+x-c+1}{x+c}{\vpar}}
{\qbin{m+x}{x+c}{\vpar}\qbin{n+x}{x+c}{\vpar}}
\right)$}.
\end{aligned}
\end{gather}

\begin{theorem}\label{theorem:eve-times-eve}
Retaining notation as above and for $\generation[v]=1$, we have
\begin{gather}\label{eq:theta-theorem}
\dmor{a\ppar^{(k)}-1,b\ppar^{(t)}-1}{v-1}{v-1}
=
\twebs
\cdot
\pjw[v{-}1]
+
\twebss
\cdot
\loopdown{S}{v{-}1},
\end{gather}
where $\twebss$ is obtained from $\twebsstild$ by specializing $\vpar$ to $\qpar$.
\end{theorem}

The proof appears below.

\begin{example}
For characteristic $\ppar=3$, \fullref{proposition:eve-times-eve} gives 
\[
\tmod(8)\hcirc\tmod(8)\cong\tmod(16)\oplus\tmod(14)
\oplus\tmod(10)\oplus\tmod(8)
\]
Let $v=15$, then $m=n=7$, 
$x=1$ and $c=6$, so that
\begin{gather*}
\twebss
=
(-1)^{7}
\left(-\frac{3}{9}
\frac{\binom{16}{1}}
{\binom{8}{1}\binom{8}{1}}
+
\big(\tfrac{3!7!}
{9!1!}\big)^{2}
\frac{\binom{10}{7}}
{\binom{8}{7}\binom{8}{7}}
\right)
=\tfrac{1}{2^{2}3}
-\tfrac{5}{2^{7}3}
=\tfrac{3^{3}}{2^{7}3}
=0.
\end{gather*}
In general, when the scalar $\twebss$ is zero, the 
corresponding digon morphism can be written as a single summand that depends only on the scalar $\twebs$. In our case,
\begin{gather*}
\dmor{8,8}{14}{14}
=
\begin{tikzpicture}[anchorbase,scale=0.25,tinynodes]
\draw[JW] (-2.5,3) rectangle (2.5,5);
\node at (0,3.8) {$\pjwm[8]$};
\draw[JW] (3.5,3) rectangle (8.5,5);
\node at (6,3.8) {$\pjwm[8]$};
\draw[pJW] (0.5,-1) rectangle (5.5,1);
\node at (3,-0.2) {$\pjwm[14]$};
\draw[pJW] (0.5,7) rectangle (5.5,9);
\node at (3,7.8) {$\pjwm[14]$};
\draw[usual] (2,5) to[out=90,in=180] (3,6)node[above,xshift=0.15cm,yshift=-0.15cm]{$1$} to[out=0,in=90] (4,5);
\draw[usual] (-1.5,5) to[out=90,in=270] (1.5,7)node[left,xshift=-0.2cm,yshift=-0.15cm]{$7$};
\draw[usual] (7.5,5) to[out=90,in=270] (4.5,7)node[right,xshift=0.2cm,yshift=-0.15cm]{$7$};
\draw[usual] (2,3) to[out=270,in=180] (3,2)node[below,xshift=0.15cm,yshift=0cm]{$1$} to[out=0,in=270] (4,3);
\draw[usual] (-1.5,3) to[out=270,in=90] (1.5,1)node[left,xshift=-0.2cm,yshift=0.00cm]{$7$};
\draw[usual] (7.5,3) to[out=270,in=90] (4.5,1)node[right,xshift=0.2cm,yshift=0.00cm]{$7$};
\end{tikzpicture}
=
\twebs
\cdot
\pjw[14]
=\tfrac{1}{4}
\cdot
\pjw[14]
=\pjw[14].
\end{gather*}
\end{example}

For the proof of \fullref{theorem:eve-times-eve} we need the following lemma
about simple JW projectors. We were not able to find it in the literature, but
it is probably known.

\begin{lemma}\label{lemma:kl-simplification}
For all $a,b,c,d\in\N[0]$ with $b\geq c$, we have the following relation in $\TL[\kkv,\vpar]$:
\begin{gather*}
\begin{tikzpicture}[anchorbase,scale=0.25,tinynodes]
\draw[JW] (3.5,3) rectangle (7.5,5);
\node at (5.5,3.8) {$\pjwm[a{+}c{+}d]$};
\draw[JW] (0.5,7) rectangle (5.5,9);
\node at (3,7.8) {$\pjwm[a{+}b]$};
\draw[JW] (3,3) rectangle (-1,5);
\node at (1,3.8) {$\pjwm[b{+}d]$};
\draw[usual] (1.5,5)node[left,yshift=0.175cm]{$b$} to[out=90,in=270] (1.5,7);
\draw[usual] (5.2,5)node[right,xshift=-0.1cm,yshift=0.175cm]{$a$} to[out=90,in=270] (4.5,7);
\draw[usual] (4.5,9) to[out=90,in=180] (5.5,9.5) to[out=0,in=90] (6.5,9)node[right,yshift=-0.5cm]{$c$} to (6.5,5);
\draw[usual] (4.5,5) to[out=90,in=0] (3.5,5.5)node[above,yshift=-0.1cm]{$d$} to[out=180,in=90] (2.5,5);
\end{tikzpicture}
=
\tfrac{\qnum{a+b-c}{\vpar}!\qnum{b}{\vpar}!}
{\qnum{a+b}{\vpar}!\qnum{b-c}{\vpar}!}
\cdot
\begin{tikzpicture}[anchorbase,scale=0.25,tinynodes]
\draw[JW] (3.5,3) rectangle (7.5,5);
\node at (5.5,3.8) {$\pjwm[a{+}c{+}d]$};
\draw[JW] (0.5,7) rectangle (5.5,9);
\node at (3,7.8) {$\pjwm[a{+}b{-}c]$};
\draw[JW] (3,3) rectangle (-1,5);
\node at (1,3.8) {$\pjwm[b{+}d]$};
\draw[usual] (1.5,5)node[left,yshift=0.175cm]{$b{-}c$} to[out=90,in=270] (1.5,7);
\draw[usual] (6.5,5)node[right,xshift=-0.1cm,yshift=0.175cm]{$a$} to[out=90,in=270] (4.5,7);
\draw[usual] (4.5,5) to[out=90,in=0] (3.5,5.5)node[above,yshift=-0.1cm]{$c{+}d$} to[out=180,in=90] (2.5,5);
\end{tikzpicture}
.
\end{gather*}
\end{lemma}

\begin{proof}
By \eqref{eq:0absorb} 
it suffices to prove the result for $d=0$.
The cases with $c=0$ or $a=0$ are trivial, 
and for $c>0$ we apply the JW recursion
and \eqref{eq:0kill} to get
\begin{gather*}
\begin{tikzpicture}[anchorbase,scale=0.25,tinynodes]
\draw[JW] (3.5,3) rectangle (7.5,5);
\node at (5.5,3.8) {$\pjwm[a{+}c]$};
\draw[JW] (0.5,7) rectangle (5.5,9);
\node at (3,7.8) {$\pjwm[a{+}b]$};
\draw[JW] (3,3) rectangle (-1,5);
\node at (1,3.8) {$\pjwm[b]$};
\draw[usual] (1.5,5)node[left,yshift=0.175cm]{$b$} to[out=90,in=270] (1.5,7);
\draw[usual] (4.5,5)node[left,yshift=0.175cm]{$a$} to[out=90,in=270] (4.5,7);
\draw[usual] (4.5,9) to[out=90,in=180] (5.5,9.5) to[out=0,in=90] (6.5,9)node[right,yshift=-0.5cm]{$c$} to (6.5,5);
\end{tikzpicture}
=
\begin{tikzpicture}[anchorbase,scale=0.25,tinynodes]
\draw[JW] (3.5,3) rectangle (7.5,5);
\node at (5.5,3.8) {$\pjwm[a{+}c]$};
\draw[JW] (0.5,7) rectangle (5.5,9);
\node at (3,7.8) {$\pjwm[a{+}b{-}1]$};
\draw[JW] (3,3) rectangle (-1,5);
\node at (1,3.8) {$\pjwm[b]$};
\draw[usual] (1.5,5)node[left,yshift=0.175cm]{$b$} to[out=90,in=270] (1.5,7);
\draw[usual] (4.5,5)node[left,xshift=0.05cm,yshift=0.175cm]{$a{-}1$} to[out=90,in=270] (4.5,7);
\draw[usual] (4.5,9) to[out=90,in=180] (5.5,9.5) to[out=0,in=90] (6.5,9)node[right,yshift=-0.5cm]{$c{-}1$} to (6.5,5);
\draw[usual] (5,5) to[out=90,in=180] (5.5,5.5) to[out=0,in=90] (6,5);
\end{tikzpicture}
+\tfrac{\qnum{a+b-1}{\vpar}}{\qnum{a+b}{\vpar}}\cdot
\begin{tikzpicture}[anchorbase,scale=0.25,tinynodes]
\draw[JW] (3.5,3) rectangle (7.5,5);
\node at (5.5,3.8) {$\pjwm[a{+}c]$};
\draw[JW] (0.5,7) rectangle (5,8);
\node at (2.75,7.3) {$\pjwm[a{+}b{-}1]$};
\draw[JW] (0.5,10) rectangle (5,11);
\node at (2.75,10.3) {$\pjwm[a{+}b{-}1]$};
\draw[JW] (3,3) rectangle (-1,5);
\node at (1,3.8) {$\pjwm[b]$};
\draw[usual] (1.5,5)node[left,yshift=0.175cm]{$b$} to[out=90,in=270] (1.5,7);
\draw[usual] (4.5,5)node[left,xshift=0.05cm,yshift=0.175cm]{$a{-}1$} to[out=90,in=270] (4.5,7);
\draw[usual] (4.5,11) to[out=90,in=180] (5.5,11.5) to[out=0,in=90] (6.5,11)node[right,yshift=-1.0cm]{$c{-}1$} to (6.5,5);
\draw[usual] (1.5,8) to[out=90,in=270] (1.5,10);
\draw[usual] (5,5) to[out=90,in=270] (5.5,8) to[out=90,in=90] (4.5,8);
\draw[usual] (6,5) to[out=90,in=270] (6,8) to[out=90,in=270] (4.5,10);
\end{tikzpicture}
=
\tfrac{\qnum{a+b-1}{\vpar}}{\qnum{a+b}{\vpar}}\cdot
\begin{tikzpicture}[anchorbase,scale=0.25,tinynodes]
\draw[JW] (3.5,3) rectangle (7.5,5);
\node at (5.5,3.8) {$\pjwm[a{+}c]$};
\draw[JW] (0.5,7) rectangle (5,8);
\node at (2.75,7.3) {$\pjwm[a{+}b{-}1]$};
\draw[JW] (0.5,10) rectangle (5,11);
\node at (2.75,10.3) {$\pjwm[a{+}b{-}1]$};
\draw[JW] (3,3) rectangle (-1,5);
\node at (1,3.8) {$\pjwm[b]$};
\draw[usual] (1.5,5)node[left,yshift=0.175cm]{$b$} to[out=90,in=270] (1.5,7);
\draw[usual] (4.5,5)node[left,xshift=0.05cm,yshift=0.175cm]{$a{-}1$} to[out=90,in=270] (4.5,7);
\draw[usual] (4.5,11) to[out=90,in=180] (5.5,11.5) to[out=0,in=90] (6.5,11)node[right,yshift=-1.0cm]{$c{-}1$} to (6.5,5);
\draw[usual] (1.5,8) to[out=90,in=270] (1.5,10);
\draw[usual] (5,5) to[out=90,in=270] (5.5,8) to[out=90,in=90] (4.5,8);
\draw[usual] (6,5) to[out=90,in=270] (6,8) to[out=90,in=270] (4.5,10);
\draw[usual,tomato,densely dotted] (-1,9) to (9,9)
node[above,yshift=-0.05cm]{top}node[below,yshift=-0.05cm]{bottom};
\end{tikzpicture}
.
\end{gather*}
If $c=1$, then the top part is what we 
want, while the bottom can be simplified using induction. 
For $c>1$, we expand both $\qjw[a{+}b{-}1]$ once more.
In doing so, we obtain a diagram of the same type (up to partial trace, which we can remove up to a scalar using \eqref{eq:0trace}), but with smaller projectors.
This process is repeated until we obtain our desired result.
\end{proof}

\begin{lemma}\label{lemma:theta-scalar}
The scalar $\twebsstild$ as in \eqref{eq:theta-scalars2} 
can be written as
\begin{gather*}
\twebsstild
=
-\lambda_{v,S}\twebstild
+\lambda_{v,S}^{2}
\frac{\qnum{m+n-2c}{\vpar}!\qnum{m}{\vpar}!}
{\qnum{m+n-c}{\vpar}!\qnum{m-c}{\vpar}!}
\twebstild[v{-}2c]
.
\end{gather*}
\end{lemma}

\begin{proof}
After recalling that $\lambda_{v,S}=(-1)^{c}
\tfrac{\qnum{m+n-2c+1}{\vpar}}{\qnum{m+n-c+1}{\vpar}}$, 
this follows by using
algebra autopilot.
\end{proof}

\begin{proof}[Proof of \fullref{theorem:eve-times-eve}]
We work in the generic setting of $\TL[\kkv,\vpar]$ and specialize to
$\TL[\kk,\qpar]$ at the end. Note that we can expand the semisimple $\ppar\lpar$JW projector $\pqjw[v{-}1]$ via \fullref{lemma:pl-jw-q} as follows
\begin{gather*}
\pqjw[v{-}1]=
\begin{tikzpicture}[anchorbase,scale=0.25,tinynodes]
\draw[pQJW] (-2.5,3) rectangle (2.5,5);
\node at (0,3.8) {$\pjwm[v{-}1]$};
\end{tikzpicture}
=
\begin{tikzpicture}[anchorbase,scale=0.25,tinynodes]
\draw[JW] (-2.5,3) rectangle (2.5,5);
\node at (0,3.8) {$\pjwm[v{-}1]$};
\end{tikzpicture}
+
\lambda_{v,S}\cdot
\begin{tikzpicture}[anchorbase,scale=0.25,tinynodes]
\draw[pQJW] (-1.5,1) rectangle (1.5,2);
\node at (0,1.3) {\scalebox{0.6}{$\pjwm[v{-}c{-}1]$}};
\draw[pQJW] (-1.5,-1) rectangle (1.5,-2);
\node at (0,-1.7) {\scalebox{0.6}{$\pjwm[v{-}c{-}1]$}};
\draw[usual] (0,-1) to (0,1);
\draw[usual] (1,-1) to[out=90,in=180] (1.5,-0.5) 
to[out=0,in=90] (2,-1) to (2,-2) node[right,xshift=-2pt,yshift=1pt]{$c$};
\draw[usual] (1,1) to[out=270,in=180] (1.5,0.5) 
to[out=0,in=270] (2,1) to (2,2) node[right,xshift=-2pt,yshift=-2pt]{$c$};
\end{tikzpicture}
=
\begin{tikzpicture}[anchorbase,scale=0.25,tinynodes]
\draw[JW] (-2.5,3) rectangle (2.5,5);
\node at (0,3.8) {$\pjwm[v{-}1]$};
\end{tikzpicture}
+
\lambda_{v,S}\cdot
\begin{tikzpicture}[anchorbase,scale=0.25,tinynodes]
\draw[JW] (-1.5,1) rectangle (1.5,2);
\node at (0,1.3) {\scalebox{0.8}{$\pjwm[v{-}c{-}1]$}};
\draw[JW] (-1.5,-1) rectangle (1.5,-2);
\node at (0,-1.7) {\scalebox{0.8}{$\pjwm[v{-}c{-}1]$}};
\draw[usual] (0,-1) to (0,1);
\draw[usual] (1,-1) to[out=90,in=180] (1.5,-0.5) 
to[out=0,in=90] (2,-1) to (2,-2) node[right,xshift=-2pt,yshift=1pt]{$c$};
\draw[usual] (1,1) to[out=270,in=180] (1.5,0.5) 
to[out=0,in=270] (2,1) to (2,2) node[right,xshift=-2pt,yshift=-2pt]{$c$};
\end{tikzpicture}
,
\end{gather*}
where we point out that $\generation[v{-}c]=0$ and thus,
$\pqjw[v{-}c{-}1]=\qjw[v{-}c{-}1]$.
The two summands of $\pqjw[v{-}1]$ 
are orthogonal idempotents.
Hence, expanding the projectors at the top and bottom 
we get
\begin{gather}\label{eq:bigdiag}
\begin{tikzpicture}[anchorbase,scale=0.25,tinynodes]
\draw[JW] (-2.5,3) rectangle (2.5,5);
\node at (0,3.8) {$\pjwm[a\ppar^{(k)}{-}1]$};
\draw[JW] (3.5,3) rectangle (8.5,5);
\node at (6,3.8) {$\pjwm[b\ppar^{(t)}{-}1]$};
\draw[pQJW] (0.5,-1) rectangle (5.5,1);
\node at (3,-0.2) {$\pjwm[v{-}1]$};
\draw[pQJW] (0.5,7) rectangle (5.5,9);
\node at (3,7.8) {$\pjwm[v{-}1]$};
\draw[usual] (2,5) to[out=90,in=180] (3,6)node[above,xshift=0.15cm,yshift=-0.1cm]{$x$} to[out=0,in=90] (4,5);
\draw[usual] (-1.5,5) to[out=90,in=270] (1.5,7);
\draw[usual] (7.5,5) to[out=90,in=270] (4.5,7);
\draw[usual] (2,3) to[out=270,in=180] (3,2)node[below,xshift=0.15cm,yshift=0.05cm]{$x$} to[out=0,in=270] (4,3);
\draw[usual] (-1.5,3) to[out=270,in=90] (1.5,1);
\draw[usual] (7.5,3) to[out=270,in=90] (4.5,1);
\end{tikzpicture}
=
\begin{tikzpicture}[anchorbase,scale=0.25,tinynodes]
\draw[JW] (-2.5,3) rectangle (2.5,5);
\node at (0,3.8) {$\pjwm[a\ppar^{(k)}{-}1]$};
\draw[JW] (3.5,3) rectangle (8.5,5);
\node at (6,3.8) {$\pjwm[b\ppar^{(t)}{-}1]$};
\draw[JW] (0.5,-1) rectangle (5.5,1);
\node at (3,-0.2) {$\pjwm[v{-}1]$};
\draw[JW] (0.5,7) rectangle (5.5,9);
\node at (3,7.8) {$\pjwm[v{-}1]$};
\draw[usual] (2,5) to[out=90,in=180] (3,6)node[above,xshift=0.15cm,yshift=-0.1cm]{$x$} to[out=0,in=90] (4,5);
\draw[usual] (-1.5,5) to[out=90,in=270] (1.5,7);
\draw[usual] (7.5,5) to[out=90,in=270] (4.5,7);
\draw[usual] (2,3) to[out=270,in=180] (3,2)node[below,xshift=0.15cm,yshift=0.05cm]{$x$} to[out=0,in=270] (4,3);
\draw[usual] (-1.5,3) to[out=270,in=90] (1.5,1);
\draw[usual] (7.5,3) to[out=270,in=90] (4.5,1);
\end{tikzpicture}
+\lambda_{v,S}^{2}\cdot
\begin{tikzpicture}[anchorbase,scale=0.25,tinynodes]
\draw[JW] (-2.5,3) rectangle (2.5,5);
\node at (0,3.8) {$\pjwm[a\ppar^{(k)}{-}1]$};
\draw[JW] (3.5,3) rectangle (8.5,5);
\node at (6,3.8) {$\pjwm[b\ppar^{(t)}{-}1]$};
\draw[JW] (0.5,-1) rectangle (5.5,1);
\node at (3,-0.2) {$\pjwm[v{-}c{-}1]$};
\draw[JW] (0.5,-3) rectangle (5.5,-5);
\node at (3,-4.2) {$\pjwm[v{-}c{-}1]$};
\draw[JW] (0.5,7) rectangle (5.5,9);
\node at (3,7.8) {$\pjwm[v{-}c{-}1]$};
\draw[JW] (0.5,11) rectangle (5.5,13);
\node at (3,11.8) {$\pjwm[v{-}c{-}1]$};
\draw[usual] (2,5) to[out=90,in=180] (3,6)node[above,xshift=0.15cm,yshift=-0.1cm]{$x$} to[out=0,in=90] (4,5);
\draw[usual] (-1.5,5) to[out=90,in=270] (1.5,7);
\draw[usual] (7.5,5) to[out=90,in=270] (4.5,7);
\draw[usual] (2,3) to[out=270,in=180] (3,2)node[below,xshift=0.15cm,yshift=0.05cm]{$x$} to[out=0,in=270] (4,3);
\draw[usual] (-1.5,3) to[out=270,in=90] (1.5,1);
\draw[usual] (7.5,3) to[out=270,in=90] (4.5,1);
\draw[usual] (4.5,-1) to[out=270,in=180] (6.25,-1.75)node[below,xshift=-0.15cm,yshift=0.05cm]{$c$} to[out=0,in=270] (8,3);
\draw[usual] (4.5,9) to[out=90,in=180] (6.26,9.75)node[above,xshift=-0.15cm,yshift=-0.12cm]{$c$} to[out=0,in=90] (8,5);
\draw[usual] (4.5,-3) to[out=90,in=180] (5.5,-2.5) to[out=0,in=90] (6.5,-3) to (6.5,-5);
\draw[usual] (4.5,11) to[out=270,in=180] (5.5,10.5) to[out=0,in=270] (6.5,11) to (6.5,13);
\draw[usual] (2,9) to (2,11);
\draw[usual] (2,-1) to (2,-3);
\end{tikzpicture}
.
\end{gather}
The first summand is covered by classical recoupling theory 
and, using \cite[Section 9.10]{KaLi-TL-recoupling}, it evaluates to
\begin{gather*}
\twebstild
\cdot
\begin{tikzpicture}[anchorbase,scale=0.25,tinynodes]
\draw[JW] (0.5,-1) rectangle (5.5,1);
\node at (3,-0.2) {$\pjwm[v{-}1]$};
\end{tikzpicture}
=
\twebstild
\cdot
\left(
\begin{tikzpicture}[anchorbase,scale=0.25,tinynodes]
\draw[pQJW] (0.5,-1) rectangle (5.5,1);
\node at (3,-0.2) {$\pjwm[v{-}1]$};
\end{tikzpicture}
-\lambda_{v,S}\cdot
\begin{tikzpicture}[anchorbase,scale=0.25,tinynodes]
\draw[JW] (-1.5,1) rectangle (1.5,2);
\node at (0,1.3) {\scalebox{0.65}{$\pjwm[v{-}c{-}1]$}};
\draw[JW] (-1.5,-1) rectangle (1.5,-2);
\node at (0,-1.7) {\scalebox{0.65}{$\pjwm[v{-}c{-}1]$}};
\draw[usual] (0,-1) to (0,1);
\draw[usual] (1,-1) to[out=90,in=180] (1.5,-0.5) 
to[out=0,in=90] (2,-1) to (2,-2) node[right,xshift=-2pt,yshift=1pt]{$c$};
\draw[usual] (1,1) to[out=270,in=180] (1.5,0.5) 
to[out=0,in=270] (2,1) to (2,2) node[right,xshift=-2pt,yshift=-2pt]{$c$};
\end{tikzpicture}
\right)
.
\end{gather*}
Here we have rewritten the scalar in 
\cite[Section 9.10]{KaLi-TL-recoupling} by collecting terms 
into quantum binomials, giving 
us the expression in \eqref{eq:theta-scalars1}, and then 
we applied the definition of $\pqjw[v{-}1]$ backwards.
The second summand in \eqref{eq:bigdiag} can be evaluated using 
\fullref{lemma:kl-simplification} twice to give
\begin{gather*}
\big(\tfrac{\qnum{m+n-2c}{\vpar}!\qnum{m}{\vpar}!}
{\qnum{m+n-c}{\vpar}!\qnum{m-c}{\vpar}!}\big)^{2}
\twebstild[v{-}2c]
\cdot
\begin{tikzpicture}[anchorbase,scale=0.25,tinynodes]
\draw[JW] (-1.5,1) rectangle (1.5,2);
\node at (0,1.3) {\scalebox{0.65}{$\pjwm[v{-}c{-}1]$}};
\draw[JW] (-1.5,-1) rectangle (1.5,-2);
\node at (0,-1.7) {\scalebox{0.65}{$\pjwm[v{-}c{-}1]$}};
\draw[usual] (0,-1) to (0,1);
\draw[usual] (1,-1) to[out=90,in=180] (1.5,-0.5) 
to[out=0,in=90] (2,-1) to (2,-2) node[right,xshift=-2pt,yshift=1pt]{$c$};
\draw[usual] (1,1) to[out=270,in=180] (1.5,0.5) 
to[out=0,in=270] (2,1) to (2,2) node[right,xshift=-2pt,yshift=-2pt]{$c$};
\end{tikzpicture}
.
\end{gather*}
Collecting terms and using \fullref{lemma:theta-scalar}, proves the claimed
formula. Finally, note that the left-hand side of \eqref{eq:theta-theorem} has
non-negative $\ppar\lpar$-adic valuation, and the morphisms $\pqjw[v{-}1]$ and
$\loopdowngen{v}{S}$  descend to the basis elements $\pjw[v{-}1]$ and
$\loopdown{S}{v{-}1}$ of the endomorphism algebra of $\pjw[v{-}1]$ . Hence, the
scalars in \eqref{eq:theta-scalars1} and \eqref{eq:theta-scalars2} can also be
specialized to $(\kk,\qpar)$.
\end{proof}

\begin{example}\label{example:thetawebs}
For characteristic $\ppar=5$ we get
\begin{gather*}
\scalebox{0.91}{$\begin{tikzpicture}[anchorbase,scale=0.25,tinynodes]
\draw[JW] (-2.5,3) rectangle (2.5,5);
\node at (0,3.8) {$\pjwm[4]$};
\draw[JW] (3.5,3) rectangle (8.5,5);
\node at (6,3.8) {$\pjwm[4]$};
\draw[pJW] (0.5,-1) rectangle (5.5,1);
\node at (3,-0.2) {$\pjwm[8]$};
\draw[pJW] (0.5,7) rectangle (5.5,9);
\node at (3,7.8) {$\pjwm[8]$};
\draw[usual] (2,5) to[out=90,in=180] (3,6)node[above,xshift=0.15cm,yshift=-0.12cm]{$0$} to[out=0,in=90] (4,5);
\draw[usual] (-1.5,5) to[out=90,in=270] (1.5,7);
\draw[usual] (7.5,5) to[out=90,in=270] (4.5,7);
\draw[usual] (2,3) to[out=270,in=180] (3,2)node[below,xshift=0.15cm,yshift=0.00cm]{$0$} to[out=0,in=270] (4,3);
\draw[usual] (-1.5,3) to[out=270,in=90] (1.5,1);
\draw[usual] (7.5,3) to[out=270,in=90] (4.5,1);
\end{tikzpicture}
=
\begin{tikzpicture}[anchorbase,scale=0.25,tinynodes]
\draw[pJW] (0.5,-1) rectangle (5.5,1);
\node at (3,-0.2) {$\pjwm[8]$};
\end{tikzpicture}
,\quad
\begin{tikzpicture}[anchorbase,scale=0.25,tinynodes]
\draw[JW] (-2.5,3) rectangle (2.5,5);
\node at (0,3.8) {$\pjwm[4]$};
\draw[JW] (3.5,3) rectangle (8.5,5);
\node at (6,3.8) {$\pjwm[4]$};
\draw[pJW] (0.5,-1) rectangle (5.5,1);
\node at (3,-0.2) {$\pjwm[6]$};
\draw[pJW] (0.5,7) rectangle (5.5,9);
\node at (3,7.8) {$\pjwm[6]$};
\draw[usual] (2,5) to[out=90,in=180] (3,6)node[above,xshift=0.15cm,yshift=-0.12cm]{$1$} to[out=0,in=90] (4,5);
\draw[usual] (-1.5,5) to[out=90,in=270] (1.5,7);
\draw[usual] (7.5,5) to[out=90,in=270] (4.5,7);
\draw[usual] (2,3) to[out=270,in=180] (3,2)node[below,xshift=0.15cm,yshift=0.00cm]{$1$} to[out=0,in=270] (4,3);
\draw[usual] (-1.5,3) to[out=270,in=90] (1.5,1);
\draw[usual] (7.5,3) to[out=270,in=90] (4.5,1);
\end{tikzpicture}
=
2\cdot
\begin{tikzpicture}[anchorbase,scale=0.25,tinynodes]
\draw[pJW] (0.5,-1) rectangle (5.5,1);
\node at (3,-0.2) {$\pjwm[6]$};
\end{tikzpicture}
+
3
\cdot
\begin{tikzpicture}[anchorbase,scale=0.25,tinynodes]
\draw[JW] (-1.5,1) rectangle (1.5,2);
\node at (0,1.25) {$\pjwm[4]$};
\draw[JW] (-1.5,-1) rectangle (1.5,-2);
\node at (0,-1.725) {$\pjwm[4]$};
\draw[usual] (0,-1) to (0,1);
\draw[usual] (1,-1) to[out=90,in=180] (1.5,-0.5) 
to[out=0,in=90] (2,-1) to (2,-2) node[right,xshift=-2pt,yshift=1pt]{$2$};
\draw[usual] (1,1) to[out=270,in=180] (1.5,0.5) 
to[out=0,in=270] (2,1) to (2,2) node[right,xshift=-2pt,yshift=-2pt]{$2$};
\end{tikzpicture}
,\quad
\begin{tikzpicture}[anchorbase,scale=0.25,tinynodes]
\draw[JW] (-2.5,3) rectangle (2.5,5);
\node at (0,3.8) {$\pjwm[4]$};
\draw[JW] (3.5,3) rectangle (8.5,5);
\node at (6,3.8) {$\pjwm[4]$};
\draw[JW] (0.5,-1) rectangle (5.5,1);
\node at (3,-0.2) {$\pjwm[4]$};
\draw[JW] (0.5,7) rectangle (5.5,9);
\node at (3,7.8) {$\pjwm[4]$};
\draw[usual] (2,5) to[out=90,in=180] (3,6)node[above,xshift=0.15cm,yshift=-0.12cm]{$2$} to[out=0,in=90] (4,5);
\draw[usual] (-1.5,5) to[out=90,in=270] (1.5,7);
\draw[usual] (7.5,5) to[out=90,in=270] (4.5,7);
\draw[usual] (2,3) to[out=270,in=180] (3,2)node[below,xshift=0.15cm,yshift=0.00cm]{$2$} to[out=0,in=270] (4,3);
\draw[usual] (-1.5,3) to[out=270,in=90] (1.5,1);
\draw[usual] (7.5,3) to[out=270,in=90] (4.5,1);
\end{tikzpicture}
=
\begin{tikzpicture}[anchorbase,scale=0.25,tinynodes]
\draw[JW] (0.5,-1) rectangle (5.5,1);
\node at (3,-0.2) {$\pjwm[4]$};
\end{tikzpicture}$},
\end{gather*}
where \fullref{proposition:eve-times-eve} gives 
$\tmod(4)\hcirc\tmod(4)\cong\tmod(8)\oplus\tmod(6)\oplus\tmod(4)$.
\end{example}

\begin{theorem}\label{theorem:eve-fusion} Retain notation as in
\fullref{lemma:catfusion-problem} and let $\tmod(v-1)$ be a direct summand of
$\tmod(a\ppar^{(k)}-1)\hcirc\tmod(b\ppar^{(t)}-1)$ with $\generation[v]=0$ or
$\generation[v]=1$ (in the latter case $v$ has a unique non-trivial
down-admissible set, denoted $S$). Then the idempotent
$\pjw[{a\ppar^{(k)}-1,b\ppar^{(t)}-1}]^{v-1}$ in $\TL[\kk,\qpar]$ realizing the
projection onto $\tmod(v-1)$ is given by the ansatz
\eqref{eq:ansatz-idempotent}, {\ie} for $\generation[v]=1$ we have
\begin{gather}\label{eq:idem-eve}
\pjw[{a\ppar^{(k)}{-}1,b\ppar^{(t)}{-}1}]^{v-1}
=
\fusscalar{(v,\emptyset)}{a\ppar^{(k)},b\ppar^{(t)}}
\cdot
\xmor{a\ppar^{(k)}-1,b\ppar^{(t)}-1}{v-1}
+ 
\fusscalar{(v,S)}{a\ppar^{(k)},b\ppar^{(t)}}
\xmor{a\ppar^{(k)}-1,b\ppar^{(t)}-1}{v[S]-1},
\end{gather}  
with scalars
\begin{gather*}
\fusscalar{(v,\emptyset)}{a\ppar^{(k)},b\ppar^{(t)}} = (\twebs)^{-1},\quad 
\fusscalar{(v,S)}{a\ppar^{(k)},b\ppar^{(t)}}=\frac{(\twebs)^{-2}\twebss\big(\tfrac{\qnum{m+n-2c}{\qpar}!\qnum{m}{\qpar}!}
{\qnum{m+n-c}{\qpar}!\qnum{m-c}{\qpar}!}\big)^{2}}{1-2(\twebs)^{-1}\lambda_{v,S}\big(\tfrac{\qnum{m+n-2c}{\qpar}!\qnum{m}{\qpar}!}
{\qnum{m+n-c}{\qpar}!\qnum{m-c}{\qpar}!}\big)^{2}\twebs[{v[S]}]}.
\end{gather*}
For $\generation[v]=0$, the second term in \eqref{eq:idem-eve} is dropped.
\end{theorem}

\begin{proof} 
For ease of notation, let
\begin{gather*}
\morstuff{A}=\xmor{a\ppar^{(k)}-1,b\ppar^{(t)}-1}{v-1} 
\quad\text{and}\quad
\morstuff{B}=\xmor{a\ppar^{(k)}-1,b\ppar^{(t)}-1}{v[S]-1}.
\end{gather*}
We now proceed to calculate the compositions $\morstuff{A}^{2}$, $\morstuff{A}\morstuff{B}$ and $\morstuff{B}^{2}$.
\fullref{theorem:eve-times-eve} and \fullref{proposition:nonzeroscalar} 
imply
\begin{gather*}
\morstuff{B}^{2}
=\twebs[{v[S]}]\cdot\morstuff{B}=0
\end{gather*}
since  $\tmod(v[S]-1)$ is not a summand of
$\tmod(a\ppar^{(k)}-1)\hcirc\tmod(b\ppar^{(t)}-1)$ ($\wmod(v[S]-1)$ is already a
Weyl factor of $\tmod(v-1)$ and
$\tmod(a\ppar^{(k)}-1)\hcirc\tmod(b\ppar^{(t)}-1)$ is Weyl-multiplicity-free).

To compute the other compositions, we work in the generic
setting of $\TL[\kkv,\vpar]$. Applying \fullref{lemma:kl-simplification} to each
of the sandwiched projectors, we obtain
\begin{gather*}
\begin{tikzpicture}[anchorbase,scale=0.25,tinynodes]
\draw[JW] (-2.5,-7) rectangle (2.5,-5);
\node at (0,-6.2) {$\qjwm[a\ppar^{(k)}{-}1]$};
\draw[JW] (3.5,-7) rectangle (8.5,-5);
\node at (6,-6.2) {$\qjwm[b\ppar^{(k)}{-}1]$};
\draw[JW] (-2.5,-19) rectangle (2.5,-17);
\node at (0,-18.2) {$\qjwm[a\ppar^{(k)}{-}1]$};
\draw[JW] (3.5,-19) rectangle (8.5,-17);
\node at (6,-18.2) {$\qjwm[b\ppar^{(k)}{-}1]$};
\draw[JW] (0.5,-11) rectangle (5.5,-9);
\node at (3,-10.2) {$\qjwm[v{-}c{-}1]$};
\draw[JW] (0.5,-15) rectangle (5.5,-13);
\node at (3,-14.2) {$\qjwm[v{-}c{-}1]$};
\draw[usual] (2,-17) to[out=90,in=180] (3,-16)node[above,xshift=0.15cm,yshift=-0.12cm]{x} to[out=0,in=90] (4,-17);
\draw[usual] (2,-7) to[out=270,in=180] (3,-8)node[below,xshift=0.15cm,yshift=0.00cm]{x} to[out=0,in=270] (4,-7);
\draw[usual] (-1.5,-17) to[out=90,in=270] (1.5,-15)node[left,xshift=-0.4cm,yshift=-0.2cm]{m};
\draw[usual] (7.5,-17) to[out=90,in=270] (4.5,-15);
\draw[usual] (7.9,-17) to[out=90,in=60] (4.5,-13)node[right,xshift=0.5cm,yshift=-0.2cm]{c};
\draw[usual] (-1.5,-7) to[out=270,in=90] (1.5,-9)node[left,xshift=-0.4cm,yshift=0.05cm]{m};
\draw[usual] (7.9,-7) to[out=270,in=300] (4.5,-11)node[right,xshift=0.5cm,yshift=0.05cm]{c};
\draw[usual] (7.5,-7) to[out=270,in=90] (4.5,-9);
\draw[usual] (1.5,-13) to (1.5,-11);
\end{tikzpicture}
=\big(\tfrac{\qnum{m+n-2c}{\vpar}!\qnum{m}{\vpar}!}
{\qnum{m+n-c}{\vpar}!\qnum{m-c}{\vpar}!}\big)^{2}
\cdot\morstuff{B},
\end{gather*}
and thus applying \fullref{theorem:eve-times-eve}, gives us
\begin{align*}
\morstuff{A}^{2}
&=
\twebstild\cdot\morstuff{A}
+
\twebsstild\big(\tfrac{\qnum{m+n-2c}{\vpar}!\qnum{m}{\vpar}!}
{\qnum{m+n-c}{\vpar}!\qnum{m-c}{\vpar}!}\big)^{2}\cdot\morstuff{B}.
\end{align*}
By expansion of the projectors $\pjw[v{-}1]$ and $\pjw[v{[S]}{-}1]$ in the following diagram
\begin{gather*}
\morstuff{B}\morstuff{A}=
\begin{tikzpicture}[anchorbase,scale=0.25,tinynodes]
\draw[JW] (-2.5,3) rectangle (2.5,5);
\node at (0,3.8) {$\qjwm[a\ppar^{(k)}{-}1]$};
\draw[JW] (3.5,3) rectangle (8.5,5);
\node at (6,3.8) {$\qjwm[b\ppar^{(k)}{-}1]$};
\draw[JW] (-2.5,-5) rectangle (2.5,-3);
\node at (0,-4.2) {$\qjwm[a\ppar^{(k)}{-}1]$};
\draw[JW] (3.5,-5) rectangle (8.5,-3);
\node at (6,-4.2) {$\qjwm[b\ppar^{(k)}{-}1]$};
\draw[pQJW] (0.5,-1) rectangle (5.5,1);
\node at (3,-0.2) {$\pjwm[v{[S]}{-}1]$};
\draw[usual] (2,-3) to[out=90,in=180] (3,-2)node[above,xshift=0.15cm,yshift=-0.1cm]{$y$} to[out=0,in=90] (4,-3);
\draw[usual] (2,3) to[out=270,in=180] (3,2)node[below,xshift=0.15cm,yshift=0.05cm]{$y$} to[out=0,in=270] (4,3);
\draw[usual] (-1.5,-3) to[out=90,in=270] (1.5,-1);
\draw[usual] (7.5,-3) to[out=90,in=270] (4.5,-1);
\draw[usual] (-1.5,3) to[out=270,in=90] (1.5,1);
\draw[usual] (7.5,3) to[out=270,in=90] (4.5,1);
\draw[JW] (-2.5,-13) rectangle (2.5,-11);
\node at (0,-12.2) {$\qjwm[a\ppar^{(k)}{-}1]$};
\draw[JW] (3.5,-13) rectangle (8.5,-11);
\node at (6,-12.2) {$\qjwm[b\ppar^{(k)}{-}1]$};
\draw[pQJW] (0.5,-9) rectangle (5.5,-7);
\node at (3,-8.2) {$\pjwm[v{-}1]$};
\draw[usual] (2,-11) to[out=90,in=180] (3,-10)node[above,xshift=0.15cm,yshift=-0.12cm]{$x$} to[out=0,in=90] (4,-11);
\draw[usual] (2,-5) to[out=270,in=180] (3,-6)node[below,xshift=0.15cm,yshift=0.00cm]{$x$} to[out=0,in=270] (4,-5);
\draw[usual] (-1.5,-11) to[out=90,in=270] (1.5,-9);
\draw[usual] (7.5,-11) to[out=90,in=270] (4.5,-9);
\draw[usual] (-1.5,-5) to[out=270,in=90] (1.5,-7);
\draw[usual] (7.5,-5) to[out=270,in=90] (4.5,-7);
\end{tikzpicture}
=
\begin{tikzpicture}[anchorbase,scale=0.25,tinynodes]
\draw[JW] (-2.5,3) rectangle (2.5,5);
\node at (0,3.8) {$\qjwm[a\ppar^{(k)}{-}1]$};
\draw[JW] (3.5,3) rectangle (8.5,5);
\node at (6,3.8) {$\qjwm[b\ppar^{(k)}{-}1]$};
\draw[JW] (-2.5,-5) rectangle (2.5,-3);
\node at (0,-4.2) {$\qjwm[a\ppar^{(k)}{-}1]$};
\draw[JW] (3.5,-5) rectangle (8.5,-3);
\node at (6,-4.2) {$\qjwm[b\ppar^{(k)}{-}1]$};
\draw[pQJW] (0.5,-1) rectangle (5.5,1);
\node at (3,-0.2) {$\pjwm[v{[S]}{-}1]$};
\draw[usual] (2,-3) to[out=90,in=180] (3,-2)node[above,xshift=0.15cm,yshift=-0.1cm]{$y$} to[out=0,in=90] (4,-3);
\draw[usual] (2,3) to[out=270,in=180] (3,2)node[below,xshift=0.15cm,yshift=0.05cm]{$y$} to[out=0,in=270] (4,3);
\draw[usual] (-1.5,-3) to[out=90,in=270] (1.5,-1);
\draw[usual] (7.5,-3) to[out=90,in=270] (4.5,-1);
\draw[usual] (-1.5,3) to[out=270,in=90] (1.5,1);
\draw[usual] (7.5,3) to[out=270,in=90] (4.5,1);
\draw[JW] (-2.5,-13) rectangle (2.5,-11);
\node at (0,-12.2) {$\qjwm[a\ppar^{(k)}{-}1]$};
\draw[JW] (3.5,-13) rectangle (8.5,-11);
\node at (6,-12.2) {$\qjwm[b\ppar^{(k)}{-}1]$};
\draw[JW] (0.5,-9) rectangle (5.5,-7);
\node at (3,-8.2) {$\pjwm[v{-}1]$};
\draw[usual] (2,-11) to[out=90,in=180] (3,-10)node[above,xshift=0.15cm,yshift=-0.12cm]{$x$} to[out=0,in=90] (4,-11);
\draw[usual] (2,-5) to[out=270,in=180] (3,-6)node[below,xshift=0.15cm,yshift=0.00cm]{$x$} to[out=0,in=270] (4,-5);
\draw[usual] (-1.5,-11) to[out=90,in=270] (1.5,-9);
\draw[usual] (7.5,-11) to[out=90,in=270] (4.5,-9);
\draw[usual] (-1.5,-5) to[out=270,in=90] (1.5,-7);
\draw[usual] (7.5,-5) to[out=270,in=90] (4.5,-7);
\end{tikzpicture}
+\lambda_{v,S}\cdot
\begin{tikzpicture}[anchorbase,scale=0.25,tinynodes]
\draw[JW] (-2.5,3) rectangle (2.5,5);
\node at (0,3.8) {$\qjwm[a\ppar^{(k)}{-}1]$};
\draw[JW] (3.5,3) rectangle (8.5,5);
\node at (6,3.8) {$\qjwm[b\ppar^{(k)}{-}1]$};
\draw[JW] (-2.5,-5) rectangle (2.5,-3);
\node at (0,-4.2) {$\qjwm[a\ppar^{(k)}{-}1]$};
\draw[JW] (3.5,-5) rectangle (8.5,-3);
\node at (6,-4.2) {$\qjwm[b\ppar^{(k)}{-}1]$};
\draw[pQJW] (0.5,-1) rectangle (5.5,1);
\node at (3,-0.2) {$\pjwm[v{[S]}{-}1]$};
\draw[usual] (2,-3) to[out=90,in=180] (3,-2)node[above,xshift=0.15cm,yshift=-0.1cm]{$y$} to[out=0,in=90] (4,-3);
\draw[usual] (2,3) to[out=270,in=180] (3,2)node[below,xshift=0.15cm,yshift=0.05cm]{$y$} to[out=0,in=270] (4,3);
\draw[usual] (-1.5,-3) to[out=90,in=270] (1.5,-1);
\draw[usual] (7.5,-3) to[out=90,in=270] (4.5,-1);
\draw[usual] (-1.5,3) to[out=270,in=90] (1.5,1);
\draw[usual] (7.5,3) to[out=270,in=90] (4.5,1);
\draw[JW] (-2.5,-17) rectangle (2.5,-15);
\node at (0,-16.2) {$\qjwm[a\ppar^{(k)}{-}1]$};
\draw[JW] (3.5,-17) rectangle (8.5,-15);
\node at (6,-16.2) {$\qjwm[b\ppar^{(k)}{-}1]$};
\draw[JW] (0.5,-9) rectangle (5.5,-7);
\node at (3,-8.2) {$\qjwm[v{-}c{-}1]$};
\draw[JW] (0.5,-13) rectangle (5.5,-11);
\node at (3,-12.2) {$\qjwm[v{-}c{-}1]$};
\draw[usual] (2,-15) to[out=90,in=180] (3,-14)node[above,xshift=0.15cm,yshift=-0.12cm]{x} to[out=0,in=90] (4,-15);
\draw[usual] (2,-5) to[out=270,in=180] (3,-6)node[below,xshift=0.15cm,yshift=0.00cm]{x} to[out=0,in=270] (4,-5);
\draw[usual] (-1.5,-15) to[out=90,in=270] (1.5,-13)node[left,xshift=-0.4cm,yshift=-0.2cm]{m};
\draw[usual] (7.5,-15) to[out=90,in=270] (4.5,-13);
\draw[usual] (7.9,-15) to[out=90,in=60] (4.5,-11)node[right,xshift=0.5cm,yshift=-0.2cm]{c};
\draw[usual] (-1.5,-5) to[out=270,in=90] (1.5,-7)node[left,xshift=-0.4cm,yshift=0.05cm]{m};
\draw[usual] (7.9,-5) to[out=270,in=300] (4.5,-9)node[right,xshift=0.5cm,yshift=0.05cm]{c};
\draw[usual] (7.5,-5) to[out=270,in=90] (4.5,-7);
\draw[usual] (1.5,-11) to (1.5,-9);
\end{tikzpicture}
=
\lambda_{v,S}\cdot
\begin{tikzpicture}[anchorbase,scale=0.25,tinynodes]
\draw[JW] (-2.5,3) rectangle (2.5,5);
\node at (0,3.8) {$\qjwm[a\ppar^{(k)}{-}1]$};
\draw[JW] (3.5,3) rectangle (8.5,5);
\node at (6,3.8) {$\qjwm[b\ppar^{(k)}{-}1]$};
\draw[JW] (-2.5,-5) rectangle (2.5,-3);
\node at (0,-4.2) {$\qjwm[a\ppar^{(k)}{-}1]$};
\draw[JW] (3.5,-5) rectangle (8.5,-3);
\node at (6,-4.2) {$\qjwm[b\ppar^{(k)}{-}1]$};
\draw[JW] (0.5,-1) rectangle (5.5,1);
\node at (3,-0.2) {$\pjwm[v{[S]}{-}1]$};
\draw[usual] (2,-3) to[out=90,in=180] (3,-2)node[above,xshift=0.15cm,yshift=-0.1cm]{$y$} to[out=0,in=90] (4,-3);
\draw[usual] (2,3) to[out=270,in=180] (3,2)node[below,xshift=0.15cm,yshift=0.05cm]{$y$} to[out=0,in=270] (4,3);
\draw[usual] (-1.5,-3) to[out=90,in=270] (1.5,-1);
\draw[usual] (7.5,-3) to[out=90,in=270] (4.5,-1);
\draw[usual] (-1.5,3) to[out=270,in=90] (1.5,1);
\draw[usual] (7.5,3) to[out=270,in=90] (4.5,1);
\draw[JW] (-2.5,-17) rectangle (2.5,-15);
\node at (0,-16.2) {$\qjwm[a\ppar^{(k)}{-}1]$};
\draw[JW] (3.5,-17) rectangle (8.5,-15);
\node at (6,-16.2) {$\qjwm[b\ppar^{(k)}{-}1]$};
\draw[JW] (0.5,-9) rectangle (5.5,-7);
\node at (3,-8.2) {$\qjwm[v{-}c{-}1]$};
\draw[JW] (0.5,-13) rectangle (5.5,-11);
\node at (3,-12.2) {$\qjwm[v{-}c{-}1]$};
\draw[usual] (2,-15) to[out=90,in=180] (3,-14)node[above,xshift=0.15cm,yshift=-0.12cm]{x} to[out=0,in=90] (4,-15);
\draw[usual] (2,-5) to[out=270,in=180] (3,-6)node[below,xshift=0.15cm,yshift=0.00cm]{x} to[out=0,in=270] (4,-5);
\draw[usual] (-1.5,-15) to[out=90,in=270] (1.5,-13)node[left,xshift=-0.4cm,yshift=-0.2cm]{m};
\draw[usual] (7.5,-15) to[out=90,in=270] (4.5,-13);
\draw[usual] (7.9,-15) to[out=90,in=60] (4.5,-11)node[right,xshift=0.5cm,yshift=-0.2cm]{c};
\draw[usual] (-1.5,-5) to[out=270,in=90] (1.5,-7)node[left,xshift=-0.4cm,yshift=0.05cm]{m};
\draw[usual] (7.9,-5) to[out=270,in=300] (4.5,-9)node[right,xshift=0.5cm,yshift=0.05cm]{c};
\draw[usual] (7.5,-5) to[out=270,in=90] (4.5,-7);
\draw[usual] (1.5,-11) to (1.5,-9);
\end{tikzpicture}
,
\end{gather*}
where all other summands in the expansion 
are annihilated 
since there are no common\makeautorefname{lemma}{Lemmas} Weyl factors. Further, by applying \fullref{lemma:kl-simplification} and \ref{lemma:ssdigon},\makeautorefname{lemma}{Lemma}
\begin{align*}
\morstuff{A}\morstuff{B}
=\lambda_{v,S} \big(\tfrac{\qnum{m+n-2c}{\vpar}!\qnum{m}{\vpar}!}
{\qnum{m+n-c}{\vpar}!\qnum{m-c}{\vpar}!}\big)^{2}
\twebstild[{v[S]}]\cdot \morstuff{B},
\end{align*}
which is symmetric, {\ie} $\morstuff{A}\morstuff{B}=\morstuff{B}\morstuff{A}$. 

It is now straightforward to verify that
$\fusscalar{(v,\emptyset)}{a\ppar^{(k)},b\ppar^{(t)}}\cdot
\morstuff{A}+\fusscalar{(v,S)}{a\ppar^{(k)},b\ppar^{(t)}}\cdot
\morstuff{B}$ is an idempotent.
\end{proof}

\begin{example}
For characteristic $\ppar=5$ we have $\tmod(4)\hcirc\tmod(4)=\tmod(8)\oplus\tmod(6)\oplus\tmod(4)$. 
We computed the digon scalars in \fullref{example:thetawebs}.
Now let us determine the idempotents for the
three summands of $\tmod(4)\hcirc\tmod(4)$. Note that $\generation[8]=\generation[4]=0$
and $\generation[6]=1$.
\begin{enumerate}

\item It follows from \fullref{example:thetawebs} 
that the idempotent endomorphisms of the summands 
$\tmod(8)$ and $\tmod(4)$, respectively, 
in the endomorphism space of $\tmod(4)\hcirc\tmod(4)$ are
\begin{gather*}
\pjw[{4,4}]^{8}=
\begin{tikzpicture}[anchorbase,scale=0.25,tinynodes]
\draw[JW] (-2.5,3) rectangle (2.5,5);
\node at (0,3.8) {$\qjwm[4]$};
\draw[JW] (3.5,3) rectangle (8.5,5);
\node at (6,3.8) {$\qjwm[4]$};
\draw[JW] (-2.5,-5) rectangle (2.5,-3);
\node at (0,-4.2) {$\qjwm[4]$};
\draw[JW] (3.5,-5) rectangle (8.5,-3);
\node at (6,-4.2) {$\qjwm[4]$};
\draw[pJW] (0.5,-1) rectangle (5.5,1);
\node at (3,-0.2) {$\pjwm[8]$};
\draw[usual] (2,-3) to[out=90,in=180] (3,-2)node[above,xshift=0.15cm,yshift=-0.12cm]{$0$} to[out=0,in=90] (4,-3);
\draw[usual] (2,3) to[out=270,in=180] (3,2)node[below,xshift=0.15cm,yshift=0.00cm]{$0$} to[out=0,in=270] (4,3);
\draw[usual] (-1.5,-3) to[out=90,in=270] (1.5,-1);
\draw[usual] (7.5,-3) to[out=90,in=270] (4.5,-1);
\draw[usual] (-1.5,3) to[out=270,in=90] (1.5,1);
\draw[usual] (7.5,3) to[out=270,in=90] (4.5,1);
\end{tikzpicture}
\quad\text{and}\quad
\pjw[{4,4}]^{4}=
\begin{tikzpicture}[anchorbase,scale=0.25,tinynodes]
\draw[JW] (-2.5,3) rectangle (2.5,5);
\node at (0,3.8) {$\qjwm[4]$};
\draw[JW] (3.5,3) rectangle (8.5,5);
\node at (6,3.8) {$\qjwm[4]$};
\draw[JW] (-2.5,-5) rectangle (2.5,-3);
\node at (0,-4.2) {$\qjwm[4]$};
\draw[JW] (3.5,-5) rectangle (8.5,-3);
\node at (6,-4.2) {$\qjwm[4]$};
\draw[JW] (0.5,-1) rectangle (5.5,1);
\node at (3,-0.2) {$\pjwm[4]$};
\draw[usual] (2,-3) to[out=90,in=180] (3,-2)node[above,xshift=0.15cm,yshift=-0.12cm]{$2$} to[out=0,in=90] (4,-3);
\draw[usual] (2,3) to[out=270,in=180] (3,2)node[below,xshift=0.15cm,yshift=0.00cm]{$2$} to[out=0,in=270] (4,3);
\draw[usual] (-1.5,-3) to[out=90,in=270] (1.5,-1);
\draw[usual] (7.5,-3) to[out=90,in=270] (4.5,-1);
\draw[usual] (-1.5,3) to[out=270,in=90] (1.5,1);
\draw[usual] (7.5,3) to[out=270,in=90] (4.5,1);
\end{tikzpicture}
.
\end{gather*}

\item It follows from \fullref{example:thetawebs} and \fullref{theorem:eve-fusion} that the idempotent
endomorphism of the summand $\tmod(6)$ in $\tmod(4)\hcirc\tmod(4)$ is
\begin{gather*}
\pjw[{4,4}]^{6}=3\cdot
\begin{tikzpicture}[anchorbase,scale=0.25,tinynodes]
\draw[JW] (-2.5,3) rectangle (2.5,5);
\node at (0,3.8) {$\qjwm[4]$};
\draw[JW] (3.5,3) rectangle (8.5,5);
\node at (6,3.8) {$\qjwm[4]$};
\draw[JW] (-2.5,-5) rectangle (2.5,-3);
\node at (0,-4.2) {$\qjwm[4]$};
\draw[JW] (3.5,-5) rectangle (8.5,-3);
\node at (6,-4.2) {$\qjwm[4]$};
\draw[pJW] (0.5,-1) rectangle (5.5,1);
\node at (3,-0.2) {$\pjwm[6]$};
\draw[usual] (2,-3) to[out=90,in=180] (3,-2)node[above,xshift=0.15cm,yshift=-0.12cm]{$1$} to[out=0,in=90] (4,-3);
\draw[usual] (2,3) to[out=270,in=180] (3,2)node[below,xshift=0.15cm,yshift=0.00cm]{$1$} to[out=0,in=270] (4,3);
\draw[usual] (-1.5,-3) to[out=90,in=270] (1.5,-1);
\draw[usual] (7.5,-3) to[out=90,in=270] (4.5,-1);
\draw[usual] (-1.5,3) to[out=270,in=90] (1.5,1);
\draw[usual] (7.5,3) to[out=270,in=90] (4.5,1);
\end{tikzpicture}
+2\cdot
\begin{tikzpicture}[anchorbase,scale=0.25,tinynodes]
\draw[JW] (-2.5,3) rectangle (2.5,5);
\node at (0,3.8) {$\qjwm[4]$};
\draw[JW] (3.5,3) rectangle (8.5,5);
\node at (6,3.8) {$\qjwm[4]$};
\draw[JW] (-2.5,-5) rectangle (2.5,-3);
\node at (0,-4.2) {$\qjwm[4]$};
\draw[JW] (3.5,-5) rectangle (8.5,-3);
\node at (6,-4.2) {$\qjwm[4]$};
\draw[pJW] (0.5,-1) rectangle (5.5,1);
\node at (3,-0.2) {$\pjwm[2]$};
\draw[usual] (2,-3) to[out=90,in=180] (3,-2)node[above,xshift=0.15cm,yshift=-0.12cm]{$3$} to[out=0,in=90] (4,-3);
\draw[usual] (2,3) to[out=270,in=180] (3,2)node[below,xshift=0.15cm,yshift=0.00cm]{$3$} to[out=0,in=270] (4,3);
\draw[usual] (-1.5,-3) to[out=90,in=270] (1.5,-1);
\draw[usual] (7.5,-3) to[out=90,in=270] (4.5,-1);
\draw[usual] (-1.5,3) to[out=270,in=90] (1.5,1);
\draw[usual] (7.5,3) to[out=270,in=90] (4.5,1);
\end{tikzpicture}.
\end{gather*}

\end{enumerate}
\end{example}

\section{Applications}\label{section:app}

We now derive some consequences for the monoidal 
structure of $\tilt[{\kk,\qpar}]$. These results are 
generalizations to mixed characteristic of well-known 
results.

\subsection{Tensor ideals, cells and Verlinde quotients}

A (two-sided) \emph{$\hcirc$-ideal} $\catstuff{I}$ in a monoidal category is a
collection of morphisms that is closed under composition and tensoring with
arbitrary morphisms. Recall also that a \emph{thick $\hcirc$-ideal}
$\catstuff{J}$ in a monoidal category is a $\hcirc$-ideal that is generated by
the identity morphisms on a set of objects. 
In general, not every $\hcirc$-ideal
is thick, so it is remarkable that these notations coincide for
$\tilt[{\kk,\qpar}]$, as we shall now see. Note that in part (d) we use again
that $\pqjw[v{-}1]$ are secretly defined over $\Qv$ to take the
$\ppar\lpar$-adic valuation.

In the following theorem, we will use $\{\placeholder\}^{\oplus}$ 
as a notation for the additive closure of a given set of objects, 
meaning the full subcategory additively (via taking direct sums) 
generated by the specified objects.

\begin{theorem}\label{theorem:cells} 
Let $\catstuff{J}_{v{-}1}$ be the thick
$\hcirc$-ideal in $\tilt[{\kk,\qpar}]$ that is monoidally generated by $\tmod(v-1)$,
{\ie} the thick ideal corresponding to set of objects containing the direct
summands of $\tmod(v-1)\hcirc\tmod(w-1)$ for any $w\in\N$.
\begin{enumerate}

\item For any $k\in\N[0]$, we have
$\catstuff{J}_{\ppar^{(k)}{-}1}=\{\tmod(v-1)\mid v\geq\ppar^{(k)}\}^{\oplus}$.

\item For any 
$\hcirc$-ideal $\catstuff{I}\neq 0$ 
(not necessarily thick), there exists $k\in\N[0]$ such that 
$\catstuff{I}=\catstuff{J}_{\ppar^{(k)}{-}1}$.

\item If $\ppar\neq 2$ and $\lpar\neq 2$, then
\begin{gather*}
\catstuff{J}_{\ppar^{(k)}{-}1}
\stackrel{\ppar{\neq}2,\lpar{\neq}2}{=}
\big\{\tmod(v-1)\mid
\ord\big(\dim_{\tilt[{\kkv,\vpar}]}\big(\overline{\tmod}(v-1)\big)\big)
\leq k\big\}^{\oplus},
\end{gather*}
where $\overline{\tmod}(v-1)$ is the image of $\pqjw[v{-}1]$ in $\tilt[{\kkv,\vpar}]$.

\end{enumerate}
\end{theorem}

In particular, by \fullref{theorem:cells} 
the $\hcirc$-ideals are the strongly connected components 
of the fusion graph of $\tmod(1)$, {\cf} \fullref{figure:fusion-graph}.

\begin{proof}
For part (a), we will use the fact that $\tmod(1)$ monoidally generates $\tilt[{\kk,\qpar}]$, 
see {\eg} \fullref{proposition:eve-times-eve}.  
We argue inductively that the summands of $\tmod(\ppar^{(k)}-1)\hcirc\tmod(1)^{\hcirc d}$ 
belong to $\{\tmod(w-1)\mid w\geq\ppar^{(k)}\}^{\oplus}$. 
Indeed, for $d=0$ there is 
nothing to show, and the case with $d>0$ then follows inductively from \fullref{proposition:times-t1}. 
It is also clear from \fullref{proposition:times-t1} that 
every element of $\{\tmod(w-1)\mid w\geq v\}^{\oplus}$ will appear for some $d$ big enough.

Next we prove part (b).
Given some $\hcirc$-ideal $\catstuff{I}\neq 0$, take any non-zero
$\morstuff{g}\in \catstuff{I}\cap
\Hom_{\tilt[{\kk,\qpar}]}\big(\tmod(w-1),\tmod(x-1)\big)$, which exists for
suitable $w,x\in\N$. By tensor-hom adjunction, there is an associated ``bent''
morphism $\morstuff{g}_{b}\in\catstuff{I}\cap
\Hom_{\tilt[{\kk,\qpar}]}\big(\munit,\tmod(x-1)\hcirc\tmod(w-1)\big)$, which is
again non-zero. Considering the direct summands of $\tmod(w-1)\hcirc\tmod(x-1)$,
we deduce that there exists a non-zero $\morstuff{f}\in\catstuff{I}\cap
\Hom_{\tilt[{\kk,\qpar}]}\big(\munit,\tmod(v-1)\big)$ and we assume that
$v\in\N$ has been chosen to be minimal with this property. Thus, $\tmod(v-1)$
has $\munit=\wmod(0)$ as a Weyl factor, which implies $v=2\ppar^{(k)}-1$ for some
$k$ by \fullref{proposition:multiplicities}. 
Composing $f$ with the inclusion
$\tmod(2\ppar^{(k)}-2)\hookrightarrow\tmod(\ppar^{(k)}-1)\hcirc\tmod(\ppar^{(k)}-1)$,
we obtain a non-zero morphism in $\catstuff{I}\cap
\Hom_{\tilt[{\kk,\qpar}]}\big(\munit,\tmod(\ppar^{(k)}-1)\hcirc\tmod(\ppar^{(k)}-1)\big)$.
Reversing the bending trick produces a non-zero endomorphism of
$\tmod(\ppar^{(k)}-1)$ that is in $\catstuff{I}$. As $\tmod(\ppar^{(k)}-1)$ is
simple, we conclude $\idmor_{\tmod(\ppar^{(k)}-1)}\in\catstuff{I}$, so
$\catstuff{I}$ contains at least $\catstuff{J}_{\ppar^{(k)}{-}1}$. By minimality
of $k$, we get $\catstuff{I}=\catstuff{J}_{\ppar^{(k)}{-}1}$.

The final statement is a consequence of 
\fullref{proposition:qdim}, using \fullref{proposition:qlucas}.
\end{proof}

A \emph{cell} in $\tilt[{\kk,\qpar}]$ is defined as an equivalence class of
indecomposable objects with respect to the preorder
\begin{gather*}
\tmod(v-1)\leq_{J}\tmod(w-1)
:\Leftrightarrow
\exists x\in\N:\tmod(w-1)\text{ is a direct summand of }
\tmod(v-1)\hcirc\tmod(x-1).
\end{gather*}

\begin{corollary}
The cells of $\tilt[{\kk,\qpar}]$ are of the form 
\[
\catstuff{J}_{k}=\{\tmod(v-1)\mid\ppar^{(k)}
\leq v<\ppar^{(k+1)}\}\quad \text{for}\;k\in\N[0].
\]
Thus, $\tilt[{\kk,\qpar}]$ has infinitely many finite cells if and only if $\ppar<\infty$ and $\lpar<\infty$.
\end{corollary}

The quotients
$\tilt[{\kk,\qpar}]_{\ppar^{(k)}}=\tilt[{\kk,\qpar}]/\catstuff{J}_{\ppar^{(k+1)}{-}1}$
are $\kk$-linear, additive, idempotent closed, Krull--Schmidt, pivotal
categories with finitely many indecomposable objects, namely $\tmod(v-1)$ for
$v\in\{1,\dots,\ppar^{(k+1)}-1\}$, and finite-dimensional hom-spaces (these
properties are sometimes called \emph{fiat}). 

To state the
following lemma we say that an object $\obstuff{X}\in\catstuff{C}$ in a category
$\catstuff{C}$ with the properties listed above is called \emph{split} if, for
any morphism $\morstuff{f}$, the morphisms
$\idmor_{\obstuff{X}}\hcirc\morstuff{f}$ and
$\morstuff{f}\hcirc\idmor_{\obstuff{X}}$ are \emph{halves-of-idempotents},
{\ie} that there exists $\morstuff{g}$ such that
$\morstuff{g}(\idmor_{\obstuff{X}}\hcirc\morstuff{f})$ and
$(\idmor_{\obstuff{X}}\hcirc\morstuff{f})\morstuff{g}$
are idempotents, and similarly 
for $\morstuff{f}\hcirc\idmor_{\obstuff{X}}$ 
(with a potentially different $\morstuff{g}$). We say
$\catstuff{C}$ if \emph{separated}, if it has enough splitting objects. (The
precise definitions of splitting and separated can be found in \cite[Section
2]{BeEtOs-pverlinde}.)

In the following, we work over $\K$ and use the existence of the category
$\allmod$ and the simple modules $\lmod(v-1)$, for which we do not have a
diagrammatic interpretation.

\begin{lemma}
$\tilt[{\K,\qpar}]_{\ppar^{(k)}}$ is separated and 
the cell $\catstuff{J}_{k}$ coincides with its ideal of splitting objects.
\end{lemma}

\begin{proof}
Using \fullref{theorem:cells} as well 
as \fullref{proposition:donkin-etc}, the arguments 
are {\muta} as in \cite[Section 3.4]{BeEtOs-pverlinde}. 
In a bit more detail, using the same arguments as in \cite[Lemma 3.3]{BeEtOs-pverlinde} \fullref{proposition:donkin-etc} implies that 
$\tmod(\ppar^{(k-1)}-1)\hcirc\lmod(w-1)$ for $w<\ppar^{(k)}-1$.
This in turn implies, 
similarly as in \cite[Proposition 3.2]{BeEtOs-pverlinde}, that 
$\tmod(\ppar^{(k-1)}-1)$ is a splitting object in 
$\tilt[{\K,\qpar}]_{\ppar^{(k)}}$. Then the proof is completed following the classification 
of $\hcirc$-ideals in \fullref{theorem:cells}.
\end{proof}

We can thus define abelianizations of $\tilt[{\K,\qpar}]_{\ppar^{(k)}}$ 
in the sense of \cite[Section 2]{BeEtOs-pverlinde}, which we denote by
$\verlinde_{\ppar^{(k)}}$ and which
could be called \emph{mixed Verlinde categories}.

\begin{proposition}\label{proposition:cells}
Let $k\in\N[0]$.

\begin{enumerate}

\item The category $\tilt[{\kk,\qpar}]_{\ppar^{(k)}}$ has cells given by the
images of $\catstuff{J}_{j}$ for $j\in\{0,\dots,k\}$, with $\catstuff{J}_{k}$
being the cell of projective objects. Thus, $\tilt[{\kk,\qpar}]_{\ppar^{(k)}}$
has $\ppar^{(k)}-\ppar^{(k-1)}$ indecomposable projectives, namely the images of
$\tmod{v-1}$ for $\ppar^{(k)} \leq v<\ppar^{(k+1)}$.

\item The Cartan matrix of $\verlinde_{\ppar^{(k)}}$ is a
$\ppar^{(k)}-\ppar^{(k-1)}$-square matrix with entries given by the common Weyl
factors of $\tmod(v-1)$ and $\tmod(w-1)$ with $\ppar^{(k)} \leq
v,w<\ppar^{(k+1)}$ (which are in turn given by
\fullref{proposition:multiplicities}).

\end{enumerate}

\end{proposition}

\begin{proof}
Part (a) is clear by the above as the projective objects 
always form the maximal cell, whilst part (b)
follows {\muta} as in \cite[Section 4]{BeEtOs-pverlinde}.
\end{proof}

The Cartan matrix of $\verlinde_{\ppar^{(k)}}$ thus has a fractal pattern, see \fullref{figure:cartan}.
Being careful with the distinction of $\ppar$ and $\lpar$ on the zeroth digit,
we leave it to the reader to generalize other results from \cite[Section 4]{BeEtOs-pverlinde} 
to the mixed Verlinde categories.

\begin{remark}
For the quantum group case $\mchar=(\infty,\lpar)$, Ostrik
\cite{Os-tensor-ideals-tilting} showed that the right cells in the affine Weyl
group are in bijection with cells in $\tilt[{\kk,\qpar}]$, which in turn are in
bijection with the thick $\hcirc$-ideals in $\tilt[{\kk,\qpar}]$. The
Riche--Williamson 
conjecture \cite{RiWi-tilting-p-canonical} implies the same in
characteristic $\ppar$ for right $\ppar$-cells. The above
discussion can be seen as a mixed characteristic version of these. We also like
to mention a statement analogous to \fullref{theorem:cells}.(c) proven in
\cite{Os-qdim-tiltings} for $\mchar=(\infty,\lpar)$.
\end{remark}

\subsection{Modified traces diagrammatically}\label{subscetion:mdim}

Recall from {\eg}
\cite{GeKuPaMi-gen-traces-modified-dimensions} 
and \cite{GeKuPaMi-ambidextrous-objects} that a (right) \emph{modified trace}
(short: \emph{mtrace}) $\mTr{\catstuff{I}}{\placeholder}$ on a $\hcirc$-ideal
$\catstuff{I}$ in a $\kk$-linear, pivotal category $\catstuff{C}$ is a
collection of $\kk$-linear functions $\{\mTr{\catstuff{I}}{\obstuff{T}}
\colon\End_{\catstuff{C}}(\obstuff{T})\to\kk\}$ satisfying \emph{cyclicity on
$\catstuff{I}$}, {\ie}
\begin{gather*}
\mTr{\catstuff{I}}{\obstuff{T}}(\morstuff{gf})=
\mTr{\catstuff{I}}{\obstuff{T}^{\prime}}(\morstuff{fg}) \quad \text{for } f\colon \obstuff{T}\to \obstuff{T}^{\prime}, g\colon \obstuff{T}^{\prime}\to \obstuff{T}
\end{gather*}
and the \emph{mtrace property on $\catstuff{I}$}. For
$\morstuff{f}\in\End_{\catstuff{C}}(\obstuff{T}\hcirc\obstuff{X})$, we let
$\MTr{r}_{\obstuff{X}}(\morstuff{f}) \in \End_{\catstuff{C}}(\obstuff{T})$
denote the partial right trace $\obstuff{X}$, determined by the pivotal
structure on $\catstuff{C}$. The mtrace property then requires that for every
object $\obstuff{T}\in\catstuff{I}$ and every object
$\obstuff{X}\in\catstuff{C}$ and any
$\morstuff{f}\in\End_{\catstuff{C}}(\obstuff{T}\hcirc\obstuff{X})$, we have
\begin{gather*}
\mTr{\catstuff{I}}{\obstuff{T}\hcirc\obstuff{X}}(\morstuff{f})
=
\mTr{\catstuff{I}}{\obstuff{T}}\big(\MTr{r}_{\obstuff{X}}(\morstuff{f})\big)
.
\end{gather*}
Below we will omit the subscripts if no confusion can arise.

Following ideas from {\eg} \cite{HeWe-mtrace}, we define a (right) mtrace
$\MTr{(k)}$ for the $\hcirc$-ideals $\catstuff{J}_{\ppar^{(k)}{-}1}$
from \fullref{theorem:cells} as follows.
We write a colored box for 
$\pjw[v{-}1]\morstuff{f}\pjw[v{-}1]$, which is a morphism 
in $\End_{\TL[{\kk,\qpar}]}(\obstuff{T}(v-1))$. 
On the indecomposable objects $\obstuff{T}(v-1)$ 
in $\catstuff{J}_{\ppar^{(k)}{-}1}$ 
and $\morstuff{f}\in\End_{\TL[{\kk,\qpar}]}
\big(\obstuff{T}(v-1)\big)$, we define
$\MTr{(k)}(\morstuff{f})$ using absorption, by
\begin{gather}\label{eq:mtrace-def}
\begin{tikzpicture}[anchorbase,scale=0.25,tinynodes]
\draw[JW] (0,-1) rectangle (5,1);
\node at (2.5,-0.2) {$\pjwm[\ppar^{(k)}{-}1]$};
\draw[pJW] (0,1) rectangle (7,3);
\node at (3.5,1.8) {$\morstuff{f}$};
\draw[JW] (0,3) rectangle (5,5);
\node at (2.5,3.8) {$\pjwm[\ppar^{(k)}{-}1]$};
\draw[usual] (6,1) to[out=270,in=180] (7,0)
to[out=0,in=270] (8,1) to[out=90,in=270] (8,3) to[out=90,in=0] (7,4)
to[out=180,in=90] (6,3);
\end{tikzpicture}
=
\mTr{(k)}{}(\morstuff{f})
\cdot
\begin{tikzpicture}[anchorbase,scale=0.25,tinynodes]
\draw[JW] (0.5,-1) rectangle (5.5,1);
\node at (3,-0.2) {$\pjwm[\ppar^{(k)}{-}1]$};
\end{tikzpicture}
.
\end{gather}
(Note that $\End_{\TL[{\kk,\qpar}]}
\big(\obstuff{T}(\ppar^{(k)}-1)\big)\cong\kk$.) 
We call this \emph{tracing down to the eve} and
extend this assignment additively to the whole of 
$\catstuff{J}_{\ppar^{(k)}{-}1}$.

\begin{lemma}\label{lemma:mtrace}
The family of $\kk$-linear functions $\MTr{(k)}$ is a 
non-trivial mtrace on $\catstuff{J}_{\ppar^{(k)}{-}1}$.
\end{lemma}

\begin{proof}
Comparing the equalities \eqref{eq:mtrace-def} and
\begin{gather*}
\MTr{(k)}_{\obstuff{T}\hcirc\obstuff{X}}(\morstuff{f})
=
\begin{tikzpicture}[anchorbase,scale=0.25,tinynodes]
\draw[pJW] (0,-1) rectangle (5,1);
\node at (2.5,-0.2) {$\obstuff{T}$};
\draw[pJW] (6,-1) rectangle (11,1);
\node at (8.5,-0.2) {$\obstuff{X}$};
\draw[pJW] (0,1) rectangle (11,3);
\node at (5.5,1.8) {$\morstuff{f}$};
\draw[pJW] (0,3) rectangle (5,5);
\node at (2.5,3.8) {$\obstuff{T}$};
\draw[pJW] (6,3) rectangle (11,5);
\node at (8.5,3.8) {$\obstuff{X}$};
\draw[usual] (10,-1) to[out=270,in=180] (11,-2)
to[out=0,in=270] (12,-1) to[out=90,in=270] (12,5) to[out=90,in=0] (11,6)
to[out=180,in=90] (10,5);
\draw[usual] (4,-1) to[out=270,in=180] (5.5,-2.5) to (11.5,-2.5)
to[out=0,in=270] (13,-1) to[out=90,in=270] (13,5) to[out=90,in=0] (11.5,6.5) to (5.5,6.5) to[out=180,in=90] (4,5);
\end{tikzpicture}
=
\MTr{(k)}_{\obstuff{T}}
\left(
\begin{tikzpicture}[anchorbase,scale=0.25,tinynodes]
\draw[pJW] (0,-1) rectangle (5,1);
\node at (2.5,-0.2) {$\obstuff{T}$};
\draw[pJW] (6,-1) rectangle (11,1);
\node at (8.5,-0.2) {$\obstuff{X}$};
\draw[pJW] (0,1) rectangle (11,3);
\node at (5.5,1.8) {$\morstuff{f}$};
\draw[pJW] (0,3) rectangle (5,5);
\node at (2.5,3.8) {$\obstuff{T}$};
\draw[pJW] (6,3) rectangle (11,5);
\node at (8.5,3.8) {$\obstuff{X}$};
\draw[usual] (10,-1) to[out=270,in=180] (11,-2)
to[out=0,in=270] (12,-1) to[out=90,in=270] (12,5) to[out=90,in=0] (11,6)
to[out=180,in=90] (10,5);
\end{tikzpicture}
\right)
=
\MTr{(k)}_{\obstuff{T}}(\MTr{r}_{\obstuff{X}}(\morstuff{f}))
,
\end{gather*}
we see that 
the mtrace property follows by construction (and absorption).
Note also that
we have $\MTr{(k)}(\pjw[\ppar^{(k)}-1])=1$, so
$\MTr{(k)}$ is non-trivial. Let $\tilde{\MTr{}}$ be the usual 
trace in $\tilt[{\kkv,\vpar}]$. Then, by absorption we have
\begin{gather}\label{eq:mtrace}
\MtTr{(k)}(\morstuff{f})
\cdot
\begin{tikzpicture}[anchorbase,scale=0.25,tinynodes]
\draw[JW] (0.5,-1) rectangle (5.5,1);
\node at (3,-0.2) {$\pjwm[\ppar^{(k)}{-}1]$};
\draw[usual] (4,-1) to[out=270,in=180] (5,-2)
to[out=0,in=270] (6,-1) to[out=90,in=270] (6,1) to[out=90,in=0] (5,2)
to[out=180,in=90] (4,1);
\end{tikzpicture}
=
\begin{tikzpicture}[anchorbase,scale=0.25,tinynodes]
\draw[JW] (0,-1) rectangle (5,1);
\node at (2.5,-0.2) {$\pjwm[\ppar^{(k)}{-}1]$};
\draw[pQJW] (0,1) rectangle (7,3);
\node at (3.5,1.8) {$\morstuff{f}$};
\draw[JW] (0,3) rectangle (5,5);
\node at (2.5,3.8) {$\pjwm[\ppar^{(k)}{-}1]$};
\draw[usual] (6,1) to[out=270,in=180] (7,0)
to[out=0,in=270] (8,1) to[out=90,in=270] (8,3) to[out=90,in=0] (7,4)
to[out=180,in=90] (6,3);
\draw[usual] (4,-1) to[out=270,in=180] (7,-2.5)
to[out=0,in=270] (10,-1) to[out=90,in=270] (10,5) to[out=90,in=0] (7,6.5)
to[out=180,in=90] (4,5);
\end{tikzpicture}
=
\begin{tikzpicture}[anchorbase,scale=0.25,tinynodes]
\draw[pQJW] (0.5,-1) rectangle (5.5,1);
\node at (3,-0.2) {$\morstuff{f}$};
\draw[usual] (4,-1) to[out=270,in=180] (5,-2)
to[out=0,in=270] (6,-1) to[out=90,in=270] (6,1) to[out=90,in=0] (5,2)
to[out=180,in=90] (4,1);
\end{tikzpicture}
=\tilde{\MTr{}}
(\morstuff{f}).
\end{gather}
Since
$\tilde{\MTr{}}(\pjw[\ppar^{(k)}-1])
=(-1)^{\ppar^{(k)}-1}\qnum{\ppar^{(k)}}{\vpar}$,  we have 
$\MtTr{(k)}(\morstuff{f})=(-1)^{\ppar^{(k)}-1}
\tilde{\MTr{}}(\morstuff{f})/\qnum{\ppar^{(k)}}{\vpar}$ in $\kkv$. 
This implies cyclicity since $\tilde{\MTr{}}$ is cyclic.
\end{proof} 

For $\obstuff{T}(v-1)\in\catstuff{J}_{\ppar^{(k)}{-}1}$ we call
$\kdim{k}{\tilt[{\kk,\qpar}]}\big(\tmod(v-1)\big)=
\MTr{(k)}(\pjw[v-1])$ its \emph{modified dimension}.

\begin{theorem}\label{theorem:mtrace}
Each $\hcirc$-ideal in $\TL[{\kk,\qpar}]$ admits a 
non-trivial mtrace $\MTr{(k)}$ which satisfies
\begin{gather*}
\kdim{k}{\tilt[{\kk,\qpar}]}\big(\tmod(v-1)\big)
=(-1)^{v-1}(-1)^{\ppar^{(k)}-1}
\tfrac{\qnum{\motherr{v}{\infty}}{\qpar}}{\qnum{\ppar^{(k)}}{\vpar}}
{
\prod_{a_{i}\neq 0}}\,\qnum{2}{\qpar^{a_{i}\ppar^{(i)}}},
\end{gather*}
where $\obstuff{T}(v-1)\in\catstuff{J}_{\ppar^{(k)}{-}1}\setminus\catstuff{J}_{\ppar^{(k+1)}{-}1}$ and the product runs over all non-zero and non-leading digits of $v$. Every
such mtrace descends to the corresponding categories $\tilt[{\kk,\qpar}]_{\ppar^{(k)}}$ and
$\verlinde[{\kk,\qpar}]_{\ppar^{(k)}}$.
\end{theorem}

\begin{proof}
By \fullref{theorem:cells}, we know that all $\hcirc$-ideals 
are of the form $\catstuff{J}_{\ppar^{(k)}{-}1}$, while 
\fullref{lemma:mtrace} defines an mtrace on these. 
Moreover, \eqref{eq:mtrace} and \fullref{proposition:qdim} 
imply the claimed formula. The property that these descend to 
$\tilt[{\kk,\qpar}]_{\ppar^{(k)}}$ and $\verlinde[{\kk,\qpar}]_{\ppar^{(k)}}$ 
is clear.
\end{proof}

\subsection{The M{\"u}ger center is often trivial}

We assume in this section that our parameter $\qpar$ 
has a square root.
Recall that we have a braiding on $\TL[{\kk,\qpar}]$, 
using the Kauffman formula \eqref{eq:kauffman}.

Under these assumptions, recall that $\tilt[{\kk,\qpar}]$ is
additive, idempotent closed, Krull--Schmidt, $\K$-linear, 
braided and pivotal. 
Let $\catstuff{C}$ be a category having these 
properties, and let $\munit$ and $\beta$ 
denote the monoidal unit and the braiding of $\catstuff{C}$.

Following \cite{Mu-modular}, we define 
the \emph{M{\"u}ger center} 
$\catstuff{MZ}(\catstuff{C})$ of $\catstuff{C}$ to be 
the full subcategory 
of $\catstuff{C}$ consisting of all objects $\obstuff{X}$ such that
$\beta_{\obstuff{X},\obstuff{Y}}
\beta_{\obstuff{Y},\obstuff{X}}=\idmor_{\obstuff{X},\obstuff{Y}}$ for any $\obstuff{Y}\in\catstuff{C}$.

Clearly $\munit^{\oplus m}\in\catstuff{MZ}(\catstuff{C})$ 
for all $m\in\N$, and we call $\catstuff{MZ}(\catstuff{C})$ 
\emph{trivial} if 
$\munit^{\oplus m}$ are the only objects in 
$\catstuff{MZ}(\catstuff{C})$. In other words, 
$\catstuff{MZ}(\catstuff{C})\simeq\catstuff{Vect}_{\kk}$, the latter
being the category of finite-dimensional $\kk$-vector spaces.

\begin{theorem}\label{theorem:mueger}
The M{\"u}ger center $\catstuff{MZ}(\tilt[\kk,\qpar])$ is non-trivial 
if and only if $\qpar=\pm 1$.
\end{theorem}

\begin{remark}
In the case $\catstuff{C}$ is a ribbon, 
finite tensor category, having a trivial
M{\"u}ger center is equivalent to $\catstuff{C}$ 
being modular in the sense of
Kerler--Lyubashenko -- see \cite[Theorem 1.1]{Sh-non-degeneracy}. 
In particular,
\fullref{theorem:mueger} is a good indication that 
tilting modules in mixed
characteristic may give rise to non-trivial link 
and $3$-manifold invariants.
\end{remark}

In order to prove \fullref{theorem:mueger}, we need two auxiliary lemmas. To this
end, recall that the endomorphism spaces in braided tensor categories have
central elements given by ``encircled identity morphisms'', illustrated below. On simple
objects these act as scalars, but $\tmod(v-1)$ is usually not simple and so we obtain the following.

\begin{lemma}\label{lemma:circle}
For $v=\plbase{a_{j},\dots,a_{0}}$, we have 
in $\tilt[{\kkvs,\vpar}]$ that
\begin{gather}\label{eq:circle}
\begin{tikzpicture}[anchorbase,scale=0.25,tinynodes]
\draw[usual] (3,2.5) to[out=90,in=0] (2,3.25) to[out=180,in=90] (1,2.5);
\draw[usual,crossline] (2,1) to[out=90,in=270] (2,4);
\draw[usual,crossline] (1,2.5) to[out=270,in=180] (2,1.75) to[out=0,in=270] (3,2.5);
\draw[pQJW] (0,0) rectangle (4,1);
\node at (2,0.3) {$\qjwm[v{-}1]$};
\end{tikzpicture}
=
\begin{tikzpicture}[anchorbase,scale=0.25,tinynodes]
\draw[usual] (3,2.5) to[out=90,in=0] (2,3.25) to[out=180,in=90] (1,2.5);
\draw[usual,crossline] (2,1) to[out=90,in=270] (2,4);
\draw[usual,crossline] (1,2.5) to[out=270,in=180] (2,1.75) to[out=0,in=270] (3,2.5);
\draw[pQJW] (0,4) rectangle (4,5);
\node at (2,4.3) {$\qjwm[v{-}1]$};
\end{tikzpicture}
=
\begin{cases}
-\qnum{2}{\vpar^{v}}\cdot
\begin{tikzpicture}[anchorbase,scale=0.25,tinynodes]
\draw[pQJW] (0,-1) rectangle (4,1);
\node at (2,-0.2) {$\qjwm[v{-}1]$};
\end{tikzpicture}
+\text{Rest}
&\text{if }a_{0}=0,
\\
-\qnum{2}{\vpar^{v}}\cdot
\begin{tikzpicture}[anchorbase,scale=0.25,tinynodes]
\draw[pQJW] (0,-1) rectangle (4,1);
\node at (2,-0.2) {$\qjwm[v{-}1]$};
\end{tikzpicture}
+s_{v}(\vpar)
\cdot
\begin{tikzpicture}[anchorbase,scale=0.25,tinynodes]
\draw[pQJW] (0,0) rectangle (3,1);
\node at (1.56,0.3) {\scalebox{0.6}{$\qjwm[v{-}a_{0}{-}1]$}};
\draw[pQJW] (0,3) rectangle (3,4);
\node at (1.56,3.3) {\scalebox{0.6}{$\qjwm[v{-}a_{0}{-}1]$}};
\draw[usual] (1,1) to (1,3);
\draw[usual] (2.5,1) to[out=90,in=180] (3,1.5) to[out=0,in=90] (3.5,1) to (3.5,0)node[below]{$a_{0}$};
\draw[usual] (2.5,3) to[out=270,in=180] (3,2.5) to[out=0,in=270] (3.5,3) to (3.5,4)node[above,yshift=-2pt]{$a_{0}$};
\end{tikzpicture}
+\text{Rest}
&\text{if }a_{0}\neq 0,
\end{cases}
\end{gather}
where Rest are potentially non-zero lower order terms, {\ie} linear combinations
of loops $\loopdowngen{S}{v{-}1}$ for $S\neq\{0\}$. The scalar is
\begin{gather}\label{eq:circle-scalar}
s_{v}(\vpar)=
(-1)^{a_{0}}(\vpar^{v-2a_{0}}-\vpar^{-v+2a_{0}})
(\vpar^{a_{0}}-\vpar^{-a_{0}}).
\end{gather}
Moreover, if $v\in\eve$, then 
there are no lower order terms in \eqref{eq:circle}.
\end{lemma}

In contrast, the scalar resulting from an unlinked
circle is $-\qnum{2}{\vpar}$ times the projector.

\begin{proof}
For $v\in\eve$, in which case $\pqjw[v{-}1]=\qjw[v{-}1]$, it is well-known that
\eqref{eq:circle} holds without lower order terms, see {\eg} \cite[Section
9.8]{KaLi-TL-recoupling}. We will use this throughout the proof. For the
other cases, we calculate that
\begin{gather}\label{eq:circle-action}
\scalebox{0.9}{$\displaystyle
\begin{tikzpicture}[anchorbase,scale=0.25,tinynodes]
\draw[usual] (3,2.5) to[out=90,in=0] (2,3.25) to[out=180,in=90] (1,2.5);
\draw[usual,crossline] (2,1) to[out=90,in=270] (2,4);
\draw[usual,crossline] (1,2.5) to[out=270,in=180] (2,1.75) to[out=0,in=270] (3,2.5);
\draw[pQJW] (0,0) rectangle (4,1);
\node at (2,0.3) {$\qjwm[v{-}1]$};
\end{tikzpicture}
=
\sum_{S\in\supp[v]}
\lambda_{v,S}
\cdot
\begin{tikzpicture}[anchorbase,scale=0.25,tinynodes]
\draw[usual] (3,4.5) to[out=90,in=0] (2,5.25) to[out=180,in=90] (1,4.5);
\draw[usual,crossline] (2,3) to[out=90,in=270] (2,6);
\draw[usual,crossline] (1,4.5) to[out=270,in=180] (2,3.75) to[out=0,in=270] (3,4.5);
\draw[JW] (0,0) rectangle (4,1.5);
\node at (2,0.55) {$\qjwm[{v[S]{-}1}]$};
\trd{5}{1.5}{$\qjwm[S]$}{-5}{0}
\tru{5}{1.5}{$\qjwm[S]$}{-5}{1.5}
\end{tikzpicture}
=
\sum_{S\in\supp[v]}
\lambda_{v,S}
\cdot
\begin{tikzpicture}[anchorbase,scale=0.25,tinynodes]
\draw[usual] (3,3) to[out=90,in=0] (2,3.75) to[out=180,in=90] (1,3);
\draw[usual,crossline] (2,1.5) to[out=90,in=270] (2,4.5);
\draw[usual,crossline] (1,3) to[out=270,in=180] (2,2.25) to[out=0,in=270] (3,3);
\draw[JW] (0,0) rectangle (4,1.5);
\node at (2,0.55) {$\qjwm[{v[S]{-}1}]$};
\trd{5}{1.5}{$\qjwm[S]$}{-5}{0}
\tru{5}{1.5}{$\qjwm[S]$}{-5}{4.5}
\end{tikzpicture}
=
-\sum_{S\in\supp[v]}
\lambda_{v,S}\qnum{2}{\vpar^{v[S]}}
\cdot
\begin{tikzpicture}[anchorbase,scale=0.25,tinynodes]
\draw[JW] (0,0) rectangle (4,1.5);
\node at (2,0.55) {$\qjwm[{v[S]{-}1}]$};
\trd{5}{1.5}{$\qjwm[S]$}{-5}{0}
\tru{5}{1.5}{$\qjwm[S]$}{-5}{1.5}
\end{tikzpicture}$}
.
\end{gather}
The right-hand side in \eqref{eq:circle-action} is 
a linear combination of loops
$\loopdowngen{S}{v{-}1}$. However, these are not 
well-defined over $\kk$, so we
need to rewrite \eqref{eq:circle-action} in terms 
of $\loopdowngen{S}{v{-}1}$ (the
loops that specialize to $\loopdown{S}{v{-}1}$). 
We do not know the full
change-of-basis matrix (see \cite[Lemma 4.8]{TuWe-quiver} 
for generation $2$), 
but we know that this matrix is unitriangular 
by an analog of \cite[Lemma
3.17]{TuWe-quiver}, which is proven {\ver} 
as therein. 
In particular, the case $a_{0}=0$ 
in \eqref{eq:circle}
follows, and also the coefficient of $\pqjw[v{-}1]$ 
is as claimed. To prove the case $a_{0}\neq 0$ in \eqref{eq:circle}, 
we expand
\begin{gather*}
\begin{tikzpicture}[anchorbase,scale=0.25,tinynodes]
\draw[JW] (0,-1) rectangle (4,1);
\node at (2,-0.2) {$\qjwm[v{-}1]$};
\end{tikzpicture}
=
\begin{tikzpicture}[anchorbase,scale=0.25,tinynodes]
\draw[pQJW] (0,-1) rectangle (4,1);
\node at (2,-0.2) {$\qjwm[v{-}1]$};
\end{tikzpicture}
-(-1)^{a_{0}}\tfrac{\qnum{v-2a_{0}}{\vpar}}{\qnum{v-a_{0}}{\vpar}}
\cdot
\begin{tikzpicture}[anchorbase,scale=0.25,tinynodes]
\draw[pQJW] (0,0) rectangle (3,1);
\node at (1.56,0.3) {\scalebox{0.6}{$\qjwm[v{-}a_{0}{-}1]$}};
\draw[pQJW] (0,3) rectangle (3,4);
\node at (1.56,3.3) {\scalebox{0.6}{$\qjwm[v{-}a_{0}{-}1]$}};
\draw[usual] (1,1) to (1,3);
\draw[usual] (2.5,1) to[out=90,in=180] (3,1.5) to[out=0,in=90] (3.5,1) to (3.5,0)node[below]{$a_{0}$};
\draw[usual] (2.5,3) to[out=270,in=180] (3,2.5) to[out=0,in=270] (3.5,3) to (3.5,4)node[above,yshift=-2pt]{$a_{0}$};
\end{tikzpicture}
+\text{Rest}
,\quad
\begin{tikzpicture}[anchorbase,scale=0.25,tinynodes]
\draw[JW] (0,0) rectangle (3,1);
\node at (1.5,0.3) {\scalebox{0.6}{$\qjwm[v{-}a_{0}{-}1]$}};
\draw[JW] (0,3) rectangle (3,4);
\node at (1.5,3.3) {\scalebox{0.6}{$\qjwm[v{-}a_{0}{-}1]$}};
\draw[usual] (1,1) to (1,3);
\draw[usual] (2.5,1) to[out=90,in=180] (3,1.5) to[out=0,in=90] (3.5,1) to (3.5,0)node[below]{$a_{0}$};
\draw[usual] (2.5,3) to[out=270,in=180] (3,2.5) to[out=0,in=270] (3.5,3) to (3.5,4)node[above,yshift=-2pt]{$a_{0}$};
\end{tikzpicture}
=
\begin{tikzpicture}[anchorbase,scale=0.25,tinynodes]
\draw[pQJW] (0,0) rectangle (3,1);
\node at (1.56,0.3) {\scalebox{0.6}{$\qjwm[v{-}a_{0}{-}1]$}};
\draw[pQJW] (0,3) rectangle (3,4);
\node at (1.56,3.3) {\scalebox{0.6}{$\qjwm[v{-}a_{0}{-}1]$}};
\draw[usual] (1,1) to (1,3);
\draw[usual] (2.5,1) to[out=90,in=180] (3,1.5) to[out=0,in=90] (3.5,1) to (3.5,0)node[below]{$a_{0}$};
\draw[usual] (2.5,3) to[out=270,in=180] (3,2.5) to[out=0,in=270] (3.5,3) to (3.5,4)node[above,yshift=-2pt]{$a_{0}$};
\end{tikzpicture}
+\text{Rest}.
\end{gather*}
By unitriangularity of the basis change, these two terms are the only ones among
the $\loopdowngen{S}{v{-}1}$ that contribute to the coefficients of
$\pqjw[v{-}1]=\loopdowngen{\emptyset}{v{-}1}$ and $\loopdowngen{\{0\}}{v{-}1}$ in \eqref{eq:circle}.

Collecting terms, we compute that the coefficient of $\loopdowngen{\{0\}}{v{-}1}$ is the one displayed in
\eqref{eq:circle-scalar}:
\begin{gather*}
s_{v}(\vpar)=
-(-1)^{a_{0}}\tfrac{\qnum{v-2a_{0}}{\vpar}}{\qnum{v-a_{0}}{\vpar}}
\big(
-\qnum{2}{\vpar^{v}}+\qnum{2}{\vpar^{v-2a_{0}}}
\big)
=
(-1)^{a_{0}}(\vpar^{v-2a_{0}}-\vpar^{-v+2a_{0}})
(\vpar^{a_{0}}-\vpar^{-a_{0}})
.
\end{gather*}
This completes the proof.
\end{proof}

\begin{remark}
In the previous proof, one may be tempted to argue that
$\qnum{2}{\qpar^{v[S]}}=\qnum{2}{\qpar^{v}}$ holds in $\kk$, and so one should
be able to factor it out from the sum in \eqref{eq:circle-action}. This is not
allowed, since the individual summands are not well-defined in
$\TL[{\kk,\qpar}]$. Indeed, this argument would predict that no lower order
terms appear, not even $\loopdowngen{\{0\}}{v{-}1}$ in \eqref{eq:circle}, which is certainly wrong.
\end{remark}

\begin{lemma}\label{lemma:scalar}
For $a_{0}\neq 0$ and $\qpar\neq\pm 1$, the scalar 
$s_{v}(\vpar)$ from 
\eqref{eq:circle-scalar} specializes to a non-zero scalar $s_{v}(\qpar)\in\kk$.
\end{lemma}

\begin{proof}
First, 
$\qpar^{a_{0}}-\qpar^{-a_{0}}=\qnum{a_{0}}{\qpar}(\qpar-\qpar^{-1})$ is non-zero since $a_{0}<\lpar$. Further note that 
$(\qpar^{v-2a_{0}}-\qpar^{-v+2a_{0}})=
\pm(\qpar^{a_{0}}-\qpar^{-a_{0}})$, so the second 
factor is non-zero by the same observation.
\end{proof}

\begin{proof}[Proof of \fullref{theorem:mueger}]
For $\qpar=\pm 1$, the Kauffman formula \eqref{eq:kauffman} shows that we have
$\catstuff{MZ}(\tilt[\kk,\qpar])=\tilt[\kk,\qpar]$, so let us focus on the case
where $\qpar\neq\pm 1$ and show that $\catstuff{MZ}(\tilt[\kk,\qpar])$ is
trivial.

To this end, it suffices to check that no indecomposable $\tmod(v-1)$ besides
$\tmod(0)=\munit$ is M{\"u}ger central. (Note that the braiding for direct sums
is defined componentwise, so objects of $\catstuff{MZ}(\tilt[\kk,\qpar])$ are
direct sums of indecomposable objects in $\catstuff{MZ}(\tilt[\kk,\qpar])$.)

Suppose, conversely, that for $v\in\N[>1]$ we have $\tmod(v-1)\in
\catstuff{MZ}(\tilt[\kk,\qpar])$. Then $\tmod(v-1)$ braids trivially with
$\tmod(1)$ and we get:
\begin{gather}\label{eq:mueger}
\begin{tikzpicture}[anchorbase,scale=0.25,tinynodes]
\draw[usual] (3,2.5) to[out=90,in=0] (2,3.25) to[out=180,in=90] (1,2.5);
\draw[usual,crossline] (2,1) to[out=90,in=270] (2,4);
\draw[usual,crossline] (1,2.5) to[out=270,in=180] (2,1.75) to[out=0,in=270] (3,2.5);
\draw[pJW] (0,0) rectangle (4,1);
\node at (2,0.3) {$\qjwm[v{-}1]$};
\end{tikzpicture}
=
\begin{tikzpicture}[anchorbase,scale=0.25,tinynodes]
\draw[usual] (5,2.5) to[out=90,in=0] (4,3.25) to[out=180,in=90] (3,2.5);
\draw[usual,crossline] (2,1) to[out=90,in=270] (2,4);
\draw[usual,crossline] (3,2.5) to[out=270,in=180] (4,1.75) to[out=0,in=270] (5,2.5);
\draw[pJW] (0,0) rectangle (4,1);
\node at (2,0.3) {$\qjwm[v{-}1]$};
\end{tikzpicture}
=
-\qnum{2}{\qpar}\cdot
\begin{tikzpicture}[anchorbase,scale=0.25,tinynodes]
\draw[usual,crossline] (2,1) to[out=90,in=270] (2,4);
\draw[pJW] (0,0) rectangle (4,1);
\node at (2,0.3) {$\qjwm[v{-}1]$};
\end{tikzpicture}.
\end{gather}
However, \eqref{eq:circle} contradicts \eqref{eq:mueger}: for $a_{0}\neq 0$ 
because the scalar $s_{v}(\qpar)$ is non-zero by \fullref{lemma:scalar}, and for $a_{0}=0$ because
$-\qnum{2}{\qpar^{v}}=\pm 2$ when $\qpar\neq\pm 1$. (To see the latter, note that
$x^{2}\mp 2x+1=0$ has only the solutions $x=\pm 1$.)
\end{proof}

\section{Some additional comments and details}

\ochanged{In this section we collect a bit more details 
than in the main text, 
for example, various fractal-like patterns 
arising from the characteristic $\ppar$ and the strictly mixed cases, 
which may be interesting in their own right.}

\ochanged{The following figures illustrate
the formulas in \fullref{proposition:multiplicities}. Precisely, 
see \fullref{figure:tiltcartan} for an illustration of $\wmult{v}{w}$
(compare to \cite[Figure 1]{JeWi-p-canonical}), 
and \fullref{figure:weylcartan} for an illustration of $\lmult{v}{w}$. 
Note that comparing these four illustrations 
shows why reciprocity only holds up to a certain point.}

\begin{figure}[ht]
\includegraphics[width=0.49\textwidth]{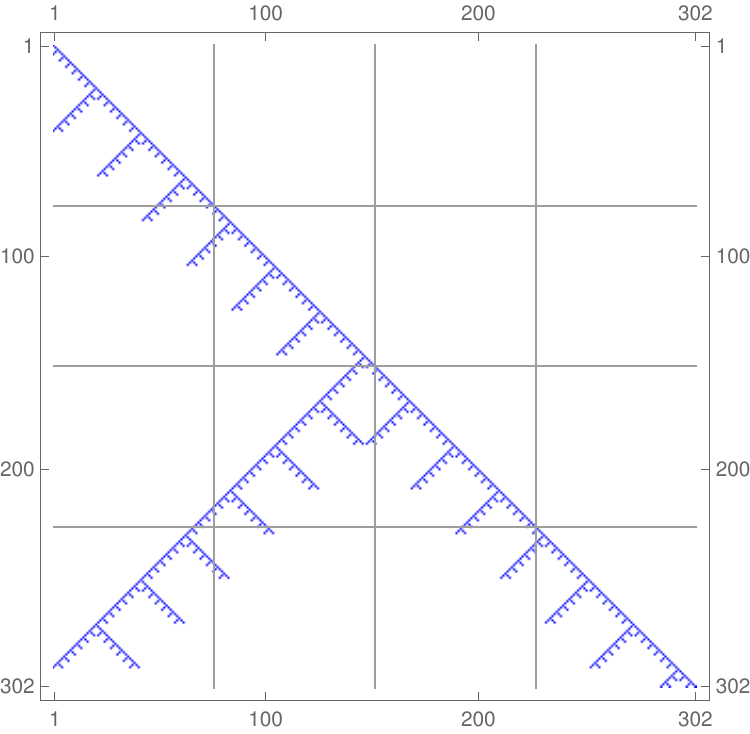}
\includegraphics[width=0.49\textwidth]{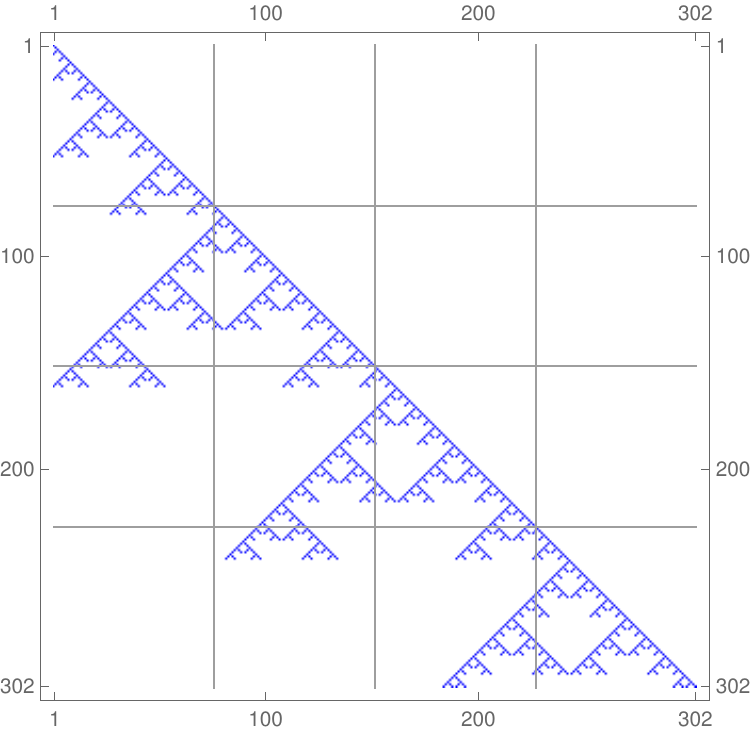}
\caption{The tilting-Cartan matrix, counting $\wmult{v}{w}$, with $v$ increasing along rows. 
The first illustration shows mixed characteristic $\mchar=(7,3)$. 
Note that the smallest visible gaps are of size $3$, the others of size $7$. 
The second illustration shows characteristic $\ppar=3$.}
\label{figure:tiltcartan}
\end{figure}

\begin{figure}[ht]
\includegraphics[width=0.49\textwidth]{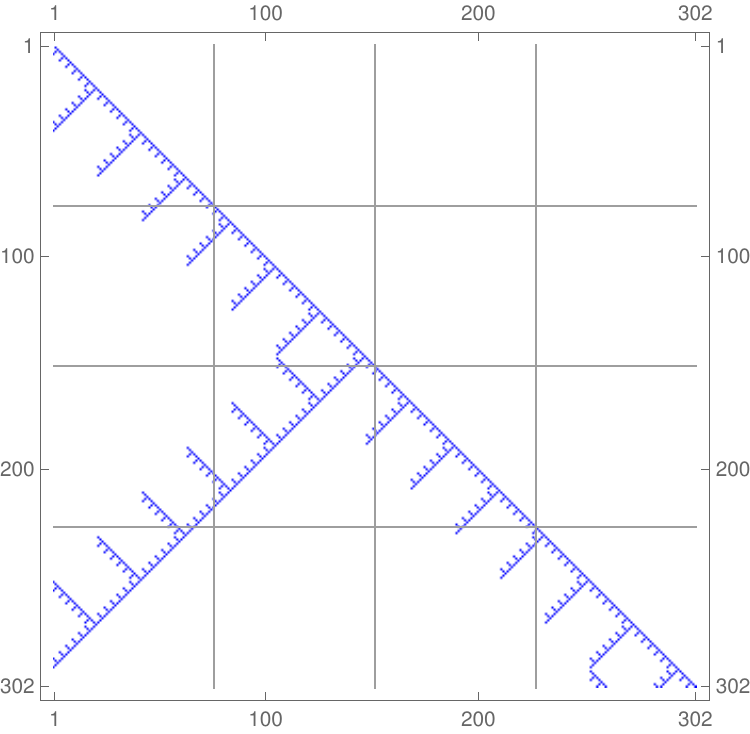}
\includegraphics[width=0.49\textwidth]{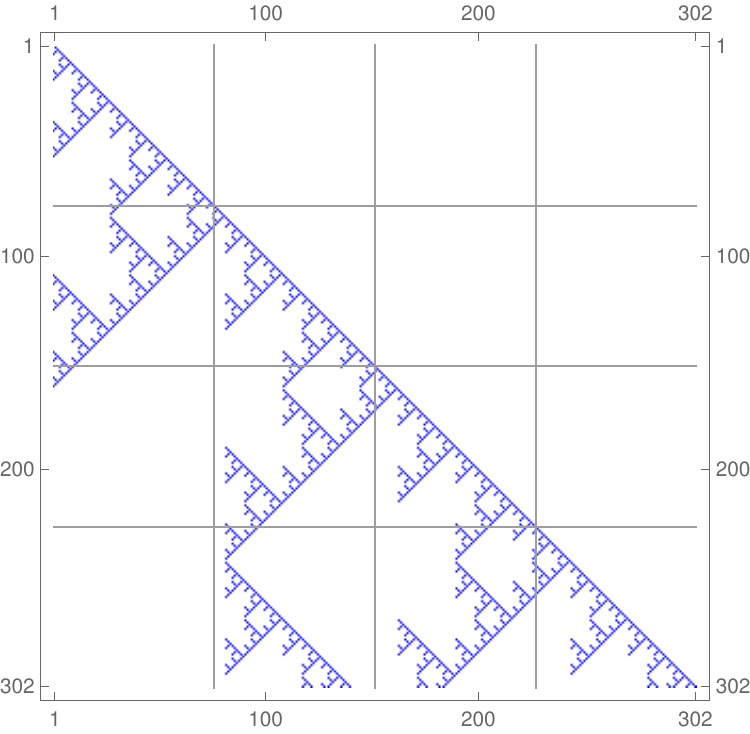}
\caption{The Weyl-Cartan matrix, counting $\lmult{v}{w}$, with $v$ increasing along rows. 
The first illustration shows mixed characteristic $\mchar=(7,3)$, the second illustration shows characteristic $\ppar=3$.}
\label{figure:weylcartan}
\end{figure}

\newpage

\ochanged{The next remark is about the projectors 
in \eqref{eq:danger}.}

\begin{remark}
We warn the reader that the projectors $\qjw[v{-}1]$ descend to well-defined
morphisms in $\TL[{\kk,\qpar}]$ if and only if $v\in\eve$. However, scalar
multiples of $\qjw[v{-}1]$ may descend to $\TL[{\kk,\qpar}]$ even if
$v\notin\eve$. For example, in characteristic $\ppar=2$ the projector $\qjw[2]$
is not well-defined but
\begin{gather*}
2\cdot
\begin{tikzpicture}[anchorbase,scale=0.25,tinynodes]
\draw[JW] (-1.5,0) rectangle (1.5,2);
\node at (0,0.9) {$\qjwm[2]$};
\end{tikzpicture}
=
2\cdot\left(
\begin{tikzpicture}[anchorbase,scale=0.25,tinynodes]
\draw[usual] (0,0) to (0,2);
\draw[usual] (2,0) to (2,2);
\end{tikzpicture}
+\tfrac{1}{2}\cdot
\begin{tikzpicture}[anchorbase,scale=0.25,tinynodes]
\draw[usual] (0,0) to[out=90,in=180] (1,0.75) to[out=0,in=90] (2,0);
\draw[usual] (0,2) to[out=270,in=180] (1,1.25) to[out=0,in=270] (2,2);
\end{tikzpicture}
\right)
=
2\cdot
\begin{tikzpicture}[anchorbase,scale=0.25,tinynodes]
\draw[usual] (0,0) to (0,2);
\draw[usual] (2,0) to (2,2);
\end{tikzpicture}
+
\begin{tikzpicture}[anchorbase,scale=0.25,tinynodes]
\draw[usual] (0,0) to[out=90,in=180] (1,0.75) to[out=0,in=90] (2,0);
\draw[usual] (0,2) to[out=270,in=180] (1,1.25) to[out=0,in=270] (2,2);
\end{tikzpicture}
\end{gather*}
can be seen as a morphism in $\TL[{\Zv,\vpar}]$ 
and specializes to a well-defined 
and non-zero morphism in $\TL[{\kk,\qpar}]$ (indeed, a nilpotent
endomorphism of $\pjw[2]$). So the white boxes need to be treated with care
whenever one works in $(\kk,\qpar)$.
\end{remark}

\ochanged{We comment on the slightly unexpected definition of the functions in \eqref{eq:quantum-fg}:}

\begin{remark}
The quantum version of $\funcf$ and $\funcg$ given in \eqref{eq:quantum-fg}
will be relevant in \fullref{subsection:fusion-mor}, but we do not need quantum numbers to describe
$\funcf$ and $\funcg$ for \fullref{theorem:main-tl-section}. 
By \fullref{proposition:qlucas} this is expected:
these functions will never be evaluated
on the zeroth digit and the quantum Lucas' theorem 
\fullref{proposition:qlucas}
implies that, up to a sign, the only relevant digit for quantum numbers 
is the zeroth digit.
\end{remark}

\ochanged{The following exemplify \fullref{theorem:main-tl-section}.}

\begin{example}
Let $v=\pbase{2,5,3,0,1}{7,3}$ and $S=\{0\}$, then 
the zigzag relation is
\begin{gather*}
\Down{0}\Up{0}\pjw[v{-}1]=
\funcg(0)\Up{\{1,0\}}\Down{\{1,0\}}\pjw[v{-}1] 
+\funcf(0)
\Up{2}\Up{\{1,0\}}\Down{\{1,0\}}\Down{2}\pjw[v{-}1]
=
5\Up{\{1,0\}}\Down{\{1,0\}}\pjw[v{-}1].
\end{gather*}
If $S=\{2\}$, then the zigzag relation reads
\begin{gather*}
\Down{2}\Up{2}\pjw[v{-}1]=
\funcg(5)\Up{2}\Down{2}\pjw[v{-}1] 
+\funcf(5)
\Up{3}\Up{2}\Down{2}\Down{3}\pjw[v{-}1]
=
3\Up{2}\Down{2}\pjw[v{-}1]
+\Up{3}\Up{2}\Down{2}\Down{3}\pjw[v{-}1].
\end{gather*}
\end{example}

\begin{example}\label{example:con-adj}
Using the containment and adjacency relations we calculate 
\begin{gather*}
\Up{0}\Down{1}\Down{0}\pjw[v{-}1]=0,
\quad
\pjw[v{-}1]\Up{0}\Up{1}\Down{0}=0,
\end{gather*}
which is a relation for the corresponding $G_{2}T$-modules, 
see \cite[Section 6.3]{An-tilting-cellular} or \cite[Section 5B]{TuWe-center} 
for the connection.
\end{example}

\begin{example}\label{example:fusion-t1}
Let us give two examples, one is {\losp}.
\begin{enumerate}

\item For $v=\pbase{4,1,6,6,6,10}{7,11}$ and $\mchar=(7,11)$, we have
\begin{align*}
\tmod(v-1)\hcirc\tmod(1)
\cong&\
\tmod\big(\pbase{4,1,6,6,6,10}{7,11}\big)
\oplus
\tmod\big(\pbase{4,1,6,6,6,8}{7,11}\big)
\oplus
\tmod\big(\pbase{4,1,6,6,4,10}{7,11}\big)
\\
&\oplus
\tmod\big(\pbase{4,1,6,4,6,10}{7,11}\big)
\oplus
\tmod\big(\pbase{4,1,4,6,6,10}{7,11}\big)
\oplus
\tmod\big(\pbase{4,-1,6,6,10}{7,11}\big)^{\oplus 2}.
\end{align*}

\item In the case $\mchar=(2,2)$ and $v=\pbase{1,1,1,1}{2,2}$, one gets
\begin{gather*}
\scalebox{0.97}{$\tmod(v-1)\hcirc\tmod(1)
\cong
\tmod\big(\pbase{1,1,1,1}{2,2}\big)
\oplus
\tmod\big(\pbase{1,1,1,-1}{2,2}\big)^{\oplus 2}
\oplus
\tmod\big(\pbase{1,1,-1,1}{2,2}\big)^{\oplus 2}
\oplus
\tmod\big(\pbase{1,-1,1,1}{2,2}\big)^{\oplus 2}.$}
\end{gather*}

\end{enumerate}
\end{example}

\ochanged{Here are addition details for the proof of \fullref{proposition:times-t1}.}

\begin{proof}[Proof of \fullref{proposition:times-t1}.]
This is a neat and direct application of \eqref{eq:partial-trace-result},
as we now explain.
First we observe that a character argument
is enough to verify the formulas.
We will use \eqref{eq:character-use}
for these character computations, where $w=\plbase{b_{j+1},\dots,b_{0}}$ 
are the relevant $\ppar\lpar$-adic expansions for the summands that appear. 
Hereby, $b_i=a_i$ whenever $\taill{v}<i\geqslant j$ and $b_{j+1}\neq 0$ if and only if $\taill{v}=j+1$.

The cases where we do not have generation drops are almost immediate (and have
already been discussed above): The first case where $a_{0}=0$ follows by
observing that the corresponding slot for $\qnum{2}{\vpar^{a_{0}\ppar^{(0)}}}=\qnum{2}{\vpar}$
is not occupied in the expression for $\ch{\tmod(v-1)}$ in
\eqref{eq:character-use}, so we end up with an indecomposable tilting module. In
the second case, the slot is occupied by a $\qnum{2}{\vpar^{\ppar^{(0)}}}$, and
we calculate that $\qnum{2}{\vpar^{\ppar^{(0)}}}\qnum{2}{\vpar}
=\qnum{2}{\vpar^{2\ppar^{(0)}}}+2$, which splits the character into two, one
appearing twice. 
The final case without generation drops, where the slot is now
given by $\qnum{2}{a_{0}\vpar^{\ppar^{(0)}}}$ for $a_{0}>1$, can then be proven
by calculating $\qnum{2}{\vpar^{a_{0}\ppar^{(0)}}}\qnum{2}{\vpar}
=\qnum{2}{\vpar^{(a_{0}+1)\ppar^{(0)}}}
+\qnum{2}{\vpar^{(a_{0}-1)\ppar^{(0)}}}$.

Let $\taill{v}=k$.
The generation drop case follows using the relations
\begin{align*}
\qnum{2}{\vpar^{(\plpar-1)\ppar^{(i)}}}\qnum{2}{\vpar^{\ppar^{(i)}}}
&=\qnum{2}{\vpar^{\ppar^{(i+1)}}}+\qnum{2}{\vpar^{(\plpar-2)\ppar^{(i)}}},
\\
\qnum{2}{\vpar^{0\ppar^{(k)}}}\qnum{2}{\vpar^{\ppar^{(k)}}}
&=
2\qnum{2}{\vpar^{\ppar^{(k)}}},\\
\qnum{2}{\vpar^{1\ppar^{(k)}}}\qnum{2}{\vpar^{\ppar^{(k)}}}
&=
\qnum{2}{\vpar^{2\ppar^{(k)}}}+2,\\
\qnum{2}{\vpar^{a_{k}\ppar^{(k)}}}\qnum{2}{\vpar^{\ppar^{(k)}}}
&=
\qnum{2}{\vpar^{(a_{k}+1)\ppar^{(k)}}}+\qnum{2}{\vpar^{(a_{k}-1)\ppar^{(k)}}} 
\text{ if }a_{k}\notin\{0,1\}.
\end{align*}
These equations imply that
\begin{gather*}
\qnum{2}{\vpar^{a_{k}\ppar^{(k)}}}
{
\prod_{i=0}^{k-1}}\,
\qnum{2}{\vpar^{(\plpar-1)\ppar^{(i)}}}
\cdot\qnum{2}{\vpar}
=
\\
\qnum{2}{\vpar^{(a_{k}+1)\ppar^{(k)}}}
+
\qnum{2}{\vpar^{a_{k}\ppar^{(k)}}}
{
\sum_{i=0}^{k-1}}\,
\big({
\prod_{j\neq i, 0\leq j\leq k-1}}\,
\qnum{2}{\vpar^{(\plpar-1)\ppar^{(j)}}}\big)
\qnum{2}{\vpar^{(\plpar-2)\ppar^{(i)}}}
+
\qnum{2}{\vpar^{(a_{k}-1)\ppar^{(k)}}},
\end{gather*}
for $a_{k}\notin\{0,1\}$. Now we collect terms, which gives the claimed 
summands in order using 
\eqref{eq:character-use}. The cases $a_{k}\in\{0,1\}$ follow 
by replacing the term $\qnum{2}{\vpar^{(a_{k}-1)\ppar^{(k)}}}$ 
by either $0$ or $2$. 
Observe that in the case $\plpar=2$ we obtain the claimed doubling of summands since $\qnum{2}{\vpar^{(\plpar-2)\ppar^{(i)}}}=2$.
\end{proof}

\ochanged{Next, let us comment on the fusion graph for $\tmod(1)$:}

\begin{remark}
For any tilting module $\tmod$ the class $[\tmod]$ acts on the Grothendieck ring
of $\tilt[{\kk,\qpar}]$ by multiplication and thus gives rise to an
$\N[0]\times\N[0]$-matrix with entries in $\N[0]$ when
evaluated on the basis of classes of indecomposable tilting modules. The fusion
graph is the (multi)graph with this adjacency matrix.

The fusion graph for $\tmod(1)$ has periodic and fractal structures. 
First, the number of edges going to lower highest weight summands repeat $l$-periodically in the order $\{0,2,1,\dots,1\}$.
This is the first generation of edges.
Secondly, the edges also become longer and repeat with bigger 
periodicity, depending on $\ppar$.
One can illustrate this fusion graph 
in many ways in order to highlight its fractal structure.
In \fullref{figure:fusion-graph}, we illustrate the fusion graph 
and a Bratteli diagram 
(which is strictly speaking 
not a fusion graph as vertices appear multiple times). 
In both of these cases, the generation drop becomes a \emph{long edge}.
\end{remark}

\begin{figure}[htp]
\begin{gather*}
\resizebox*{!}{\dimexpr\textheight-6.75\baselineskip\relax}{
\begin{tikzpicture}[anchorbase,tinynodes,scale=1]
\pgfmathsetmacro{\fixedprime}{3}
\pgfmathsetmacro{\fixednumber}{25}
\draw[thick] (-25/2,-25) to[out=315,in=180] (-25/2+1.5,-25-0.5)
to[out=0,in=180] (-25/2+17/2-1.5,-25-0.5) to[out=0,in=135](-25/2+17/2,-25-1);
\foreach \n in {0,...,\fixednumber} {
\pgfmathsetmacro{\upperbound}{floor(\n/2)}
\foreach \kk in {0,...,\upperbound} {
\pgfmathparse{\n-\kk-\kk-1 > -1 ? 1 : 0}
\ifnum\pgfmathresult=1
\pgfmathparse{Mod(\n-\kk-\kk+1,\fixedprime) == 0}
\ifnum\pgfmathresult=1
\draw[mygray,dotted,thick] (\kk-\n/2,-\n) to (\kk-\n/2+1/2,-\n-1);
\else
\pgfmathparse{Mod(\n-\kk-\kk+1,\fixedprime) == 1}
\ifnum\pgfmathresult=1
\draw[spinach,double] (\kk-\n/2,-\n) to (\kk-\n/2+1/2,-\n-1);
\else
\draw[thick] (\kk-\n/2,-\n) to (\kk-\n/2+1/2,-\n-1);
\fi     
\fi
\fi
\pgfmathsetmacro\jump{(\fixedprime+\fixedprime-1)/2}
\pgfmathparse{\n-\kk-\kk-(\fixedprime+\fixedprime)+1 > -1 ? 1 : 0}
\ifnum\pgfmathresult=1
\pgfmathparse{Mod(\n-\kk-\kk+1,\fixedprime) == (\fixedprime-1)}
\ifnum\pgfmathresult=1
\pgfmathparse{Mod((\n-\kk-\kk+1-(\fixedprime-1))/\fixedprime,\fixedprime) == 0}
\ifnum\pgfmathresult=1
\draw[mygray,dotted,thick] (\kk-\n/2,-\n) to[out=315,in=135] (\kk-\n/2+\jump,-\n-1);    
\fi
\pgfmathparse{Mod((\n-\kk-\kk+1-(\fixedprime-1))/\fixedprime,\fixedprime) == 1}
\ifnum\pgfmathresult=1
\draw[spinach,double] (\kk-\n/2,-\n) to[out=315,in=135] (\kk-\n/2+\jump,-\n-1); 
\fi
\pgfmathparse{Mod((\n-\kk-\kk+1-(\fixedprime-1))/\fixedprime,\fixedprime) > 1 ? 1 : 0}
\ifnum\pgfmathresult=1
\draw[thick] (\kk-\n/2,-\n) to[out=315,in=135] (\kk-\n/2+\jump,-\n-1);  
\fi
\fi
\fi
\draw[thick] (\kk-\n/2,-\n) to (\kk-\n/2-1/2,-\n-1);
\draw[white,fill=white] (\kk-\n/2,-\n) circle (0.26cm);
\pgfmathsetmacro\result{\n-\kk-\kk}
\node[] at (\kk-\n/2,-\n) {$\tmod(\pgfmathprintnumber{\result})$};
}}
\node at (-10,-10) {\includegraphics[width=0.49\textwidth]{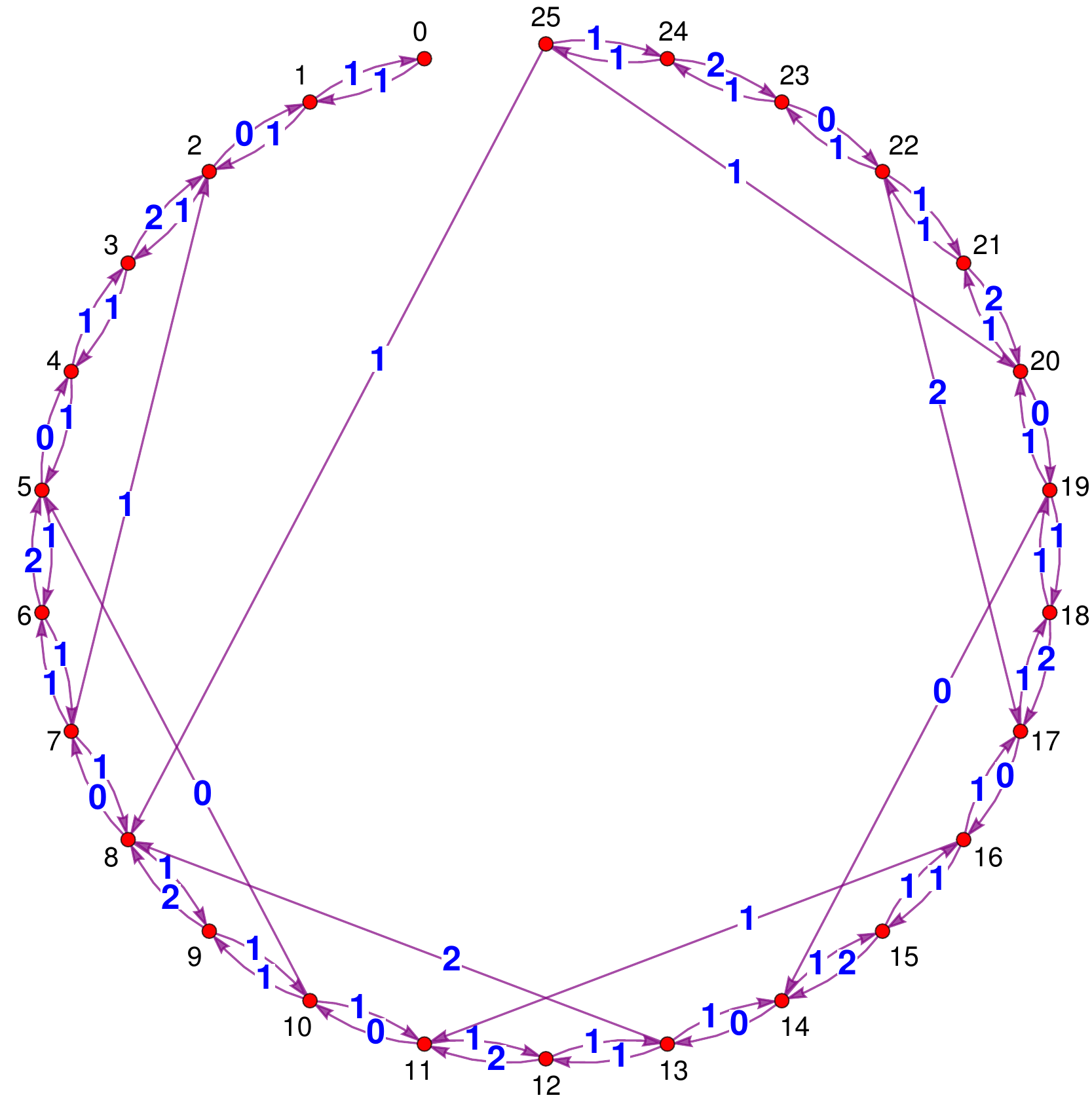}};
\end{tikzpicture}
}
\end{gather*}
\caption{       
\emph{Top left picture:} The full subgraph of the fusion graph for $\tmod(1)$
in characteristic $3$ containing the first $26$ vertices. The vertex labels are
$\tmod(v)$, the edge labels indicate multiplicities. 
\emph{Bottom right picture:} The first $26$
rows of the Bratteli diagram in characteristic $3$. The solid and double edges
indicate multiplicities $1$ and $2$, respectively; the dotted edges indicate
multiplicity $0$. (In both cases, the edges labeled $0$ are only to highlight the
periodicity.) Note that the last $2.5$ rows in the Bratteli diagram are another
illustration of the graph on the top left.
}
\label{figure:fusion-graph}
\end{figure}

\newpage

\makeautorefname{lemma}{Lemmas}

\ochanged{The extended proofs of \fullref{lemma:fusiontotrace} and \ref{lemma:tracetofusion} are as follows.}

\makeautorefname{lemma}{Lemma}

\begin{proof}[Proof of \fullref{lemma:fusiontotrace}]
We first consider $i=0$, where we have $w=1$ and 
assume $a_{0}\neq 0$, {\ie} we aim to prove
\begin{gather}\label{eq:ptrclaim2}
\begin{tikzpicture}[anchorbase,scale=0.25,tinynodes]
\draw[pJW] (1.5,-1) rectangle (-2.5,1);
\node at (-0.5,-0.1) {$\pjwm[v{-}1]$};
\draw[usual] (1,1) to[out=90,in=180] (1.5,1.5) to[out=0,in=90] 
(2,1) to (2,0) to (2,-1) 
to[out=270,in=0] (1.5,-1.5) to[out=180,in=270] (1,-1);
\node at (0,2.5) {$\phantom{a}$};
\node at (0,-2.5) {$\phantom{a}$};
\end{tikzpicture} 
=
\funcg_{\qpar}(a_{0}-1)\cdot
\begin{tikzpicture}[anchorbase,scale=0.25,tinynodes]
\draw[pJW] (1.5,-1) rectangle (-2.5,1);
\node at (-0.5,-0.1) {$\pjwm[v{-}2]$};
\node at (0,2.5) {$\phantom{a}$};
\node at (0,-2.5) {$\phantom{a}$};
\end{tikzpicture}
+\funcf_{\qpar}(a_{0}-1)\cdot
\loopdown{0}{v{-}2}.
\end{gather}
To verify this, we will use the fusion rule for $v-1$ in reverse to expand the
projector $\pjw[v{-}1]$. If $a_{0}=1$, then we have
$\pjw[v{-}1]=\pjw[v{-}2]\hcirc\idtl[1]$. The claimed statement follows since
the circle value is $-\qnum{2}{\qpar}=\funcg_{\qpar}(0)$ and the second term is zero by definition. If $\lpar=2$, then we are done. Thus, from now on suppose
that $\lpar>2$. 

We consider the case
$a_{0}=2$, where the fusion 
rule involves $\fusidema{v{-}1}{0}+\fusidemb{v{-}1}{0}$:
\begin{gather*}
\begin{tikzpicture}[anchorbase,scale=0.25,tinynodes]
\draw[pJW] (1.5,-1) rectangle (-2.5,1);
\node at (-0.5,-0.1) {$\pjwm[v{-}1]$};
\draw[usual] (1,1) to[out=90,in=180] (1.5,1.5) to[out=0,in=90] 
(2,1) to (2,0) to (2,-1) 
to[out=270,in=0] (1.5,-1.5) to[out=180,in=270] (1,-1);
\node at (0,2.5) {$\phantom{a}$};
\node at (0,-2.5) {$\phantom{a}$};
\end{tikzpicture}
=
\begin{tikzpicture}[anchorbase,scale=0.25,tinynodes]
\draw[pJW] (0.5,-1) rectangle (-2.5,1);
\node at (-1,-0.1) {$\pjwm[v{-}2]$};
\draw[usual] (1,1) to[out=90,in=180] (1.5,1.5) to[out=0,in=90] 
(2,1) to (2,0) to (2,-1) 
to[out=270,in=0] (1.5,-1.5) to[out=180,in=270] (1,-1) to (1,1);
\node at (0,2.5) {$\phantom{a}$};
\node at (0,-2.5) {$\phantom{a}$};
\end{tikzpicture}
-
\left(
\begin{tikzpicture}[anchorbase,scale=0.4,tinynodes]
\draw[pJW] (-0.5,-1.5) rectangle (3.2,-2.5);
\node at (1.35,-2.2) {$\pjwm[v{-}2]$};
\draw[pJW] (-0.5,-3.5) rectangle (2.5,-4.5);
\node at (1,-4.2) {$\pjwm[v{-}3]$};
\draw[pJW] (-0.5,-5.5) rectangle (3.2,-6.5);
\node at (1.35,-6.2) {$\pjwm[v{-}2]$};
\draw[usual] (2,-2.5) to[out=270,in=180] (2.5,-2.85) 
to[out=0,in=270] (3,-2.5);
\draw[usual] (2,-3.5) to[out=90,in=180] (2.5,-3.2) 
to (3,-3.2) to [out=0,in=270] (3.5,-1.5);
\draw[usual] (2,-4.5) to (2,-5.5);
\draw[usual] (2.5,-5.5) to[out=90,in=180] (3,-5.15) 
to[out=0,in=90] (3.5,-5.5) to (3.5,-6.5);
\draw[usual] (0.5,-3.5) to (0.5,-2.5);
\draw[usual] (0.5,-4.5) to (0.5,-5.5);
\draw[usual] (3.5,-6.5) to[out=270,in=180] (4,-7) to[out=0,in=270]
(4.5,-6.5) to (4.5,-1.5) to[out=90,in=0] (4,-1) to[out=180,in=90]
(3.5,-1.5);
\end{tikzpicture}
+\,
\begin{tikzpicture}[anchorbase,scale=0.4,tinynodes]
\draw[pJW] (-0.5,-1.5) rectangle (3.2,-2.5);
\node at (1.35,-2.2) {$\pjwm[v{-}2]$};
\draw[pJW] (-0.5,-3.5) rectangle (2.5,-4.5);
\node at (1,-4.2) {$\pjwm[v{-}3]$};
\draw[pJW] (-0.5,-5.5) rectangle (3.2,-6.5);
\node at (1.35,-6.2) {$\pjwm[v{-}2]$};
\draw[usual] (2,-2.5) to[out=270,in=180] (2.5,-2.85) 
to[out=0,in=270] (3,-2.5);
\draw[usual] (2,-3.5) to[out=90,in=180] (2.5,-3.2) 
to (3,-3.2) to [out=0,in=270] (3.5,-1.5);
\draw[usual] (2,-4.5) to[out=270,in=180] (2.5,-4.8) 
to (3,-4.8) to [out=0,in=90] (3.5,-6.5);
\draw[usual] (2,-5.5) to[out=90,in=180] (2.5,-5.15) 
to[out=0,in=90] (3,-5.5);
\draw[usual] (0.5,-3.5) to (0.5,-2.5);
\draw[usual] (0.5,-4.5) to (0.5,-5.5);
\draw[usual] (3.5,-6.5) to[out=270,in=180] (4,-7) to[out=0,in=270]
(4.5,-6.5) to (4.5,-1.5) to[out=90,in=0] (4,-1) to[out=180,in=90]
(3.5,-1.5);
\end{tikzpicture}
\right)
-\left(
\begin{tikzpicture}[anchorbase,scale=0.4,tinynodes]
\draw[pJW] (-0.5,-1.5) rectangle (3.2,-2.5);
\node at (1.35,-2.2) {$\pjwm[v{-}2]$};
\draw[pJW] (-0.5,-3.5) rectangle (2.5,-4.5);
\node at (1,-4.2) {$\pjwm[v{-}3]$};
\draw[pJW] (-0.5,-5.5) rectangle (3.2,-6.5);
\node at (1.35,-6.2) {$\pjwm[v{-}2]$};
\draw[usual] (2.5,-2.5) to[out=270,in=180] (3,-2.85) 
to[out=0,in=270] (3.5,-2.5) to (3.5,-1.5);
\draw[usual] (2,-3.5) to (2,-2.5);
\draw[usual] (2,-4.5) to[out=270,in=180] (2.5,-4.8) 
to (3,-4.8) to [out=0,in=90] (3.5,-6.5);
\draw[usual] (2,-5.5) to[out=90,in=180] (2.5,-5.15) 
to[out=0,in=90] (3,-5.5);
\draw[usual] (0.5,-3.5) to (0.5,-2.5);
\draw[usual] (0.5,-4.5) to (0.5,-5.5);
\draw[usual] (3.5,-6.5) to[out=270,in=180] (4,-7) to[out=0,in=270]
(4.5,-6.5) to (4.5,-1.5) to[out=90,in=0] (4,-1) to[out=180,in=90]
(3.5,-1.5);
\end{tikzpicture}
+\,
\begin{tikzpicture}[anchorbase,scale=0.4,tinynodes]
\draw[pJW] (-0.5,-1.5) rectangle (3.2,-2.5);
\node at (1.35,-2.2) {$\pjwm[v{-}2]$};
\draw[pJW] (-0.5,-3.5) rectangle (2.5,-4.5);
\node at (1,-4.2) {$\pjwm[v{-}3]$};
\draw[pJW] (-0.5,-5.5) rectangle (3.2,-6.5);
\node at (1.35,-6.2) {$\pjwm[v{-}2]$};
\draw[usual] (2,-2.5) to[out=270,in=180] (2.5,-2.85) 
to[out=0,in=270] (3,-2.5);
\draw[usual] (2,-3.5) to[out=90,in=180] (2.5,-3.2) 
to (3,-3.2) to [out=0,in=270] (3.5,-1.5);
\draw[usual] (2,-4.5) to[out=270,in=180] (2.5,-4.8) 
to (3,-4.8) to [out=0,in=90] (3.5,-6.5);
\draw[usual] (2,-5.5) to[out=90,in=180] (2.5,-5.15) 
to[out=0,in=90] (3,-5.5);
\draw[usual] (0.5,-3.5) to (0.5,-2.5);
\draw[usual] (0.5,-4.5) to (0.5,-5.5);
\draw[usual] (3.5,-6.5) to[out=270,in=180] (4,-7) to[out=0,in=270]
(4.5,-6.5) to (4.5,-1.5) to[out=90,in=0] (4,-1) to[out=180,in=90]
(3.5,-1.5);
\end{tikzpicture}
\right)
.
\end{gather*}
The first term in each of the brackets 
is $\loopdown{0}{v{-}2}$; the second is
$(\loopdown{0}{v{-}2})^{2}$, which is 
zero by \eqref{eq:dud}. Since 
$\funcg_{\qpar}(1)=-\qnum{2}{\qpar}$ and
$\funcf_{\qpar}(1)=-2$ the claim 
follows for $a_{0}=2$.

Next we consider $a_{0}\in\{3,\dots,\lpar-1\}$.  
Here we use $-\qnum{2}{\qpar}+\tfrac{\qnum{a_{0}-2}{\qpar}}
{\qnum{a_{0}-1}{\qpar}}=-
\tfrac{\qnum{a_{0}}{\qpar}}{\qnum{a_{0}-1}{\qpar}}$ 
to get
\begin{gather}\label{eq:ptrfusion2}
\begin{aligned}
\begin{tikzpicture}[anchorbase,scale=0.25,tinynodes]
\draw[pJW] (1.5,-1) rectangle (-2.5,1);
\node at (-0.5,-0.1) {$\pjwm[v{-}1]$};
\draw[usual] (1,1) to[out=90,in=180] (1.5,1.5) to[out=0,in=90] 
(2,1) to (2,0) to (2,-1) 
to[out=270,in=0] (1.5,-1.5) to[out=180,in=270] (1,-1);
\node at (0,2.5) {$\phantom{a}$};
\node at (0,-2.5) {$\phantom{a}$};
\end{tikzpicture}
&=
\begin{tikzpicture}[anchorbase,scale=0.25,tinynodes]
\draw[pJW] (0.5,-1) rectangle (-2.5,1);
\node at (-1,-0.1) {$\pjwm[v{-}2]$};
\draw[usual] (1,1) to[out=90,in=180] (1.5,1.5) to[out=0,in=90] 
(2,1) to (2,0) to (2,-1) 
to[out=270,in=0] (1.5,-1.5) to[out=180,in=270] (1,-1) to (1,1);
\node at (0,2.5) {$\phantom{a}$};
\node at (0,-2.5) {$\phantom{a}$};
\end{tikzpicture}
-
\left(
\tfrac{1}{\funcg_{\qpar}(a_{0}-2)}\cdot
\begin{tikzpicture}[anchorbase,scale=0.4,tinynodes]
\draw[pJW] (-0.5,-1.5) rectangle (2.5,-2.5);
\node at (1,-2.2) {$\pjwm[v{-}2]$};
\draw[pJW] (-0.5,-4.5) rectangle (2.5,-5.5);
\node at (1,-5.2) {$\pjwm[v{-}2]$};
\draw[pJW] (-0.5,-3.25) rectangle (0.5,-3.75);
\draw[usual] (2,-2.5) to[out=270,in=180] (2.5,-3) to[out=0,in=270] 
(3,-2.5) to (3,-1.5);
\draw[usual] (2,-4.5) to[out=90,in=180] (2.5,-4) to[out=0,in=90] 
(3,-4.5) to (3,-5.5);
\draw[usual] (0,-2.5) to (0,-3.25);
\draw[usual] (0,-3.75) to (0,-4.5);
\draw[usual] (3,-5.5) to[out=270,in=180] (3.5,-6) to[out=0,in=270]
(4,-5.5) to (4,-1.5) to[out=90,in=0] (3.5,-1) to[out=180,in=90]
(3,-1.5);
\end{tikzpicture}  
-\tfrac{\funcf_{\qpar}(a_{0}-1)}{\funcg_{\qpar}(a_{0}-2)}\cdot
\begin{tikzpicture}[anchorbase,scale=0.4,tinynodes]
\draw[pJW] (-0.5,-1.5) rectangle (2.5,-2.5);
\node at (1,-2.2) {$\pjwm[v{-}2]$};
\draw[pJW] (-0.5,-4.5) rectangle (2.5,-5.5);
\node at (1,-5.2) {$\pjwm[v{-}2]$};
\draw[pJW] (-0.5,-3.25) rectangle (0.5,-3.75);
\draw[usual] (2,-2.5) to[out=270,in=180] (2.5,-3) to[out=0,in=270] 
(3,-2.5) to (3,-1.5);
\draw[usual] (2,-4.5) to[out=90,in=180] (2.5,-4) to[out=0,in=90] 
(3,-4.5) to (3,-5.5);
\draw[usual] (0,-2.5) to (0,-3.25);
\draw[usual] (0,-3.75) to (0,-4.5);
\draw[usual] (0.5,-4.5) to[out=90,in=180] (1,-4) to[out=0,in=90] 
(1.5,-4.5);
\draw[usual] (0.5,-2.5) to[out=270,in=180] (1,-3) to[out=0,in=270] 
(1.5,-2.5);
\node at (1,-3.65) {$S_{0}$};
\draw[usual] (3,-5.5) to[out=270,in=180] (3.5,-6) to[out=0,in=270]
(4,-5.5) to (4,-1.5) to[out=90,in=0] (3.5,-1) to[out=180,in=90]
(3,-1.5);
\end{tikzpicture}
\right)
\\
&=
\funcg_{\qpar}(a_{0}-1)\cdot
\begin{tikzpicture}[anchorbase,scale=0.25,tinynodes]
\draw[pJW] (1.5,-1) rectangle (-2.5,1);
\node at (-0.5,-0.1) {$\pjwm[v{-}2]$};
\node at (0,2.5) {$\phantom{a}$};
\node at (0,-2.5) {$\phantom{a}$};
\end{tikzpicture}
+
\tfrac{\funcf_{\qpar}(a_{0}-1)}{\funcg_{\qpar}(a_{0}-2)}\cdot
\begin{tikzpicture}[anchorbase,scale=0.4,tinynodes]
\draw[pJW] (-0.5,-1.5) rectangle (2.5,-2.5);
\node at (1,-2.2) {$\pjwm[v{-}2]$};
\draw[pJW] (-0.5,-4.5) rectangle (2.5,-5.5);
\node at (1,-5.2) {$\pjwm[v{-}2]$};
\draw[pJW] (-0.5,-3.25) rectangle (0.5,-3.75);
\draw[usual] (2,-4.5) to (2,-2.5);
\draw[usual] (0,-2.5) to (0,-3.25);
\draw[usual] (0,-3.75) to (0,-4.5);
\draw[usual] (0.5,-4.5) to[out=90,in=180] (1,-4) to[out=0,in=90] 
(1.5,-4.5);
\draw[usual] (0.5,-2.5) to[out=270,in=180] (1,-3) to[out=0,in=270] 
(1.5,-2.5);
\node at (1,-3.65) {$S_{0}$};
\end{tikzpicture}
,
\end{aligned}
\end{gather}
and it remains to compute the final term. To this end, 
we will use the fusion rule for $v-2a_{0}+2$ 
on the mini box (using induction), which corresponds to
$\pjw[v{-}2a_{0}{+}1]$. The last digit of its
relevant $\ppar\lpar$-adic expansion 
is an element of $\{3,\dots,\lpar-1\}$, namely 
$\lpar-a_{0}+2$. We claim that fusion results in:
\begin{gather}\label{eq:middleminifusion2}
\begin{tikzpicture}[anchorbase,scale=0.4,tinynodes]
\draw[pJW] (-0.5,-1.5) rectangle (2.5,-2.5);
\node at (1,-2.2) {$\pjwm[v{-}2]$};
\draw[pJW] (-0.5,-4.5) rectangle (2.5,-5.5);
\node at (1,-5.2) {$\pjwm[v{-}2]$};
\draw[pJW] (-0.5,-3.25) rectangle (0.5,-3.75);
\draw[usual] (2,-4.5) to (2,-2.5);
\draw[usual] (0,-2.5) to (0,-3.25);
\draw[usual] (0,-3.75) to (0,-4.5);
\draw[usual] (0.5,-4.5) to[out=90,in=180] (1,-4) to[out=0,in=90] 
(1.5,-4.5);
\draw[usual] (0.5,-2.5) to[out=270,in=180] (1,-3) to[out=0,in=270] 
(1.5,-2.5);
\node at (1,-3.65) {$S_{0}$};
\end{tikzpicture}
=
\begin{tikzpicture}[anchorbase,scale=0.4,tinynodes]
\draw[pJW] (-0.5,-1.5) rectangle (2.5,-2.5);
\node at (1,-2.2) {$\pjwm[v{-}2]$};
\draw[pJW] (-0.5,-4.5) rectangle (2.5,-5.5);
\node at (1,-5.2) {$\pjwm[v{-}2]$};
\draw[pJW] (-0.5,-3.25) rectangle (2.5,-3.75);
\draw[usual] (2,-3.25) to (2,-2.5);
\draw[usual] (2,-4.5) to (2,-3.75);
\draw[usual] (0,-2.5) to (0,-3.25);
\draw[usual] (0,-3.75) to (0,-4.5);
\draw[usual] (0.5,-4.5) to[out=90,in=180] (1,-4) to[out=0,in=90] 
(1.5,-4.5);
\draw[usual] (0.5,-2.5) to[out=270,in=180] (1,-3) to[out=0,in=270] 
(1.5,-2.5);
\end{tikzpicture}
+
\tfrac{1}{\funcg_{\qpar}(\lpar-a_{0}+1)}
\loopdown{0}{v{-}2}
+\text{lower order terms}
=
\tfrac{1}{\funcg_{\qpar}(\lpar-a_{0}+1)}
\loopdown{0}{v{-}2}
.
\end{gather}
Here, we have three things to check. To start, the first term in
the middle is zero because $\tmod(v-2)$ and $\tmod(v-2a_{0}+2)$ do not share
any common Weyl factors. To see this, first observe that
the relevant zeroth digits are $a_{0}-1$ and respectively $\lpar-a_{0}+3$.
Now the claim follows from the condition 
in \fullref{lemma:no-delta-overlap} for any odd
$\lpar$ since one has $a_{0}-1\neq\lpar-a_{0}+3$ and $a_{0}-1\neq a_{0}-3$. For 
even $\lpar$ it can happen that $a_{0}-1=\lpar-a_{0}+3$, namely 
for $a_{0}=\tfrac{\lpar+4}{2}$. However, in this 
case the next digit of $v-1$ and 
$v-2a_{0}+3$ differ by one, so 
\fullref{lemma:no-delta-overlap} also applies.
Second, the fusion rule
includes a term of the form $\fusidem{v{-}2a_{0}{+}2}{0}$, which is typically a sum of two
diagrams (although the second may not appear in some cases). The first diagram
combines with the present caps and cups to form
$\tfrac{1}{\funcg_{\qpar}(\lpar-a_{0}+1)}\loopdown{0}{v{-}2}$. The second
diagram (if it is present at all) vanishes when it is sandwiched because of the containment
relation $\Up{0}\Up{0}=0$. Third, all possible terms of even lower order arising
from the fusion rule become zero when sandwiched. Such terms only arise if
$\taill{v-2a_{0}+2}>0$, {\ie} if $a_{0}=3$.

Sandwiching a term $\fusidem{v{-}2a_{0}{+}2}{i}$ for $i>0$ produces a linear
combination of (at most) two diagrams:
\begin{align*}
\begin{tikzpicture}[anchorbase,scale=0.4,tinynodes]
\draw[pJW] (-1.5,-1.5) rectangle (8,-2.5);
\node at (3.25,-2.2) {$\pjwm[v{-}2]$};
\draw[pJW] (-1.5,-7.75) rectangle (7,-8.5);
\node at (2.75,-8.25) {$\pjwm[v{-}3]$};
\draw[pJW] (-1.5,-8.5) rectangle (8,-9.5);
\node at (3.25,-9) {$\pjwm[v{-}2]$};
\draw[usual] (-1,-2.5) to (-1,-7.75);
\draw[usual] (1,-2.5) to[out=270,in=180] (4.25,-4.25) to[out=0,in=270] 
(7.5,-2.5);
\draw[usual] (1.5,-2.5) to[out=270,in=180] (2.5,-3.25) to[out=0,in=270] 
(3.5,-2.5);
\node at (3,-3.75) {$w-1$};
\node at (6,-3.75) {$1$};
\node at (3,-6.75) {$w-1$};
\node at (6,-6.75) {$1$};
\node at (8,-3) {$1$};
\node at (8,-7.5) {$1$};
\draw[usual] (4.5,-2.5) to[out=270,in=180] (5.5,-3.25) to[out=0,in=270] 
(6.5,-2.5);
\draw[usual] (1,-7.75) to[out=90,in=180] (4.25,-5.5) to[out=0,in=90] 
(7.5,-8.5);
\draw[usual] (1.5,-7.75) to[out=90,in=180] (2.5,-7) to[out=0,in=90] 
(3.5,-7.75);
\draw[usual] (4.5,-7.75) to[out=90,in=180] (5.5,-7) to[out=0,in=90] 
(6.5,-7.75);
\draw[usual] (-0.5,-7.75) to (-0.5,-2.5) (0.5,-2.5) to
(0.5,-7.75);
\draw[thick,tomato,dashed] (-1.5,-5.25) rectangle (7,-8.5);
\end{tikzpicture}
,\qquad
\begin{tikzpicture}[anchorbase,scale=0.4,tinynodes]
\draw[pJW] (-1.5,-1.5) rectangle (8,-2.5);
\node at (3.25,-2.2) {$\pjwm[v{-}2]$};
\draw[pJW] (-1.5,-7.75) rectangle (7,-8.5);
\node at (2.75,-8.25) {$\pjwm[v{-}3]$};
\draw[pJW] (-1.5,-8.5) rectangle (8,-9.5);
\node at (3.25,-9) {$\pjwm[v{-}2]$};
\draw[usual] (-1,-2.5) to (-1,-7.75);
\draw[usual] (1,-2.5) to[out=270,in=180] (4.25,-4.25) to[out=0,in=270] 
(7.5,-2.5);
\draw[usual] (1.5,-2.5) to[out=270,in=180] (2.5,-3.25) to[out=0,in=270] 
(3.5,-2.5);
\node at (3,-3.75) {$w-1$};
\node at (6,-3.75) {$1$};
\node at (3,-6.75) {$w-1$};
\node at (6,-6.75) {$1$};
\node at (8,-3) {$1$};
\node at (8,-7.5) {$1$};
\draw[usual] (4.5,-2.5) to[out=270,in=180] (5.5,-3.25) to[out=0,in=270] 
(6.5,-2.5);
\draw[usual] (1,-7.75) to[out=90,in=180] (4.25,-5.5) to[out=0,in=90] 
(7.5,-8.5);
\draw[usual] (1.5,-7.75) to[out=90,in=180] (2.5,-7) to[out=0,in=90] 
(3.5,-7.75);
\draw[usual] (4.5,-7.75) to[out=90,in=180] (5.5,-7) to[out=0,in=90] 
(6.5,-7.75);
\draw[usual] (-0.5,-7.75) to (-.5,-5.25) to [out=90,in=180] (0,-4.75) to [out=0,in=90] (0.5,-5.25) to  
(0.5,-7.75);
\draw[usual] (-0.5,-2.5) to (-.5,-3.75) to [out=270,in=180] (0,-4.25) to [out=0,in=270] (0.5,-3.75) to  
(0.5,-2.5);
\draw[thick,tomato,dashed] (-1.5,-5.25) rectangle (7,-8.5);
\end{tikzpicture}
,
\end{align*}
where $w=\ppar^{(i)}$. 
We focus on the common outlined portion. 
Once we bend up the
free strand, this is in fact the morphism
$\Down{S}\pjw[v{-}5]\Down{0}\pjw[v{-}3]$ 
for the stretch $S=\{0,1,\dots,i-1\}$,
which vanishes by the containment relation.

The terms $\fusidema{v{-}2a_{0}{+}2}{i}$ 
or $\fusidemb{v{-}2a_{0}{+}2}{i}$ for
$i>0$ only occur when $i=\taill{v-2a_{0}+2}$ 
and $a_{i}=2$. Sandwiching such a term results
in a diagram that factors through a 
morphism from $\pjw[v{-}2]$ to
$\pjw[v{-}2a_{0}{+}2{-}2\ppar^{(i)}]$, and we 
claim that such morphisms are zero
because the corresponding tilting modules 
have no common Weyl factors. Indeed,
the relevant $\ppar\lpar$-adic expansions are
\[
v-1=\plbase{\dots,a_{i},0,0,\dots,2}\quad\text{and}\quad
v{-}2a_0+3-2\ppar^{(i)}=
\plbase{\dots,a_{i}-2,0,\dots,0},
\] 
and thus the claim follows from
\fullref{lemma:no-delta-overlap}. 

We have finished showing that 
\eqref{eq:middleminifusion2} holds, which we use to
rewrite \eqref{eq:ptrfusion2}. 
The coefficient of $\loopdown{0}{v{-}2}$ is
\begin{gather*}
\tfrac{\funcf_{\qpar}(a_{0}-1)}{\funcg_{\qpar}(a_{0}-2)}\cdot
\tfrac{1}{\funcg_{\qpar}(\lpar-a_{0}+1)}
=
\tfrac{\funcf_{\qpar}(a_{0}-1)}{\funcg_{\qpar}(a_{0}-2)}\cdot
\funcg_{\qpar}(a_{0}-2)
=
\funcf_{\qpar}(a_{0}-1),
\end{gather*}
where we have simplified $\funcg_{\qpar}(\lpar-a_{0}+1)^{-1}=
\funcg_{\qpar}(a_{0}-2)$ using $\qnum{\lpar}{\qpar}=0$ (this does not hold in
the semisimple case). Thus we have verified the partial trace claim
\eqref{eq:ptrclaim2}. 

Finally we consider the case $0<i\leq\taill{v}$, which, surprisingly, is much
easier to prove. Recall that we assume that $a_{i}\neq 0$ and aim to prove: 
\begin{gather*}
\begin{tikzpicture}[anchorbase,scale=0.25,tinynodes]
\draw[pJW] (1.5,-1) rectangle (-2.5,1);
\node at (-0.5,-0.1) {$\pjwm[v{-}w]$};
\draw[usual] (1,1) to[out=90,in=180] (1.5,1.5) to[out=0,in=90] 
(2,1) to (2,0) node[right,xshift=-2pt]{$w$} to (2,-1) 
to[out=270,in=0] (1.5,-1.5) to[out=180,in=270] (1,-1);
\node at (0,2.5) {$\phantom{a}$};
\node at (0,-2.5) {$\phantom{a}$};
\end{tikzpicture} 
=
\funcg_{\qpar}(a_{i}-1)\cdot
\begin{tikzpicture}[anchorbase,scale=0.25,tinynodes]
\draw[pJW] (1.5,-1) rectangle (-2.5,1);
\node at (-0.5,-0.1) {$\pjwm[v{-}2w]$};
\node at (0,2.5) {$\phantom{a}$};
\node at (0,-2.5) {$\phantom{a}$};
\end{tikzpicture}
+\funcf_{\qpar}(a_{i}-1)\cdot
\loopdown{i}{v{-}2w}.
\end{gather*}
To verify this claim, we calculate that 
$v-w+1=\plbase{a_{j},\dots,a_{i},0,\dots,0}$.
In particular, we can use (a slight generalization of) \fullref{proposition:more-ptrace} 
to trace off $w-1=\ppar^{(i)}-1$ strands
and we get
\begin{gather*}
\begin{tikzpicture}[anchorbase,scale=0.25,tinynodes]
\draw[pJW] (1.5,-1) rectangle (-2.5,1);
\node at (-0.5,-0.1) {$\pjwm[v{-}w]$};
\draw[usual] (1,1) to[out=90,in=180] (1.5,1.5) to[out=0,in=90] 
(2,1) to (2,0) node[right,xshift=-2pt]{$w-1$} to (2,-1) 
to[out=270,in=0] (1.5,-1.5) to[out=180,in=270] (1,-1);
\end{tikzpicture}   
=
\funcg_{\qpar}(a_{i}-1)\cdot
\begin{tikzpicture}[anchorbase,scale=0.25,tinynodes]
\draw[pJW] (2,0.5) rectangle (-1,-0.5);
\draw[pJW] (2,3.5) rectangle (-1,4.5);
\node at (0.5,-0.2) {$\pjwm[x{-}1]$};
\node at (0.5,3.8) {$\pjwm[x{-}1]$};
\draw[usual] (0,0.5) to[out=90,in=180] (0.5,1) to[out=0,in=90] (1,0.5);
\draw[usual] (0,3.5) to[out=270,in=180] (0.5,3) to[out=0,in=270] (1,3.5);
\draw[usual] (-0.5,0.5) to (-0.5,3.5);
\end{tikzpicture}
+
\funcf_{\qpar}(a_{i}-1)\cdot
\begin{tikzpicture}[anchorbase,scale=0.25,tinynodes]
\draw[pJW] (2,0.5) rectangle (-1,-0.5);
\draw[pJW] (2,3.5) rectangle (-1,4.5);
\node at (0.5,-0.2) {$\pjwm[x{-}1]$};
\node at (0.5,3.8) {$\pjwm[x{-}1]$};
\draw[usual] (0,0.5) to[out=90,in=180] (0.5,1) to[out=0,in=90] (1,0.5);
\draw[usual] (-0.25,0.5) to[out=90,in=180] (0.5,1.25) to[out=0,in=90] (1.25,0.5);
\draw[usual] (0,3.5) to[out=270,in=180] (0.5,3) to[out=0,in=270] (1,3.5);
\draw[usual] (-0.25,3.5) to[out=270,in=180] (0.5,2.75) to[out=0,in=270] (1.25,3.5);
\draw[usual] (-0.5,0.5) to (-0.5,3.5);
\node at (1,3.5) {$\phantom{a}$};
\node at (1,-0.5) {$\phantom{a}$};
\end{tikzpicture},
\end{gather*}
where $x=\plbase{a_{j},\dots,a_{i}-1,0,\dots,1}$ (if
$i=j$, the $\funcf$-term vanishes but the $\funcg$-term does not because we
assume $a_j>1$). Now we use shortening to compute the full partial trace as:
\begin{gather*}
\begin{tikzpicture}[anchorbase,scale=0.25,tinynodes]
\draw[pJW] (1.5,-1) rectangle (-2.5,1);
\node at (-0.5,-0.1) {$\pjwm[v{-}w]$};
\draw[usual] (1,1) to[out=90,in=180] (1.5,1.5) to[out=0,in=90] 
(2,1) to (2,0)node[right]{$w$} to (2,-1) 
to[out=270,in=0] (1.5,-1.5) to[out=180,in=270] (1,-1);
\end{tikzpicture} 
=
\funcg_{\qpar}(a_{i}-1)\cdot
\begin{tikzpicture}[anchorbase,scale=0.25,tinynodes]
\draw[pJW] (2,0.5) rectangle (-1,-0.5);
\draw[pJW] (2,3.5) rectangle (-1,4.5);
\node at (0.5,-0.2) {$\pjwm[x{-}1]$};
\node at (0.5,3.8) {$\pjwm[x{-}1]$};
\draw[usual] (0,0.5) to[out=90,in=180] (0.5,1) to[out=0,in=90] (1,0.5);
\draw[usual] (0,3.5) to[out=270,in=180] (0.5,3) to[out=0,in=270] (1,3.5);
\draw[usual] (-0.5,0.5) to (-0.5,3.5);
\draw[usual] (1.5,4.5) to[out=90,in=180] (2,5) to[out=0,in=90] 
(2.5,4.5) to (2.5,-0.5) 
to[out=270,in=0] (2,-1) to[out=180,in=270] (1.5,-0.5);
\end{tikzpicture}
+
\funcf_{\qpar}(a_{i}-1)\cdot
\begin{tikzpicture}[anchorbase,scale=0.25,tinynodes]
\draw[pJW] (2,0.5) rectangle (-1,-0.5);
\draw[pJW] (2,3.5) rectangle (-1,4.5);
\node at (0.5,-0.2) {$\pjwm[x{-}1]$};
\node at (0.5,3.8) {$\pjwm[x{-}1]$};
\draw[usual] (0,0.5) to[out=90,in=180] (0.5,1) to[out=0,in=90] (1,0.5);
\draw[usual] (-0.25,0.5) to[out=90,in=180] (0.5,1.25) to[out=0,in=90] (1.25,0.5);
\draw[usual] (0,3.5) to[out=270,in=180] (0.5,3) to[out=0,in=270] (1,3.5);
\draw[usual] (-0.25,3.5) to[out=270,in=180] (0.5,2.75) to[out=0,in=270] (1.25,3.5);
\draw[usual] (-0.5,0.5) to (-0.5,3.5);
\node at (1,3.5) {$\phantom{a}$};
\node at (1,-0.5) {$\phantom{a}$};
\draw[usual] (1.5,4.5) to[out=90,in=180] (2,5) to[out=0,in=90] 
(2.5,4.5) to (2.5,-0.5) 
to[out=270,in=0] (2,-1) to[out=180,in=270] (1.5,-0.5);
\end{tikzpicture}
=
\funcg_{\qpar}(a_{i}-1)\cdot
\begin{tikzpicture}[anchorbase,scale=0.25,tinynodes]
\draw[pJW] (0.75,0.5) rectangle (-1,-0.5);
\draw[pJW] (0.75,3.5) rectangle (-1,4.5);
\draw[usual] (0,0.5) to[out=90,in=180] (0.5,1) to[out=0,in=90] (1,0.5);
\draw[usual] (0,3.5) to[out=270,in=180] (0.5,3) to[out=0,in=270] (1,3.5);
\draw[usual] (-0.5,0.5) to (-0.5,3.5);
\draw[usual] (1,3.5) to (1,4.5) to[out=90,in=180] (1.5,5) to[out=0,in=90] 
(2,4.5) to (2,-0.5) 
to[out=270,in=0] (1.5,-1) to[out=180,in=270] (1,-0.5) to (1,0.5);
\end{tikzpicture}
+
\funcf_{\qpar}(a_{i}-1)\cdot
\begin{tikzpicture}[anchorbase,scale=0.25,tinynodes]
\draw[pJW] (1.025,0.5) rectangle (-1,-0.5);
\draw[pJW] (1.025,3.5) rectangle (-1,4.5);
\draw[usual] (0,0.5) to[out=90,in=180] (0.5,1) to[out=0,in=90] (1,0.5);
\draw[usual] (-0.25,0.5) to[out=90,in=180] (0.5,1.25) to[out=0,in=90] (1.25,0.5);
\draw[usual] (0,3.5) to[out=270,in=180] (0.5,3) to[out=0,in=270] (1,3.5);
\draw[usual] (-0.25,3.5) to[out=270,in=180] (0.5,2.75) to[out=0,in=270] (1.25,3.5);
\draw[usual] (-0.5,0.5) to (-0.5,3.5);
\node at (1,3.5) {$\phantom{a}$};
\node at (1,-0.5) {$\phantom{a}$};
\draw[usual] (1.25,3.5) to (1.25,4.5) to[out=90,in=180] (1.75,5) to[out=0,in=90] 
(2.25,4.5) to (2.25,-0.5) 
to[out=270,in=0] (1.75,-1) to[out=180,in=270] (1.25,-0.5) to (1.25,0.5);
\end{tikzpicture}
\; ,
\end{gather*}
which we pull straight to get the claimed partial trace formula.
\end{proof}

\begin{proof}[Proof of \fullref{lemma:tracetofusion}]
We start with a few observations. First,
\fullref{proposition:times-t1} ensures that we know how many orthogonal,
primitive idempotents to expect in the categorified fusion rule. Second, by the
same arguments as in the proof of 
\fullref{theorem:well-defined}
(however, it is easier in this case since we only need to tensor with $\tmod(1)$) the idempotents
projecting onto isotypic components are uniquely determined by the
property of absorbing $\pjw[v{-}1]\hcirc\idtl[1]$. These isotypic idempotents
are automatically orthogonal because a straightforward computation, using
\fullref{proposition:multiplicities} and \fullref{lemma:clebsch-gordan}, shows
that the isotypic components share no Weyl factors, which implies that there are
no non-zero morphisms between them by \fullref{lemma:no-delta-overlap}.

Combining these observations, it remains to show that the morphisms
$\fusidem{v}{i}$, $\fusidema{v}{i}$, $\fusidemb{v}{i}$ from
\fullref{definition:fusionidempotents} satisfy the absorption property and are indeed
idempotents whenever they appear in \eqref{eq:timest1-first}. Finally, we also
check that $\fusidema{v}{i}$ and $\fusidemb{v}{i}$ are orthogonal.

\textit{Absorption.} Let us first check that all of the candidate
idempotents appearing in \fullref{theorem:timest1} absorb
$\pjw[v{-}1]\hcirc\idtl[1]$. For $\pjw[v]$, $\fusidem{v}{0}$, $\fusidema{v}{0}$,
and $\fusidemb{v}{0}$, this follows immediately from the absorption properties of
the $\ppar\lpar$JW projectors -- see \fullref{proposition:jw-properties}.(a). For
the cases with $i>0$ we use the shortening property from
\fullref{proposition:jw-properties}.(b), {\eg} for 
terms of the form $\fusidem{v}{i}$ it suffices to observe
\begin{gather*}
\begin{tikzpicture}[anchorbase,scale=0.4,tinynodes]
\draw[pJW] (-0.5,-1.5) rectangle (3.25,-0.5);
\node at (1.375,-1.2) {$\pjwm[v{-}1]$};
\draw[pJW] (-0.5,-1.5) rectangle (2.5,-2.5);
\node at (1,-2.2) {$\pjwm[v{-}w]$};
\node at (2.7,-3.4){ \scalebox{0.75}{$w-1$}};
\draw[pJW] (-0.5,-2.5) rectangle (1.75,-3.5);
\node at (0.675,-3.2) {\scalebox{0.75}{$\pjwm[v{-}2w{+}1]$}};
\draw[usual] (2,-2.5) to[out=270,in=180] (2.5,-3) to[out=0,in=270] 
(3,-2.5) to (3,-1.5);
\draw[usual] (3.5,-0.5) to (3.5,-3.5) to[out=270,in=0] (2.5,-4) to[out=180,in=270] (1.5,-3.5);
\node at (0,-0.5){$\phantom{a}$};
\node at (0,-4){$\phantom{a}$};
\end{tikzpicture}
=
\begin{tikzpicture}[anchorbase,scale=0.4,tinynodes]
\draw[pJW] (-0.5,-1.5) rectangle (3.25,-0.5);
\node at (1.375,-1.2) {$\pjwm[v{-}1]$};
\draw[pJW] (-0.5,-2.5) rectangle (1.75,-1.5);
\node at (2.7,-2.4){ \scalebox{0.75}{$w-1$}};
\node at (0.675,-2.2) {\scalebox{0.75}{$\pjwm[v{-}2w{+}1]$}};
\draw[usual] (2,-1.5) to[out=270,in=180] (2.5,-2) to[out=0,in=270] 
(3,-1.5);
\draw[usual] (3.5,-0.5) to (3.5,-2.5) to[out=270,in=0] (2.5,-3) to[out=180,in=270] (1.5,-2.5);
\node at (0,-0.5){$\phantom{a}$};
\node at (0,-4){$\phantom{a}$};
\end{tikzpicture}
=
\begin{tikzpicture}[anchorbase,scale=0.4,tinynodes]
\draw[pJW] (-0.5,-1.5) rectangle (2.75,-0.5);
\node at (1.125,-1.2) {$\pjwm[v{-}w]$};
\draw[pJW] (-0.5,-2.5) rectangle (1.75,-1.5);
\node at (2.7,-2.4){ \scalebox{0.75}{$w-1$}};
\node at (0.675,-2.2) {\scalebox{0.75}{$\pjwm[v{-}2w{+}1]$}};
\draw[usual] (2,-1.5) to[out=270,in=180] (2.5,-2) to[out=0,in=270] 
(3,-1.5) to (3,-0.5);
\draw[usual] (3.5,-0.5) to (3.5,-2.5) to[out=270,in=0] (2.5,-3) to[out=180,in=270] (1.5,-2.5);
\node at (0,-0.5){$\phantom{a}$};
\node at (0,-4){$\phantom{a}$};
\end{tikzpicture}
=
\begin{tikzpicture}[anchorbase,scale=0.4,tinynodes]
\draw[pJW] (-0.5,-0.5) rectangle (2.5,-1.5);
\node at (1,-1.2) {$\pjwm[v{-}w]$};
\node at (3.4,-1.4){ \scalebox{0.75}{$w$}};
\draw[usual] (2,-1.5) to[out=270,in=180] (2.5,-2) to[out=0,in=270] 
(3,-1.5) to (3,-0.5);
\node at (0,-0.5){$\phantom{a}$};
\node at (0,-4){$\phantom{a}$};
\end{tikzpicture}
.
\end{gather*}
Absorption for the remaining cases, namely $\fusidema{v}{i}$ and $\fusidemb{v}{i}$
with $i>0$, can be similarly shown using shortening.

We will now verify the idempotency of the candidate expressions on a case-by-case basis.

\textit{Idempotency (and orthogonality) of $\fusidema{v}{i}$ and $\fusidemb{v}{i}$.}
We start with the \fix{terms $\fusidema{v}{i}$ and $\fusidemb{v}{i}$}, which \fix{are} defined \fix{using} the two diagrams
\begin{gather*}
\fusidema{v}{i}=\morstuff{X}+\fix{\alpha}\morstuff{Y},\quad
\fix{\fusidemb{v}{i}=\fliph[\morstuff{X}]+\beta\morstuff{Y}},\quad\text{where}\quad
\morstuff{X}=
\begin{tikzpicture}[anchorbase,scale=0.4,tinynodes]
\draw[pJW] (-0.5,-1.5) rectangle (3.2,-2.5);
\node at (1.35,-2.2) {$\pjwm[v{-}w]$};
\draw[pJW] (-0.5,-3.5) rectangle (2.5,-4.5);
\node at (1,-4.2) {$\pjwm[v{-}2w]$};
\draw[pJW] (-0.5,-5.5) rectangle (3.2,-6.5);
\node at (1.35,-6.2) {$\pjwm[v{-}w]$};
\draw[usual] (2,-2.5) to[out=270,in=180] (2.5,-2.85)node[left,xshift=-1pt,yshift=-2.5pt]{$w$} 
to[out=0,in=270] (3,-2.5);
\draw[usual] (2,-3.5) to[out=90,in=180] (2.5,-3.2) 
to (3,-3.2) to [out=0,in=270] (3.5,-1.5)
node[right,xshift=-2pt,yshift=-2pt]{$w$};
\draw[usual] (2,-4.5) to (2,-5.5);
\draw[usual] (2.5,-5.5) to[out=90,in=180] (3,-5.15) 
to[out=0,in=90] (3.5,-5.5) to (3.5,-6.5)node[right,xshift=-2pt,yshift=1pt]{$w$};
\draw[usual] (0.5,-3.5) to (0.5,-2.5);
\draw[usual] (0.5,-4.5) to (0.5,-5.5);
\end{tikzpicture}
\quad\text{and}\quad
\morstuff{Y}=
\begin{tikzpicture}[anchorbase,scale=0.4,tinynodes]
\draw[pJW] (-0.5,-1.5) rectangle (3.2,-2.5);
\node at (1.35,-2.2) {$\pjwm[v{-}w]$};
\draw[pJW] (-0.5,-3.5) rectangle (2.5,-4.5);
\node at (1,-4.2) {$\pjwm[v{-}2w]$};
\draw[pJW] (-0.5,-5.5) rectangle (3.2,-6.5);
\node at (1.35,-6.2) {$\pjwm[v{-}w]$};
\draw[usual] (2,-2.5) to[out=270,in=180] (2.5,-2.85)node[left,xshift=-1pt,yshift=-2.5pt]{$w$} 
to[out=0,in=270] (3,-2.5);
\draw[usual] (2,-3.5) to[out=90,in=180] (2.5,-3.2) 
to (3,-3.2) to [out=0,in=270] (3.5,-1.5)
node[right,xshift=-2pt,yshift=-2pt]{$w$};
\draw[usual] (2,-4.5) to[out=270,in=180] (2.5,-4.8) 
to (3,-4.8) to [out=0,in=90] (3.5,-6.5)
node[right,xshift=-2pt,yshift=1pt]{$w$};
\draw[usual] (2,-5.5) to[out=90,in=180] (2.5,-5.15)node[left,xshift=-1pt]{$w$}
to[out=0,in=90] (3,-5.5);
\draw[usual] (0.5,-3.5) to (0.5,-2.5);
\draw[usual] (0.5,-4.5) to (0.5,-5.5);
\end{tikzpicture}
,
\end{gather*}
in which $w=\ppar^{(i)}$. Next we compute the pairwise products of
$\morstuff{X}$, $\fliph[\morstuff{X}]$, 
and $\morstuff{Y}$. First we use shortening
and absorption of mixed JW projectors to compute
\begin{gather*}
\morstuff{X}^{2}=
\begin{tikzpicture}[anchorbase,scale=0.4,tinynodes]
\draw[pJW] (-0.5,-1.5) rectangle (3.2,-2.5);
\node at (1.35,-2.2) {$\pjwm[v{-}w]$};
\draw[pJW] (-0.5,-3.5) rectangle (2.5,-4.5);
\node at (1,-4.2) {$\pjwm[v{-}2w]$};
\draw[pJW] (-0.5,-5.5) rectangle (3.2,-6.5);
\node at (1.35,-6.2) {$\pjwm[v{-}w]$};
\draw[usual] (2,-2.5) to[out=270,in=180] (2.5,-2.85)node[left,xshift=-1pt,yshift=-2.5pt]{$w$} 
to[out=0,in=270] (3,-2.5);
\draw[usual] (2,-3.5) to[out=90,in=180] (2.5,-3.2) 
to (3,-3.2) to [out=0,in=270] (3.5,-1.5)node[right,xshift=-2pt,yshift=-2pt]{$w$};
\draw[usual] (2,-4.5) to (2,-5.5);
\draw[usual]  (3,-5.5) to[out=90,in=180] (3.25,-5.15) node[above,yshift=-2pt]{$w$}
to[out=0,in=90] (3.5,-5.5) to (3.5,-6.5);
\draw[usual] (0.5,-3.5) to (0.5,-2.5);
\draw[usual] (0.5,-4.5) to (0.5,-5.5);
\draw[pJW] (-0.5,-6.5) rectangle (3.2,-7.5);
\node at (1.35,-7.2) {$\pjwm[v{-}w]$};
\draw[pJW] (-0.5,-8.5) rectangle (2.5,-9.5);
\node at (1,-9.2) {$\pjwm[v{-}2w]$};
\draw[pJW] (-0.5,-10.5) rectangle (3.2,-11.5);
\node at (1.35,-11.2) {$\pjwm[v{-}w]$};
\draw[usual] (2,-7.5) to[out=270,in=180] (2.5,-7.85) 
to[out=0,in=270] (3,-7.5);
\draw[usual] (2,-8.5) to[out=90,in=180] (2.5,-8.2) 
to (3,-8.2) to [out=0,in=270] (3.5,-6.5);
\draw[usual] (2,-9.5) to (2,-10.5);
\draw[usual] (2.5,-10.5) to[out=90,in=180] (3,-10.15) 
to[out=0,in=90] (3.5,-10.5) to (3.5,-11.5)node[right,xshift=-2pt,yshift=1pt]{$w$};
\draw[usual] (0.5,-8.5) to (0.5,-7.5);
\draw[usual] (0.5,-9.5) to (0.5,-10.5);
\end{tikzpicture}
=
\begin{tikzpicture}[anchorbase,scale=0.4,tinynodes]
\draw[pJW] (-0.5,-1.5) rectangle (3.2,-2.5);
\node at (1.35,-2.2) {$\pjwm[v{-}w]$};
\draw[pJW] (-0.5,-3.5) rectangle (2.5,-4.5);
\node at (1,-4.2) {$\pjwm[v{-}2w]$};
\draw[usual] (2,-2.5) to[out=270,in=180] (2.5,-2.85)node[left,xshift=-1pt,yshift=-2.5pt]{$w$} 
to[out=0,in=270] (3,-2.5);
\draw[usual] (2,-3.5) to[out=90,in=180] (2.5,-3.2) 
to (3,-3.2) to [out=0,in=270] (3.5,-1.5)node[right,xshift=-2pt,yshift=-2pt]{$w$};
\draw[usual] (2,-4.5) to (2,-6.5);
\draw[usual]  (3,-5.5) to[out=90,in=180] (3.25,-5.15) node[above,yshift=-2pt]{$w$}
to[out=0,in=90] (3.5,-5.5) to (3.5,-6.5);
\draw[usual] (0.5,-3.5) to (0.5,-2.5);
\draw[usual] (0.5,-4.5) to (0.5,-6.5);
\draw[pJW] (-0.5,-6.5) rectangle (2.5,-7.5);
\node at (1,-7.2) {$\pjwm[v{-}2w]$};
\draw[pJW] (-0.5,-8.5) rectangle (2.5,-9.5);
\node at (1,-9.2) {$\pjwm[v{-}2w]$};
\draw[pJW] (-0.5,-10.5) rectangle (3.2,-11.5);
\node at (1.35,-11.2) {$\pjwm[v{-}w]$};
\draw[usual] (2,-7.5) to[out=270,in=180] (2.5,-7.85) 
to[out=0,in=270] (3,-7.5) to (3,-5.5);
\draw[usual] (2,-8.5) to[out=90,in=180] (2.5,-8.2) 
to (3,-8.2) to [out=0,in=270] (3.5,-6.5);
\draw[usual] (2,-9.5) to (2,-10.5);
\draw[usual] (2.5,-10.5) to[out=90,in=180] (3,-10.15) 
to[out=0,in=90] (3.5,-10.5) to (3.5,-11.5)node[right,xshift=-2pt,yshift=1pt]{$w$};
\draw[usual] (0.5,-8.5) to (0.5,-7.5);
\draw[usual] (0.5,-9.5) to (0.5,-10.5);
\end{tikzpicture}
=\morstuff{X}
\quad\text{and}\quad
\morstuff{X}\morstuff{Y} =
\begin{tikzpicture}[anchorbase,scale=0.4,tinynodes]
\draw[pJW] (-0.5,-1.5) rectangle (3.2,-2.5);
\node at (1.35,-2.2) {$\pjwm[v{-}w]$};
\draw[pJW] (-0.5,-3.5) rectangle (2.5,-4.5);
\node at (1,-4.2) {$\pjwm[v{-}2w]$};
\draw[pJW] (-0.5,-5.5) rectangle (3.2,-6.5);
\node at (1.35,-6.2) {$\pjwm[v{-}w]$};
\draw[usual] (2,-2.5) to[out=270,in=180] (2.5,-2.85)node[left,xshift=-1pt,yshift=-2.5pt]{$w$} 
to[out=0,in=270] (3,-2.5);
\draw[usual] (2,-3.5) to[out=90,in=180] (2.5,-3.2) 
to (3,-3.2) to [out=0,in=270] (3.5,-1.5)node[right,xshift=-2pt,yshift=-2pt]{$w$};
\draw[usual] (2,-4.5) to (2,-5.5);
\draw[usual]  (3,-5.5) to[out=90,in=180] (3.25,-5.15) node[above,yshift=-2pt]{$w$}
to[out=0,in=90] (3.5,-5.5) to (3.5,-6.5);
\draw[usual] (0.5,-3.5) to (0.5,-2.5);
\draw[usual] (0.5,-4.5) to (0.5,-5.5);
\draw[pJW] (-0.5,-6.5) rectangle (3.2,-7.5);
\node at (1.35,-7.2) {$\pjwm[v{-}w]$};
\draw[pJW] (-0.5,-8.5) rectangle (2.5,-9.5);
\node at (1,-9.2) {$\pjwm[v{-}2w]$};
\draw[pJW] (-0.5,-10.5) rectangle (3.2,-11.5);
\node at (1.35,-11.2) {$\pjwm[v{-}w]$};
\draw[usual] (2,-7.5) to[out=270,in=180] (2.5,-7.85) 
to[out=0,in=270] (3,-7.5);
\draw[usual] (2,-8.5) to[out=90,in=180] (2.5,-8.2) 
to (3,-8.2) to [out=0,in=270] (3.5,-6.5);
\draw[usual] (0.5,-8.5) to (0.5,-7.5);
\draw[usual] (2,-9.5) to[out=270,in=180] (2.5,-9.8) 
to (3,-9.8) to [out=0,in=90] (3.5,-11.5)
node[right,xshift=-2pt,yshift=1pt]{$w$};
\draw[usual] (2,-10.5) to[out=90,in=180] 
(2.5,-10.15)node[left,xshift=-1pt]{$w$}
to[out=0,in=90] (3,-10.5);
\end{tikzpicture}
=
\begin{tikzpicture}[anchorbase,scale=0.4,tinynodes]
\draw[pJW] (-0.5,-1.5) rectangle (3.2,-2.5);
\node at (1.35,-2.2) {$\pjwm[v{-}w]$};
\draw[pJW] (-0.5,-3.5) rectangle (2.5,-4.5);
\node at (1,-4.2) {$\pjwm[v{-}2w]$};
\draw[usual] (2,-2.5) to[out=270,in=180] (2.5,-2.85)node[left,xshift=-1pt,yshift=-2.5pt]{$w$} 
to[out=0,in=270] (3,-2.5);
\draw[usual] (2,-3.5) to[out=90,in=180] (2.5,-3.2) 
to (3,-3.2) to [out=0,in=270] (3.5,-1.5)node[right,xshift=-2pt,yshift=-2pt]{$w$};
\draw[usual] (2,-4.5) to (2,-6.5);
\draw[usual]  (3,-5.5) to[out=90,in=180] (3.25,-5.15) node[above,yshift=-2pt]{$w$}
to[out=0,in=90] (3.5,-5.5) to (3.5,-6.5);
\draw[usual] (0.5,-3.5) to (0.5,-2.5);
\draw[usual] (0.5,-4.5) to (0.5,-6.5);
\draw[pJW] (-0.5,-6.5) rectangle (2.5,-7.5);
\node at (1,-7.2) {$\pjwm[v{-}2w]$};
\draw[pJW] (-0.5,-8.5) rectangle (2.5,-9.5);
\node at (1,-9.2) {$\pjwm[v{-}2w]$};
\draw[pJW] (-0.5,-10.5) rectangle (3.2,-11.5);
\node at (1.35,-11.2) {$\pjwm[v{-}w]$};
\draw[usual] (2,-7.5) to[out=270,in=180] (2.5,-7.85) 
to[out=0,in=270] (3,-7.5) to (3,-5.5);
\draw[usual] (2,-8.5) to[out=90,in=180] (2.5,-8.2) 
to (3,-8.2) to [out=0,in=270] (3.5,-6.5);
\draw[usual] (0.5,-8.5) to (0.5,-7.5);
\draw[usual] (2,-9.5) to[out=270,in=180] (2.5,-9.8) 
to (3,-9.8) to [out=0,in=90] (3.5,-11.5)
node[right,xshift=-2pt,yshift=1pt]{$w$};
\draw[usual] (2,-10.5) to[out=90,in=180] 
(2.5,-10.15)node[left,xshift=-1pt]{$w$}
to[out=0,in=90] (3,-10.5);
\end{tikzpicture}
=\morstuff{Y}.
\end{gather*}
Symmetrically, we also have 
$(\fliph[\morstuff{X}])^{2}=\fliph[\morstuff{X}]$ and
$\morstuff{Y}(\fliph[\morstuff{X}])=\morstuff{Y}$. 
Now we claim that \fix{the following}
products are zero, namely
$(\fliph[\morstuff{X}])\morstuff{X}=\morstuff{Y}\morstuff{X}=
(\fliph[\morstuff{X}])\morstuff{Y}=\morstuff{Y}^{2}=0$.
This can be seen as follows. Up to symmetry, 
these statements all follow from
\begin{gather*}
\begin{tikzpicture}[anchorbase,scale=0.4,tinynodes]
\draw[pJW] (-0.5,-3.5) rectangle (2.5,-4.5);
\node at (1,-4.2) {$\pjwm[v{-}2w]$};
\draw[pJW] (-0.5,-5.5) rectangle (3.5,-6.5);
\node at (1.5,-6.2) {$\pjwm[v{-}w]$};
\draw[usual] (2,-4.5) to[out=270,in=180] (2.5,-4.8) 
to (3,-4.8) to [out=0,in=90] (4,-6.5) to (4,-7.5)
node[right,xshift=-1pt,yshift=1pt]{$w$};
\draw[usual] (2,-5.5) to[out=90,in=180] (2.5,-5.2)
to[out=0,in=90] (3,-5.5);
\draw[usual] (3,-6.5) to[out=270,in=0] (2.5,-6.8)node[below]{$w$} 
to[out=180,in=270] (2,-6.5);
\draw[usual] (0.5,-4.5) to (0.5,-5.5);
\draw[usual] (0.5,-6.5) to (0.5,-7.5);
\end{tikzpicture}
=0
\quad \Leftrightarrow\quad 
\begin{tikzpicture}[anchorbase,scale=0.4,tinynodes]
\draw[pJW] (-0.5,-3.5) rectangle (2.5,-4.5);
\node at (1,-4.2) {$\pjwm[v{-}2w]$};
\draw[pJW] (-0.5,-5.5) rectangle (3.5,-6.5);
\node at (1.5,-6.2) {$\pjwm[v{-}w]$};
\draw[usual] (2,-4.5) to[out=270,in=180] (2.5,-4.8) 
to (3,-4.8) to [out=0,in=90] (4,-6.5) to[out=270,in=180] (4.5,-6.8) 
to[out=0,in=270] (5,-6.5) to (5,-3.5)node[right,xshift=-1pt,yshift=-5pt]{$w$};
\draw[usual] (2,-5.5) to[out=90,in=180] (2.5,-5.2)
to[out=0,in=90] (3,-5.5);
\draw[usual] (3,-6.5) to[out=270,in=0] (2.5,-6.8)node[below]{$w$} 
to[out=180,in=270] (2,-6.5);
\draw[usual] (0.5,-4.5) to (0.5,-5.5);
\draw[usual] (0.5,-6.5) to (0.5,-7.5);
\end{tikzpicture}
=
\begin{tikzpicture}[anchorbase,scale=0.4,tinynodes]
\draw[pJW] (-0.5,-3.5) rectangle (3.5,-4.5);
\node at (1.5,-4.2) {$\pjwm[v{-}w]$};
\draw[pJW] (-0.5,-5.5) rectangle (3.5,-6.5);
\node at (1.5,-6.2) {$\pjwm[v{-}w]$};
\draw[usual] (2,-4.5) to[out=270,in=180] (2.5,-4.8) 
to[out=0,in=270] (3,-4.5);
\draw[usual] (2,-5.5) to[out=90,in=180] (2.5,-5.2)node[right,yshift=0.5pt]{$w$}  
to[out=0,in=90] (3,-5.5);
\draw[usual] (3,-6.5) to[out=270,in=0] (2.5,-6.8)node[below]{$w$}  
to[out=180,in=270] (2,-6.5);
\draw[usual] (0.5,-4.5) to (0.5,-5.5);
\draw[usual] (0.5,-6.5) to (0.5,-7.5);
\end{tikzpicture}
=\pjw[v{-}w]\Up{i}\Down{i}\Up{i} = 0.
\end{gather*}
The equivalence is given by bending, as illustrated. 
On the right-hand side we undid
shortening and translated the caps and cups into morphisms $\Up{i}$ and
$\Down{i}$ respectively (this is possible since $a_{i}=1$), and then applied \eqref{eq:dud}. 
\fix{The remaining product is
$\morstuff{X}(\fliph[\morstuff{X}])=\funcg_{\qpar}(0)\morstuff{Y}$ by the zigzag relation.}
Taking all of these together shows that
\begin{gather*}
(\morstuff{X}+\fix{\alpha}\morstuff{Y})^{2}
=
\morstuff{X}+\fix{\alpha}\morstuff{Y}
,\quad
(\fliph[\morstuff{X}]+\fix{\beta}\morstuff{Y})^{2}
=
\fliph[\morstuff{X}]+\fix{\beta}\morstuff{Y}
,\\
(\morstuff{X}+\fix{\alpha}\morstuff{Y})
(\fliph[\morstuff{X}]+\fix{\beta}\morstuff{Y})
=
\fix{(\funcg_{\qpar}(0)+\alpha+\beta)\morstuff{Y}=0},
\\
(\fliph[\morstuff{X}]+\fix{\beta}\morstuff{Y})
(\morstuff{X}+\fix{\alpha}\morstuff{Y})=0,
\end{gather*}
which expresses $\fusidema{v}{i}$ and $\fusidemb{v}{i}$ as orthogonal idempotents.

\textit{Idempotency of $\fusidem{v}{i}$.} Next we check that the terms
$\fusidem{v}{i}$ are idempotents. Recall that $\fusidem{v}{i}$ for $i\neq j$ is
defined as a linear combination of the following two morphisms
\begin{align*}
\morstuff{X}=
\begin{tikzpicture}[anchorbase,scale=0.4,tinynodes]
\draw[pJW] (-0.5,-1.5) rectangle (2.5,-2.5);
\node at (1,-2.2) {$\pjwm[v{-}w]$};
\draw[pJW] (-0.5,-4.5) rectangle (2.5,-5.5);
\node at (1,-5.2) {$\pjwm[v{-}w]$};
\draw[pJW] (-0.5,-3.25) rectangle (0.5,-3.75);
\draw[usual] (2,-2.5) to[out=270,in=180] (2.5,-3) to[out=0,in=270] 
(3,-2.5) to (3,-1.5) node[above]{$w$};
\draw[usual] (2,-4.5) to[out=90,in=180] (2.5,-4) to[out=0,in=90] 
(3,-4.5) to (3,-5.5) node[below]{$w$};
\draw[usual] (0,-2.5) to (0,-3.25);
\draw[usual] (0,-3.75) to (0,-4.5);
\end{tikzpicture}
=
\begin{tikzpicture}[anchorbase,scale=0.4,tinynodes]
\draw[pJW] (-0.5,-1.5) rectangle (2.5,-2.5);
\node at (1,-2.2) {$\pjwm[v{-}w]$};
\draw[pJW] (-0.5,-4.5) rectangle (2.5,-5.5);
\node at (1,-5.2) {$\pjwm[v{-}w]$};
\draw[usual] (2,-2.5) to[out=270,in=180] (2.5,-3) to[out=0,in=270] 
(3,-2.5) to (3,-1.5) node[above]{$w$};
\draw[usual] (2,-4.5) to[out=90,in=180] (2.5,-4) to[out=0,in=90] 
(3,-4.5) to (3,-5.5) node[below]{$w$};
\draw[usual] (0,-2.5) to (0,-4.5);
\end{tikzpicture}
\quad\text{and}\quad
\morstuff{Y}=
\begin{tikzpicture}[anchorbase,scale=0.4,tinynodes]
\draw[pJW] (-0.5,-1.5) rectangle (2.5,-2.5);
\node at (1,-2.2) {$\pjwm[v{-}w]$};
\draw[pJW] (-0.5,-4.5) rectangle (2.5,-5.5);
\node at (1,-5.2) {$\pjwm[v{-}w]$};
\draw[pJW] (-0.5,-3.25) rectangle (0.5,-3.75);
\draw[usual] (2,-2.5) to[out=270,in=180] (2.5,-3) to[out=0,in=270] 
(3,-2.5) to (3,-1.5) node[above]{$w$};
\draw[usual] (2,-4.5) to[out=90,in=180] (2.5,-4) to[out=0,in=90] 
(3,-4.5) to (3,-5.5) node[below]{$w$};
\draw[usual] (0,-2.5) to (0,-3.25);
\draw[usual] (0,-3.75) to (0,-4.5);
\draw[usual] (0.5,-4.5) to[out=90,in=180] (1,-4) to[out=0,in=90] 
(1.5,-4.5);
\draw[usual] (0.5,-2.5) to[out=270,in=180] (1,-3) to[out=0,in=270] 
(1.5,-2.5);
\node at (1,-3.65) {$S_{i}$};
\end{tikzpicture}
=
\begin{tikzpicture}[anchorbase,scale=0.4,tinynodes]
\draw[pJW] (-0.5,-1.5) rectangle (2.5,-2.5);
\node at (1,-2.2) {$\pjwm[v{-}w]$};
\draw[pJW] (-0.5,-4.5) rectangle (2.5,-5.5);
\node at (1,-5.2) {$\pjwm[v{-}w]$};
\draw[usual] (2,-2.5) to[out=270,in=180] (2.5,-3) to[out=0,in=270] 
(3,-2.5) to (3,-1.5) node[above]{$w$};
\draw[usual] (2,-4.5) to[out=90,in=180] (2.5,-4) to[out=0,in=90] 
(3,-4.5) to (3,-5.5) node[below]{$w$};
\draw[usual] (0,-2.5) to (0,-4.5);
\draw[usual] (0.5,-4.5) to[out=90,in=180] (1,-4) to[out=0,in=90] 
(1.5,-4.5);
\draw[usual] (0.5,-2.5) to[out=270,in=180] (1,-3) to[out=0,in=270] 
(1.5,-2.5);
\node at (1,-3.65) {$S_{i}$};
\end{tikzpicture}
, 
\end{align*}
where we also write $w=\ppar^{(i)}$. For $i=j$, we use the same definition for
$\morstuff{X}$, but set $\morstuff{Y}=0$. To compute the various products of
these elements, we will use the following partial trace rule for $\pjw[v{-}w]$ from
\eqref{eq:ptrstatement}, which holds by assumption $\trp(v)$:
\begin{gather*}
\begin{tikzpicture}[anchorbase,scale=0.25,tinynodes]
\draw[pJW] (1.5,-1) rectangle (-2.5,1);
\node at (-0.5,-0.1) {$\pjwm[v{-}w]$};
\draw[usual] (1,1) to[out=90,in=180] (1.5,1.5) to[out=0,in=90] 
(2,1) to (2,0) node[right,xshift=-2pt]{$w$} to (2,-1) 
to[out=270,in=0] (1.5,-1.5) to[out=180,in=270] (1,-1);
\node at (0,2.5) {$\phantom{a}$};
\node at (0,-2.5) {$\phantom{a}$};
\end{tikzpicture} 
=
\funcg_{\qpar}(a_{i}-1)\cdot
\begin{tikzpicture}[anchorbase,scale=0.25,tinynodes]
\draw[pJW] (1.5,-1) rectangle (-2.5,1);
\node at (-0.5,-0.1) {$\pjwm[v{-}2w]$};
\node at (0,2.5) {$\phantom{a}$};
\node at (0,-2.5) {$\phantom{a}$};
\end{tikzpicture}
+\funcf_{\qpar}(a_{i}-1)\cdot
\loopdown{i}{v{-}2w}.
\end{gather*}
Since we know from \eqref{eq:dud} that the loop $\loopdown{i}{v{-}2w}$ is annihilated 
by postcomposing with
another down morphism $\Down{i}\pjw[v{-}2w]$, we 
also obtain from \eqref{eq:ptrclaim} that
\begin{gather*}
\begin{tikzpicture}[anchorbase,scale=0.25,tinynodes]
\draw[pJW] (1.5,-1) rectangle (-2.5,1);
\node at (-0.5,-0.1) {$\pjwm[v{-}w]$};
\draw[usual] (1,1) to[out=90,in=180] (1.5,1.5) to[out=0,in=90] 
(2,1) to (2,0)node[right,xshift=-2pt]{$w$} to (2,-1) 
to[out=270,in=0] (1.5,-1.5) to[out=180,in=270] (1,-1);
\draw[usual] (0.5,1) to[out=90,in=0] (0,1.5) to[out=180,in=90] (-0.5,1);
\node at (0,2.5) {$\phantom{a}$};
\node at (0,-2.5) {$\phantom{a}$};
\end{tikzpicture}=
\funcg_{\qpar}(a_{i}-1)\cdot
\begin{tikzpicture}[anchorbase,scale=0.25,tinynodes]
\draw[pJW] (1.5,-1) rectangle (-2.5,1);
\node at (-0.5,-0.1) {$\pjwm[v{-}2w]$};
\draw[usual] (0.5,1) to[out=90,in=0] (0,1.5) to[out=180,in=90] (-0.5,1);
\node at (0,2.5) {$\phantom{a}$};
\node at (0,-2.5) {$\phantom{a}$};
\end{tikzpicture}
\;\quad\text{and}\quad
\begin{tikzpicture}[anchorbase,scale=0.25,tinynodes]
\draw[pJW] (1.5,-1) rectangle (-2.5,1);
\node at (-0.5,-0.1) {$\pjwm[v{-}w]$};
\draw[usual] (1,1) to[out=90,in=180] (1.5,1.5) to[out=0,in=90] 
(2,1) to (2,0)node[right,xshift=-2pt]{$w$} to (2,-1) 
to[out=270,in=0] (1.5,-1.5) to[out=180,in=270] (1,-1);
\draw[usual] (0.5,1) to[out=90,in=0] (0,1.5) to[out=180,in=90] (-0.5,1);
\draw[usual] (0.5,-1) to[out=270,in=0] (0,-1.5) to[out=180,in=270] (-0.5,-1);
\node at (0,2.5) {$\phantom{a}$};
\node at (0,-2.5) {$\phantom{a}$};
\end{tikzpicture}
=
\funcg_{\qpar}(a_{i}-1)\cdot
\begin{tikzpicture}[anchorbase,scale=0.25,tinynodes]
\draw[pJW] (1.5,-1) rectangle (-2.5,1);
\node at (-0.5,-0.1) {$\pjwm[v{-}2w]$};
\draw[usual] (0.5,1) to[out=90,in=0] (0,1.5) to[out=180,in=90] (-0.5,1);
\draw[usual] (0.5,-1) to[out=270,in=0] (0,-1.5) to[out=180,in=270] (-0.5,-1);
\node at (0,2.5) {$\phantom{a}$};
\node at (0,-2.5) {$\phantom{a}$};
\end{tikzpicture}.
\end{gather*}
Thus we calculate
\begin{align*}
\morstuff{X}^{2}&=
\funcg_{\qpar}(a_{i}-1)\cdot
\begin{tikzpicture}[anchorbase,scale=0.4,tinynodes]
\draw[pJW] (-0.5,-1.5) rectangle (2.5,-2.5);
\node at (1,-2.2) {$\pjwm[v{-}w]$};
\draw[pJW] (-0.5,-4.5) rectangle (2.5,-5.5);
\node at (1,-5.2) {$\pjwm[v{-}w]$};
\draw[usual] (2,-2.5) to[out=270,in=180] (2.5,-3) to[out=0,in=270] 
(3,-2.5) to (3,-1.5) node[above]{$w$};
\draw[usual] (2,-4.5) to[out=90,in=180] (2.5,-4) to[out=0,in=90] 
(3,-4.5) to (3,-5.5) node[below]{$w$};
\draw[usual] (0,-2.5) to (0,-4.5);
\end{tikzpicture}
+\funcf_{\qpar}(a_{0}-1)\cdot
\begin{tikzpicture}[anchorbase,scale=0.4,tinynodes]
\draw[pJW] (-0.5,-1.5) rectangle (2.5,-2.5);
\node at (1,-2.2) {$\pjwm[v{-}w]$};
\draw[pJW] (-0.5,-4.5) rectangle (2.5,-5.5);
\node at (1,-5.2) {$\pjwm[v{-}w]$};
\draw[usual] (2,-2.5) to[out=270,in=180] (2.5,-3) to[out=0,in=270] 
(3,-2.5) to (3,-1.5) node[above]{$w$};
\draw[usual] (2,-4.5) to[out=90,in=180] (2.5,-4) to[out=0,in=90] 
(3,-4.5) to (3,-5.5) node[below]{$w$};
\draw[usual] (0,-2.5) to (0,-4.5);
\draw[usual] (0.5,-4.5) to[out=90,in=180] (1,-4) to[out=0,in=90] 
(1.5,-4.5);
\draw[usual] (0.5,-2.5) to[out=270,in=180] (1,-3) to[out=0,in=270] 
(1.5,-2.5);
\node at (1,-3.65) {$S_{i}$};
\end{tikzpicture}
=
\funcg_{\qpar}(a_{i}-1)\morstuff{X}
+\funcf_{\qpar}(a_{i}-1)\morstuff{Y}
,
\\
\morstuff{X}\morstuff{Y}&=
\funcg_{\qpar}(a_{i}-1)\cdot
\begin{tikzpicture}[anchorbase,scale=0.4,tinynodes]
\draw[pJW] (-0.5,-1.5) rectangle (2.5,-2.5);
\node at (1,-2.2) {$\pjwm[v{-}w]$};
\draw[pJW] (-0.5,-4.5) rectangle (2.5,-5.5);
\node at (1,-5.2) {$\pjwm[v{-}w]$};
\draw[usual] (2,-2.5) to[out=270,in=180] (2.5,-3) to[out=0,in=270] 
(3,-2.5) to (3,-1.5) node[above]{$w$};
\draw[usual] (2,-4.5) to[out=90,in=180] (2.5,-4) to[out=0,in=90] 
(3,-4.5) to (3,-5.5) node[below]{$w$};
\draw[usual] (0,-2.5) to (0,-4.5);
\draw[usual] (0.5,-4.5) to[out=90,in=180] (1,-4) to[out=0,in=90] 
(1.5,-4.5);
\draw[usual] (0.5,-2.5) to[out=270,in=180] (1,-3) to[out=0,in=270] 
(1.5,-2.5);
\node at (1,-3.65) {$S_{i}$};
\end{tikzpicture}
=\funcg_{\qpar}(a_{i}-1)\morstuff{Y}
\quad\text{and}\quad
\morstuff{Y}^{2}
=0.
\end{align*}
The latter holds since $(\loopdown{i}{v{-}2w})^{2}=0$, which follows from \eqref{eq:dud}. (Observe that the above relations also hold in the special case
$i=j$ where $\morstuff{Y}=0$.) Now we verify that $\fusidem{v}{i}$ is an
idempotent:
\begin{align*}
(\fusidem{v}{i})^{2}=\big(\tfrac{1}{\funcg_{\qpar}(a_{i}-1)}\morstuff{X}
-\tfrac{\funcf_{\qpar}(a_{i})}{\funcg_{\qpar}(a_{i}-1)}\morstuff{Y}\big)^{2}
&=
\tfrac{1}{\funcg_{\qpar}(a_{i}-1)^{2}}\morstuff{X}^{2}
-
\tfrac{\funcf_{\qpar}(a_{i})}{\funcg_{\qpar}(a_{i}-1)^{2}}
(\morstuff{X}\morstuff{Y}+\morstuff{Y}\morstuff{X})
+0
\\
&=
\tfrac{1}{\funcg_{\qpar}(a_{i}-1)}\morstuff{X} 
+
\tfrac{\funcf_{\qpar}(a_{i}-1)}{\funcg_{\qpar}(a_{i}-1)^{2}} \morstuff{Y} 
-2\tfrac{\funcf_{\qpar}(a_{i})}{\funcg_{\qpar}(a_{i}-1)}\morstuff{Y}
\\
&=\tfrac{1}{\funcg_{\qpar}(a_{i}-1)}\morstuff{X}
-\tfrac{\funcf_{\qpar}(a_{i})}{\funcg_{\qpar}(a_{i}-1)}\morstuff{Y} = \fusidem{v}{i},
\end{align*}
where we have used 
$\morstuff{X}\morstuff{Y}=\morstuff{Y}\morstuff{X}$ and
$\tfrac{\funcf_{\qpar}(a_{i}-1)}{\funcg_{\qpar}(a_{i}-1)}
=\funcf_{\qpar}(a_{i})$.
\end{proof}

\ochanged{The next figures, \fullref{figure:fusion-graph2} and \fullref{figure:cartan}, illustrates the eve fusion rule and more fractal behavior, respectively.}

\begin{figure}[ht]
\includegraphics[width=0.49\textwidth]{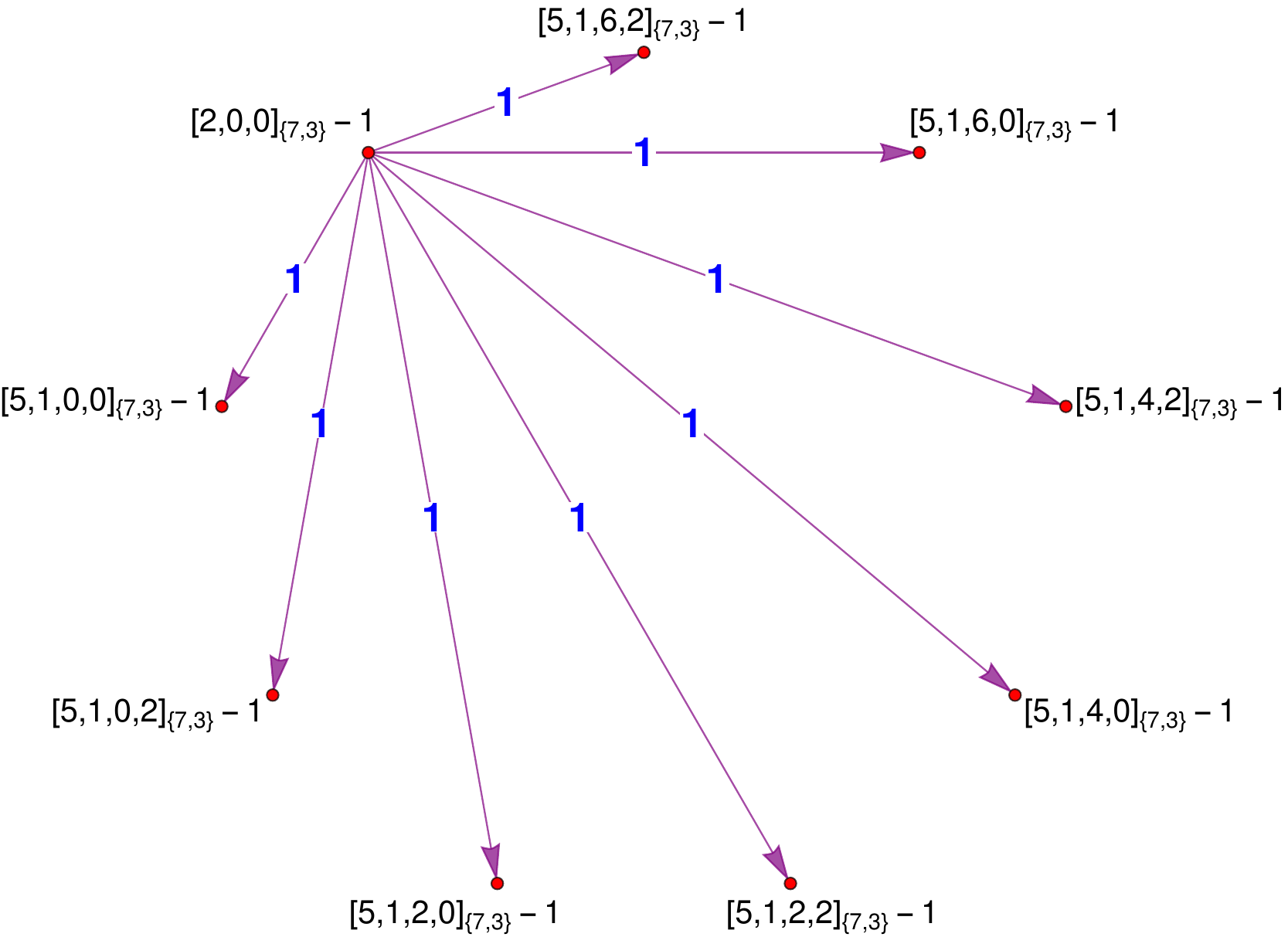}
\caption{An illustration of the fusion rule for
$\tmod\big(\pbase{5,0,0,0}{7,3}-1\big)
\hcirc\tmod\big(\pbase{2,0,0}{7,3}-1\big)$ (a subgraph of the fusion graph as
given in \fullref{figure:fusion-graph} for
$\tmod\big(\pbase{5,0,0,0}{7,3}-1\big)$).}
\label{figure:fusion-graph2}
\end{figure}

\begin{figure}[ht]
\includegraphics[width=0.49\textwidth]{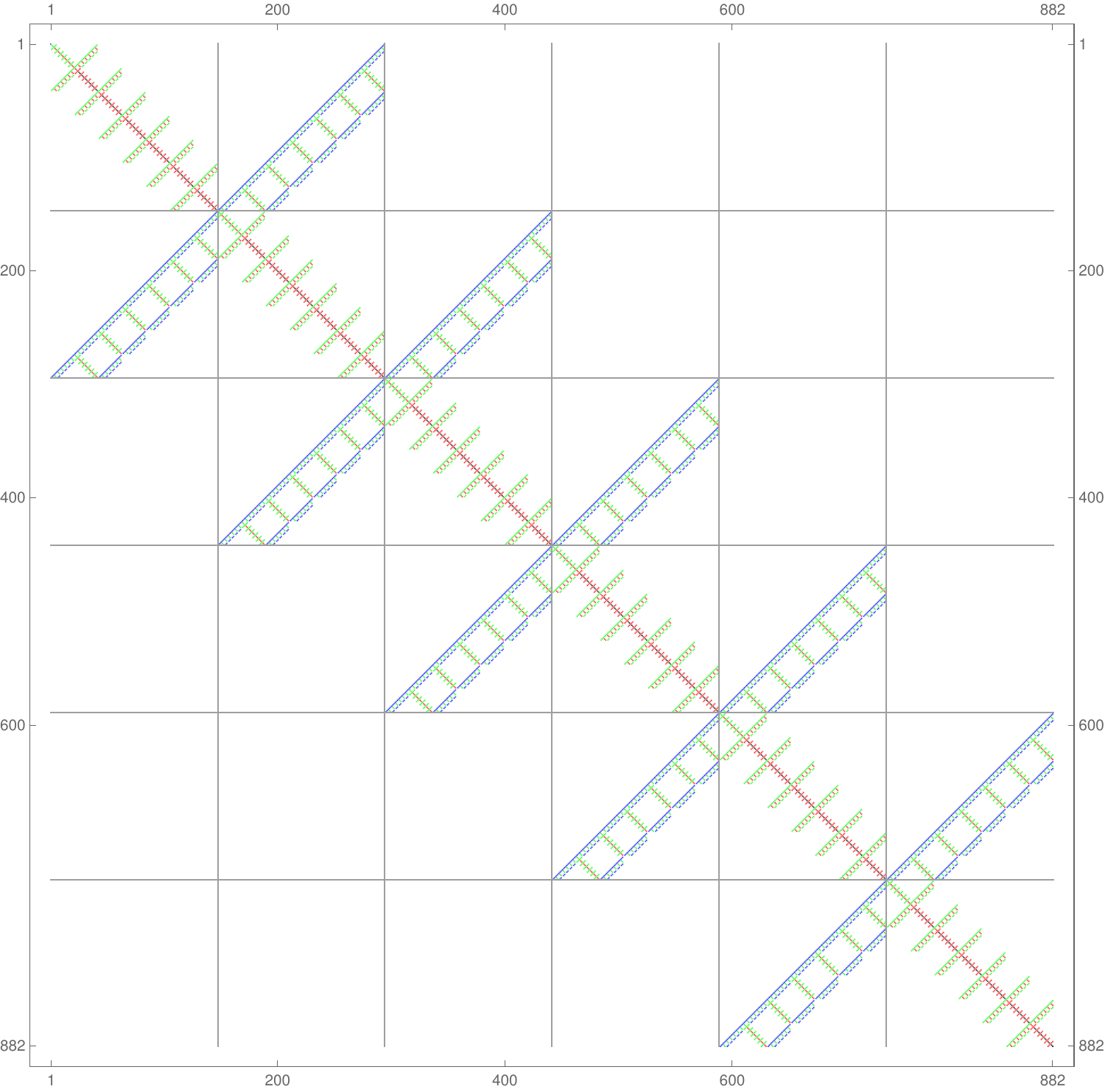}
\includegraphics[width=0.49\textwidth]{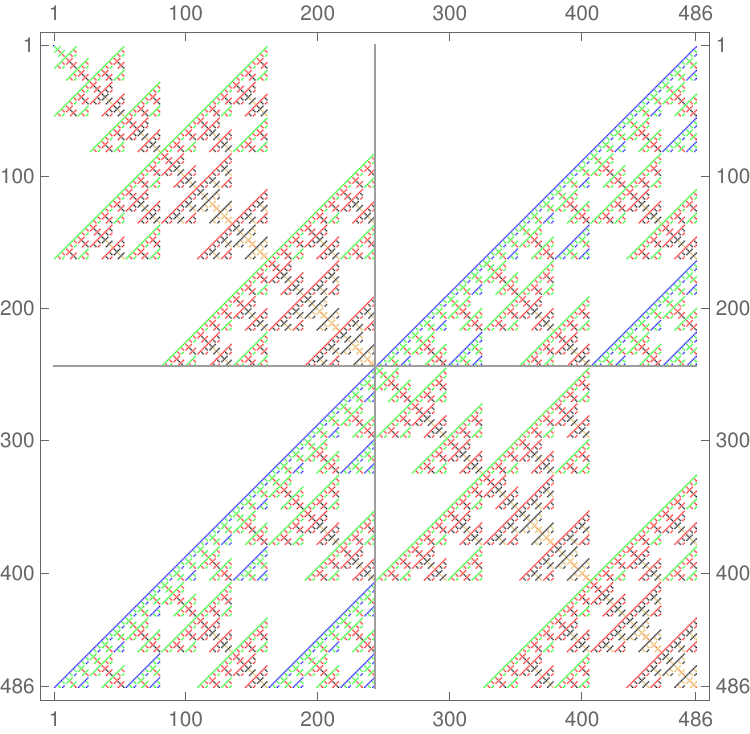}
\caption{\emph{From left to right:} The Cartan matrix of $\verlinde_{\ppar^{(4)}}$ in mixed characteristic 
$\mchar=(7,3)$; the Cartan matrix of $\verlinde_{\ppar^{(6)}}$ in characteristic $\ppar=3$. The colors indicate higher multiplicities.}
\label{figure:cartan}
\end{figure}


\newpage

\begin{thebibliography}{GKPM13}

\bibitem[Alp76]{Al-modules-fractional-groups}
J.~Alperin.
\newblock On modules for the linear fractional groups.
\newblock In {\em International Symposium on the Theory of Finite Groups},
Finite Groups, Sapporo and Kyoto, 1974, pages 157--163. Japan Society for the
Promotion of Science, 1976.

\bibitem[And19]{An-simple-tl}
H.H.~Andersen.
\newblock Simple modules for {T}emperley--{L}ieb algebras and related algebras.
\newblock {\em J. Algebra}, 520:276--308, 2019.
\newblock URL: \url{https://arxiv.org/abs/1709.00248}, \href
{https://doi.org/10.1016/j.jalgebra.2018.10.035}
{\path{doi:10.1016/j.jalgebra.2018.10.035}}.

\bibitem[And92]{An-tensor-q-tilting-modules}
H.H.~Andersen.
\newblock Tensor products of quantized tilting modules.
\newblock {\em Comm. Math. Phys.}, 149(1):149--159, 1992.

\bibitem[And20]{An-tilting-cellular}
H.H.~Andersen.
\newblock Tilting modules and cellular categories.
\newblock {\em J. Pure Appl. Algebra}, 224(9):106366, 2020.
\newblock URL: \url{https://arxiv.org/abs/1912.00817}, \href
{https://doi.org/10.1016/j.jpaa.2020.106366}
{\path{doi:10.1016/j.jpaa.2020.106366}}.

\bibitem[AJL83]{AnJoLa-sl2-projectives}
H.H.~Andersen, J.~J{\o}rgensen, and P.~Landrock.
\newblock The projective indecomposable modules of {$\mathrm{SL}(2,\,p^{n})$}.
\newblock {\em Proc. London Math. Soc. (3)}, 46(1):38--52, 1983.
\newblock URL: \href {https://doi.org/10.1112/plms/s3-46.1.38}
{\path{doi:10.1112/plms/s3-46.1.38}}.

\bibitem[AK92]{AnWe-mixed-qgroup}
H.H.~Andersen and W.X.~Kexin.
\newblock Representations of quantum algebras. {T}he mixed case.
\newblock {\em J. Reine Angew. Math.}, 427:35--50, 1992.
\newblock URL: \href {https://doi.org/10.1515/crll.1992.427.35}
{\path{doi:10.1515/crll.1992.427.35}}.

\bibitem[APW91]{AnPoWe-representation-qalgebras}
H.H.~Andersen, P.~Polo, and K.X. Wen.
\newblock Representations of quantum algebras.
\newblock {\em Invent. Math.}, 104(1):1--59, 1991.
\newblock URL: \href {https://doi.org/10.1007/BF01245066}
{\path{doi:10.1007/BF01245066}}.

\bibitem[AST15]{AnStTu-cellular-tilting-2}
H.H.~Andersen, C.~Stroppel, and D.~Tubbenhauer.
\newblock Additional notes for the paper ``{C}ellular structures using
$\textbf{U}_q$-tilting modules''.
\newblock 2015.
\newblock Draft version which is not intended for publication, eprint, URL:
\url{http://www.math.uni-bonn.de/ag/stroppel/cell-tilt-proofs_neu.pdf},
\url{http://www.dtubbenhauer.com/cell-tilt-proofs.pdf}.

\bibitem[AST17]{AnStTu-semisimple-tilting}
H.H.~Andersen, C.~Stroppel, and D.~Tubbenhauer.
\newblock Semisimplicity of {H}ecke and (walled) {B}rauer algebras.
\newblock {\em J. Aust. Math. Soc.}, 103(1):1--44, 2017.
\newblock URL: \url{https://arxiv.org/abs/1507.07676}, \href
{https://doi.org/10.1017/S1446788716000392}
{\path{doi:10.1017/S1446788716000392}}.

\bibitem[AST18]{AnStTu-cellular-tilting}
H.H.~Andersen, C.~Stroppel, and D.~Tubbenhauer.
\newblock Cellular structures using {$\mathrm{U}_q$}-tilting modules.
\newblock {\em Pacific J. Math.}, 292(1):21--59, 2018.
\newblock URL: \url{https://arxiv.org/abs/1503.00224}, \href
{https://doi.org/10.2140/pjm.2018.292.21}
{\path{doi:10.2140/pjm.2018.292.21}}.

\bibitem[AT17]{AnTu-tilting}
H.H.~Andersen and D.~Tubbenhauer.
\newblock Diagram categories for {$\textbf{U}_q$}-tilting modules at roots of
unity.
\newblock {\em Transform. Groups}, 22(1):29--89, 2017.
\newblock URL: \url{http://arxiv.org/abs/1409.2799}, \href
{https://doi.org/10.1007/s00031-016-9363-z}
{\path{doi:10.1007/s00031-016-9363-z}}.

\bibitem[BD09]{BeDo-schur-weyl-finite-fields}
D.J.~Benson and S.~Doty.
\newblock Schur--{W}eyl duality over finite fields.
\newblock {\em Arch. Math.}, 93:425--435, 2009.
\newblock URL: \url{https://arxiv.org/abs/0805.1235}, \href
{https://doi.org/10.1007/s00013-009-0066-8}
{\path{doi:10.1007/s00013-009-0066-8}}.

\bibitem[BDRM19]{BlDeReMu-dia-small-qgroup}
C.~Blanchet, M.~De~Renzi, and J.~Murakami.
\newblock Diagrammatic construction of representations of small quantum
{$\mathfrak{sl}_{2}$}.
\newblock 2019.
\newblock URL: \url{https://arxiv.org/abs/1910.12427}.

\bibitem[BEO20]{BeEtOs-pverlinde}
D.J.~Benson, P.~Etingof, and V.~Ostrik.
\newblock New incompressible symmetric tensor categories in positive
characteristic.
\newblock 2020.
\newblock URL: \url{https://arxiv.org/abs/2003.10499v3}.

\bibitem[BLS19]{BuLiSe-tl-char-p}
G.~Burrull, N.~Libedinsky, and P.~Sentinelli.
\newblock p-{J}ones--{W}enzl idempotents.
\newblock {\em Adv. Math.}, 352(20):246--264, 2019.
\newblock URL: \url{https://arxiv.org/abs/1902.00305}, \href
{https://doi.org/10.1016/j.aim.2019.06.005}
{\path{doi:10.1016/j.aim.2019.06.005}}.

\bibitem[BS18]{BrSt-semi-infinite}
J.~Brundan and C.~Stroppel.
\newblock Semi-infinite highest weight categories.
\newblock 2018.
\newblock URL: \url{https://arxiv.org/abs/1808.08022}.

\bibitem[CC76]{CaCl-submodule-weyl-a1}
R.~Carter and E.~Cline.
\newblock The submodule structure of {W}eyl modules for groups of type
{$A_{1}$}.
\newblock Proceedings of the Conference on Finite Groups (Univ. Utah, Park City, Utah, 1975), pp. 303--311. Academic Press, New York, 1976.

\bibitem[Cra13]{Cr-tensor-simple-modules}
D.A. Craven.
\newblock On tensor products of simple modules for simple groups.
\newblock {\em Algebr. Represent. Theory}, 16(2):377--404, 2013.
\newblock URL: \url{https://arxiv.org/abs/1102.3447}, \href
{https://doi.org/10.1007/s10468-011-9311-5}
{\path{doi:10.1007/s10468-011-9311-5}}.

\bibitem[DGGPR22]{DGGPR19}
M.~De~Renzi, A.~Gainutdinov, N.~Geer, B.~Patureau-Mirand, and I.~Runkel.
\newblock 3-Dimensional {TQFT}s From Non-Semisimple Modular Categories.
\newblock {\em Selecta Math. (N.S.)} 28 (2022), no. 2, Paper No. 42, 60 pp. 
\newblock URL: \url{https://arxiv.org/abs/1912.02063}, \href
{https://10.1007/s00029-021-00737-z}
{\path{doi:10.1007/s00029-021-00737-z}}.

\bibitem[DH05]{DoHe-char-p-sl2}
S.~Doty and A.~Henke.
\newblock Decomposition of tensor products of modular irreducibles for
{$\mathrm{SL}_{2}$}.
\newblock {\em Q. J. Math.}, 56(2):189--207, 2005.
\newblock URL: \url{https://arxiv.org/abs/math/0205186}, \href
{https://doi.org/10.1093/qmath/hah027} {\path{doi:10.1093/qmath/hah027}}.

\bibitem[Don93]{Do-tilting-alg-groups}
S.~Donkin.
\newblock On tilting modules for algebraic groups.
\newblock {\em Math. Z.}, 212(1):39--60, 1993.
\newblock URL: \href {https://doi.org/10.1007/BF02571640}
{\path{doi:10.1007/BF02571640}}.

\bibitem[Don98]{Do-q-schur}
S.~Donkin.
\newblock {\em The {$q$}-{S}chur algebra}, volume 253 of {\em London
Mathematical Society Lecture Note Series}.
\newblock Cambridge University Press, Cambridge, 1998.
\newblock URL: \href {https://doi.org/10.1017/CBO9780511600708}
{\path{doi:10.1017/CBO9780511600708}}.

\bibitem[DPS98]{DuPaSc-schur-weyl}
J.~Du, B.~Parshall, and L.~Scott.
\newblock Quantum {W}eyl reciprocity and tilting modules.
\newblock {\em Comm. Math. Phys.}, 195(2):321--352, 1998.
\newblock URL: \href {https://doi.org/10.1007/s002200050392}
{\path{doi:10.1007/s002200050392}}.

\bibitem[EH02a]{ErHe-ringel-schur}
K.~Erdmann and A.~Henke.
\newblock On {R}ingel duality for {S}chur algebras.
\newblock {\em Math. Proc. Cambridge Philos. Soc.}, 132(1):97--116, 2002.
\newblock URL: \href {https://doi.org/10.1017/S0305004101005485}
{\path{doi:10.1017/S0305004101005485}}.

\bibitem[EH02b]{ErHe-ringel-schur-symmetric-group}
K.~Erdmann and A.~Henke.
\newblock On {S}chur algebras, {R}ingel duality and symmetric groups.
\newblock {\em J. Pure Appl. Algebra}, 169(2-3):175--199, 2002.
\newblock URL: \href {https://doi.org/10.1016/S0022-4049(01)00071-8}
{\path{doi:10.1016/S0022-4049(01)00071-8}}.

\bibitem[Eli15]{El-ladders-clasps}
B.~Elias.
\newblock Light ladders and clasp conjectures.
\newblock 2015.
\newblock URL: \url{https://arxiv.org/abs/1510.06840}.

\bibitem[Eli17]{El-q-satake}
B.~Elias.
\newblock Quantum {S}atake in type {$A$}. {P}art {I}.
\newblock {\em J. Comb. Algebra}, 1(1):63--125, 2017.
\newblock URL: \url{https://arxiv.org/abs/1403.5570}, \href
{https://doi.org/10.4171/JCA/1-1-4} {\path{doi:10.4171/JCA/1-1-4}}.

\bibitem[EL16]{ElLa-trace-hecke}
B.~Elias and A.D.~Lauda.
\newblock Trace decategorification of the {H}ecke category.
\newblock {\em J. Algebra}, 449:615--634, 2016.
\newblock URL: \url{https://arxiv.org/abs/1504.05267}, \href
{https://doi.org/10.1016/j.jalgebra.2015.11.028}
{\path{doi:10.1016/j.jalgebra.2015.11.028}}.

\bibitem[EGNO15]{EtGeNiOs-tensor-categories}
P.~Etingof, S.~Gelaki, D.~Nikshych, and V.~Ostrik.
\newblock {\em Tensor categories}, volume 205 of {\em Mathematical Surveys and
Monographs}.
\newblock American Mathematical Society, Providence, RI, 2015.
\newblock URL: \href {https://doi.org/10.1090/surv/205}
{\path{doi:10.1090/surv/205}}.

\bibitem[FP18]{FlPe-jwalgebras1}
S.M.~Flores and E.~Peltola.
\newblock Standard modules, radicals, and the valenced {T}emperley--{L}ieb
algebra.
\newblock 2018.
\newblock URL: \url{https://arxiv.org/abs/1801.10003}.

\bibitem[GKPM11]{GeKuPaMi-gen-traces-modified-dimensions}
N.~Geer, J.~Kujawa, and B.~Patureau-Mirand.
\newblock Generalized trace and modified dimension functions on ribbon
categories.
\newblock {\em Selecta Math. (N.S.)}, 17(2):453--504, 2011.
\newblock URL: \url{https://arxiv.org/pdf/1001.0985.pdf}, \href
{https://doi.org/10.1007/s00029-010-0046-7}
{\path{doi:10.1007/s00029-010-0046-7}}.

\bibitem[GKPM13]{GeKuPaMi-ambidextrous-objects}
N.~Geer, J.~Kujawa, and B.~Patureau-Mirand.
\newblock Ambidextrous objects and trace functions for nonsemisimple
categories.
\newblock {\em Proc. Amer. Math. Soc.}, 141(9):2963--2978, 2013.
\newblock URL: \url{https://arxiv.org/abs/1106.4477}, \href
{https://doi.org/10.1090/S0002-9939-2013-11563-7}
{\path{doi:10.1090/S0002-9939-2013-11563-7}}.

\bibitem[GW93]{GoWe-tl-at-roots-of-1}
F.M.~Goodman and H.~Wenzl.
\newblock The {T}emperley--{L}ieb algebra at roots of unity.
\newblock {\em Pacific J. Math.}, 161(2):307--334, 1993.

\bibitem[GL96]{GrLe-cellular}
J.J.~Graham and G.~Lehrer.
\newblock Cellular algebras.
\newblock {\em Invent. Math.}, 123(1):1--34, 1996.
\newblock URL: \href {https://doi.org/10.1007/BF01232365}
{\path{doi:10.1007/BF01232365}}.

\bibitem[HW22]{HeWe-mtrace}
T.~Heidersdorf and H.~Wenzl.
\newblock Generalized negligible morphisms and their tensor ideals.
\newblock {\em Selecta Math. (N.S.)} 28 (2022), no. 2, Paper No. 31, 39 pp. 
\newblock URL: \url{https://arxiv.org/abs/1912.09457v2}, \href
{https://doi.org/10.1007/s00029-021-00749-9}
{\path{doi:10.1007/s00029-021-00749-9}}.

\bibitem[Jan96]{Ja-lectures-qgroups}
J.C.~Jantzen.
\newblock {\em Lectures on quantum groups}, volume~6 of {\em Graduate Studies
in Mathematics}.
\newblock American Mathematical Society, Providence, RI, 1996.

\bibitem[JW17]{JeWi-p-canonical}
L.T.~Jensen and G.~Williamson.
\newblock The {$p$}-canonical basis for {H}ecke algebras.
\newblock In {\em Categorification and higher representation theory}, volume
683 of {\em Contemp. Math.}, pages 333--361. Amer. Math. Soc., Providence,
RI, 2017.
\newblock URL: \url{https://arxiv.org/abs/1510.01556}, \href
{https://doi.org/10.1090/conm/683} {\path{doi:10.1090/conm/683}}.

\bibitem[Jo83]{Jo-index-subfactors}
V.F.R.~Jones.
\newblock Index for subfactors.
\newblock {\em Invent. Math.} 72 (1983), no. 1, 1--25.
\newblock URL: \href
{https://doi.org/10.1007/BF01389127} {\path{doi:10.1007/BF01389127}}.

\bibitem[Kas95]{Ka-quantum-groups}
C.~Kassel.
\newblock {\em Quantum groups}, volume 155 of {\em Graduate Texts in
Mathematics}.
\newblock Springer-Verlag, New York, 1995.
\newblock URL: \href {https://doi.org/10.1007/978-1-4612-0783-2}
{\path{doi:10.1007/978-1-4612-0783-2}}.

\bibitem[KL94]{KaLi-TL-recoupling}
L.H.~Kauffman and S.L.~Lins.
\newblock {\em Temperley--{L}ieb recoupling theory and invariants of
{$3$}-manifolds}, volume 134 of {\em Annals of Mathematics Studies}.
\newblock Princeton University Press, Princeton, NJ, 1994.
\newblock URL: \href {https://doi.org/10.1515/9781400882533}
{\path{doi:10.1515/9781400882533}}.

\bibitem[KS10]{KoSt-fusion}
C.~Korff and C.~Stroppel.
\newblock The sl(n)-{WZNW} fusion ring: a combinatorial construction and a
realisation as quotient of quantum cohomology.
\newblock {\em Adv. Math.}, 225(1):200--268, 2010.
\newblock URL: \url{https://arxiv.org/abs/0909.2347}, \href
{https://doi.org/10.1016/j.aim.2010.02.021}
{\path{doi:10.1016/j.aim.2010.02.021}}.

\bibitem[Lam01]{La-first-non-commutative-rings}
T.Y. Lam.
\newblock {\em A first course in noncommutative rings}, volume 131 of {\em
Graduate Texts in Mathematics}.
\newblock Springer-Verlag, New York, second edition, 2001.
\newblock URL: \href {https://doi.org/10.1007/978-1-4419-8616-0}
{\path{doi:10.1007/978-1-4419-8616-0}}.

\bibitem[MS21]{MaSp-lp-jw}
S.~Martin and R.A.~Spencer.
\newblock {$(\ell,p)$}-{J}ones--{W}enzl idempotents.
\newblock 2021.
\newblock To appear in {\em J. Algebra}.
\newblock URL: \url{https://arxiv.org/abs/2102.08205}.

\bibitem[M{\"u}g03]{Mu-modular}
M.~M{\"u}ger.
\newblock On the structure of modular categories.
\newblock {\em Proc. Lond. Math. Soc. (3)}, 87(2):291--308, 2003.
\newblock URL: \url{https://arxiv.org/abs/math/0201017}, \href
{https://doi.org/10.1112/S0024611503014187}
{\path{doi:10.1112/S0024611503014187}}.

\bibitem[Oli65]{Ol-gpower}
G.~Olive.
\newblock Generalized powers.
\newblock {\em Amer. Math. Monthly}, 72(6):619--627, 1965.
\newblock URL: \href {https://doi.org/10.2307/2313851} {\path{doi:10.2307/2313851}}.

\bibitem[Ost97]{Os-tensor-ideals-tilting}
V.~Ostrik.
\newblock Tensor ideals in the category of tilting modules.
\newblock {\em Transform. Groups}, 2(3):279--287, 1997.
\newblock URL: \url{https://arxiv.org/abs/q-alg/9611033}, \href
{https://doi.org/10.1007/BF01234661} {\path{doi:10.1007/BF01234661}}.

\bibitem[Ost01]{Os-qdim-tiltings}
V.~Ostrik.
\newblock Dimensions of quantized tilting modules.
\newblock {\em Mosc. Math. J.}, 1(1):65--71, 2001.
\newblock URL: \url{https://arxiv.org/abs/math/9902093}, \href
{https://doi.org/10.17323/1609-4514-2001-1-1-65-71}
{\path{doi:10.17323/1609-4514-2001-1-1-65-71}}.

\bibitem[RH03]{RyHa-integral-kempf}
S.~Ryom-Hansen.
\newblock A q-analogue of {K}empf's vanishing theorem.
\newblock {\em Mosc. Math. J.}, 3(1):173--187, 2003.
\newblock URL: \url{https://arxiv.org/abs/0905.0236}, \href
{https://doi.org/10.17323/1609-4514-2003-3-1-173-187}
{\path{doi:10.17323/1609-4514-2003-3-1-173-187}}.

\bibitem[Rin91]{Ri-good-filtrations}
C.M.~Ringel.
\newblock The category of modules with good filtrations over a quasi-hereditary
algebra has almost split sequences.
\newblock {\em Math. Z.}, 208(2):209--223, 1991.
\newblock URL: \href {https://doi.org/10.1007/BF02571521}
{\path{doi:10.1007/BF02571521}}.

\bibitem[RT16]{RoTu-symmetric-howe}
D.E.V.~Rose and D.~Tubbenhauer.
\newblock Symmetric webs, {J}ones--{W}enzl recursions, and {$q$}-{H}owe
duality.
\newblock {\em Int. Math. Res. Not. IMRN}, (17):5249--5290, 2016.
\newblock URL: \url{https://arxiv.org/abs/1501.00915}, \href
{https://doi.org/10.1093/imrn/rnv302} {\path{doi:10.1093/imrn/rnv302}}.

\bibitem[RTW32]{RuTeWe-cup-diagrams}
G.~Rumer, E.~Teller, and H.~Weyl.
\newblock Eine f{\"u}r die {V}alenztheorie geeignete {B}asis der bin{\"a}ren
{V}ektorinvarianten.
\newblock {\em Nachrichten von der Ges. der Wiss. Zu G{\"o}ttingen. Math.-Phys.
Klasse}, 1932:498--504, 1932.
\newblock URL: \url{http://eudml.org/doc/59396}.

\bibitem[RW18]{RiWi-tilting-p-canonical}
S.~Riche and G.~Williamson.
\newblock Tilting modules and the {$p$}-canonical basis.
\newblock {\em Ast{\'e}risque}, (397):ix+184, 2018.
\newblock URL: \url{https://arxiv.org/abs/1512.08296}.

\bibitem[Shi19]{Sh-non-degeneracy}
K.~Shimizu.
\newblock Non-degeneracy conditions for braided finite tensor categories.
\newblock {\em Adv. Math.}, 355:106778, 36, 2019.
\newblock URL: \url{https://arxiv.org/abs/1602.06534}, \href
{https://doi.org/10.1016/j.aim.2019.106778}
{\path{doi:10.1016/j.aim.2019.106778}}.

\bibitem[Soe97]{So-tilting-a}
W.~Soergel.
\newblock Kazhdan--{L}usztig polynomials and a combinatoric[s] for tilting
modules.
\newblock {\em Represent. Theory}, 1:83--114, 1997.
\newblock URL: \href {https://doi.org/10.1090/S1088-4165-97-00021-6}
{\path{doi:10.1090/S1088-4165-97-00021-6}}.

\bibitem[Soe98]{So-tilting-b}
W.~Soergel.
\newblock Character formulas for tilting modules over {K}ac--{M}oody algebras.
\newblock {\em Represent. Theory}, 2:432--448, 1998.
\newblock URL: \href {https://doi.org/10.1090/S1088-4165-98-00057-0}
{\path{doi:10.1090/S1088-4165-98-00057-0}}.

\bibitem[Spe21]{Sp-modular-jwtl-algebra}
R.A. Spencer.
\newblock Modular valenced {T}emperley--{L}ieb algebras.
\newblock 2021.
\newblock URL: \url{https://arxiv.org/abs/2108.10011}.

\bibitem[Spe20]{Sp-modular-tl-algebra}
R.A. Spencer.
\newblock The modular {T}emperley--{L}ieb algebra.
\newblock 2020.
\newblock URL: \url{https://arxiv.org/abs/2011.01328}.

\bibitem[Str97]{St-diplom}
C.~Stroppel.
\newblock Untersuchungen zu den parabolischen {K}azhdan--{L}usztig-{P}olynomen
f{\"u}r affine {W}eyl-{G}ruppen.
\newblock 1997.
\newblock Diploma Thesis (1997), 74 pages (German).
\newblock URL:
\url{http://www.math.uni-bonn.de/ag/stroppel/arbeit\_Stroppel.pdf}.

\bibitem[TW21]{TuWe-quiver}
D.~Tubbenhauer and P.~Wedrich.
\newblock Quivers for {SL2} tilting modules.
\newblock {\em Represent. Theory} 25 (2021), 440--480.
\newblock URL: \url{https://arxiv.org/abs/1907.11560}, \href
{https://doi.org/10.1090/ert/569}
{\path{doi:10.1090/ert/569}}.

\bibitem[TW20]{TuWe-center}
D.~Tubbenhauer and P.~Wedrich.
\newblock The center of {SL2} tilting modules.
\newblock {\em  Glasg. Math. J.} 64 (2022), no. 1, 165--184.
\newblock URL: \url{https://arxiv.org/abs/2004.10146}, \href
{https://doi.org/10.1017/S001708952100001X}
{\path{doi:10.1017/S001708952100001X}}.

\bibitem[We87]{We-projectors}
H.~Wenzl.
\newblock On sequences of projections.
\newblock {\em C. R. Math. Rep. Acad. Sci. Canada} 9 (1987), no. 1, 5--9.

\bibitem[Wes09]{We-tensors-cellular-categories}
B.W.~Westbury.
\newblock Invariant tensors and cellular categories.
\newblock {\em J. Algebra}, 321(11):3563--3567, 2009.
\newblock URL: \url{https://arxiv.org/abs/0806.4045}, \href
{https://doi.org/10.1016/j.jalgebra.2008.07.004}
{\path{doi:10.1016/j.jalgebra.2008.07.004}}.

\end{thebibliography}
\end{document}